\documentclass[12pt,a4paper]{article}
\usepackage[left=2cm,right=2cm,top=2cm,bottom=2cm]{geometry}

\usepackage[utf8]{inputenc}
\usepackage{amsmath,amsfonts,amssymb}
\usepackage{graphicx}
\usepackage{textcomp}
\usepackage{subcaption}
\usepackage{amsthm}
\usepackage{setspace}
\usepackage[english]{babel}
\usepackage{dsfont}
\usepackage{xcolor}

\newcommand{\R}{\mathbb{R}}
\newcommand{\N}{\mathbb{N}}
\newcommand{\Y}{\mathcal{Y}}

\newcommand{\demi}{\frac{1}{2}}
\newcommand{\eps}{\varepsilon}
\newcommand{\dps}{\displaystyle}
\newcommand{\VKL}{\mathcal{V}_{\rm KL}}

\newcommand{\M}{\mathcal{M}}
\newcommand{\GKL}{\mathcal{G}_{\rm KL}}

\let\div\relax
\DeclareMathOperator{\div}{div}
 
\DeclareMathOperator{\m}{m}

\let\P\relax
\newcommand{\P}{\mathcal{P}}
\let\O\relax
\newcommand{\O}{\mathcal{O}}
\newcommand{\E}{\mathcal{E}}

\newcommand{\Id}{\text{Id}}

\newtheorem{theorem}{Theorem}[section]
\newtheorem{lemma}[theorem]{Lemma}
\newtheorem{remark}[theorem]{Remark}

\newtheorem{proposition-definition}[theorem]{Proposition-Definition}

\title{Convergence of two-scale expansions for elastic heterogeneous plates}
\author{Virginie Ehrlacher$^{1,3}$, Arthur Leb\'ee$^2$, Fr\'ed\'eric Legoll$^{2,3}$ and Adrien Lesage$^{1,2,3}$
\\
{\footnotesize $^1$ CERMICS, ENPC, Institut Polytechnique de Paris, Marne-la-Vall\'ee, France}
\\
{\footnotesize $^2$ Navier, ENPC, Institut Polytechnique de Paris, Univ Gustave Eiffel, CNRS, Marne-la-Vall\'ee, France}
\\
{\footnotesize $^3$ MATHERIALS project-team, Inria, Paris, France}
\\
{\footnotesize \tt \{virginie.ehrlacher,arthur.lebee,frederic.legoll\}@enpc.fr, adr.lesage@gmail.com}
}

\date{\today}

\begin{document}

\maketitle

\begin{abstract}

%
The aim of this article is to prove strong convergence results on the difference between the solution to highly oscillatory problems posed in thin domains and its two-scale expansion. We first consider the case of the linear diffusion equation and establish such results in arbitrary dimensions, by using a straightforward adaptation of the classical arguments used for the homogenization of highly oscillatory problems posed on fixed (non-thin) domains. We next consider the linear elasticity problem, which raises challenging difficulties in its full generality. Under some classical assumptions on the symmetries of the elasticity tensor, the problem can be split into two independent problems, the {\em membrane} problem and the {\em bending} problem. Focusing on two-dimensional problems, we show that the membrane case can actually be addressed using a careful adaptation of classical arguments. In the bending case, the scheme of the proof used in the membrane and diffusion cases can however not be straightforwardly adapted. In that bending case, we establish the desired strong convergence results by using a different strategy of proof, which seems, up to our knowledge, to be new.
%

%
\end{abstract}

\section{Introduction}




In this article, we consider highly oscillatory problems posed in thin domains of $\mathbb{R}^d$. These problems typically read as
\begin{equation}\label{eq:diff_eq}
- \div \left( \mathcal{A}^\eps \nabla \widetilde{u}^\eps \right) = f \quad \mbox{in} \quad \Omega^\eps,
\end{equation}
where the matrix $\mathcal{A}^\eps$ (which is -- uniformly in $\eps$ -- bounded from below and from above, to ensure ellipticity of the problem and thus its well-posedness) varies at the small characteristic length-scale $\eps$. We concurrently consider two types of PDEs: (i) the diffusion equation~\eqref{eq:diff_eq}, where $\mathcal{A}^\eps$ is a $\mathbb{R}^{d \times d}$ matrix and $\widetilde{u}^\eps$ is scalar-valued and (ii) the linear elasticity problem, which can again be written in the form~\eqref{eq:diff_eq} where $\mathcal{A}^\eps$ is a fourth-order tensor and $\widetilde{u}^\eps$ is vector-valued.

The domain $\Omega^\eps$ is thin, in the sense that its width $\eps$ in the $d$-direction is small. A typical example is when
\begin{equation} \label{eq:thin_domain}
\Omega^\eps = \omega \times \left( -\eps/2, \eps/2\right),
\end{equation}
where $\omega$ is a bounded open subset of $\mathbb{R}^{d-1}$. Note that the width of the domain $\Omega^\eps$ is here equal to the characteristic length-scale of $\mathcal{A}^\eps$. Other choices could have been made, as discussed in Remark~\ref{rem:regime} below. For simplicity, we have assumed in~\eqref{eq:diff_eq} that the right-hand side $f$ does not depend on $\eps$. More general cases are considered below. Of course, Problem~\eqref{eq:diff_eq} should be complemented by appropriate boundary conditions, that are also made precise below. 

\medskip

The question we consider here is to identify the limit of $\widetilde{u}^\eps$ when $\eps$ goes to zero. In the case when the domain $\Omega$ on which the equation is posed actually does not depend on $\eps$, this is a very classical question of homogenization theory (see e.g. the classical textbooks~\cite{bensoussan2011asymptotic,jikov}, \cite{livre_blanc_lebris}, \cite[Chapter~1]{allaire2012shape} and also~\cite{engquist2008asymptotic,le2005systemes}). For a simple diffusive equation such as~\eqref{eq:diff_eq}, and assuming for instance homogeneous Dirichlet boundary conditions on $\partial \Omega$ and periodicity of the matrix $\mathcal{A}^\eps$ (that is $\mathcal{A}^\eps = \mathcal{A}_{\rm per}(\cdot/\epsilon)$ for a fixed periodic matrix $\mathcal{A}_{\rm per}$), it is well-known that $\widetilde{u}^\eps \in H^1_0(\Omega)$ converges to some $\widetilde{u}^\star \in H^1_0(\Omega)$, solution to a homogenized problem of the same form with a diffusion coefficient $\mathcal{A}^\star$ which is constant. The value of $\mathcal{A}^\star$ can easily be computed using the so-called {\em corrector functions}, which are solutions of some auxiliary problems posed over the unit periodic cell. The convergence of $\widetilde{u}^\eps$ to $\widetilde{u}^\star$ is strong in $L^2(\Omega)$ and weak in $H^1(\Omega)$. It is furthermore possible to introduce a two-scale expansion $\widetilde{u}^{\eps,1}$, explicitly built using the homogenized solution $\widetilde{u}^\star$ and the corrector functions, so that the difference $\widetilde{u}^\eps - \widetilde{u}^{\eps,1}$ {\em strongly} converges to 0 in $H^1(\Omega)$ when $\eps \to 0$. Similar results have been obtained for many different equations (besides the simple diffusion equation), and in particular for linear elasticity problems (see e.g.~\cite[Chap.~10]{CD}) of specific interest in this article. Likewise, many other settings have been considered besides the periodic setting (including the quasi-periodic setting, the random stationary setting, \dots).

\medskip

In this article, we consider the situation where the domain $\Omega^\eps$ on which the oscillatory problem is posed actually depends on $\eps$ and is given by~\eqref{eq:thin_domain}. In other words, we study problems posed on plates composed of an heterogeneous medium, where the typical size of the heterogeneities is of the same order as the (small) thickness of the plate (see Figure~\ref{fig:plate1}). For such problems, the homogenized limit of~\eqref{eq:diff_eq}--\eqref{eq:thin_domain} has been identified (both for the diffusion equation and the linear elasticity problem) in various cases, including the stratified case~\cite{gustafsson2003non,gustafsson2006compensated,marohnic2016non} (that is when $\mathcal{A}^\eps$ only depends on $x_d \in (-\eps/2,\eps/2)$), and the case of periodic heterogeneities in the transverse, in-plane directions~\cite{caillerieDiffusion,caillerieElasticite} (that is when $\mathcal{A}^\eps$ is $\eps \mathbb{Z}^{d-1}$-periodic with respect to $(x_1, \dots, x_{d-1}) \in \omega$), to name but a few. We also refer to~\cite{hornung2018stochastic,marohnic2015homogenization,tomasz2000plates} for recent homogenization results on plates with more general heterogeneities. Results have also been obtained for nonlinear problems: we refer e.g. to~\cite{hornung2014derivation,velvcic2015derivation} for nonlinear elasticity models. In all these works, the weak convergence of $\widetilde{u}^\eps$ towards the solution $\widetilde{u}^\star$ to a homogenized problem has been established.

\begin{figure}[htbp]
  \begin{center}
    \includegraphics[scale=0.45,angle=-90]{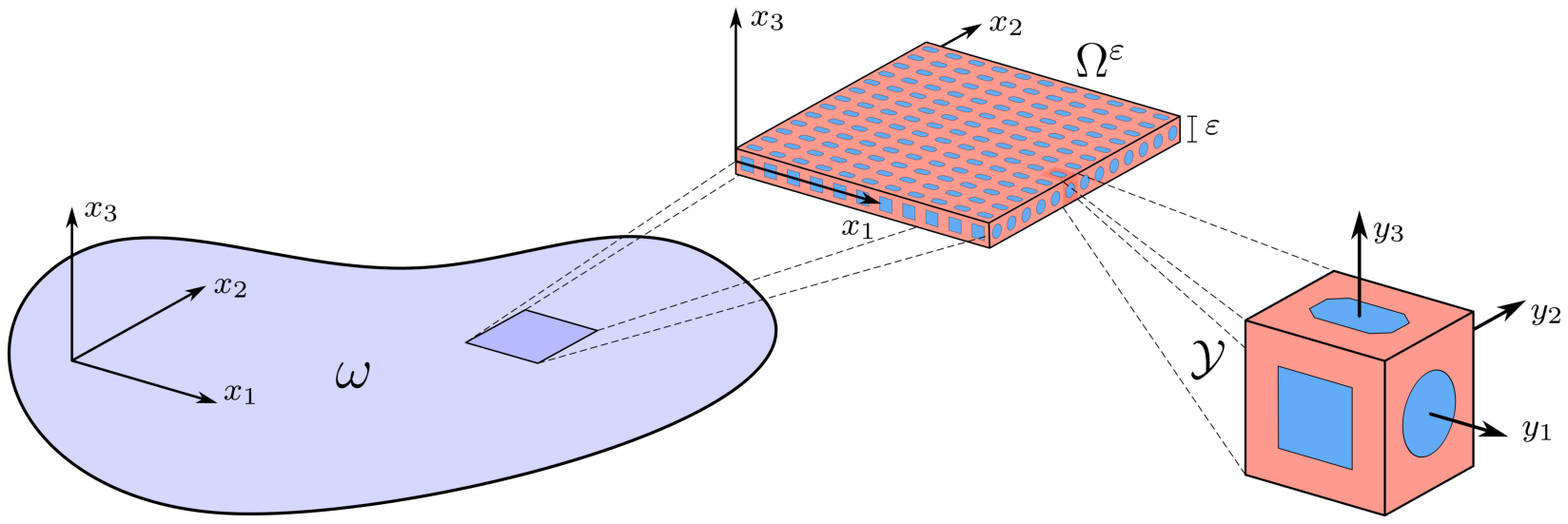}
  \end{center}
  \caption{The plate and its microstructure for $d=3$. \label{fig:plate1}}
\end{figure}

%

Following the general path of homogenization theory, the next step is to obtain a strong convergence (say in $H^1$), that is to build a relevant two-scale expansion $\widetilde{u}^{\eps,1}$ and to prove that the difference $\widetilde{u}^\eps - \widetilde{u}^{\eps,1}$ {\em strongly} converges to 0 in $H^1$ when $\eps \to 0$ (of course, since the domain $\Omega^\eps$ on which the oscillatory equation is posed depends itself on $\eps$, the domain on which the $H^1$ norm is considered should be carefully chosen). Surprisingly, this question has been addressed in very few cases, at least up to our knowledge. In that direction, strong convergence results have been obtained for {\em homogeneous} plates (where the parameter $\eps$ only encodes the small thickness of the computational domain) in~\cite{ciarlet1979justification,dauge,destuynder1981comparaison}. In this article, we focus on the case where the plate has periodic heterogeneities in its in-plane directions, which is a setting similar to the one considered in~\cite{caillerieDiffusion,caillerieElasticite} and for which weak convergence results have been established.

At first sight, it may be thought that such strong convergence results may easily be obtained by extending standard arguments used in the classical case (i.e. when the domain on which the equation is posed does not depend on $\eps$). This is indeed the case for the (scalar-valued) diffusion equation. However, the analysis that we present here shows that this is not always the case for the linear elasticity (vector-valued) problem, and that additional difficulties arise.

\medskip

As the title of this article suggests, we are mainly interested here in the linear elasticity problem. However, to proceed in a pedagogical manner, we first consider the case of the diffusion equation in Section~\ref{sec:diff}. In that first case, strong convergence results (such as Theorem~\ref{thconvforte_diffusion}, our main result in that case) can indeed be obtained, in arbitrary dimensions, by using standard arguments. The situation turns out to be different in the case of linear elasticity, which we address in Section~\ref{sec:elasticity}. We assume there that the mechanical composition of the heterogeneous material is symmetric with respect to its medium plane, which corresponds to assuming that the components of the elasticity tensor are either odd or even functions with respect to $x_d$ (see~\eqref{hyp:symA} below). This assumption is classical in the literature, and is in particular satisfied by isotropic materials. Under this assumption, we distinguish two situations, depending on the symmetries of the loading (i.e. the function $f$ in the right-hand side of the PDE) imposed on the plate: the {\em membrane case} and the {\em bending case}. The membrane case can be analyzed using a careful adaptation of standard arguments used in classical homogenization theory (our main result in that case is Theorem~\ref{thconvforte2}). However, in the bending case, the standard proof does not go through. The analysis of this case turns out to actually require specific arguments, inspired by some ideas present in~\cite{destuynder1981comparaison} to handle homogeneous plates, but the adaptation of which to the heterogeneous case is far from immediate. This alternative strategy of proof is thus new, at least up to our knowledge, and it culminates in Theorem~\ref{thconvforte3_a}, our main result in the bending case.



In both the membrane and the bending cases, we establish the desired strong convergence results only in dimension $d=2$. The restriction of our analysis to the two-dimensional case stems from the fact that we are only able to establish a technical result needed in the proof (namely Lemma~\ref{lemma:minZ2}, a result which generalizes to linear elasticity problems posed on thin domains a well-known result in homogenization theory, see e.g.~\cite[p.~27]{jikov}) when $d=2$. We leave the extension of this technical result in general dimensions to future works.

\medskip

Obtaining a strong convergence result as discussed above is of course interesting from the theoretical viewpoint, since it provides an accurate description of the solution to the oscillatory problem in its natural energy norm. It is also helpful for proving numerical analysis results. In particular, this type of results is a key ingredient to prove error bounds for the Multiscale Finite Element Method (MsFEM). This numerical approach, which is dedicated to approximating the solution to highly oscillatory problems of the type~\eqref{eq:diff_eq} (for a small, but not asymptotically small, scale $\eps$), proceeds by performing a variational approximation of~\eqref{eq:diff_eq} using pre-computed basis functions that are {\em adapted} to the problem (we refer to~\cite{efendiev2009multiscale,livre_blanc_lebris} and references therein, and also to~\cite{lebris_legoll_jcp}). They are indeed solutions to local problems defined using the same differential operator as the problem of interest. Using these problem-specific basis functions, the MsFEM approach yields an accurate approximation of the oscillatory solution using only a limited number of degrees of freedom, in contrast to standard Finite Element approaches (which would request taking a mesh size smaller than $\eps$ in order to provide an accurate solution). In addition, the MsFEM approach is applicable in general situations, and is not limited to the case when the highly oscillatory coefficient of the equation is periodic. In our companion article~\cite{MsFEM-ELLL}, we introduce several variants of the MsFEM approach for the case of elastic heterogeneous plates, and establish error bounds for these. The strong convergence results shown here are pivotal for their numerical analysis. 

%
%
%

\medskip

The article is organized as follows. In Section~\ref{sec:diff}, we consider the diffusion case (i.e. the scalar version of the problem). After introducing the plate problem (in Section~\ref{sec:def_plate_pb_diff}) and recalling the corresponding weak convergence result (namely, Theorem~\ref{limitdiff} in Section~\ref{sec:weak_diff}), we turn in Section~\ref{sec:strong_diff} to establishing a strong convergence result, namely Theorem~\ref{thconvforte_diffusion}. In Section~\ref{sec:elasticity}, we investigate the elasticity case (i.e. the vector version of the problem), and first introduce the plate problem (in Section~\ref{sec:def_plate_pb_elas}) and then recall the corresponding weak convergence result (namely, Theorem~\ref{limitel} in Section~\ref{sec:weakel}). We then explain how to split the problem into a membrane problem and a bending problem in Section~\ref{sec:prelim_el}. Section~\ref{sec:membrane} is devoted to establishing a strong convergence result, namely Theorem~\ref{thconvforte2}, in the membrane case. The bending case is eventually considered in Section~\ref{sec:bending}, which culminates with Theorem~\ref{thconvforte3_a} stating a strong convergence result in that case. We collect in the appendices several technical results, concerning the rescaling of plate problems (in Appendix~\ref{app:scaling}), $H_{\rm div}$ spaces (in Appendix~\ref{app:Hdiv}) and Korn inequalities (in Appendix~\ref{app:korn}). For the sake of completeness, we provide a proof of the homogenization limits (i.e. the weak convergence results) in Appendix~\ref{app:diff}.






\section{The diffusion case}\label{sec:diff}

Let $\omega \subset \mathbb{R}^{d-1}$ be a bounded and smooth domain (where $d \in \N^\star$ is the ambient dimension) and let $\dps \Omega := \omega \times \left( -\frac{1}{2},\frac{1}{2} \right)$. For a small parameter $0 < \eps \leq 1$, we introduce $\dps \Omega^\eps := \omega \times \left( -\frac{\eps}{2},\frac{\eps}{2} \right)$. The domain $\Omega^\eps$ is called a plate because $\eps$ is small compared to the characteristic size of $\omega$ (see Figure~\ref{fig:plate2}). We also denote by $n$ (respectively $n^\eps$) the outward normal unit vector to $\partial \Omega$ (respectively $\partial \Omega^\eps$). 

\smallskip

\begin{figure}[htbp]
  \begin{center}
    \includegraphics[scale=0.6]{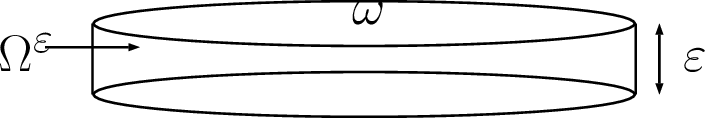}
    \end{center}
  \caption{Schematic representation of the plate $\Omega^\eps$. \label{fig:plate2}}
\end{figure}

Let $(e_i)_{1 \leq i \leq d}$ be the canonical basis of $\mathbb{R}^d$. For any $x=(x_i)_{1 \leq i \leq d} \in \mathbb{R}^d$, we set $x' := (x_i)_{1 \leq i \leq d-1} \in \R^{d-1}$. For any $M := (M_{ij})_{1 \leq i,j \leq d} \in \mathbb{R}^{d \times d}$, we set $M':= (M_{ij})_{1 \leq i,j \leq d-1} \in \mathbb{R}^{(d-1) \times (d-1)}$. The set of $d\times d$ symmetric matrices is denoted by $\mathbb{R}^{d\times d}_s$ and $c_-, c_+ >0$ are some fixed positive constants. We also define the periodic cells
\begin{equation} \label{eq:def_Y_Y}
Y := (0,1)^{d-1} \qquad \text{and} \qquad \mathcal{Y} := Y \times \left( -\frac{1}{2},\frac{1}{2} \right).
\end{equation}
For any $\dps f: \R^{d-1} \times \left(-\demi,\demi \right) \rightarrow \R^p$ and any $x' \in \R^{d-1}$, we denote
\begin{equation} \label{eq:def_m}
\m(f)(x'):= \int_{-1/2}^{1/2} f(x',x_d) \, dx_d
\end{equation}
the mean of $f$ over its last variable.

Throughout the article, we use the Einstein summation convention. Latin letters are used for indices running between $1$ and $d$ and greek letters for indices running between $1$ and $d-1$. Also, in all statements and proofs, $C$ denotes an arbitrary constant independent of $\eps$ and $\omega$.

\subsection{Definition of the plate problem} \label{sec:def_plate_pb_diff}

The notations we introduce now are specific to the diffusion case (i.e. to the current Section~\ref{sec:diff}). We refer to Section~\ref{sec:def_plate_pb_elas} for similar notations for the elasticity case. We denote by $\mathcal{M} \subset \mathbb{R}^{d\times d}_s$ the set of symmetric matrices $M$ such that
$$
\forall \xi \in \mathbb{R}^d, \quad |M \xi | \leq c_+ \, |\xi| \quad \text{and} \quad \xi^T M \xi \geq c_- \, |\xi|^2.
$$
Let $\dps A: \mathbb{R}^{d-1} \times \left( -\frac{1}{2}, \frac{1}{2} \right) \to \mathcal{M}$ be a matrix-valued field such that, for any $\dps x_d \in \left( -\frac{1}{2},\frac{1}{2} \right)$, the function $x' \in \mathbb{R}^{d-1} \mapsto A(x',x_d)$ is $Y$-periodic. For any $x = (x',x_d) \in \Omega$, we set
\begin{equation} \label{eq:utile}
  A^\eps(x) = A \left( \frac{x'}{\eps}, x_d \right).
\end{equation}
In addition, we define $\mathcal{A}^\eps$ by
$$
\forall x \in \Omega^\eps, \quad \mathcal{A}^\eps(x):=A^\eps\left(x', \frac{x_d}{\eps}\right) = A\left( \frac{x'}{\eps}, \frac{x_d}{\eps} \right).
$$
We set
$$
V^\eps := \left\{ v \in H^1(\Omega^\eps), \quad v=0 \text{ on } \partial \omega \times \left( -\frac{\eps}{2},\frac{\eps}{2} \right) \right\}.
$$
A function in $V^\eps$ thus vanishes on the lateral boundary of $\Omega^\eps$ (see Figure~\ref{fig:plate2}).

\medskip

For any $\eps >0$, let $\widetilde{f}^\eps \in L^2(\Omega^\eps)$, $g^\eps \in H^1(\omega)$ and $h^\eps_\pm \in L^2(\omega)$. We consider the following diffusion problem: find $\widetilde{u}^\eps \in V^\eps$ such that
\begin{equation} \label{pb:diff1}
\left\{
\begin{aligned}
-\div (\mathcal{A}^\eps \nabla \widetilde{u}^\eps) &= \widetilde{f}^\eps + \div(\mathcal{A}^\eps \nabla g^\eps) \quad \text{in $\Omega^\eps$},
\\
\mathcal{A}^\eps \nabla \widetilde{u}^\eps \cdot n^\eps &= \eps \, h^\eps_\pm - \mathcal{A}^\eps \nabla g^\eps \cdot n^\eps \quad \text{on $\omega \times \left\{ \pm \frac{\eps}{2} \right\}$}.
\end{aligned}
\right.
\end{equation}
In~\eqref{pb:diff1}, $\widetilde{f}^\eps$ is the load imposed in $\Omega^\eps$. The function $g^\eps$ is inserted as a possible extension of a non-trivial Dirichlet boundary condition (so that $\widetilde{u}^\eps + g^\eps$ does not necessarily vanish on $\partial \omega \times \left( -\eps/2,\eps/2 \right)$). A motivation for considering this general case is the fact that we use these strong convergence results in our companion article~\cite{MsFEM-ELLL} for the numerical analysis of MsFEM approaches, where such general Dirichlet boundary conditions appear. Note that $g^\eps$ does not depend on $x_d$. The function $h^\eps_\pm$ plays the role of a Neumann boundary condition on the top and bottom faces of the plate $\Omega^\eps$.

\begin{remark} \label{rem:regime}
  As is obvious on~\eqref{pb:diff1}, the thickness (denoted $\eta$ in this remark) of the plate is equal to the characteristic length-scale $\eps$ of variations of $\mathcal{A}^\eps$ in the in-plane directions. As pointed out above, different regimes for $\eta$ vs $\eps$ have been considered in the literature (e.g. sending $\eta$ to 0 before or after sending $\eps$ to 0). We refer e.g. to~\cite{caillerieDiffusion,griso2018homogenization,gustafsson2003non} for such studies. 
  The case of plates of rapidly varying thickness has also been considered in the literature, we refer e.g. to~\cite{kohn_vogelius86}.
\end{remark}

Since our goal is to study the asymptotic behaviour of $\widetilde{u}^\eps$ when $\eps$ goes to 0, it is convenient to rescale the problem and recast~\eqref{pb:diff1} as a problem set on $\Omega$, a domain independent of $\eps$ (see Figure~\ref{fig:rescaling}). This is of course a standard step when studying plate problems (see e.g.~\cite[eq.~(3.1)]{caillerieDiffusion}).

\smallskip

\begin{figure}[!htb]
\centering 
\includegraphics[scale=0.8]{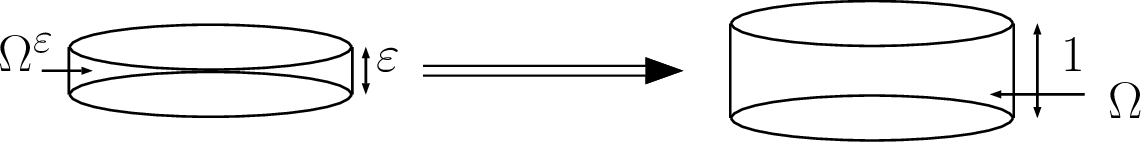}   
\caption{Rescaling of the domain \label{fig:rescaling}}
\end{figure}

\medskip

For any $u \in {\mathcal{D}}'(\R^d)$ and any $T \in ({\mathcal{D}}'(\R^d))^d$, we set
\begin{equation} \label{eq:def_nabla_prime}
\nabla' u = \partial_\alpha u \, e_\alpha = \sum_{\alpha = 1}^{d-1} \partial_\alpha u \, e_\alpha.
\end{equation}
We next introduce scaled operators as
$$
\nabla^\eps u := \partial_\alpha u \, e_\alpha + \frac{1}{\eps} \, \partial_d u \, e_d = \sum_{\alpha = 1}^{d-1} \partial_\alpha u \, e_\alpha + \frac{1}{\eps} \, \partial_d u \, e_d
$$
and
$$
\div^\eps T := \partial_\alpha T_\alpha+ \frac{1}{\eps} \, \partial_d T_d = \sum_{\alpha = 1}^{d-1} \partial_\alpha T_\alpha+ \frac{1}{\eps} \, \partial_d T_d.
$$
By construction, for any $u \in {\mathcal{D}}(\Omega)$ and any $T \in ({\mathcal{D}}(\Omega))^d$, we have $\dps \int_\Omega T \cdot \nabla^\eps u = - \int_\Omega u \, \div^\eps T$. We introduce
\begin{equation} \label{def:V}
V := \left\{ v \in H^1(\Omega), \quad v=0 \text{ on } \partial \omega \times \left( -\frac{1}{2},\frac{1}{2} \right) \right\}.
\end{equation}
It can be easily checked (see Appendix~\ref{app:scaling_scalaire} for details) that problem~\eqref{pb:diff1} is equivalent to finding $u^\eps \in V$ such that
\begin{equation} \label{pb:diff2}
\left\{
\begin{aligned}
  -\div^\eps (A^\eps \nabla^\eps u^\eps) &= f^\eps + \div^\eps(A^\eps \nabla^\eps g^\eps) \quad \text{in $\Omega$},
  \\
  A^\eps \nabla^\eps u^\eps \cdot n &= \eps \, h^\eps_\pm - A^\eps \nabla^\eps g^\eps \cdot n \quad \text{on $\omega \times \left\{ \pm \frac{1}{2} \right\}$}, 
\end{aligned}
\right.
\end{equation}
where $A^\eps$ is given by~\eqref{eq:utile} and where, for any $x = (x',x_d) \in \Omega$,
\begin{equation} \label{eq:scaling_u_scal}
u^\eps(x) = \widetilde{u}^\eps(x',\eps \, x_d) \qquad \text{and} \qquad f^\eps(x) = \widetilde{f}^\eps(x',\eps \, x_d).
\end{equation}
Note that, since $g^\eps$ does not depend on $x_d$, the same function appears in~\eqref{pb:diff1} and~\eqref{pb:diff2}. The scaled operator $\nabla^\eps$ is defined in such a way that $(\nabla^\eps u^\eps)(x) = (\nabla \widetilde{u}^\eps)(x',\eps \, x_d)$.

\medskip

The variational formulation of~\eqref{pb:diff2} reads as follows: find $u^\eps \in V$ such that
\begin{equation} \label{var:diff}
\forall v \in V, \quad a^\eps(u^\eps,v) = b^\eps(v), 
\end{equation}
where, for all $u,v \in V$, 
\begin{equation}\label{eq:aeps}
a^\eps(u,v) := \int_\Omega A^\eps \nabla^\eps u \cdot \nabla^\eps v
\end{equation}
and
\begin{equation}\label{eq:beps}
b^\eps(v) := \int_\Omega f^\eps \, v - \int_\Omega A^\eps \nabla^\eps g^\eps \cdot \nabla^\eps v + \int_\omega h^\eps_+ \, v\left(\cdot,\frac{1}{2}\right) + \int_\omega h^\eps_- \, v\left(\cdot,-\frac{1}{2}\right).
\end{equation}
The coercivity of the bilinear form $a^\eps$ (with a constant which is uniform with respect to $\eps$) is an easy consequence of the following Poincaré inequality and of the fact that $\| \nabla u \|_{L^2(\Omega)} \leq \| \nabla^\eps u \|_{L^2(\Omega)}$ for any $u \in V$ (since we have assumed that $0 < \eps \leq 1$). Using the Lax-Milgram theorem, we deduce that there exists a unique solution to~\eqref{var:diff}. 

\begin{lemma} \label{lemma:poincare}
Let $V$ be defined by~\eqref{def:V}. There exists a constant $C(\Omega)>0$ such that
$$
\forall u \in V, \quad \| u \|_{L^2(\Omega)} \leq C(\Omega) \, \| \nabla u \|_{L^2(\Omega)}.
$$
\end{lemma}

\subsection{Homogenization result: weak convergence} \label{sec:weak_diff}


We first establish a priori bounds on $u^\eps$. Taking $v = u^\eps$ in~\eqref{var:diff}, using the fact that $g^\eps$ is independent of $x_d$ (so that $\nabla^\eps g^\eps = \nabla g^\eps$) for the second term of $b^\eps$ and a trace inequality for the last term, we get
$$
c_- \| \nabla^\eps u^\eps \|^2_{L^2(\Omega)} \leq \| f^\eps \|_{L^2(\Omega)} \| u^\eps \|_{L^2(\Omega)} + c_+ \| g^\eps \|_{H^1(\omega)} \| \nabla^\eps u^\eps \|_{L^2(\Omega)} + C \|h^\eps_\pm \|_{L^2(\omega)} \|u^\eps \|_{H^1(\Omega)}.
$$
Using Lemma~\ref{lemma:poincare} and next the fact that $\| \nabla u^\eps \|_{L^2(\Omega)} \leq \| \nabla^\eps u^\eps \|_{L^2(\Omega)}$, we deduce that
\begin{align*}
  c_- \| \nabla^\eps u^\eps \|^2_{L^2(\Omega)}
  &\leq
  C \| f^\eps \|_{L^2(\Omega)} \| \nabla u^\eps \|_{L^2(\Omega)} + c_+ \| g^\eps \|_{H^1(\omega)} \| \nabla^\eps u^\eps \|_{L^2(\Omega)} + C \|h^\eps_\pm \|_{L^2(\omega)} \| \nabla u^\eps \|_{L^2(\Omega)}
  \\
  &\leq
  C \| f^\eps \|_{L^2(\Omega)} \| \nabla^\eps u^\eps \|_{L^2(\Omega)} + c_+ \| g^\eps \|_{H^1(\omega)} \| \nabla^\eps u^\eps \|_{L^2(\Omega)} + C \|h^\eps_\pm \|_{L^2(\omega)} \| \nabla^\eps u^\eps \|_{L^2(\Omega)},
\end{align*}
hence
\begin{equation} \label{bornesigma}
  \| \nabla^\eps u^\eps \|_{L^2(\Omega)} \leq C \left( \|f^\eps \|_{L^2(\Omega)} + \|g^\eps\|_{H^1(\omega)} + \|h^\eps_\pm \|_{L^2(\omega)} \right),
\end{equation}
for some constant $C$ independent of $\eps$. Using again Lemma~\ref{lemma:poincare}, this implies that
\begin{equation} \label{borneu}
  \| u^\eps \|_{H^1(\Omega)} \leq C \left( \|f^\eps \|_{L^2(\Omega)} + \|g^\eps\|_{H^1(\omega)} + \|h^\eps_\pm \|_{L^2(\omega)} \right),
\end{equation}
for some constant $C$ independent of $\eps$.

\medskip

From now on in Section~\ref{sec:diff}, we assume that there exist $f \in L^2(\Omega)$, $g \in H^1(\omega)$ and $h_\pm \in L^2(\omega)$ such that, for all $\eps >0$, 
\begin{equation} \label{eq:inde_eps}
f^\eps = f, \qquad g^\eps = g \qquad \mbox{ and } \qquad h_\pm^\eps = h_\pm.
\end{equation}

\begin{remark} \label{rem:hyp_f_ind_eps}
Note that, in Theorem~\ref{limitdiff} below, it would be sufficient to assume that the sequence $(f^\eps)_{\eps >0}$ (respectively $(h_\pm^\eps)_{\eps >0}$) weakly converges to $f$ (respectively $h_\pm$) in $L^2(\Omega)$ (respectively in $L^2(\omega)$).
\end{remark}

Under Assumption~\eqref{eq:inde_eps}, we infer from~\eqref{borneu} that there exists $u^\star \in V$ such that, up to the extraction of a subsequence,
$$
u^\eps \underset{\eps \rightarrow 0}{\rightharpoonup} u^\star \text{ weakly in } {H^1(\Omega)}.
$$

\medskip

We now recall the well-known homogenization result of~\cite{caillerieDiffusion}. To that aim, we introduce the corrector functions associated to the problem. Let
$$
\mathcal{W}(\mathcal{Y}) := \left\{ v \in H^1_{\rm loc}\left(\mathbb{R}^{d-1} \times \left( -\frac{1}{2},\frac{1}{2} \right)\right), \forall z \in \left( -\frac{1}{2},\frac{1}{2} \right), \, v(\cdot,z) \text{ is $Y$-periodic and } \int_{\mathcal{Y}} v = 0 \right\}.
$$ 
For all $1\leq \alpha \leq d-1$, let $w^\alpha \in \mathcal{W}(\mathcal{Y})$ be the unique solution to
\begin{equation} \label{prcor1} 
\forall v\in\mathcal{W}(\mathcal{Y}), \qquad \int_{\mathcal{Y}} A (\nabla w^\alpha + e_\alpha) \cdot \nabla v = 0.
\end{equation}
The function $w^\alpha \in \mathcal{W}(\mathcal{Y})$ is equivalently the unique solution in $\mathcal{W}(\mathcal{Y})$ to
\begin{equation} \label{prcor2}
\left\{
\begin{aligned}
-\div A(\nabla w^\alpha + e_\alpha) &= 0 \quad \text{ in } \mathcal{Y}, \\
A(\nabla w^\alpha + e_\alpha) \cdot e_d &= 0 \quad \text{ on } \mathcal{Y}_+ \cup \mathcal{Y}_-,
\end{aligned}
\right.
\end{equation}
where $\dps \mathcal{Y}_\pm := Y \times \left\{ \pm \frac{1}{2} \right\}$ is the top (resp. bottom) face of $\mathcal{Y}$. We are now in position to state the homogenization theorem for the plate problem (recall that the notations $\nabla'$ and $\m(\cdot)$ used below are defined by~\eqref{eq:def_nabla_prime} and~\eqref{eq:def_m}, respectively). 

\begin{theorem}[from Theorem~8.1 of~\cite{caillerieDiffusion}] \label{limitdiff}
Under Assumption~\eqref{eq:inde_eps}, the sequence $(u^\eps)_{\eps >0}$ of solutions to~\eqref{var:diff} weakly converges to $u^\star$ in $H^1(\Omega)$ as $\eps$ goes to $0$, where the function $u^\star$ does not depend on $x_d$, belongs to $H^1_0(\omega)$ and is the unique solution to
\begin{equation} \label{eq:fv_star}
\forall \phi \in H^1_0(\omega), \quad \int_\omega A^\star \nabla' u^\star \cdot \nabla'\phi = \int_\omega \big( \m(f) + h_+ + h_- \big) \, \phi - \int_\omega A^\star \nabla' g \cdot \nabla' \phi,
\end{equation}
where $A^\star:=(A^\star_{\alpha \beta})_{1 \leq \alpha, \beta \leq d-1}$ is the homogenized matrix defined by
$$
A^\star_{\alpha \beta} := \int_{\mathcal{Y}} A (\nabla w^\alpha + e_\alpha) \cdot (\nabla w^\beta + e_\beta),
$$
where $w^\alpha$ is the solution to~\eqref{prcor2}. By construction, $A^\star$ is symmetric and coercive, and we also have $\dps A^\star_{\alpha \beta} = \int_{\mathcal{Y}} A (\nabla w^\alpha + e_\alpha) \cdot e_\beta$.
\end{theorem}

We note that the homogenized problem~\eqref{eq:fv_star} is again a diffusion type problem, posed on $\omega$, and with a constant diffusion coefficient. For the sake of completeness, we provide a proof of this result in Appendix~\ref{app:diff_preuve1}, using the method of the oscillating test function. 

\subsection{Strong convergence of the two-scale expansion} \label{sec:strong_diff}

From now on, we assume that the domain $\omega$ satisfies the following ``shape regularity'' assumption. Let $\rho_0$ be the diameter of the largest ball included in $\omega$. We assume that
\begin{equation} \label{eq:shape_regul}
  \frac{\text{diam $\omega$}}{\rho_0} \leq \eta,
\end{equation}
for some constant $\eta$. The motivation for this assumption is explained below Theorem~\ref{thconvforte_diffusion}.

\medskip

For any function $u$ of $H^1(\Omega)$, we define the norm
\begin{equation}\label{eq:defnorm}
\| u \|_{H^1_\eps(\Omega)} := \sqrt{ \frac{\| u \|^2_{L^2(\Omega)}}{\left[ \max\left(\eps,|\omega|^{\frac{1}{d-1}}\right) \right]^2} + \| \nabla^\eps u \|^2_{L^2(\Omega)} }.
\end{equation}
The second term $\| \nabla^\eps u \|_{L^2(\Omega)}$ is indeed the relevant energy norm for~\eqref{var:diff}, and the scaling of the first term is motivated by Lemma~\ref{poincarebis} below. 

\medskip

The aim of this section is to prove that the classical two-scale expansion, built upon the homogenized solution and the correctors, yields an approximation of $u^\eps$ which is converging in the $H^1_\eps(\Omega)$ norm as $\eps$ goes to $0$. More precisely, we are now in position to state our main result for the diffusion problem.

\begin{theorem} \label{thconvforte_diffusion}
Under Assumption~\eqref{eq:inde_eps}, consider the solution $u^\eps$ to~\eqref{var:diff} and its homogenized limit $u^\star$, solution to~\eqref{eq:fv_star}. Assume that $\omega$ satisfies~\eqref{eq:shape_regul}, that $\nabla (u^\star + g) \in (W^{1,\infty}(\omega))^d$ and that, for any $1 \leq \alpha \leq d-1$, the corrector $w^\alpha$ defined by~\eqref{prcor2} satisfies $\dps w^\alpha \in W^{1,\infty}\left(\mathbb{R}^{d-1} \times \left( -\demi, \demi \right)\right)$. Then, introducing the two-scale expansion
$$
u^{\eps,1}(x) := u^\star(x') + \eps \, w^\alpha \left(\frac{x'}{\eps},x_d\right) \partial_\alpha (u^\star+g)(x')
$$
for any $x=(x',x_d) \in \Omega$, it holds that
\begin{multline} \label{eq:titi7}
  \| u^\eps - u^{\eps,1} \|_{H^1_\eps(\Omega)} \leq C \, \sqrt{\eps} \, \Big( |\omega|^{\frac{d-2}{2(d-1)}} \, \|\nabla (u^\star+g) \|_{L^\infty(\omega)} \\ + \sqrt{\eps} \, |\omega|^{1/2} \, \| \nabla^2 (u^\star+g) \|_{L^\infty(\omega)} + \sqrt{\eps} \, \|h_\pm \|_{L^2(\omega)} + \sqrt{\eps} \, \| f \|_{L^2(\Omega)} \Big) 
\end{multline}
for some constant $C>0$ independent of $\eps$, $\omega$ (but depending on the constant $\eta$ of~\eqref{eq:shape_regul}), $u^\star$, $f$, $g$ and $h_\pm$. 
\end{theorem}

Up to lower order terms, we have
$$
\big( \nabla^\eps (u^{\eps,1}+g) \big)(x) \approx \nabla' (u^\star+g)(x') + (\nabla w)\left(\frac{x'}{\eps},x_d\right) \nabla'(u^\star+g)(x'),
$$
where the matrix $\nabla w$ is defined, for any $1 \leq i \leq d$ and $1 \leq \alpha \leq d-1$, by $(\nabla w)_{i \alpha} = \partial_i w^\alpha$. 

\medskip

Note that we do not simply prove that $\| u^\eps - u^{\eps,1} \|_{H^1_\eps(\Omega)} \leq \overline{C} \, \sqrt{\eps}$ for some constant $\overline{C}$ independent of $\eps$, but that we explicitly write how this constant depends on $\omega$. Tracking this dependence is important for our applications to MsFEM approaches considered in~\cite{MsFEM-ELLL}, where we write the above estimate for problems posed on local elements (and where $|\omega|$ is thus directly related to the size of the coarse mesh). In the same spirit, we will ask there that all the elements of the mesh satisfy the shape regularity assumption~\eqref{eq:shape_regul} for some constant $\eta$ independent of the element.

\begin{remark} \label{rem:regul_ustar_diff}
  We wish to point out that the assumption $\nabla (u^\star + g) \in (W^{1,\infty}(\omega))^d$ is a standard assumption when proving convergence rates of two-scale expansions (see, e.g.,~\cite[p.~28]{jikov}). Note that, in view of~\eqref{eq:fv_star}, this assumption implies that $\m(f) + h_+ + h_-$ belongs to $L^\infty(\omega)$. 
\end{remark}  

\medskip

The rest of this section is devoted to the proof of Theorem~\ref{thconvforte_diffusion}, which requires some preliminary results stated below. 
We first state a Poincar\'e estimate for the norm $\| \cdot \|_{H^1_\eps(\Omega)}$ defined in~\eqref{eq:defnorm}.

\begin{lemma} \label{poincarebis}
Let $V$ be defined by~\eqref{def:V}. Assume that $\omega$ satisfies~\eqref{eq:shape_regul}. There exists a constant $C>0$ independent of $\eps$ and $\omega$ (but depending on the constant $\eta$ of~\eqref{eq:shape_regul}) such that
\begin{equation} \label{eq:poinK}
\forall u \in V, \quad \| u \|_{L^2(\Omega)} \leq C \max\left(\eps,|\omega|^{\frac{1}{d-1}}\right) \| \nabla^\eps u \|_{L^2(\Omega)}.
\end{equation}
As a direct consequence of~\eqref{eq:poinK}, we have that
\begin{equation} \label{eq:poinK_corro}
\forall u \in V, \quad \| \nabla^\eps u \|_{L^2(\Omega)} \leq \| u \|_{H^1_\eps(\Omega)} \leq C \, \| \nabla^\eps u \|_{L^2(\Omega)},
\end{equation}
for some constant $C>0$ independent of $\eps$ and $\omega$.
\end{lemma}

\begin{proof}
For any smooth bounded domain $\widetilde{\omega} \subset \R^{d-1}$, we denote 
$$
V(\widetilde{\Omega}) := \left\{ v \in H^1(\widetilde{\Omega}), \quad v=0 \text{ on } \partial \widetilde{\omega} \times \left( -\frac{1}{2},\frac{1}{2} \right) \right\}
$$
where $\dps \widetilde{\Omega}:= \widetilde{\omega} \times \left( -\frac{1}{2}, \frac{1}{2}\right)$. 

Let $\widehat{\omega} := (0,1)^{d-1}$ and $\dps \widehat{\Omega} := \widehat{\omega} \times \left(-\demi, \demi \right)$. Using the Poincar\'e inequality in $V(\widehat{\Omega})$, there exists some constant $C_P(\widehat{\Omega})$ such that
\begin{equation} \label{eq:toto20}
  \forall \widehat{u} \in V(\widehat{\Omega}), \quad \| \widehat{u} \|_{L^2(\widehat{\Omega})} \leq C_P(\widehat{\Omega}) \, \| \nabla \widehat{u} \|_{L^2(\widehat{\Omega})}.
\end{equation}
We next proceed by scaling. We introduce $\omega_K := (0,K)^{d-1}$ and $\dps \Omega_K := \omega_K \times \left(-\demi, \demi \right)$ for some $K>0$. For any $u_K \in V(\Omega_K)$, the function $\dps \widehat{u} : \widehat{\Omega} \ni x \mapsto u_K \left( K \, x',x_d \right)$ belongs to $V(\widehat{\Omega})$. A simple computation shows that
\begin{align*}
  \| u_K \|_{L^2(\Omega_K)}^2 &= K^{d-1} \, \| \widehat{u} \|_{L^2(\widehat{\Omega})}^2,
  \\
  \| \nabla' u_K \|_{L^2(\Omega_K)}^2 &= K^{d-3} \, \| \nabla' \widehat{u} \|_{L^2(\widehat{\Omega})}^2,
  \\
  \| \eps^{-1} \partial_d u_K\|_{L^2(\Omega_K)}^2 &= K^{d-1} \, \| \eps^{-1} \partial_d \widehat{u} \|_{L^2(\widehat{\Omega})}^2.
\end{align*}
Using~\eqref{eq:toto20}, we thus get that
\begin{align*}
  \| u_K \|_{L^2(\Omega_K)}^2
  & \leq C_P(\widehat{\Omega})^2 \, \big (K^2 \, \| \nabla' u_K \|_{L^2(\Omega_K)}^2 + \eps^2 \, \| \eps^{-1} \partial_d u_K\|_{L^2(\Omega_K)}^2 \big)
  \\
  & \leq C_P(\widehat{\Omega})^2 \, \max(K^2,\eps^2) \, \| \nabla^\eps u_K \|_{L^2(\Omega_K)}^2.
\end{align*}
Since $K=|\omega_K|^{\frac{1}{d-1}}$, we get
$$
\| u_K \|_{L^2(\Omega_K)} \leq C_P(\widehat{\Omega}) \, \max\left(|\omega_K|^{\frac{1}{d-1}}, \eps\right) \| \nabla^\eps u_K \|_{L^2(\Omega_K)},
$$
which proves~\eqref{eq:poinK} in the case $\omega=\omega_K$. In the case of a more general, shape regular domain $\omega$, the proof can be performed using the same arguments.


The estimate~\eqref{eq:poinK_corro} is a direct consequence of~\eqref{eq:poinK} and~\eqref{eq:defnorm}. This concludes the proof of Lemma~\ref{poincarebis}.
\end{proof}

We now proceed with a form of Poincar\'e-Wirtinger inequality.

\begin{lemma} \label{potmoy}
Let $h \in L^2(\omega)$ and $v \in H^1(\Omega)$. Then, for any $\dps z \in \left[ -\frac{1}{2},\frac{1}{2} \right]$, we have
$$
\left| \int_\omega \big( v(\cdot,z)-\m(v) \big) \, h \right| \leq \eps \, \| h \|_{L^2(\omega)} \, \|\nabla^\eps v \|_{L^2(\Omega)}.
$$
\end{lemma}

\begin{proof}
Let $v \in H^1(\Omega)$. For any $z \in [-1/2,1/2]$, we have
\begin{align*}
  \left| v(\cdot,z)-\m(v) \right|
  &=
  \left| \int_{-1/2}^{1/2} \big( v(\cdot,z)-v(\cdot,t) \big) \, dt \right|
  \\
  &=
  \left| \int_{-1/2}^{1/2} \int_t^z \partial_d v(\cdot,s) \, ds \, dt \right|
  \\
  &\leq
  \eps \int_{-1/2}^{1/2} \left| \frac{1}{\eps} \, \partial_d v(\cdot,s) \right| ds.
\end{align*}
It follows that
\begin{align*}
  \| v(\cdot,z)-\m(v) \|^2_{L^2(\omega)}
  &=
  \int_\omega \left| v(\cdot,z)-\m(v) \right|^2
  \\
  &
  \leq \eps^2 \int_\omega \left| \int_{-1/2}^{1/2} \left| \frac{1}{\eps} \, \partial_d v(\cdot,s) \right| ds \right|^2
  \\ 
  & \leq
  \eps^2 \left\| \frac{1}{\eps} \, \partial_d v \right\|_{L^2(\Omega)}^2
  \\
  & \leq \eps^2 \| \nabla^\eps v \|_{L^2(\Omega)}^2.
\end{align*}	
For any $z \in [-1/2,1/2]$, we now write
$$
\left| \int_\omega \big( v(\cdot,z)-\m(v) \big) \, h \right| \leq \| h \|_{L^2(\omega)} \ \| v(\cdot,z)-\m(v) \|_{L^2(\omega)}.
$$
Collecting the above two estimates concludes the proof of Lemma~\ref{potmoy}.
\end{proof}

Lemma~\ref{minZ} is an adaptation of a technical result, already present in~\cite[p.~27]{jikov}, to the case of plates.

\begin{lemma} \label{minZ}
Let $V$ be defined by~\eqref{def:V}. Let $\dps Z \in \left[ L^2_{\rm loc}\left(\mathbb{R}^{d-1} \times \left( -\frac{1}{2},\frac{1}{2} \right) \right) \right]^d$ be a vector field such that
\begin{itemize}
 \item [(i)] for almost all $\dps z \in \left( -\frac{1}{2},\frac{1}{2} \right)$, the function $Z(\cdot,z)$ is $Y$-periodic;
 \item[(ii)] $\dps \int_{\mathcal{Y}} Z = 0$; 
 \item[(iii)] $\div Z = 0$ in $\dps \mathcal{D}'\left(\mathbb{R}^{d-1} \times \left( -\frac{1}{2},\frac{1}{2} \right) \right)$; 
 \item [(iv)] $Z \cdot e_d = 0$ on $\mathbb{R}^{d-1} \times \{-1/2\}$ and on $\mathbb{R}^{d-1} \times \{1/2\}$.
\end{itemize}
Then, there exists some $C$ independent of $\eps$ and $\omega$ such that, for any $\varphi$ in $W^{1,\infty}(\omega)$ and any $v \in V$,
\begin{equation}\label{eq:covid}
\left| \int_\Omega \varphi(x') \, Z\left( \frac{x'}{\eps},x_d \right) \cdot \nabla^\eps v(x',x_d) \, dx' \, dx_d \right| \leq C \, \eps \, |\omega|^{1/2} \, \| \nabla \varphi \|_{L^\infty(\omega)} \, \| \nabla^\eps v \|_{L^2(\Omega)}.
\end{equation}
\end{lemma}

\begin{remark}
We recall that any function in $H_{\rm div}(D) := \left\{ Z \in (L^2(D))^d, \ \ \div Z \in L^2(D)\right\}$ (for any smooth domain $D\subset \R^d$) has a well-defined normal trace on $\partial D$ (see Appendix~\ref{app:Hdiv} for details). Assumption~(iv) in Lemma~\ref{minZ} thus makes sense.
\end{remark}

\begin{remark} \label{rem:CN_scalaire}
Assumption~(ii) is a necessary condition for~\eqref{eq:covid} to hold true. Consider indeed a function $v \in V$ independent of $x_d$: then, on the one hand, the right-hand side of~\eqref{eq:covid} obviously converges to 0 when $\eps \to 0$, and thus also the left-hand side, which in this case yields that
$$
\lim_{\eps \to 0} \int_\Omega \varphi(x') \, Z\left( \frac{x'}{\eps},x_d \right) \cdot \nabla v(x') \, dx' \, dx_d = 0.
$$
On the other hand, using Lemma~\ref{limmoyenne} and the fact that $Z$ is $Y$-periodic, we are in position to state that the limit of the left-hand side of~\eqref{eq:covid} is $\dps \int_\Omega \varphi(x') \, Z^\star(x_d) \cdot \nabla v(x') \, dx' \, dx_d$, with $\dps Z^\star(x_d) = \int_Y Z(x',x_d) dx'$ for any $x_d \in (-1/2,1/2)$. We thus deduce that $\dps \int_\Omega \varphi(x') \, Z^\star(x_d) \cdot \nabla v(x') \, dx' \, dx_d = 0$ for any $\varphi$ in $W^{1,\infty}(\omega)$ and any $v \in V$ independent of $x_d$. This implies that $\dps \int_{-1/2}^{1/2} Z^\star_\alpha(x_d) \, dx_d = 0$ for any $1 \leq \alpha \leq d-1$, and thus, in view of the definition of $Z^\star$, that $\dps \int_{\mathcal Y} Z_\alpha = 0$.

To show that Assumption~(ii) is indeed a necessary condition, we are left to show that~\eqref{eq:covid} implies $\dps \int_{\mathcal Y} Z_d = 0$. To this end, consider the function $v(x',x_d) = \eps \, \psi(x') \, \theta(x_d)$ for some $\psi \in C^\infty_0(\omega)$ and some $\theta \in C^\infty[-1/2,1/2]$. It belongs to $V$ and satisfies that $\nabla^\eps v$ is bounded. The right-hand side, and therefore the left-hand side, of~\eqref{eq:covid} hence converges to 0 when $\eps \to 0$. For this function $v$, the left-hand side of~\eqref{eq:covid} writes
$$
\eps \int_\Omega \varphi(x') \, Z_\alpha\left( \frac{x'}{\eps},x_d \right) \partial_\alpha \psi(x') \, \theta(x_d) \, dx' \, dx_d + \int_\Omega \varphi(x') \, Z_d\left( \frac{x'}{\eps},x_d \right) \psi(x') \, \theta'(x_d) \, dx' \, dx_d,
$$
which converges to $\dps \int_\Omega \varphi(x') \, Z_d^\star(x_d) \, \psi(x') \, \theta'(x_d) \, dx' \, dx_d$ in view of Lemma~\ref{limmoyenne}. Since $\varphi$ is arbitrary, this implies that $\dps \int_{-1/2}^{1/2} Z_d^\star(x_d) \, \theta'(x_d) \, dx_d = 0$, and therefore $Z_d^\star(x_d) = 0$ for any $x_d \in (-1/2,1/2)$ since $\theta \in C^\infty[-1/2,1/2]$ is arbitrary. This of course implies that $\dps \int_{-1/2}^{1/2} Z^\star_d(x_d) \, dx_d = 0$, and thus $\dps \int_{\mathcal Y} Z_d = 0$, in view of the definition of $Z^\star_d$. This completes the argument showing that Assumption~(ii) is indeed a necessary condition for~\eqref{eq:covid} to hold true.
\end{remark}

\begin{remark} \label{rem:CN_scalaire_comp}
  In view of the second part of Remark~\ref{rem:CN_scalaire}, one could think that a necessary condition for~\eqref{eq:covid} to hold true is $Z_d^\star(x_d) = 0$ for any $x_d \in (-1/2,1/2)$ (and not only $\dps \int_{\mathcal Y} Z_d = 0$). Actually, under Assumption~(iii), the statement
  \begin{equation} \label{eq:stat1}
    \int_{\mathcal Y} Z_d = 0
  \end{equation}
  is equivalent to the statement
  \begin{equation} \label{eq:stat2}
    Z_d^\star(x_d) = 0 \quad \text{for any $x_d \in (-1/2,1/2)$},
  \end{equation}
  as we now show. Of course, \eqref{eq:stat2} implies~\eqref{eq:stat1}. Assuming now that~\eqref{eq:stat1} holds, we integrate on $Y$ the equation $\div Z = 0$ (see Assumption~(iii)). Using the $Y$-periodicity of $Z$, we find $\dps \partial_d \int_Y Z_d = 0$, which means that $Z^\star_d$ is actually a constant. The condition~\eqref{eq:stat1} now implies that this constant vanishes, and we hence deduce~\eqref{eq:stat2}.
\end{remark}

The proof of Lemma~\ref{minZ} requires the following result (see~\cite[p.~6]{jikov}). We recall that $\dps L^2_{\rm per}(\mathbb{R}^d) = \left\{ f \in L^2_{\rm loc}(\mathbb{R}^d), \ \ \text{$f$ is $(0,1)^d$-periodic} \right\}$.

\begin{lemma}[from p.~6 of~\cite{jikov}]\label{lemmajikov}
Let $\dps p \in \left\{q \in (L^2_{\rm per}(\mathbb{R}^d))^d, \ \ \div q = 0 \text{ in } \mathbb{R}^d \right\}$. Then there exists a skew-symmetric matrix $J \in \left( H^1_{\rm per}(\mathbb{R}^d) \right)^{d \times d}$ such that
$$
\forall 1 \leq j \leq d, \qquad p_j - \int_{(0,1)^d} p_j = \partial_i J_{ij}, \qquad \int_{(0,1)^d} J_{ij} = 0.
$$
\end{lemma}

\begin{proof}[Proof of Lemma~\ref{minZ}]
Let $\overline{Z}$ be the periodic extension of $Z$ in the $e_d$ direction. We then have $\overline{Z} \in (L^2_{\rm per}(\mathbb{R}^d))^d$. Let us prove that $\div \overline{Z} = 0$ in $\mathcal{D}'(\R^d)$. Let $\psi \in \mathcal{D}(\R^d)$. There exists a compact set $K\subset \R^{d-1}$ and an integer $m \in \mathbb{N}^\star$ such that $\mbox{\rm Supp}\, \psi \subset K \times [-m-1/2 ; m+1/2]$. We compute
\begin{align*}
  \langle \div \overline{Z}, \psi \rangle_{\mathcal{D}'(\R^d), \mathcal{D}(\R^d)}
  & = -\int_{K\times (-m-1/2 ; m+1/2)} \overline{Z} \cdot \nabla \psi
  \\
  &= \sum_{k=-m}^m -\int_{K\times (k-1/2; k+1/2)} \overline{Z} \cdot \nabla \psi
  \\
  &= \sum_{k=-m}^m -\int_{K\times (-1/2; 1/2)} Z \cdot \nabla \psi(\cdot+k \, e_d).
\end{align*}
Using Assumption~(iv) and that $\psi(\cdot+k \, e_d) = 0$ on $\partial K \times \R$, we compute that
$$
\int_{K\times (-1/2; 1/2)} Z \cdot \nabla \psi(\cdot+k \, e_d) = - \int_{K\times (-1/2; 1/2)} \psi(\cdot+k \, e_d) \, \div Z. 
$$
Using Assumption~(iii), we deduce that $\dps \int_{K\times (-1/2; 1/2)} Z \cdot \nabla \psi(\cdot+k \, e_d) = 0$. We thus obtain that $\div \overline{Z} = 0$ in $\mathcal{D}'(\R^d)$.
	
\medskip
	
We are thus in position to use Lemma~\ref{lemmajikov} for $\overline{Z}$. Since $\dps \int_{\mathcal{Y}} \overline{Z} = 0$, there exists a skew-symmetric matrix-valued field $J \in \left( H^1_{\rm per}(\mathbb{R}^d) \right)^{d \times d}$ such that $\overline{Z}_j = \partial_i J_{ij}$ for any $1 \leq j \leq d$.

\medskip

Let $\varphi \in W^{1,\infty}(\omega)$ and $1 \leq j \leq d$. Denoting $J_{\cdot j} = \left( J_{ij} \right)_{1 \leq i \leq d} \in \R^d$, we compute, for any $x_d$ in $(-1/2,1/2)$, that
\begin{align}
  \left[ Z \left( \frac{\cdot}{\eps},x_d \right) \varphi \right]_j
  &= \left[ \overline{Z} \left( \frac{\cdot}{\eps},x_d \right) \varphi \right]_j
  \nonumber
  \\
  &= \partial_i J_{ij}\left( \frac{\cdot}{\eps},x_d \right) \varphi
  \nonumber
  \\
  &= \eps \div^\eps \left[ J_{\cdot j}\left( \frac{\cdot}{\eps},x_d \right) \right] \varphi
  \nonumber
  \\
  &= \eps \div^\eps \left[ J_{\cdot j}\left( \frac{\cdot}{\eps},x_d \right) \varphi \right] - \eps \, J_{\cdot j}\left( \frac{\cdot}{\eps},x_d \right) \cdot \nabla^\eps \varphi
  \nonumber
  \\
  &= \eps \, \widetilde{B}_j(\cdot,x_d) - \eps \, B_j(\cdot,x_d),
  \label{eq:diffe1}
\end{align}
where $\dps \widetilde{B}_j(\cdot,x_d) := \div^\eps \left[ J_{\cdot j}\left( \frac{\cdot}{\eps},x_d \right) \varphi \right]$ and $\dps B_j(\cdot,x_d) := J_{\cdot j}\left( \frac{\cdot}{\eps},x_d \right) \cdot \nabla^\eps \varphi$. Note that
\begin{align}
  & \div^\eps \widetilde{B}
  \nonumber
  \\
  &= \div^\eps \left( \widetilde{B}_j \, e_j \right) 
  \nonumber
  \\
  &= \div^\eps \left( \partial_\alpha \left[ J_{\alpha j}\left( \frac{\cdot}{\eps},x_d \right) \varphi \right] e_j + \frac{1}{\eps} \partial_d \left[ J_{d j}\left( \frac{\cdot}{\eps},x_d \right) \varphi \right] e_j \right) 
  \nonumber
  \\
  &= \partial_{\beta \alpha} \left[ J_{\alpha \beta} \left( \frac{\cdot}{\eps},x_d \right) \varphi \right] + \frac{1}{\eps} \left( \partial_{d \alpha} \left[ J_{\alpha d}\left( \frac{\cdot}{\eps},x_d \right) \varphi \right] + \partial_{\beta d} \left[ J_{d \beta}\left( \frac{\cdot}{\eps},x_d \right) \varphi \right] \right) + \frac{1}{\eps^2} \partial_{dd} \left[ J_{dd}\left( \frac{\cdot}{\eps},x_d \right) \varphi \right] 
  \nonumber
  \\
  &= 0 \quad [\text{because $J$ is skew symmetric}]. \label{eq:toto6}
\end{align}
We now write that, for any $v$ in $V$,
\begin{equation} \label{eq:toto}
  \int_\Omega \varphi(x') \, Z\left(\frac{x'}{\eps},x_d \right) \cdot \nabla^\eps v(x',x_d) \, dx' \, dx_d = \eps \int_\Omega \widetilde{B} \cdot \nabla^\eps v - \eps \int_\Omega B \cdot \nabla^\eps v.
\end{equation}
We are going to successively bound the two terms of~\eqref{eq:toto}. For the first term, we note that, by definition, $\widetilde{B}$ is in $(L^2(\Omega))^d$ and $\div^\eps \widetilde{B}$ is also in $L^2(\Omega)$, in view of~\eqref{eq:toto6}. We thus note that $\widetilde{B}$ has a well-defined normal trace (see Appendix~\ref{app:Hdiv} for details). Using~\eqref{eq:toto6} and an integration by part (see~\eqref{eq:IPP_scalaire}), we obtain
\begin{equation} \label{eq:toto10_a}
  \int_\Omega \widetilde{B} \cdot \nabla^\eps v = \frac{1}{\eps} \int_\omega \widetilde{B} \left(\cdot,\frac{1}{2}\right) \cdot e_d \ v\left(\cdot,\frac{1}{2}\right) - \frac{1}{\eps} \int_\omega \widetilde{B} \left(\cdot,-\frac{1}{2}\right) \cdot e_d \ v\left(\cdot,-\frac{1}{2}\right). 
\end{equation}
We next infer from~\eqref{eq:diffe1}, the definition of $B$ and the fact that $J$ is skew-symmetric that, for any $x_d$ in $(-1/2,1/2)$,
\begin{equation} \label{eq:diffe2}
\widetilde{B}(\cdot,x_d) = \frac{1}{\eps} \, Z\left(\frac{\cdot}{\eps}, x_d\right) \varphi - J\left(\frac{\cdot}{\eps}, x_d\right) \nabla^\eps \varphi.
\end{equation}
Since $\varphi \in W^{1,\infty}(\omega)$, we have $\dps \nabla^\eps \varphi = \nabla \varphi = \left( \begin{array}{c} \nabla' \varphi \\ 0 \end{array} \right) \in (L^\infty(\omega))^d$. Besides, since $J \in \left( H^1_{\rm per}(\R^d) \right)^{d \times d}$, we have
$$
J\left( \frac{\cdot}{\eps}, -\frac{1}{2}\right) \nabla^\eps \varphi = J\left( \frac{\cdot}{\eps}, \frac{1}{2}\right) \nabla^\eps \varphi = J\left( \frac{\cdot}{\eps}, \frac{1}{2}\right) \nabla \varphi \in (L^2(\omega))^d.
$$
We then infer from~\eqref{eq:diffe2}, the above relation and Assumption~(iv) that
$$
\widetilde{B}\left(\cdot, \pm \frac{1}{2}\right)\cdot e_d 
=
-e_d^T J\left( \frac{\cdot}{\eps}, \frac{1}{2}\right) \nabla \varphi.
$$
We hence deduce from~\eqref{eq:toto10_a} that
\begin{align}
  \left| \int_\Omega \widetilde{B} \cdot \nabla^\eps v \right| 
  &= \frac{1}{\eps} \left| \int_\omega e_d^T J\left( \frac{\cdot}{\eps}, \frac{1}{2}\right) \nabla \varphi \left[ v\left(\cdot,\frac{1}{2}\right) - v\left(\cdot,-\frac{1}{2}\right) \right] \right|
  \nonumber
  \\
  &= \left| \int_\omega \int_{t=-1/2}^{1/2} \frac{1}{\eps} \partial_d v(x',t) \ e_d^T J\left(\frac{x'}{\eps},\frac{1}{2}\right)\nabla \varphi(x') \,dt \,dx' \right|
  \nonumber
  \\
  & \leq \| \nabla^\eps v \|_{L^2(\Omega)} \left\| J\left( \frac{\cdot}{\eps}, \frac{1}{2} \right) \right\|_{L^2(\omega)} \| \nabla \varphi \|_{L^\infty(\omega)} 
  \nonumber
  \\
  & \leq C |\omega|^{1/2} \, \| \nabla^\eps v \|_{L^2(\Omega)} \left\| J\left(\cdot, \frac{1}{2}\right) \right\|_{L^2(Y)} \| \nabla \varphi \|_{L^\infty(\omega)} 
  \nonumber
  \\
  & \leq C |\omega|^{1/2} \, \| \nabla^\eps v \|_{L^2(\Omega)} \, \| J \|_{H^1(\mathcal{Y})} \, \| \nabla \varphi \|_{L^\infty(\omega)},
  \label{eq:bound1}
\end{align}
where we have used a trace inequality in the last line.

We now bound the second term of~\eqref{eq:toto}:
\begin{align}
  \left| \int_\Omega B \cdot \nabla^\eps v \right|
  &= \left| \int_\omega \int_{-1/2}^{1/2} (\nabla^\eps v(x',x_d))^T \, J\left( \frac{x'}{\eps},x_d \right) \nabla^\eps \varphi(x') \, dx' \, dx_d \right|
  \nonumber
  \\
  & \leq \| \nabla \varphi \|_{L^\infty(\omega)} \, \| \nabla^\eps v \|_{L^2(\Omega)} \left\| J\left( \frac{x'}{\eps},x_d \right) \right\|_{L^2(\Omega)}
  \nonumber
  \\
  & \leq C \, |\omega|^{1/2} \, \| \nabla \varphi \|_{L^\infty(\omega)} \, \| \nabla^\eps v \|_{L^2(\Omega)} \, \| J\|_{L^2(\mathcal{Y})}.
  \label{eq:bound2_a}
\end{align} 
Collecting~\eqref{eq:toto}, \eqref{eq:bound1} and~\eqref{eq:bound2_a}, we deduce that
$$
\left| \int_\Omega \varphi(x') \, Z\left( \frac{x'}{\eps},x_d \right) \cdot \nabla^\eps v(x',x_d) \, dx' \, dx_d \right| \leq C \eps \, |\omega|^{1/2} \, \| \nabla \varphi \|_{L^\infty(\omega)} \, \| \nabla^\eps v \|_{L^2(\Omega)}.
$$
This concludes the proof of Lemma~\ref{minZ}.
\end{proof}

We are now in position to prove Theorem~\ref{thconvforte_diffusion}.

\begin{proof}[Proof of Theorem~\ref{thconvforte_diffusion}]
The proof falls in four steps. In the first step, we correct for the boundary mismatch between $u^\eps$ and its approximation $u^{\eps,1}$. In Steps~2 and~3, we show that our approximation is accurate in the bulk of the domain. In Step~4, we collect all the estimates to reach the conclusion.
	
\medskip

\noindent
{\bf Step~1.} Let $\tau_\eps \in \mathcal{D}(\omega)$ such that $0 \leq \tau_\eps \leq 1$ in $\omega$ and such that $\tau_\eps(x') = 1$ for any $x' \in \omega$ such that $\text{dist}(x',\partial \omega) \geq \eps$. Since $\omega$ is smooth, we can choose $\tau_\eps$ such that $\eps \|\nabla \tau_\eps \|_{L^\infty(\omega)} \leq C$ for some $C>0$ independent of $\omega$ and $\eps$. We define $\omega_\eps := \{ x' \in \omega \text{ such that } \text{dist}(x',\partial \omega) \geq \eps \}$ and $\dps \Omega_\eps := \omega_\eps \times \left( -\frac{1}{2},\frac{1}{2} \right)$. Note that $|\Omega \setminus \Omega_\eps| \leq C \, \eps \, |\omega|^{\frac{d-2}{d-1}}$.

We introduce the function $v^{\eps,1}$ defined for $x = (x',x_d) \in \Omega$ by
$$
v^{\eps,1}(x) := u^\star(x') + \eps \, \tau_\eps(x') \, w^\alpha \left( \frac{x'}{\eps},x_d \right) \partial_\alpha (u^\star+g)(x').
$$
By definition of $\tau_\eps$, we have $v^{\eps,1} \in V$ and $v^{\eps,1} = u^{\eps,1}$ in $\Omega_\eps$. In this first step of the proof, we bound $\| u^{\eps,1} - v^{\eps,1} \|_{H^1_\eps(\Omega)}$. We compute that
$$
u^{\eps,1}(x) - v^{\eps,1}(x) = \eps \, (1-\tau_\eps(x')) \, w^\alpha \left( \frac{x'}{\eps},x_d \right) \partial_\alpha (u^\star+g)(x'),
$$
and thus
\begin{align}
  \| v^{\eps,1} - u^{\eps,1} \|^2_{L^2(\Omega)}
  &\leq
  \eps^2 \, \| 1-\tau_\eps \|^2_{L^2(\Omega \setminus \Omega_\eps)} \sup_{1 \leq \alpha \leq d-1} \| w^\alpha \|_{L^\infty}^2 \, \| \nabla(u^\star+g) \|^2_{L^\infty(\omega)}
  \nonumber
  \\
  &\leq
  C \, \eps^2 \, |\Omega \setminus \Omega_\eps| \, \| \nabla(u^\star+g) \|^2_{L^\infty(\omega)}
  \nonumber
  \\
  &\leq
  C \, \eps^3 \, |\omega|^{\frac{d-2}{d-1}} \, \| \nabla(u^\star+g) \|^2_{L^\infty(\omega)}
  \nonumber
  \\
  &\leq
  C \, \eps \, |\omega|^{\frac{d-2}{d-1}} \, \left[ \max\left(\eps,|\omega|^{\frac{1}{d-1}}\right) \right]^2 \, \| \nabla(u^\star+g) \|^2_{L^\infty(\omega)}.
  \label{diffapp3_L2}
\end{align}
We next compute that $\nabla^\eps u^{\eps,1} - \nabla^\eps v^{\eps,1} = E_0^\eps - E_1^\eps + \eps \, E_2^\eps$, where
\begin{align*}
  E_0^\eps(x) &= (1-\tau_\eps(x')) \, \nabla w^\alpha \left( \frac{x'}{\eps},x_d \right) \partial_\alpha (u^\star+g)(x'),
  \\ 
  E_1^\eps(x) &= \eps \, \nabla^\eps \tau_\eps(x') \, w^\alpha \left( \frac{x'}{\eps},x_d \right) \partial_\alpha (u^\star+g)(x'),
  \\
  E_2^\eps(x) &= (1-\tau_\eps(x')) \, w^\alpha \left( \frac{x'}{\eps},x_d \right) \nabla^\eps \partial_\alpha (u^\star+g)(x').
\end{align*}
We bound the above three terms in $L^2(\Omega)$ norm, using that $w^\alpha \in W^{1,\infty}$, $\nabla (u^\star+g)$ belongs to $(W^{1,\infty}(\omega))^d$, $0 \leq \tau_\eps \leq 1$ and $\eps \, \|\nabla \tau_\eps \|_{L^\infty(\omega)} \leq C$. Note that, by definition, $\partial_d \tau_\eps = 0$, and therefore $\nabla^\eps \tau_\eps = \nabla \tau_\eps$. Likewise, we have $\nabla^\eps \partial_\alpha (u^\star+g) = \nabla \partial_\alpha (u^\star+g)$. We thus obtain
\begin{align*}
  \|E_2^\eps\|^2_{L^2(\Omega)} & \leq \sup_{1 \leq \alpha \leq d-1} \| w^\alpha \|_{L^\infty}^2 \, \| \nabla^2 (u^\star+g) \|^2_{L^2(\Omega)}
  \\
  & \leq C \, |\omega| \, \| \nabla^2 (u^\star+g) \|^2_{L^\infty(\omega)},
  \\
  \| E_1^\eps\|^2_{L^2(\Omega)} & \leq C \, |\Omega \setminus \Omega_\eps| \, \sup_{1 \leq \alpha \leq d-1} \| w^\alpha \|_{L^\infty}^2 \, \| \nabla(u^\star+g) \|^2_{L^\infty(\omega)}
  \\
  & \leq C \, \eps \, |\omega|^{\frac{d-2}{d-1}} \, \| \nabla (u^\star+g) \|^2_{L^\infty(\omega)},
  \\
  \| E_0^\eps\|^2_{L^2(\Omega)} & \leq C \, |\Omega \setminus \Omega_\eps| \, \sup_{1 \leq \alpha \leq d-1} \| \nabla w^\alpha \|_{L^\infty}^2 \, \| \nabla(u^\star+g) \|^2_{L^\infty(\omega)}
  \\
  & \leq C \, \eps \, |\omega|^{\frac{d-2}{d-1}} \, \| \nabla (u^\star+g) \|^2_{L^\infty(\omega)}.
\end{align*}
We thus obtain
\begin{equation} \label{diffapp3}
  \| \nabla^\eps v^{\eps,1} - \nabla^\eps u^{\eps,1} \|^2_{L^2(\Omega)} \leq C \left( \eps \, |\omega|^{\frac{d-2}{d-1}} \, \|\nabla (u^\star+g) \|^2_{L^\infty(\omega)} + \eps^2 \, |\omega| \, \| \nabla^2 (u^\star+g) \|^2_{L^\infty(\omega)} \right).
\end{equation}
Collecting~\eqref{diffapp3_L2} and~\eqref{diffapp3}, we deduce
\begin{equation} \label{diffapp3_H1}
  \| v^{\eps,1} - u^{\eps,1} \|^2_{H^1_\eps(\Omega)} \leq C \, \eps \left( |\omega|^{\frac{d-2}{d-1}} \, \|\nabla (u^\star+g) \|^2_{L^\infty(\omega)} + \eps \, |\omega| \, \| \nabla^2 (u^\star+g) \|^2_{L^\infty(\omega)} \right),
\end{equation}
where we recall that the norm $\| \cdot \|_{H^1_\eps(\Omega)}$ is defined by~\eqref{eq:defnorm}.

\medskip

\noindent
{\bf Step~2.} We now bound $\overline{v}^\eps := u^\eps - v^{\eps,1}$. Using the coercivity of $A$, we write
\begin{align}
  c_- \| \nabla^\eps \overline{v}^\eps \|_{L^2(\Omega)}^2
  &\leq
  \int_\Omega \nabla^\eps \overline{v}^\eps \cdot A^\eps \nabla^\eps \overline{v}^\eps
  \nonumber
  \\
  &=
  \int_\Omega \nabla^\eps \overline{v}^\eps \cdot A^\eps \nabla^\eps (u^\eps - u^{\eps,1}) + \int_\Omega \nabla^\eps \overline{v}^\eps \cdot A^\eps \nabla^\eps (u^{\eps,1} - v^{\eps,1}),
  \label{estim2}
\end{align}
and we bound the second term of~\eqref{estim2} using~\eqref{diffapp3}:
\begin{align}
  & \int_\Omega \nabla^\eps \overline{v}^\eps \cdot A^\eps \nabla^\eps (u^{\eps,1} - v^{\eps,1})
  \nonumber
  \\
  &\leq
  C \, \|\nabla^\eps \overline{v}^\eps \|_{L^2(\Omega)} \, \|\nabla^\eps (u^{\eps,1} - v^{\eps,1}) \|_{L^2(\Omega)}
  \nonumber
  \\
  &\leq
  C \left( \sqrt{\eps} \, |\omega|^{\frac{d-2}{2(d-1)}} \, \|\nabla (u^\star+g) \|_{L^\infty(\omega)} + \eps \, |\omega|^{1/2} \, \| \nabla^2 (u^\star+g) \|_{L^\infty(\omega)} \right) \|\nabla^\eps \overline{v}^\eps \|_{L^2(\Omega)}.
  \label{eq:titi}
\end{align}
For the first term of~\eqref{estim2}, we write
$$
\int_\Omega \nabla^\eps \overline{v}^\eps \cdot A^\eps \nabla^\eps (u^\eps - u^{\eps,1}) = \int_\Omega A^\eps \nabla^\eps (u^\eps + g) \cdot \nabla^\eps \overline{v}^\eps - \int_\Omega A^\eps \nabla^\eps (u^{\eps,1}+g) \cdot \nabla^\eps \overline{v}^\eps.
$$
Defining the remainder terms $R^\eps_1$ and $R^\eps_2$ by
\begin{equation} \label{eq:r1}
  R^\eps_1 := \int_\Omega A^\eps \nabla^\eps (u^\eps + g) \cdot \nabla^\eps \overline{v}^\eps - \int_\Omega \left[ \int_{\mathcal{Y}} A (e_\alpha + \nabla w^\alpha) \right] \partial_\alpha (u^\star +g) \cdot \nabla^\eps \overline{v}^\eps
\end{equation}
and
\begin{equation} \label{eq:r2}
  R^\eps_2 := \int_\Omega A^\eps \nabla^\eps (u^{\eps,1} + g) \cdot \nabla^\eps \overline{v}^\eps - \int_\Omega A^\eps \left( e_\alpha + \nabla w^\alpha_\eps \right) \partial_\alpha (u^\star+g) \cdot \nabla^\eps \overline{v}^\eps,
\end{equation}
where we have used the short-hand notation $\dps \nabla w^\alpha_\eps(x) = \nabla w^\alpha \left(\frac{x'}{\eps},x_d \right)$, we deduce that
\begin{multline} \label{eq:titi2}
  \int_\Omega \nabla^\eps \overline{v}^\eps \cdot A^\eps \nabla^\eps (u^\eps - u^{\eps,1}) \\ = \int_\Omega \left( \left[ \int_{\mathcal{Y}} A ( e_\alpha + \nabla w^\alpha) \right] - A^\eps ( e_\alpha + \nabla w^\alpha_\eps) \right) \partial_\alpha (u^\star+g) \cdot \nabla^\eps \overline{v}^\eps + R^\eps_1 - R^\eps_2.
\end{multline}
To estimate the first term of~\eqref{eq:titi2}, we introduce the vector-valued function
\begin{equation} \label{eq:def_Z_scalaire}
Z_\alpha := \left[ \int_{\mathcal{Y}} A (e_\alpha + \nabla w^\alpha) \right] - A \left( e_\alpha + \nabla w^\alpha \right)
\end{equation}
and we recast~\eqref{eq:titi2} as
\begin{equation} \label{eq:titi3}
  \int_\Omega \nabla^\eps \overline{v}^\eps \cdot A^\eps \nabla^\eps (u^\eps - u^{\eps,1}) = \int_\Omega Z_\alpha \left(\frac{x'}{\eps},x_d\right) \partial_\alpha (u^\star+g)(x') \cdot \nabla^\eps \overline{v}^\eps(x) \, dx + R^\eps_1 - R^\eps_2.
\end{equation}
In view of the properties~\eqref{prcor2} satisfied by $w^\alpha$, we observe that $Z_\alpha$ satisfies the assumptions of Lemma~\ref{minZ}. This is obvious for Assumptions~(i), (ii) and~(iii). Assumption~(iv) stems from the second line of~\eqref{prcor2} and from~\eqref{prcor1} where we choose $v(x) = x_d$ as test function, which yields
\begin{equation} \label{eq:titi6}
\forall 1 \leq \alpha \leq d-1, \qquad \int_{\mathcal{Y}} A (e_\alpha + \nabla w^\alpha) \cdot e_d = 0.
\end{equation}
In addition, the function $\partial_\alpha (u^\star+g)$ belongs to $W^{1,\infty}(\omega)$ and $\overline{v}^\eps$ belongs to $V$. We are thus in position to use Lemma~\ref{minZ}. For any $1 \leq \alpha \leq d-1$, we thus infer that
\begin{equation} \label{eq:titi4}
\left| \int_\Omega Z_\alpha \left( \frac{x'}{\eps},x_d \right) \partial_\alpha (u^\star+g)(x') \cdot \nabla^\eps \overline{v}^\eps(x) \, dx \right| \leq C \, \eps \, |\omega|^{1/2} \, \| \nabla^2 (u^\star+g) \|_{L^\infty(\omega)} \, \| \nabla^\eps \overline{v}^\eps \|_{L^2(\Omega)}.
\end{equation}
Collecting~\eqref{estim2}, \eqref{eq:titi}, \eqref{eq:titi3} and~\eqref{eq:titi4}, we have shown that
\begin{multline} \label{eq:vendredi_}
\| \nabla^\eps \overline{v}^\eps \|_{L^2(\Omega)}^2 \leq C \left( \sqrt{\eps} \, |\omega|^{\frac{d-2}{2(d-1)}} \, \| \nabla (u^\star+g) \|_{L^\infty(\omega)} + \eps \, |\omega|^{1/2} \, \| \nabla^2 (u^\star+g) \|_{L^\infty(\omega)} \right) \|\nabla^\eps \overline{v}^\eps \|_{L^2(\Omega)} \\ + |R^\eps_1| + |R^\eps_2|.
\end{multline}
In the next step, we are going to show that
\begin{equation} \label{eq:titi5}
|R_1^\eps| + |R_2^\eps| \leq C \, \eps \left( |\omega|^{1/2} \, \| \nabla^2 (u^\star+g) \|_{L^\infty(\omega)} + \|h_\pm \|_{L^2(\omega)} + \| f \|_{L^2(\Omega)} \right) \| \nabla^\eps \overline{v}^\eps \|_{L^2(\Omega)}.
\end{equation}
Inserting this bound in~\eqref{eq:vendredi_}, we deduce that
\begin{multline*}
\| \nabla^\eps \overline{v}^\eps \|_{L^2(\Omega)} \leq C \, \Big( \sqrt{\eps} \, |\omega|^{\frac{d-2}{2(d-1)}} \, \| \nabla (u^\star+g) \|_{L^\infty(\omega)} + \eps \, |\omega|^{1/2} \, \| \nabla^2 (u^\star+g) \|_{L^\infty(\omega)} \\ + \eps \, \|h_\pm \|_{L^2(\omega)} + \eps \, \| f \|_{L^2(\Omega)} \Big).
\end{multline*}
Since $\overline{v}^\eps$ belongs to $V$, we can use the estimate~\eqref{eq:poinK_corro} stated in Lemma~\ref{poincarebis}, and we obtain from the above bound that
\begin{multline} \label{eq:france_inter}
\| \overline{v}^\eps \|_{H^1_\eps(\Omega)} \leq C \, \sqrt{\eps} \, \Big( |\omega|^{\frac{d-2}{2(d-1)}} \, \|\nabla (u^\star+g) \|_{L^\infty(\omega)} + \sqrt{\eps} \, |\omega|^{1/2} \, \| \nabla^2 (u^\star+g) \|_{L^\infty(\omega)} \\ + \sqrt{\eps} \, \|h_\pm \|_{L^2(\omega)} + \sqrt{\eps} \, \| f \|_{L^2(\Omega)} \Big).
\end{multline}

\medskip

\noindent
{\bf Step~3.} We show in this step the claimed bound~\eqref{eq:titi5} on $R^\eps_1$ and $R^\eps_2$. The bound on $R^\eps_2$ stems from inserting in~\eqref{eq:r2} the definition of $u^{\eps,1}$. Using again the short-hand notation $\dps \nabla w^\alpha_\eps(x) = \nabla w^\alpha \left(\frac{x'}{\eps},x_d \right)$, we compute
\begin{align*}
  R^\eps_2
  &=
  \int_\Omega A^\eps \left[ \nabla^\eps (u^{\eps,1}+g) - \left( e_\alpha + \nabla w^\alpha_\eps \right) \partial_\alpha (u^\star+g) \right] \cdot \nabla^\eps \overline{v}^\eps
  \\
  &= \eps \int_\Omega A^\eps(x) \left[ w^\alpha \left(\frac{x'}{\eps},x_d \right) \nabla \partial_\alpha (u^\star+g)(x') \right] \cdot \nabla^\eps \overline{v}^\eps(x),
\end{align*}
and thus
\begin{equation} \label{eq:bound_R2eps}
  | R^\eps_2 |
  \leq
  C \, \eps \, \| \nabla^2 (u^\star+g) \|_{L^2(\Omega)} \, \| \nabla^\eps \overline{v}^\eps \|_{L^2(\Omega)}
  \leq
  C \, \eps \, |\omega|^{1/2} \, \| \nabla^2 (u^\star+g) \|_{L^\infty(\omega)} \, \| \nabla^\eps \overline{v}^\eps \|_{L^2(\Omega)}.
\end{equation}
We now bound $R^\eps_1$ defined by~\eqref{eq:r1}. Using the variational formulation~\eqref{var:diff} with the test function $\overline{v}^\eps \in V$ defined in Step~2, the first term of $R^\eps_1$ reads
\begin{equation} \label{eq:laposte1}
\int_\Omega A^\eps \nabla^\eps (u^\eps+g) \cdot \nabla^\eps \overline{v}^\eps = \int_\Omega f \, \overline{v}^\eps + \int_{\Gamma_\pm} h_\pm \, \overline{v}^\eps,
\end{equation}
where $\dps \Gamma_\pm = \omega \times \left\{ \pm \frac{1}{2} \right\}$. Using~\eqref{eq:titi6}, the second term of $R^\eps_1$ reads 
\begin{align*}
  \int_\Omega \left[ \int_{\mathcal{Y}} A ( e_\alpha + \nabla w^\alpha) \right] \partial_\alpha (u^\star+g) \cdot \nabla^\eps \overline{v}^\eps
  &=
  \int_\Omega A^\star_{\alpha \beta} \, \partial_\alpha (u^\star+g) \, \partial_\beta \overline{v}^\eps
  \\
  &=
  \int_\Omega A^\star \nabla' (u^\star+g) \cdot \nabla' \overline{v}^\eps
  \\
  &=
  \int_\omega A^\star \nabla' (u^\star+g) \cdot \nabla' \m(\overline{v}^\eps),
\end{align*}
where we have used in the last line that $u^\star+g$ is independent from $x_d$. We next proceed by using the variational formulation~\eqref{eq:fv_star} of the homogenized problem and the fact that $\m(\overline{v}^\eps)$ belongs to $H^1_0(\omega)$ (and is thus an admissible test function for~\eqref{eq:fv_star}). We thus obtain that the second term of $R^\eps_1$ reads 
\begin{equation} \label{eq:laposte2}
\int_\Omega \left[ \int_{\mathcal{Y}} A ( e_\alpha + \nabla w^\alpha) \right] \partial_\alpha (u^\star+g) \cdot \nabla^\eps \overline{v}^\eps
=
\int_\omega (\m(f) + h_\pm) \m(\overline{v}^\eps)
=
\int_\Omega f \, \m(\overline{v}^\eps) + \int_\omega h_\pm \, \m(\overline{v}^\eps).
\end{equation}
We thus deduce from~\eqref{eq:laposte1} and~\eqref{eq:laposte2} that
\begin{equation} \label{eq:laposte3}
  |R^\eps_1|
  \leq
  \left| \int_\Omega f \, \big( \overline{v}^\eps - \m(\overline{v}^\eps) \big) \right|
  +
  \left| \int_{\Gamma_\pm} h_\pm \, \overline{v}^\eps - \int_\omega h_\pm \, \m(\overline{v}^\eps) \right|.
\end{equation}
Using Lemma~\ref{potmoy} with $z=\pm 1/2$, we can bound the second term of~\eqref{eq:laposte3}:
\begin{equation} \label{eq:laposte4}
  \left| \int_{\Gamma_\pm} h_\pm \, \overline{v}^\eps - \int_\omega h_\pm \, \m(\overline{v}^\eps) \right|
  \leq
  \eps \, \|h_\pm \|_{L^2(\omega)} \, \| \nabla^\eps \overline{v}^\eps \|_{L^2(\Omega)}.
\end{equation}
For the first term of~\eqref{eq:laposte3}, we write, using again Lemma~\ref{potmoy}, that
$$
\left| \int_\omega f(\cdot,z) \, \big( \overline{v}^\eps(\cdot,z) - \m(\overline{v}^\eps) \big) \right|
\leq
\eps \, \| f(\cdot,z) \|_{L^2(\omega)} \, \| \nabla^\eps \overline{v}^\eps \|_{L^2(\Omega)},
$$
and thus
\begin{equation} \label{eq:laposte5}
\left| \int_\Omega f \, \big( \overline{v}^\eps - \m(\overline{v}^\eps) \big) \right|
\leq
\eps \int_{-1/2}^{1/2} \| f(\cdot,z) \|_{L^2(\omega)} \, \| \nabla^\eps \overline{v}^\eps \|_{L^2(\Omega)} \, dz
\leq
\eps \, \| f \|_{L^2(\Omega)} \, \| \nabla^\eps \overline{v}^\eps \|_{L^2(\Omega)}.
\end{equation}
Collecting~\eqref{eq:laposte3}, \eqref{eq:laposte4} and~\eqref{eq:laposte5}, we deduce that
$$
|R^\eps_1| \leq \eps \, \big( \| f \|_{L^2(\Omega)} + \|h_\pm \|_{L^2(\omega)} \big) \, \| \nabla^\eps \overline{v}^\eps \|_{L^2(\Omega)}.
$$
Collecting this estimate with~\eqref{eq:bound_R2eps}, we obtain the claimed bound~\eqref{eq:titi5}.

\medskip

\noindent
{\bf Step~4.} Collecting~\eqref{diffapp3_H1} and~\eqref{eq:france_inter}, we deduce~\eqref{eq:titi7}, which concludes the proof of Theorem~\ref{thconvforte_diffusion}.
\end{proof}

\section{The elasticity case}\label{sec:elasticity}

We now turn to the linear elasticity problem.

\subsection{Definition of the plate problem} \label{sec:def_plate_pb_elas}

We use here the same notations as introduced at the beginning of Section~\ref{sec:diff}. We furthermore introduce some notations specific to the elasticity case considered in this Section~\ref{sec:elasticity}. First, for any $A,B\in\R^{d \times d}$, we denote by $A:B = A_{ij} \, B_{ij}$. In particular, for any $1 \leq i,j\leq d$ and any $A\in\R^{d \times d}$, we have $A : (e_i \otimes e_j) = e_i^T A e_j$.

Recalling that $\mathbb{R}^{d\times d}_s$ is the set of symmetric matrices of size $d \times d$, we denote by $\mathcal{M}$ the set of tensors $M \in \R^{d \times d \times d \times d}$ such that
$$
\forall \xi \in \mathbb{R}^{d\times d}_s, \qquad |M \xi| \leq c_+ \, |\xi| \quad \text{and} \quad \xi : M \xi \geq c_- \, |\xi|^2,
$$
and which have the following symmetries: 
\begin{equation} \label{eq:symm_A}
\forall 1 \leq i,j,k,l \leq d, \quad M_{ijkl} = M_{jikl} = M_{ijlk} = M_{klij}.
\end{equation}
Let $\dps A: \mathbb{R}^{d-1} \times \left( -\frac{1}{2}, \frac{1}{2} \right) \to \mathcal{M}$ be such that, for any $\dps x_d \in \left(-\frac{1}{2},\frac{1}{2} \right)$, the function $x' \in \R^{d-1} \mapsto A(x',x_d)$ is $Y$-periodic. For any $x = (x',x_d) \in \Omega$, we set
\begin{equation} \label{eq:utile_elas}
A^\eps(x) = A \left( \frac{x'}{\eps}, x_d \right).
\end{equation}
In addition, we define $\mathcal{A}^\eps$ by
$$
\forall x \in \Omega^\eps, \quad \mathcal{A}^\eps(x):=A^\eps\left(x', \frac{x_d}{\eps}\right) = A\left( \frac{x'}{\eps}, \frac{x_d}{\eps} \right).
$$
We denote
$$
V^\eps := \left\{v \in \left( H^1(\Omega^\eps) \right)^d, \quad v=0 \text{ on } \partial \omega \times \left( -\frac{\eps}{2},\frac{\eps}{2} \right) \right\}.
$$
For any $v \in \left( {\mathcal{D}}'(\R^d) \right)^d$, let $\dps e(v):= \frac{1}{2} \left( \nabla v + \nabla v^T \right)$ denote the symmetric gradient of $v$.

\medskip

For any $\eps >0$, let $\widetilde{f}^\eps \in \left( L^2(\Omega^\eps) \right)^d$, $\widetilde{g}^\eps \in \left( H^1(\Omega^\eps) \right)^d$ and $\widetilde{h}^\eps_\pm \in \left( L^2(\omega) \right)^d$. The plate linear elasticity problem reads as follows: find $\widetilde{u}^\eps \in V^\eps$ such that
\begin{equation}\label{pbelas_eps}
\left\{
\begin{aligned}
  -\div (\mathcal{A}^\eps e(\widetilde{u}^\eps)) &= \widetilde{f}^\eps + \div(\mathcal{A}^\eps e(\widetilde{g}^\eps)) \quad \text{in $\Omega^\eps$},
  \\
\mathcal{A}^\eps e(\widetilde{u}^\eps) \cdot n^\eps &= \eps \, \widetilde{h}^\eps_\pm - \mathcal{A}^\eps e(\widetilde{g}^\eps) \cdot n^\eps \quad \text{on $\omega \times \left\{ \pm \frac{\eps}{2} \right\}$}. 
\end{aligned}
\right.
\end{equation}
Similarly to the plate diffusion problem~\eqref{pb:diff1}, $\widetilde{f}^\eps$ in~\eqref{pbelas_eps} is the load imposed in $\Omega^\eps$. The function $\widetilde{g}^\eps$ is inserted as a possible extension of a non-trivial Dirichlet boundary condition (so that $\widetilde{u}^\eps + \widetilde{g}^\eps$ does not necessarily vanish on $\partial \omega \times \left( -\eps/2,\eps/2 \right)$). Note that, in contrast to the plate diffusion problem, $\widetilde{g}^\eps$ may depend on $x_d$. The function $\widetilde{h}^\eps_\pm$ plays the role of a Neumann boundary condition (i.e. a traction boundary condition for this elasticity problem) on the top and bottom faces of the plate $\Omega^\eps$.

\medskip

As already pointed out in Section~\ref{sec:diff}, we consider the case of non-homogeneous Dirichlet boundary conditions and we track below the dependence of our estimates with respect to the size of $\omega$ since we have in mind the application of these results to the numerical analysis of MsFEM approaches (see~\cite{MsFEM-ELLL}), where these two points are needed.

\medskip

As in the diffusion case, we are going to rescale the problem in the $d$-direction, in order to work with problems posed on a domain $\Omega$ independent of $\eps$ (recall Figure~\ref{fig:rescaling}). We follow here~\cite[eq.~(1.4)]{dauge} (see also~\cite[Sec.~1.3]{ciarlet1988mathematical_vol2}).

\medskip

For any $u \in \left( \mathcal{D}'(\R^d) \right)^d$ and $T \in \left( \mathcal{D}'(\R^d) \right)^{d \times d}$, we define the scaled operators $e^\eps$ and $\div^\eps$ by
$$
e^\eps_{\alpha \beta}(u) := e_{\alpha \beta}(u), \qquad e^\eps_{\alpha d}(u) := \frac{1}{\eps} \, e_{\alpha d}(u) \qquad \text{and} \qquad e^\eps_{dd}(u) := \frac{1}{\eps^2} \, e_{dd}(u)
$$
for any $1 \leq \alpha, \beta \leq d-1$, and
$$
\begin{cases}
  \dps [ \div^\eps T ]_\alpha := \partial_\beta T_{\alpha \beta} + \frac{1}{\eps} \, \partial_d T_{\alpha d} = \sum_{\beta = 1}^{d-1} \partial_\beta T_{\alpha \beta} + \frac{1}{\eps} \, \partial_d T_{\alpha d},
  \\
  \dps [ \div^\eps T]_d := \frac{1}{\eps} \, \partial_\beta T_{d \beta} + \frac{1}{\eps^2} \, \partial_d T_{dd} = \sum_{\beta = 1}^{d-1} \frac{1}{\eps} \, \partial_\beta T_{d \beta} + \frac{1}{\eps^2} \, \partial_d T_{dd}.
  \end{cases}
$$
For any $u \in \left( \mathcal{D}(\Omega) \right)^d$ and any $T \in \left( \mathcal{D}(\Omega) \right)^{d \times d}$ such that $T(x)$ is a symmetric matrix for any $x \in \Omega$, we have $\dps \int_\Omega T : e^\eps(u) = - \int_\Omega u \cdot \div^\eps T$. We denote
\begin{equation}\label{def:V2}
V := \left\{ v \in \left( H^1(\Omega) \right)^d, \quad v=0 \text{ on } \partial\omega \times \left( -\frac{1}{2}, \frac{1}{2} \right) \right\}.
\end{equation} 
It can then be easily seen (see Appendix~\ref{app:scaling_vectoriel} for details) that problem~\eqref{pbelas_eps} is equivalent to finding $u^\eps \in V$ such that
\begin{equation} \label{pbelas}
\left\{
\begin{aligned}
  - \div^\eps (A^\eps e^\eps(u^\eps)) &= f^\eps + \div^\eps(A^\eps e^\eps(g^\eps)) \quad \text{in $\Omega$},
  \\
  A^\eps e^\eps(u^\eps) \cdot n &= \begin{pmatrix} \eps \, (h^\eps_\pm)' \\ \eps^2 \, (h^\eps_\pm)_d \end{pmatrix} - A^\eps e^\eps(g^\eps) \cdot n \quad \text{on $\omega \times \left\{\pm \frac{1}{2}\right\}$},
\end{aligned}
\right.
\end{equation}
where $A^\eps$ is given by~\eqref{eq:utile_elas} and where, for any $x = (x',x_d) \in \Omega$ and any $1 \leq \alpha \leq d-1$,
\begin{equation} \label{eq:scaling_u_el}
  \begin{array}{rlcrl}
  u_\alpha^\eps(x) &= \widetilde{u}_\alpha^\eps(x',\eps \, x_d) \qquad &\text{and}& \qquad u_d^\eps(x) &= \eps \, \widetilde{u}_d^\eps(x',\eps \, x_d),
  \\ \noalign{\vskip 2pt}
  f_\alpha^\eps(x) &= \widetilde{f}_\alpha^\eps(x',\eps \, x_d) \qquad &\text{and}& \qquad f_d^\eps(x) &= \eps^{-1} \, \widetilde{f}_d^\eps(x',\eps \, x_d),
  \\ \noalign{\vskip 2pt}
  h_\alpha^\eps(x') &= \widetilde{h}_\alpha^\eps(x') \qquad &\text{and}& \qquad h_d^\eps(x') &= \eps^{-1} \, \widetilde{h}_d^\eps(x'),
  \\ \noalign{\vskip 2pt}
  g_\alpha^\eps(x) &= \widetilde{g}_\alpha^\eps(x',\eps \, x_d) \qquad &\text{and}& \qquad g_d^\eps(x) &= \eps \, \widetilde{g}_d^\eps(x',\eps \, x_d),
  \end{array}
\end{equation}
where, in the third line, we have written $h$ as a short-hand notation for $h_\pm$. Note that the scaled operator $e^\eps$ and the displacement $u^\eps$ are defined in such a way that $(e^\eps(u^\eps))(x) = (e(\widetilde{u}^\eps))(x',\eps \, x_d)$.

\medskip

The variational formulation of~\eqref{pbelas} reads as follows:
\begin{equation} \label{formvarelast}
\text{Find $u^\eps \in V$ such that, for any $v \in V$}, \qquad a^\eps(u^\eps,v) = b^\eps(v),
\end{equation}
where
$$
a^\eps(u^\eps,v) := \int_\Omega A^\eps e^\eps(u^\eps) : e^\eps(v)
$$
and
$$
b^\eps(v) := \int_\Omega f^\eps \cdot v - \int_\Omega A^\eps e^\eps(g^\eps) : e^\eps(v) + \int_\omega h^\eps_+ \cdot v\left(\cdot,\frac{1}{2}\right) + \int_\omega h^\eps_- \cdot v\left(\cdot,-\frac{1}{2}\right).
$$ 
We are going to use the following Korn inequality (see~\cite{ciarlet1988mathematical}), a proof of which is given in Appendix~\ref{app:korn} (see Lemma~\ref{lem:Korn3}): there exists $C(\Omega)$ such that
\begin{equation} \label{eq:korn}
  \forall u \in V, \quad \| u \|_{H^1(\Omega)} \leq C(\Omega) \, \| e(u) \|_{L^2(\Omega)}.
\end{equation}
Using this inequality, the fact that $\| e(u) \|_{L^2(\Omega)} \leq \| e^\eps(u) \|_{L^2(\Omega)}$ for any $u \in V$ (since we have assumed that $0 < \eps \leq 1$) and the Lax-Milgram theorem, one easily obtains that there exists a unique solution to~\eqref{formvarelast}.

\medskip

We now establish a priori bounds on $u^\eps$. Taking $v = u^\eps$ in~\eqref{formvarelast} and using a trace inequality, we get
\begin{multline*}
c_- \| e^\eps(u^\eps) \|^2_{L^2(\Omega)} \leq \|f^\eps \|_{L^2(\Omega)} \, \| u^\eps \|_{L^2(\Omega)} \\ + c_+ \, \|e^\eps(g^\eps)\|_{L^2(\Omega)} \, \| e^\eps(u^\eps) \|_{L^2(\Omega)} + C \, \|h^\eps_\pm\|_{L^2(\omega)} \, \| u^\eps \|_{H^1(\Omega)}.
\end{multline*}
Using~\eqref{eq:korn} and the fact that $\| e(u^\eps) \|_{L^2(\Omega)} \leq \| e^\eps(u^\eps) \|_{L^2(\Omega)}$, we deduce that
\begin{multline*}
  c_- \| e^\eps(u^\eps) \|^2_{L^2(\Omega)} \leq C \, \|f^\eps \|_{L^2(\Omega)} \, \| e^\eps(u^\eps) \|_{L^2(\Omega)} \\ + c_+ \, \|e^\eps(g^\eps)\|_{L^2(\Omega)} \, \| e^\eps(u^\eps) \|_{L^2(\Omega)} + C \, \|h^\eps_\pm\|_{L^2(\omega)} \, \| e^\eps(u^\eps) \|_{L^2(\Omega)}
\end{multline*}
and thus
\begin{equation}\label{est:sig}
  \| e^\eps(u^\eps) \|_{L^2(\Omega)} \leq C \left( \|f^\eps \|_{L^2(\Omega)} + \|e^\eps(g^\eps)\|_{L^2(\Omega)} + \|h^\eps_\pm\|_{L^2(\omega)} \right)
\end{equation}
for some constant $C>0$ independent of $\eps$. Using again~\eqref{eq:korn}, we infer that
\begin{equation}\label{est:u}
 \| u^\eps \|_{H^1(\Omega)} \leq C \left( \|f^\eps \|_{L^2(\Omega)} + \|e^\eps(g^\eps)\|_{L^2(\Omega)} + \|h^\eps_\pm\|_{L^2(\omega)} \right)
\end{equation}
for some constant $C>0$ independent of $\eps$.

\medskip

To obtain bounds independent of $\eps>0$ on $\left( \| u^\eps \|_{H^1(\Omega)} \right)_{\eps>0}$ and $\left( \| e^\eps(u^\eps) \|_{L^2(\Omega)} \right)_{\eps >0}$, it is convenient to work with sequences $\left( \|f^\eps \|_{L^2(\Omega)} \right)_{\eps >0}$, $\left( \|e^\eps(g^\eps)\|_{L^2(\Omega)} \right)_{\eps >0}$ and $\left( \|h^\eps_\pm\|_{L^2(\omega)} \right)_{\eps>0}$ which are bounded. 

\medskip

From now on in this Section~\ref{sec:elasticity}, we assume that there exist $f \in (L^2(\Omega))^d$ and $h_\pm \in (L^2(\omega))^d$ such that, for all $\eps >0$, 
\begin{equation} \label{eq:inde_eps_vec}
f^\eps = f \quad \mbox{ and } \quad h_\pm^\eps = h_\pm.
\end{equation}
As pointed out in Remark~\ref{rem:hyp_f_ind_eps} in the diffusion case, the assumptions on $(f^\eps)_{\eps >0}$ and $(h_\pm^\eps)_{\eps >0}$ could probably be weakened. We postpone the precise assumption we make on the sequence $(g^\eps)_{\eps >0}$ until Section~\ref{sec:weakel}. Provided that the sequence $\left(\|e^\eps(g^\eps)\|_{L^2(\Omega)} \right)_{\eps >0}$ is bounded, we infer from~\eqref{est:u} and~\eqref{eq:inde_eps_vec} that there exists $u^\star \in V$ such that, up to the extraction of a subsequence, 
$$
u^\eps \underset{\eps \rightarrow 0}{\rightharpoonup} u^\star \text{ weakly in } \left(H^1(\Omega)\right)^d.
$$

\subsection{Homogenization result: weak convergence}\label{sec:weakel}

The aim of this section is to recall the result of~\cite{caillerieElasticite}, where the homogenized limit of~\eqref{formvarelast} has been identified. We introduce the set
$$
\mathcal{W}(\mathcal{Y}) := \left\{ v \in \left(H^1_{\rm loc}\left(\mathbb{R}^{d-1} \times \left( -\frac{1}{2},\frac{1}{2} \right)\right)\right)^d, \forall z \in \left( -\frac{1}{2},\frac{1}{2} \right), \, v(\cdot,z) \text{ is $Y$-periodic and } \int_{\mathcal{Y}} v = 0 \right\}.
$$
Let $w^{\alpha \beta} \in \mathcal{W}(\mathcal{Y})$ be the unique solution to the problem
\begin{equation} \label{el-prcor1} 
\forall v \in \mathcal{W}(\mathcal{Y}), \quad \int_{\mathcal{Y}} A \big( e(w^{\alpha \beta}) + e_\alpha \otimes e_\beta \big) : e(v) = 0 
\end{equation}
for any $1\leq \alpha,\beta \leq d-1$. Note that, because of the symmetry properties of $A$, we also have $\dps \int_{\mathcal{Y}} A \left( e(w^{\alpha \beta}) + \frac{1}{2} (e_\alpha \otimes e_\beta + e_\beta \otimes e_\alpha) \right) : e(v) = 0$ for any $v \in \mathcal{W}(\mathcal{Y})$, and thus $w^{\alpha \beta} = w^{\beta \alpha}$. The function $w^{\alpha \beta}$ is equivalently the unique solution in $\mathcal{W}(\mathcal{Y})$ to
\begin{equation} \label{el-prcor1_edp} 
\left\{
\begin{aligned}
-\div A \big( e(w^{\alpha \beta}) + e_\alpha \otimes e_\beta \big) &= 0 \quad \text{ in } \mathcal{Y}, \\ 
\left[ A \big( e(w^{\alpha \beta}) + e_\alpha \otimes e_\beta \big) \right] e_d &= 0 \quad \text{ on } \mathcal{Y}_\pm,
\end{aligned}
\right.
\end{equation}
where $\dps \mathcal{Y}_\pm := Y \times \left\{\pm \frac{1}{2} \right\}$. 

\medskip

Let $W^{\alpha \beta} \in \mathcal{W}(\mathcal{Y})$ be the unique solution to the problem
\begin{equation} \label{el-prcor2} 
  \forall v \in \mathcal{W}(\mathcal{Y}), \quad \int_{\mathcal{Y}} A \big( e(W^{\alpha \beta}) - x_d \, e_\alpha \otimes e_\beta \big) : e(v) = 0
\end{equation}
for any $1\leq \alpha,\beta \leq d-1$. The function $W^{\alpha \beta}$ is equivalently the unique solution in $\mathcal{W}(\mathcal{Y})$ to
\begin{equation} \label{el-prcor2_edp} 
  \left\{
  \begin{aligned}
    -\div A \big( e(W^{\alpha \beta}) - x_d \, e_\alpha \otimes e_\beta \big) &= 0 \quad \text{ in } \mathcal{Y}, \\
    \left[ A \big( e(W^{\alpha \beta}) - x_d \, e_\alpha \otimes e_\beta \big) \right] e_d &= 0 \quad \text{ on } \mathcal{Y}_\pm.
  \end{aligned}
  \right.
\end{equation}	

\medskip

We define the set of the {\em Kirchoff-Love displacements} as follows:
$$
\VKL := \left\{ v \in \left(H^1(\Omega)\right)^{d-1} \times H^2_0(\omega), \, \exists \widehat{v} \in \left(H^1_0(\omega)\right)^{d-1} \times H^2_0(\omega), \, v_\alpha = \widehat{v}_\alpha - x_d \, \partial_\alpha \widehat{v}_d, \, v_d = \widehat{v}_d \right\},
$$
where $H^2_0(\omega)$ is the closure of $\mathcal{D}(\omega)$ in $H^2(\omega)$. For any $v \in \VKL$, we use the notation $\widehat{v}$ to denote the (unique) corresponding element of $\left(H^1_0(\omega)\right)^{d-1} \times H^2_0(\omega)$.

Let us also denote 
$$
\GKL := \left\{ g \in \left(H^1(\Omega)\right)^{d-1} \times H^2(\omega), \, \exists \widehat{g} \in \left(H^1(\omega)\right)^{d-1} \times H^2(\omega), \, g_\alpha = \widehat{g}_\alpha - x_d \, \partial_\alpha \widehat{g}_d, \, g_d = \widehat{g}_d \right\}.
$$
For any $g \in \GKL$, we denote by $\widehat{g}$ the (unique) corresponding element of $\left(H^1(\omega)\right)^{d-1} \times H^2(\omega)$.

It then holds that $\VKL \subset \GKL$ and that, for any $g \in \GKL$, the sequence $\left( \|e^\eps(g)\|_{L^2(\Omega)} \right)_{\eps >0}$ is bounded. Observe indeed that
$$
\forall g \in \GKL, \quad \forall 1 \leq i \leq d, \quad e^\eps_{id}(g) = 0,
$$
and therefore $e^\eps(g) = e(g)$ for any $g \in \GKL$. In addition, for any $1 \leq \alpha,\beta \leq d-1$, we have $e_{\alpha \beta}(g) = e_{\alpha \beta}(\widehat{g}) - x_d \, \partial_{\alpha \beta} \widehat{g}_d$. 

From now on, we make the assumption that the function $g^\eps$ appearing in the linear form $b^\eps$ of~\eqref{formvarelast} is such that
\begin{equation} \label{eq:assump_g}
  \text{there exists $g \in \GKL$ such that, for all $\eps>0$, we have $g^\eps = g$.}
\end{equation}  

\medskip

We recall the following result from~\cite{caillerieElasticite} (for the sake of completeness, a partial proof of Theorem~\ref{limitel} is provided in Appendices~\ref{app:diff_preuve2_memb} and~\ref{app:diff_preuve2_bend}). 

\begin{theorem}[from Theorem~6.2 of~\cite{caillerieElasticite}] \label{limitel}
Under Assumptions~\eqref{eq:inde_eps_vec} and~\eqref{eq:assump_g}, the sequence $(u^\eps)_{\eps >0}$ of solutions to~\eqref{formvarelast} weakly converges to $u^\star$ in $\left(H^1(\Omega)\right)^d$, where the function $u^\star$ belongs to $\VKL$ and is the unique solution to
\begin{multline} \label{formvarelasthomog}
  \forall \phi \in \VKL, \quad \int_\omega K^\star \, \mathcal{P} u^\star : \mathcal{P} \phi = \int_\omega (\m(f)+h_++h_-) \cdot \widehat{\phi} - \int_\omega \m(x_d \, f_\alpha) \, \partial_\alpha \widehat{\phi}_d \\ - \frac{1}{2} \int_\omega \big( (h_+)_\alpha - (h_-)_\alpha \big) \, \partial_\alpha \widehat{\phi}_d - \int_\omega K^\star \, \mathcal{P} g : \mathcal{P} \phi,
\end{multline}
where the homogenized tensor is given by
$$
K^\star := \begin{pmatrix} K^\star_{11} & K^\star_{12} \\ (K^\star_{12})^T & K^\star_{22} \end{pmatrix},
$$
where each subtensor is defined as follows: for all $1\leq \alpha, \beta, \gamma, \delta \leq d-1$,
\begin{align}
  (K^\star_{11})_{\alpha \beta \gamma \delta} &:= \int_\Y A \big( e(w^{\alpha \beta}) + e_\alpha \otimes e_\beta \big) : \big( e(w^{\gamma \delta}) + e_\gamma \otimes e_\delta \big),
  \nonumber
  \\
  (K^\star_{12})_{\alpha \beta \gamma \delta} &:= \int_\Y A \big( e(w^{\alpha \beta}) + e_\alpha \otimes e_\beta \big) : \big( e(W^{\gamma \delta}) - x_d \, e_\gamma \otimes e_\delta \big),
  \label{eq:def_Kstar12}
  \\
  (K^\star_{22})_{\alpha \beta \gamma \delta} &:= \int_\Y A \big( e(W^{\alpha \beta}) - x_d \, e_\alpha \otimes e_\beta \big) : \big( e(W^{\gamma \delta}) - x_d \, e_\gamma \otimes e_\delta \big),
  \label{eq:def_Kstar22}
  \\
  ((K^\star_{12})^T)_{\alpha \beta \gamma \delta} &:= (K^\star_{12})_{\gamma \delta \alpha \beta},
  \nonumber
\end{align}
and where
\begin{equation*}
\mathcal{P} : \left\{ 
\begin{array}{ccc}
\GKL &\rightarrow &L^2\left(\omega;\left( \R_s^{(d-1)\times (d-1)}\right)^2 \right) \\
v &\mapsto& \left(
\begin{array}{c} e'(\widehat{v}') \\ \nabla^2_{d-1} \widehat{v}_d \\ \end{array}\right)
\end{array}\right. .
\end{equation*}
In the above expression of $K^\star$, the correctors $w^{\alpha\beta}$ and $W^{\alpha\beta}$ are the solutions to~\eqref{el-prcor1_edp} and~\eqref{el-prcor2_edp}, respectively.

By construction, $K^\star$ is symmetric, in the sense that $(K^\star_{11})_{\alpha \beta \gamma \delta} = (K^\star_{11})_{\beta \alpha \gamma \delta} = (K^\star_{11})_{\alpha \beta \delta \gamma}$ (and likewise for $K^\star_{22}$ and $K^\star_{12}$) and $(K^\star_{11})_{\alpha \beta \gamma \delta} = (K^\star_{11})_{\gamma \delta \alpha \beta}$ (and likewise for $K^\star_{22}$). In addition, $K^\star$ is coercive, in the sense that there exists $\widetilde{c}_- > 0$ (one can take $\widetilde{c}_- = c_-/12$) such that, for any $\sigma$ and $\tau$ in $\R_s^{(d-1)\times(d-1)}$, we have $\dps K^\star \begin{pmatrix} \sigma \\ \tau \end{pmatrix} : \begin{pmatrix} \sigma \\ \tau \end{pmatrix} \geq \widetilde{c}_- \begin{pmatrix} \sigma \\ \tau \end{pmatrix} : \begin{pmatrix} \sigma \\ \tau \end{pmatrix}$.
\end{theorem}

In the definition of $\mathcal{P}$, we recall that, to any $v \in \GKL$ is associated a (unique) function $\widehat{v} \in \left(H^1(\omega)\right)^{d-1} \times H^2(\omega)$. We then write $\dps \widehat{v} = \begin{pmatrix} \widehat{v}' \\ \widehat{v}_d \end{pmatrix}$ with $\widehat{v}' \in \left(H^1(\omega)\right)^{d-1}$ and $\widehat{v}_d \in H^2(\omega)$. Then, $e'(\widehat{v}')$ is the $(d-1) \times (d-1)$ symmetric matrix given by $\dps \Big[ e'(\widehat{v}') \Big]_{\alpha \beta} = e_{\alpha \beta}(\widehat{v})$ and $\nabla^2_{d-1} \widehat{v}_d$ is the Hessian matrix of $\widehat{v}_d$, that is the $(d-1) \times (d-1)$ symmetric matrix given by $\dps \Big[ \nabla^2_{d-1} \widehat{v}_d \Big]_{\alpha \beta} = \partial_{\alpha \beta}\widehat{v}_d$.

\medskip

As is well-known, in the simple case when $K^\star_{12}$ vanishes, the homogenized problem~\eqref{formvarelasthomog} can be written in the form of two uncoupled problems. The first one consists in finding $U^\star \in \left(H^1_0(\omega)\right)^{d-1}$ (where $U^\star$ corresponds to $(\widehat{u^\star})'$, that is the $(d-1)$-dimensional vector collecting the $d-1$ first components of the KL representative of $u^\star$) such that
$$
\forall \Phi \in \left(H^1_0(\omega)\right)^{d-1}, \quad \int_\omega K^\star_{11} \, e'(U^\star) : e'(\Phi) = \int_\omega (\m(f')+h'_++h'_-) \cdot \Phi - \int_\omega K^\star_{11} \, e'(\widehat{g}') : e'(\Phi),
$$
where $e'(\Phi)$ is the $(d-1)$-dimensional symmetric gradient of $\Phi: \omega \to \R^{d-1}$. We hence observe that $U^\star$ is the solution to a linear elasticity problem posed in $\omega$, with a constant elasticity tensor.

The second one consists in finding $U^\star \in H^2_0(\omega)$ (where $U^\star$ now corresponds to $(\widehat{u^\star})_d$, that is the $d$-th component of the KL representative of $u^\star$) such that, for any $\Phi \in H^2_0(\omega)$,
\begin{multline*}
\int_\omega K^\star_{22} \, \nabla^2_{d-1} U^\star : \nabla^2_{d-1} \Phi = \int_\omega \big( \m(f_d)+(h_+)_d+(h_-)_d \big) \, \Phi - \int_\omega \m(x_d \, f_\alpha) \, \partial_\alpha \Phi \\ - \frac{1}{2} \int_\omega \big( (h_+)_\alpha - (h_-)_\alpha \big) \, \partial_\alpha \Phi - \int_\omega K^\star_{22} \, \nabla^2_{d-1} \widehat{g}_d : \nabla^2_{d-1} \Phi,
\end{multline*}
where $\nabla^2_{d-1} \Phi$ is the $(d-1) \times (d-1)$ Hessian matrix of $\Phi: \omega \to \R$. We hence observe that $U^\star$ is the solution to a beam-type problem posed in $\omega$.

\subsection{Strong convergence of a two-scale expansion: preliminary ingredients} \label{sec:prelim_el}

In what follows, we are going to state some strong convergence results similar to Theorem~\ref{thconvforte_diffusion} in the elasticity case. We mention that, in the case of homogeneous plates, expansions at an arbitrary order in $\eps$ (which then simply encodes the small thickness of the plate) were studied in~\cite{dauge}. 

We first present some preliminary results in Section~\ref{sec:prelel}. Then, in Section~\ref{sec:symmetry}, we present the symmetry assumptions on the elasticity tensor $A$ needed to state our results. These assumptions are classical, and using them, it is possible to split the problem of interest into two uncoupled problems, a membrane problem and a bending problem. We state our main results on strong convergence in the following sections, namely in Section~\ref{sec:membrane} for the membrane case (see Theorem~\ref{thconvforte2}) and in Section~\ref{sec:bending} for the bending case (see Theorem~\ref{thconvforte3_a}). 

We stress here the following point. While, in the membrane case, the proof of strong convergence follows from a careful adaptation of the arguments used in the diffusion case investigated in Section~\ref{sec:diff}, the proof in the bending case is much more involved and requires a different strategy of proof. We will discuss this in more details at the beginning of Section~\ref{sec:bending}. 

\subsubsection{Preliminary results}\label{sec:prelel} 

The aim of this section is to prove some auxiliary results which are useful in the sequel. 

\begin{lemma} \label{potmoy2}
  Let $h \in L^2(\omega)$ and $v \in (H^1(\Omega))^d$. Then, for any $\dps z \in \left[ -\frac{1}{2},\frac{1}{2} \right]$, we have
  $$
  \left| \int_\omega \big( v_d(\cdot,z)-\m(v_d) \big) \, h \right| \leq \eps^2 \, \| h \|_{L^2(\omega)} \, \| e^\eps(v) \|_{L^2(\Omega)}.
  $$
\end{lemma}
The proof of Lemma~\ref{potmoy2} is an easy adaptation of the proof of Lemma~\ref{potmoy}.

\begin{proof}
Let $v \in (H^1(\Omega))^d$. For any $z \in [-1/2,1/2]$, we have
\begin{align*}
  \left| v_d(\cdot,z)-\m(v_d) \right|
  &=
  \left| \int_{-1/2}^{1/2} \big( v_d(\cdot,z)-v_d(\cdot,t) \big) \, dt \right|
  \\
  &=
  \left| \int_{-1/2}^{1/2} \int_t^z \partial_d v_d(\cdot,s) \, ds \, dt \right|
  \\
  &\leq
  \eps^2 \int_{-1/2}^{1/2} \left| \frac{1}{\eps^2} \, \partial_d v_d(\cdot,s) \right| ds.
\end{align*}
It follows that
\begin{align*}
  \| v_d(\cdot,z)-\m(v_d) \|^2_{L^2(\omega)}
  &=
  \int_\omega \left| v_d(\cdot,z)-\m(v_d) \right|^2
  \\
  &
  \leq \eps^4 \int_\omega \left| \int_{-1/2}^{1/2} \left| \frac{1}{\eps^2} \, \partial_d v_d(\cdot,s) \right| ds \right|^2
  \\ 
  & \leq
  \eps^4 \, \left\| \frac{1}{\eps^2} \, \partial_d v_d \right\|_{L^2(\Omega)}^2
  \\
  & \leq \eps^4 \, \| e^\eps (v) \|_{L^2(\Omega)}^2.
\end{align*}	
For any $z \in [-1/2,1/2]$, we now write
$$
\left| \int_\omega \big( v_d(\cdot,z)-\m(v_d) \big) \, h \right| \leq \| h \|_{L^2(\omega)} \ \| v_d(\cdot,z)-\m(v_d) \|_{L^2(\omega)}.
$$
Collecting the above two estimates concludes the proof of Lemma~\ref{potmoy2}.
\end{proof}

\medskip

Our next result states some estimates on the $L^2$ norm of the trace on $\dps \Gamma_\pm = \omega \times \left\{ \pm \frac{1}{2} \right\}$ of a function $v \in L^2(\Omega)$ such that $\partial_d v \in L^2(\Omega)$.

\begin{lemma}\label{lem:trace}
  Let $v\in L^2(\Omega)$ such that $\partial_d v\in L^2(\Omega)$. We have
  \begin{equation}\label{eq:ineqtrace}
    \|v\|_{L^2(\Gamma_\pm)} \leq \sqrt{2} \left( \|v\|_{L^2(\Omega)} + \|\partial_d v\|_{L^2(\Omega)} \right).
  \end{equation}
\end{lemma}

\begin{proof}
For almost all $\dps (x',z) \in \omega \times (-1/2,1/2)$, we have 
$$
v\left(x', \pm \frac{1}{2}\right) = v(x',z) + \int_z^{\pm 1/2} \partial_d v(x',t)\,dt.
$$
Integrating with respect to $x'$, this implies that, for almost all $z \in (-1/2,1/2)$, 
\begin{align*}
  \int_{\Gamma_\pm}|v|^2
  & = \int_\omega \left| v\left(x', \pm\frac{1}{2} \right) \right|^2 \, dx'
  \\
  & \leq 2 \left( \int_\omega |v(x',z)|^2 \, dx' + \int_\omega \left| \int_z^{\pm 1/2} \partial_d v(x',t) \, dt \right|^2 \, dx' \right)
  \\
  & \leq 2 \left( \int_\omega |v(x',z)|^2 \, dx' + \int_\omega \int_{-1/2}^{1/2} \left| \partial_d v(x',t) \right|^2 \, dt \, dx' \right)
  \\
  & = 2 \left( \int_\omega |v(x',z)|^2 \, dx' + \int_\Omega \left| \partial_d v \right|^2 \right).
\end{align*}
Integrating the above inequality with respect to $z$ in $(-1/2,1/2)$ yields that
$$
\|v\|_{L^2(\Gamma_\pm)}^2 \leq 2 \left( \|v\|_{L^2(\Omega)}^2 + \|\partial_d v\|_{L^2(\Omega)}^2 \right),
$$
and thus the claimed result.
\end{proof}

\medskip

We also need the following Poincaré estimate, which is a generalization of Lemma~\ref{poincarebis} to the elasticity setting. To properly state this result, we introduce a weighted $L^2$ norm: for any $v \in (L^2(\Omega))^d$, we set
\begin{equation} \label{eq:def_L2_w}
\| v \|_{L^2_w(\Omega)} = \sqrt{ |\omega|^{\frac{-2}{d-1}} \, \| v_d \|^2_{L^2(\Omega)} + \sum_{\alpha=1}^{d-1} \| v_\alpha \|^2_{L^2(\Omega)} }.
\end{equation}

\begin{lemma}\label{poincarebis2}
  Let $V$ be defined by~\eqref{def:V2}. Assume that $\omega$ satisfies the shape regularity assumption~\eqref{eq:shape_regul}. There exists a constant $C>0$ independent of $\eps$ and $\omega$ (but depending on the constant $\eta$ of~\eqref{eq:shape_regul}) such that, for any $u$ in $V$,
  \begin{equation} \label{ineq:1_vec}
    \| u \|_{L^2_w(\Omega)} \leq C \, \max\left(|\omega|^{\frac{1}{d-1}}, \eps^2 \, |\omega|^{-\frac{1}{d-1}}\right) \| e^\eps(u) \|_{L^2(\Omega)}.
  \end{equation}
\end{lemma}

\begin{proof}
For any smooth bounded domain $\widetilde{\omega} \subset \R^{d-1}$, we denote 
$$
V(\widetilde{\Omega}) := \left\{ v \in \left( H^1(\widetilde{\Omega}) \right)^d, \quad v=0 \text{ on } \partial \widetilde{\omega} \times \left( -\frac{1}{2},\frac{1}{2} \right) \right\}
$$
where $\dps \widetilde{\Omega}:= \widetilde{\omega} \times \left( -\frac{1}{2}, \frac{1}{2}\right)$.

Let $\widehat{\omega} := (0,1)^{d-1}$ and $\dps \widehat{\Omega} := \widehat{\omega} \times \left(-\demi, \demi \right)$. Using the Korn inequality in $V(\widehat{\Omega})$ (see Lemma~\ref{lem:Korn3}), there exists some constant $C_{\rm Korn}(\widehat{\Omega})$ such that
\begin{equation}
\label{eq:toto200}
\forall \widehat{u} \in V(\widehat{\Omega}), \quad \| \widehat{u} \|_{L^2(\widehat{\Omega})} \leq C_{\rm Korn}(\widehat{\Omega}) \, \| e(\widehat{u}) \|_{L^2(\widehat{\Omega})}.
\end{equation}
To prove~\eqref{ineq:1_vec}, we proceed by scaling. Introduce $\omega_K := (0,K)^{d-1}$ and $\dps \Omega_K := \omega_K \times \left(-\demi, \demi \right)$ for some $K>0$. For any $u_K \in V(\Omega_K)$, we define the function $\widehat{u}$ on $\widehat{\Omega}$ by
$$
\forall x \in \widehat{\Omega}, \qquad \widehat{u}_\alpha(x) = K \, u_{K,\alpha}(K \, x', x_d), \qquad \widehat{u}_d(x) = u_{K,d}(K \, x', x_d).
$$
The function $\widehat{u}$ belongs to $V(\widehat{\Omega})$ and simple computations show that
\begin{align*}
  \| u_K \|_{L^2_w(\Omega_K)}^2 &= K^{d-1} \left( K^{-2} \, \| \widehat{u}' \|_{L^2(\widehat{\Omega})}^2 + K^{-2} \, \| \widehat{u}_d \|_{L^2(\widehat{\Omega})}^2 \right),
  \\
  \| e^\eps_{\alpha \beta}(u_K) \|_{L^2(\Omega_K)}^2 &= K^{d-5} \, \| e_{\alpha \beta}(\widehat{u}) \|_{L^2(\widehat{\Omega})}^2,
  \\
  \| e^\eps_{\alpha d}(u_K) \|_{L^2(\Omega_K)}^2 &= K^{d-3} \, \eps^{-2} \, \| e_{\alpha d}(\widehat{u}) \|_{L^2(\widehat{\Omega})}^2,
  \\
  \| e^\eps_{dd}(u_K) \|_{L^2(\Omega_K)}^2 &= K^{d-1} \, \eps^{-4} \, \| e_{dd}(\widehat{u}) \|_{L^2(\widehat{\Omega})}^2,
\end{align*}
where we recall that the weighted $L^2_w(\Omega_K)$ norm is defined by~\eqref{eq:def_L2_w}. Using~\eqref{eq:toto200}, we thus get that
\begin{align*}
  \| u_K \|_{L^2_w(\Omega_K)}^2
  & \leq K^{d-3} \, \| \widehat{u} \|_{L^2(\widehat{\Omega})}^2
  \\
  & \leq C_{\rm Korn}(\widehat{\Omega})^2 \, K^{d-3} \, \| e(\widehat{u}) \|_{L^2(\widehat{\Omega})}^2
  \\
  & \leq C_{\rm Korn}(\widehat{\Omega})^2 \, \max\left(K^2, \eps^2, \eps^4 \, K^{-2} \right) \| e^\eps(u_K) \|_{L^2(\Omega_K)}^2.
\end{align*}
Using that $\max\left(1, \eps^2 \, K^{-2}, \eps^4 \, K^{-4} \right) = \max\left(1, \eps^4 \, K^{-4} \right)$ and that $K = |\omega_K|^{\frac{1}{d-1}}$, we deduce that
$$
\| u_K \|_{L^2_w(\Omega_K)} \leq C_{\rm Korn}(\widehat{\Omega}) \, \max\left(|\omega_K|^{\frac{1}{d-1}}, \eps^2 \, |\omega_K|^{-\frac{1}{d-1}}\right) \| e^\eps(u_K) \|_{L^2(\Omega_K)},
$$
which proves~\eqref{ineq:1_vec} in the case $\omega=\omega_K$. In the case of a more general, shape regular domain $\omega$, the proof can be performed using the same arguments. This concludes the proof of Lemma~\ref{poincarebis2}.
%
\end{proof}

\subsubsection{Use of symmetries}\label{sec:symmetry}

In all what follows, we make additional symmetry assumptions on the problem. These assumptions enable us to split the problem into two independent problems, which are commonly called the {\em membrane case} and the {\em bending case} in the literature. We define
\begin{align*}
  \E &:= \left\{ v\in L^2(\Omega) \text{ s.t., for almost any $x'\in\omega$}, \ \left( -\frac{1}{2}, \frac{1}{2}\right) \ni x_d \mapsto v(x',x_d) \text{ is even} \right\},
  \\
  \O &:= \left\{ v\in L^2(\Omega) \text{ s.t., for almost any $x'\in\omega$}, \ \left( -\frac{1}{2}, \frac{1}{2}\right) \ni x_d \mapsto v(x',x_d) \text{ is odd} \right\},
\end{align*}
and point out that
$$
(L^2(\Omega))^d = (\E^{d-1} \times \O) \oplus (\O^{d-1} \times \E),
$$
where the decomposition is orthogonal with respect to the $(L^2(\Omega))^d$ scalar product. This orthogonal decomposition has the following consequence: it holds that
$$
\VKL = \VKL^{\mathcal{M}} \oplus \VKL^{\mathcal{B}},
$$
where
$$
\VKL^{\mathcal{M}} := (H^1_0(\omega))^{d-1} \times \{0\} \subset \E^{d-1} \times \O,
$$
and
$$
\VKL^{\mathcal{B}} := \left\{ v \in \left(H^1(\Omega)\right)^{d-1} \times H^2_0(\omega), \ \ \exists \widehat{v_d} \in H^2_0(\omega), \ v_\alpha = - x_d \, \partial_\alpha \widehat{v}_d, \ v_d = \widehat{v}_d \right\} \subset \O^{d-1} \times \E.
$$
Similarly, we have
$$
\GKL = \GKL^{\mathcal{M}} \oplus \GKL^{\mathcal{B}},
$$
where
$$
\GKL^{\mathcal{M}} := (H^1(\omega))^{d-1} \times \{0\} \subset \E^{d-1} \times \O,
$$
and
$$
\GKL^{\mathcal{B}} := \left\{ v \in \left(H^1(\Omega)\right)^{d-1} \times H^2(\omega), \ \ \exists \widehat{v_d} \in H^2(\omega), \ v_\alpha = - x_d \, \partial_\alpha \widehat{v}_d, \ v_d = \widehat{v}_d \right\}\subset \O^{d-1} \times \E.
$$
From now on, we make the following additional assumptions on the tensor-valued field $A$: for all $1\leq \alpha, \beta, \gamma, \delta \leq d-1$,
\begin{equation} \label{hyp:symA}
  A_{\alpha \beta \gamma \delta}, A_{\alpha \beta d d}, A_{\alpha d\beta d}, A_{dddd} \in \E,
  \qquad
  A_{\alpha d d d}, A_{\alpha \beta \sigma d} \in \O.
\end{equation}
This is a classical assumption for plate problems (see e.g.~\cite[Section~7]{caillerieElasticite}, and in particular Eq.~(7.2) there). 

\begin{remark}
  In the case when the plate is only composed of isotropic materials, Assumption~\eqref{hyp:symA} amounts to assuming that the material is symmetric with respect to its medium plane $\{ x \in \Omega, \ x_d = 0 \}$. Indeed, for an isotropic material, many coefficients of $A$ vanish. The only non-vanishing coefficients are $A_{iiii}(x) = 2\mu(x) + \lambda(x)$ for any $1 \leq i \leq d$, $A_{ijij}(x) = 2\mu(x)$ and $A_{iijj}(x) = \lambda(x)$ for any $1 \leq i,j \leq d$ with $i \neq j$. Assume the material is symmetric with respect to its medium plane: then $\mu$ and $\lambda$ are even functions of $x_d$, and hence all non-vanishing coefficients of $A$ belong to $\E$. The assumption~\eqref{hyp:symA} is therefore satisfied.
\end{remark}

We also introduce two sets of assumptions on the data of the problem, namely $f$, $g$, $h_-$ and $h_+$. In the membrane case, we assume that
\begin{equation} \label{eq:ass_membrane}
  f \in \E^{d-1} \times \O, \quad g \in \GKL^{\mathcal{M}}, \quad \text{$(h_+)_\alpha = (h_-)_\alpha$ for all $1 \leq \alpha \leq d-1$ and $(h_+)_d = - (h_-)_d$}.
\end{equation}
In the bending case, we assume that
\begin{equation} \label{eq:ass_bending}
f \in \O^{d-1} \times \E, \quad g \in \GKL^{\mathcal{B}}, \quad \text{$(h_+)_\alpha = -(h_-)_\alpha$ for all $1\leq \alpha \leq d-1$ and $(h_+)_d = (h_-)_d$}.
\end{equation}
A typical example of a membrane (resp. bending) loading case is when $g = 0$, $h_\pm = 0$ and $f = e_1$ (resp. $f = e_d$). 

\medskip

We then have the following result on the symmetry properties of the solution $u^\eps$ to~\eqref{formvarelast} in the membrane and the bending cases. 

\begin{lemma} \label{lem:membrane}
Assume that $A$ satisfies~\eqref{hyp:symA}. Then, the solution $u^\eps$ to~\eqref{formvarelast} satisfies $u^\eps \in \E^{d-1} \times \O$ in the membrane case~\eqref{eq:ass_membrane}, whereas it satisfies $u^\eps \in \O^{d-1} \times \E$ in the bending case~\eqref{eq:ass_bending}.
\end{lemma}

\begin{proof}
For any $v \in V \cap (\O^{d-1} \times \E)$, we can check that $e_{\alpha \beta}(v) \in \O$, $e_{\alpha d}(v) \in \E$ and $e_{dd}(v) \in \O$. For any $u \in V \cap (\E^{d-1} \times \O)$, we likewise have $e_{\alpha \beta}(u) \in \E$, $e_{\alpha d}(u) \in \O$ and $e_{dd}(u) \in \E$. Using~\eqref{hyp:symA}, we thus obtain that, for any $u \in V \cap (\E^{d-1} \times \O)$ and any $v \in V \cap (\O^{d-1} \times \E)$,
$$
\int_\Omega A^\eps e^\eps(u) : e^\eps(v) = 0.
$$
We next see that, in the membrane case~\eqref{eq:ass_membrane}, we have $b^\eps(v) = 0$ for any $v \in V \cap (\O^{d-1} \times \E)$, where we recall that $b^\eps$ is the linear form in the right-hand side of~\eqref{formvarelast}. Similarly, in the bending case~\eqref{eq:ass_bending}, we have $b^\eps(v) = 0$ for any $v \in V \cap (\E^{d-1} \times \O)$.

By orthogonal decomposition, we now write the solution $u^\eps$ to~\eqref{formvarelast} as $u^\eps = u^\eps_{\rm even} + u^\eps_{\rm odd}$, where $u^\eps_{\rm even} \in \E^{d-1} \times \O$ and $u^\eps_{\rm odd} \in \O^{d-1} \times \E$. Assume we are in the membrane case (a similar argument can be made in the bending case). We then have
$$
a^\eps(u^\eps_{\rm odd},u^\eps_{\rm odd}) = a^\eps(u^\eps_{\rm even} + u^\eps_{\rm odd},u^\eps_{\rm odd}) = a^\eps(u^\eps,u^\eps_{\rm odd}) = b^\eps(u^\eps_{\rm odd}) = 0,
$$
which implies, since $a^\eps$ is coercive, that $u^\eps_{\rm odd}$ vanishes and thus that $u^\eps \in \E^{d-1} \times \O$.
\end{proof}

The symmetry of $A$ also implies the following symmetry properties on the corrector functions.

\begin{lemma}\label{lem:symcorr}
  Assume that $A$ satisfies~\eqref{hyp:symA}. Then, for all $1\leq \alpha, \beta \leq d-1$, the membrane corrector $w^{\alpha\beta}$ solution to~\eqref{el-prcor1} satisfies $w^{\alpha\beta} \in \E^{d-1} \times \O$ and likewise, the bending corrector $W^{\alpha\beta}$ solution to~\eqref{el-prcor2} satisfies $W^{\alpha\beta} \in \O^{d-1}\times \E$. In addition, the tensor $K^\star_{12}$ defined by~\eqref{eq:def_Kstar12} satisfies $K^\star_{12} = 0$. 
\end{lemma}

The proof of Lemma~\ref{lem:symcorr} follows the same lines as the proof of Lemma~\ref{lem:membrane} and we omit it here for the sake of brevity. 


\medskip

Similarly to Lemma~\ref{lem:membrane} which characterizes the symmetry properties of the solution to the oscillatory problem, we have the following result, which characterizes the symmetry properties of the solution $u^\star$ to the homogenized problem~\eqref{formvarelasthomog}.

\begin{lemma}\label{lem:ustarsym}
  Assume that $A$ satisfies~\eqref{hyp:symA} and let $u^\star$ be the solution to~\eqref{formvarelasthomog}. In the membrane case~\eqref{eq:ass_membrane}, we have $\dps u^\star = \left( \begin{array}{c} \widehat{u}^\star \\ 0 \end{array} \right)$, where $\widehat{u}^\star \in (H^1_0(\omega))^{d-1}$ is the unique solution to
\begin{equation} \label{formvar:pb_membrane}
  \forall \widehat{v} \in (H^1_0(\omega))^{d-1}, \quad \int_\omega K^\star_{11} \, e'(\widehat{u}^\star) : e'(\widehat{v}) = \int_\omega \big( \m(f') + h'_+ + h'_- \big) \cdot \widehat{v} - \int_\omega K^\star_{11} \, e'(g') : e'(\widehat{v}).
\end{equation}
We thus have $u^\star \in \VKL^{\mathcal{M}}$.

\medskip

In the bending case~\eqref{eq:ass_bending}, we have $u^\star = \left( \begin{array}{c} -x_d \, \nabla' \widehat{u}_d^\star \\ \widehat{u}_d^\star \end{array} \right)$, where $\widehat{u}_d^\star \in H^2_0(\omega)$ is the unique solution to
\begin{multline} \label{formvar:pb_bending}
  \forall v_d \in H^2_0(\omega), \quad \int_\omega K^\star_{22} \, \nabla^2_{d-1} \widehat{u}_d^\star : \nabla^2_{d-1} v_d = \int_\omega \big( \m(f_d) + (h_+)_d + (h_-)_d \big) \, v_d \\ - \int_\omega \m(x_d \, f') \cdot \nabla' v_d - \frac{1}{2} \int_\omega (h'_+-h'_-) \cdot \nabla' v_d - \int_\omega K^\star_{22} \, \nabla^2_{d-1} g_d : \nabla^2_{d-1} v_d,
\end{multline}
where $\nabla^2_{d-1} v_d$ is the $(d-1) \times (d-1)$ Hessian matrix of $v_d : \omega \to \R$. We thus have $u^\star \in \VKL^{\mathcal{B}}$.
\end{lemma}



We next proceed with Lemma~\ref{lemma:minZ2}, which is an extension of Lemma~\ref{minZ} to the elasticity case. Its proof is much more challenging than that of Lemma~\ref{minZ}. In particular, we have to restrict ourselves to the two-dimensional case, because the divergence of the matrix $\widetilde{B}^T$ which appears in the proof has a simple expression in this case (see~\eqref{eq:toto62_ba} and~\eqref{eq:toto62_bb}). The extension of this result to higher-dimensional cases remains an open question.


\begin{lemma} \label{lemma:minZ2}
Let $V$ be defined by~\eqref{def:V2}. Let $\dps Z \in \left[ L^2_{\rm loc}\left(\mathbb{R}^{d-1} \times \left( -\frac{1}{2},\frac{1}{2} \right) \right) \right]^{d\times d}$ be a matrix field such that
\begin{itemize}
\item[(i)] for almost all $\dps z \in \left( -\frac{1}{2},\frac{1}{2} \right)$, the function $Z(\cdot,z)$ is $Y$-periodic;
\item[(ii)] $\dps \int_{\mathcal{Y}} Z = 0$;
\item[(iii)] $\dps \int_{\mathcal{Y}} x_d \, Z = 0$;
\item[(iv)] $\div Z = 0$ in $\dps \left[ \mathcal{D}'\left(\mathbb{R}^{d-1} \times \left( -\frac{1}{2},\frac{1}{2} \right) \right) \right]^d$; 
\item[(v)] $Z \, e_d = 0$ on $\mathbb{R}^{d-1} \times \{-1/2\}$ and on $\mathbb{R}^{d-1} \times \{1/2\}$;
\item[(vi)] $Z$ is symmetric in the sense that $Z_{ij} = Z_{ji}$ for all $1\leq i,j \leq d$.
\end{itemize}
We assume that $d=2$. Then, there exists some $C$ independent of $\eps$ and $\omega$ such that, for any $\varphi$ in $W^{2,\infty}(\omega)$ and any $v \in V$,
\begin{equation} \label{eq:covid2}
\left| \int_\Omega \varphi(x') \, Z\left( \frac{x'}{\eps},x_d \right) : e^\eps(v)(x',x_d) \, dx' \, dx_d \right| \leq C \, \eps \, |\omega|^{1/2} \, \| \nabla \varphi \|_{W^{1,\infty}(\omega)} \, \| e^\eps(v) \|_{L^2(\Omega)}.
\end{equation}
\end{lemma}


We recall that any function in $H_{\rm div}(D)$ (for any smooth domain $D\subset \R^d$) has a well-defined normal trace on $\partial D$ (see Appendix~\ref{app:Hdiv} for details). Assumption~(v) in Lemma~\ref{lemma:minZ2} thus makes sense.

\begin{remark} \label{rem:CN_vectorielle}
In the spirit of Remark~\ref{rem:CN_scalaire}, we show that Assumptions~(ii) and~(iii) are necessary conditions for~\eqref{eq:covid2} to hold true. 

Consider indeed some $v \in \VKL$. We then observe that $\dps e^\eps(v) = \left( \begin{array}{cc} e'(\widehat{v}') - x_d \nabla^2 \, \widehat{v}_d & 0 \\ 0 & 0 \end{array} \right)$ is independent of $\eps$. On the one hand, the right-hand side, and therefore the left-hand side, of~\eqref{eq:covid2} thus converges to 0 when $\eps \to 0$, which yields that
$$
\lim_{\eps \to 0} \int_\Omega \varphi(x') \, Z\left( \frac{x'}{\eps},x_d \right) : \left( \begin{array}{cc} e'(\widehat{v}') - x_d \, \nabla^2 \widehat{v}_d & 0 \\ 0 & 0 \end{array} \right) dx' \, dx_d = 0.
$$
On the other hand, using the fact that $Z$ is $Y$-periodic and Lemma~\ref{limmoyenne}, we obtain that the left-hand side of~\eqref{eq:covid2} converges to $\dps \int_\Omega \varphi(x') \, Z^\star(x_d) : \left( \begin{array}{cc} e'(\widehat{v}) - x_d \, \nabla^2 \widehat{v}_d & 0 \\ 0 & 0 \end{array} \right) dx' \, dx_d$, with $\dps Z^\star(x_d) = \int_Y Z(x',x_d) dx'$ for any $x_d \in (-1/2,1/2)$. Since $\varphi$ is arbitrary, we deduce that $\dps \int_{-1/2}^{1/2} Z^\star_{\gamma \delta}(x_d) \, \big( e_{\gamma \delta}(\widehat{v}) - x_d \, \partial_{\gamma \delta} \widehat{v}_d \big) dx_d = 0$. Since $\widehat{v}'$ and $\widehat{v}_d$ only depend on $x'$ and are independent, this implies that $\dps \int_{-1/2}^{1/2} Z^\star_{\gamma \delta}(x_d) \, dx_d = 0$ and $\dps \int_{-1/2}^{1/2} x_d \, Z^\star_{\gamma \delta}(x_d) \, dx_d = 0$ for any $1 \leq \gamma,\delta \leq d-1$, which means, in view of the definition of $Z^\star$, that $\dps \int_{\mathcal Y} Z_{\gamma \delta} = 0$ and $\dps \int_{\mathcal Y} x_d \, Z_{\gamma \delta} = 0$.

To show that Assumptions~(ii) and~(iii) are indeed necessary conditions, we are left to show that~\eqref{eq:covid2} implies $\dps \int_{\mathcal Y} Z_{id} = \int_{\mathcal Y} x_d \, Z_{id} = 0$ for any $1 \leq i \leq d$. To this end, we first consider the function $v(x',x_d) = \eps^2 \, \psi(x') \, \theta(x_d) \, e_d$ for some $\psi \in C^\infty_0(\omega)$ and some $\theta \in C^\infty[-1/2,1/2]$. It belongs to $V$ and satisfies that $e^\eps(v)$ is bounded, thus the right-hand side, and therefore the left-hand side, of~\eqref{eq:covid2} converges to 0 when $\eps \to 0$. For this function $v$, and up to terms that obviously vanish in the limit $\eps \to 0$, the left-hand side of~\eqref{eq:covid2} writes
$$
\int_\Omega \varphi(x') \, Z_{dd}\left( \frac{x'}{\eps},x_d \right) \psi(x') \, \theta'(x_d) \, dx' \, dx_d,
$$
and thus converges to $\dps \int_\Omega \varphi(x') \, Z_{dd}^\star(x_d) \, \psi(x') \, \theta'(x_d) \, dx' \, dx_d$ in view of Lemma~\ref{limmoyenne}. Since $\varphi$ is arbitrary, this implies that $\dps \int_{-1/2}^{1/2} Z_{dd}^\star(x_d) \, \theta'(x_d) \, dx_d = 0$, and therefore $Z_{dd}^\star(x_d) = 0$ for any $x_d \in (-1/2,1/2)$ since $\theta \in C^\infty[-1/2,1/2]$ is arbitrary. This of course implies that $\dps \int_{-1/2}^{1/2} Z_{dd}^\star(x_d) \, dx_d = 0$ and $\dps \int_{-1/2}^{1/2} x_d \, Z_{dd}^\star(x_d) \, dx_d = 0$, and thus $\dps \int_{\mathcal Y} Z_{dd} = 0$ and $\dps \int_{\mathcal Y} x_d \, Z_{dd} = 0$, in view of the definition of $Z_{dd}^\star$.

Second, we consider the function $v(x',x_d) = \eps \, \psi(x') \, \theta(x_d) \, e_\alpha$ for some $1 \leq \alpha \leq d-1$, $\psi \in C^\infty_0(\omega)$ and some $\theta \in C^\infty[-1/2,1/2]$. The same computations as above lead to the fact that $\dps \int_\Omega \varphi(x') \, Z_{d\alpha}^\star(x_d) \, \psi(x') \, \theta'(x_d) \, dx' \, dx_d = 0$, and thus $Z_{d\alpha}^\star(x_d) = 0$ for any $x_d \in (-1/2,1/2)$. This of course implies that $\dps \int_{-1/2}^{1/2} Z_{d\alpha}^\star(x_d) \, dx_d = 0$ and $\dps \int_{-1/2}^{1/2} x_d \, Z_{d\alpha}^\star(x_d) \, dx_d = 0$, and thus $\dps \int_{\mathcal Y} Z_{d\alpha} = 0$ and $\dps \int_{\mathcal Y} x_d \, Z_{d\alpha} = 0$, in view of the definition of $Z_{d\alpha}^\star$. This completes the argument showing that Assumptions~(ii) and~(iii) are indeed necessary conditions for~\eqref{eq:covid2} to hold true.
\end{remark}

\begin{remark}
  As in the scalar case (see Remark~\ref{rem:CN_scalaire_comp}), one could think, in view of the second part of Remark~\ref{rem:CN_vectorielle}, that a necessary condition for~\eqref{eq:covid2} to hold true is $Z_{di}^\star(x_d) = 0$ for any $x_d \in (-1/2,1/2)$ and any $1 \leq i \leq d$. Actually, under Assumption~(iv), the statement
  \begin{equation} \label{eq:stat3}
    \forall 1 \leq i \leq d, \quad \int_{\mathcal Y} Z_{di} = 0 \quad \text{and} \quad \int_{\mathcal Y} x_d \, Z_{di} = 0
  \end{equation}
  is equivalent to the statement
  \begin{equation} \label{eq:stat4}
    \forall 1 \leq i \leq d, \quad Z_{di}^\star(x_d) = 0 \quad \text{for any $x_d \in (-1/2,1/2)$},
  \end{equation}
  as we now show. Of course, \eqref{eq:stat4} implies~\eqref{eq:stat3}. Assuming now that~\eqref{eq:stat3} holds, we integrate on $Y$ the vectorial equation $\div Z = 0$ (see Assumption~(iv)). Using the $Y$-periodicity of $Z$, we find $\dps \partial_d \int_Y Z_{di} = 0$ for any $1 \leq i \leq d$, which means that $Z^\star_{di}$ is actually a constant. The first part of condition~\eqref{eq:stat3} now implies that this constant vanishes, and we hence deduce~\eqref{eq:stat4}.
\end{remark}
  


\begin{proof}[Proof of Lemma~\ref{lemma:minZ2}]
For any $1 \leq i \leq d$, let $Z^i$ be the vector-valued function such that $[ Z^i ]_j = Z_{ij}$ for any $1 \leq j \leq d$. By definition of the divergence operator, we have $[ \div Z ]_i = \div Z^i$. The vector field $Z^i$ thus satisfies all assumptions of Lemma~\ref{minZ}. Introducing the periodic extension $\overline{Z} \in (L^2_{\rm per}(\mathbb{R}^d))^{d \times d}$ of $Z$ in the $e_d$ direction, and using arguments similar to the ones used at the beginning of the proof of Lemma~\ref{minZ}, we obtain that, for any $1 \leq i \leq d$, there exists a skew-symmetric matrix-valued field $J^i \in \left( H^1_{\rm per}(\mathbb{R}^d) \right)^{d \times d}$ such that $\overline{Z}_{ij} = [ \, \overline{Z}^i \, ]_j = \partial_k J^i_{k,j}$ for any $1 \leq i,j \leq d$. The proof falls in three steps.


\medskip

\noindent
{\bf Step~1.} Let $\varphi \in W^{2,\infty}(\omega)$ and $1 \leq i,j \leq d$. Denoting $J^i_{\cdot j} = \left( J^i_{k,j} \right)_{1 \leq k \leq d} \in \R^d$, we compute, for any $x_d$ in $(-1/2,1/2)$, that
\begin{align}
  \left[ Z \left( \frac{\cdot}{\eps},x_d \right) \varphi \right]_{ij}
  &= \left[ \overline{Z} \left( \frac{\cdot}{\eps},x_d \right) \varphi \right]_{ij}
  \nonumber
  \\
  &= \partial_k J^i_{k,j}\left( \frac{\cdot}{\eps},x_d \right) \varphi
  \nonumber
  \\
  &= \eps \div^\eps \left[ J^i_{\cdot,j}\left( \frac{\cdot}{\eps},x_d \right) \right] \varphi
  \nonumber
  \\
  &= \eps \div^\eps \left[ J^i_{\cdot,j}\left( \frac{\cdot}{\eps},x_d \right) \varphi \right] - \eps \, J^i_{\cdot,j}\left( \frac{\cdot}{\eps},x_d \right) \cdot \nabla^\eps \varphi
  \nonumber
  \\
  &= \eps \, \widetilde{B}_{ij}(\cdot,x_d) - \eps \, B_{ij}(\cdot,x_d),
  \label{eq:diffe1_el}
\end{align}
where $\dps \widetilde{B}_{ij}(\cdot,x_d) := \div^\eps \left[ J^i_{\cdot,j}\left( \frac{\cdot}{\eps},x_d \right) \varphi \right]$ and $\dps B_{ij}(\cdot,x_d) := J^i_{\cdot,j}\left( \frac{\cdot}{\eps},x_d \right) \cdot \nabla^\eps \varphi$. Note that
$$
\left[ \div^\eps \widetilde{B} \right]_\gamma = \div^\eps (\widetilde{B}_{\gamma \, \cdot})
\qquad \text{and} \qquad
\left[ \div^\eps \widetilde{B} \right]_d = \frac{1}{\eps} \, \div^\eps (\widetilde{B}_{d \, \cdot}).
$$
We next compute, for any $1 \leq i \leq d$, that
\begin{align*}
  & \div^\eps (\widetilde{B}_{i \, \cdot})
  \\
  &= \div^\eps (\widetilde{B}_{ij} \, e_j)
  \\
  &= \div^\eps \left( \partial_\alpha \left[ J^i_{\alpha,j}\left( \frac{\cdot}{\eps},x_d \right) \varphi \right] e_j + \frac{1}{\eps} \partial_d \left[ J^i_{d,j}\left( \frac{\cdot}{\eps},x_d \right) \varphi \right] e_j \right) 
  \\
  &= \partial_{\beta \alpha} \left[ J^i_{\alpha, \beta} \left( \frac{\cdot}{\eps},x_d \right) \varphi \right] + \frac{1}{\eps} \left( \partial_{d \alpha} \left[ J^i_{\alpha, d}\left( \frac{\cdot}{\eps},x_d \right) \varphi \right] + \partial_{\beta d} \left[ J^i_{d, \beta}\left( \frac{\cdot}{\eps},x_d \right) \varphi \right] \right) + \frac{1}{\eps^2} \partial_{dd} \left[ J^i_{d, d}\left( \frac{\cdot}{\eps},x_d \right) \varphi \right] 
  \\
  &= 0 \quad [\text{because $J^i$ is skew symmetric}]
\end{align*}
and thus
\begin{equation} \label{eq:toto62_a}
  \div^\eps \widetilde{B} = 0.
\end{equation}
We note that the matrix $\widetilde{B}$ is not symmetric, in general. There is no reason for $\div^\eps (\widetilde{B}^T)$ to vanish. We now compute this quantity, the expression of which will be useful below. We observe that
\begin{equation} \label{eq:covid3}
\left[ \div^\eps (\widetilde{B}^T) \right]_\gamma = \div^\eps (\widetilde{B}_{\cdot \, \gamma})
\qquad \text{and} \qquad
\left[ \div^\eps (\widetilde{B}^T) \right]_d = \frac{1}{\eps} \, \div^\eps (\widetilde{B}_{\cdot \, d}).
\end{equation}
We next compute, for any $1 \leq i \leq d$, that
\begin{align*}
  & \div^\eps (\widetilde{B}_{\cdot \, i})
  \\
  &= \div^\eps (\widetilde{B}_{ji} \, e_j)
  \\
  &= \div^\eps \left( \partial_\alpha \left[ J^j_{\alpha,i}\left( \frac{\cdot}{\eps},x_d \right) \varphi \right] e_j + \frac{1}{\eps} \partial_d \left[ J^j_{d,i}\left( \frac{\cdot}{\eps},x_d \right) \varphi \right] e_j \right) 
  \\
  &= \partial_{\beta \alpha} \left[ J^\beta_{\alpha, i} \left( \frac{\cdot}{\eps},x_d \right) \varphi \right] + \frac{1}{\eps} \left( \partial_{d \alpha} \left[ J^d_{\alpha, i}\left( \frac{\cdot}{\eps},x_d \right) \varphi \right] + \partial_{\beta d} \left[ J^\beta_{d, i}\left( \frac{\cdot}{\eps},x_d \right) \varphi \right] \right) + \frac{1}{\eps^2} \, \partial_{dd} \left[ J^d_{d, i}\left( \frac{\cdot}{\eps},x_d \right) \varphi \right].
\end{align*}
We are next going to use that $d=2$. We first compute the above quantity for $i=1$. In that case, the first two terms vanish (because the only choice for $\alpha$ is $\alpha = 1 = i$ and $J^\beta$ and $J^d$ are skew-symmetric), and hence
\begin{align}
  \div^\eps (\widetilde{B}_{\cdot \, 1})
  &=
  \frac{1}{\eps} \, \partial_{12} \left[ J^1_{2,1}\left( \frac{\cdot}{\eps},x_2 \right) \varphi \right] + \frac{1}{\eps^2} \, \partial_{22} \left[ J^2_{2,1}\left( \frac{\cdot}{\eps},x_2 \right) \varphi \right]
  \nonumber
  \\
  &=
  \frac{1}{\eps} \, \partial_1 \left[ \partial_2 J^1_{2,1}\left( \frac{\cdot}{\eps},x_2 \right) \varphi \right] + \frac{1}{\eps^2} \, \partial_{22} J^2_{2,1}\left( \frac{\cdot}{\eps},x_2 \right) \varphi \qquad \text{[$\varphi$ ind. of $x_2$]}
  \nonumber
  \\
  &=
  \frac{1}{\eps} \, \partial_1 \left[ \overline{Z}_{11}\left( \frac{\cdot}{\eps},x_2 \right) \varphi \right] + \frac{1}{\eps^2} \, \partial_2 \overline{Z}_{21}\left( \frac{\cdot}{\eps},x_2 \right) \varphi
  \nonumber
  \\
  &=
  \frac{1}{\eps} \, \overline{Z}_{11}\left( \frac{\cdot}{\eps},x_2 \right) \partial_1 \varphi,
  \label{eq:toto62_ba}
\end{align}
where we have used, in the last line, that $\overline{Z}$ is symmetric and that $\overline{Z}_{1\cdot}$ is divergence-free. Considering now the case $i=2$, we compute
\begin{align}
  \div^\eps (\widetilde{B}_{\cdot \, 2})
  &=
  \partial_{11} \left[ J^1_{1,2} \left( \frac{\cdot}{\eps},x_2 \right) \varphi \right] + \frac{1}{\eps} \, \partial_{12} \left[ J^2_{1,2}\left( \frac{\cdot}{\eps},x_2 \right) \varphi \right]
  \nonumber
  \\
  &=
  \partial_1 \left[ J^1_{1,2} \left( \frac{\cdot}{\eps},x_2 \right) \partial_1 \varphi \right] + \frac{1}{\eps} \, \partial_1 \left[ \partial_1 J^1_{1,2} \left( \frac{\cdot}{\eps},x_2 \right) \varphi \right] + \frac{1}{\eps} \, \partial_1 \left[ \partial_2 J^2_{1,2}\left( \frac{\cdot}{\eps},x_2 \right) \varphi \right]
  \nonumber
  \\
  &=
  \partial_1 \left[ J^1_{1,2} \left( \frac{\cdot}{\eps},x_2 \right) \partial_1 \varphi \right] + \frac{1}{\eps} \, \partial_1 \left[ \overline{Z}_{12} \left( \frac{\cdot}{\eps},x_2 \right) \varphi \right] - \frac{1}{\eps} \, \partial_1 \left[ \overline{Z}_{21}\left( \frac{\cdot}{\eps},x_2 \right) \varphi \right]
  \nonumber
  \\
  &=
  \partial_1 \left[ J^1_{1,2} \left( \frac{\cdot}{\eps},x_2 \right) \partial_1 \varphi \right],
  \label{eq:toto62_bb}
\end{align}
where we have used at the third line that $J^2$ is skew-symmetric and at the last line that $\overline{Z}$ is symmetric. These computations confirm that, in general, $\div^\eps (\widetilde{B}^T)$ does not vanish.

\medskip

Since $\widetilde{B}$ is not symmetric, it is useful for the sequel to introduce the symmetric matrix $\overline{B} = \widetilde{B} + \widetilde{B}^T$. In the following, we will need to evaluate $\overline{B}$ on $\R^{d-1} \times \{ \pm 1/2 \}$. To this end, we proceed as follows. We infer from~\eqref{eq:diffe1_el} and from the definition of $B$ that, for any $x_d$ in $(-1/2,1/2)$ and any $1\leq i,j \leq d$,
\begin{equation} \label{eq:diffe2_el}
\widetilde{B}_{ij}(\cdot,x_d) = \frac{1}{\eps}Z_{ij}\left(\frac{\cdot}{\eps}, x_d\right) \varphi + J^i_{\cdot,j}\left(\frac{\cdot}{\eps}, x_d\right) \cdot \nabla^\eps \varphi.
\end{equation}
Since $\varphi$ is independent of $x_d$, we have $\dps \nabla^\eps \varphi = \nabla \varphi = \left( \begin{array}{c} \nabla' \varphi \\ 0 \end{array} \right)$. Besides, since $J^i$ is $\mathcal{Y}$-periodic, we have, for any $1 \leq i,j \leq d$, 
$$
J^i_{\cdot,j} \left( \frac{\cdot}{\eps}, -\frac{1}{2}\right) \cdot \nabla^\eps \varphi = J^i_{\cdot,j} \left( \frac{\cdot}{\eps}, \frac{1}{2}\right) \cdot \nabla^\eps \varphi = J^i_{\alpha, j}\left( \frac{\cdot}{\eps}, \frac{1}{2}\right) \partial_\alpha \varphi.
$$
We then infer from~\eqref{eq:diffe2_el}, the above relation and Assumptions~(v) and~(vi) that
$$
\widetilde{B}_{dj}\left(\cdot, \pm \frac{1}{2}\right)
=
J^d_{\alpha, j} \left( \frac{\cdot}{\eps}, \frac{1}{2}\right) \partial_\alpha \varphi,
\qquad
\widetilde{B}_{jd}\left(\cdot, \pm \frac{1}{2}\right)
=
J^j_{\alpha, d} \left( \frac{\cdot}{\eps}, \frac{1}{2}\right) \partial_\alpha \varphi,
$$
and therefore
\begin{equation} \label{eq:covid5}
\overline{B}_{dj}\left(\cdot, \pm \frac{1}{2}\right)
=
\widetilde{B}_{dj}\left(\cdot, \pm \frac{1}{2}\right) + \widetilde{B}_{jd}\left(\cdot, \pm \frac{1}{2}\right)
=
\left[ J^d_{\alpha, j} \left( \frac{\cdot}{\eps}, \frac{1}{2}\right) + J^j_{\alpha, d} \left( \frac{\cdot}{\eps}, \frac{1}{2}\right) \right] \partial_\alpha \varphi.
\end{equation}

\medskip

\noindent
{\bf Step~2.} Inserting~\eqref{eq:diffe1_el} in the quantity we wish to bound, we write that, for any $v \in V$,
\begin{align}
  \int_\Omega \varphi(x') \, Z\left( \frac{x'}{\eps},x_d \right) : e^\eps(v)(x',x_d) \, dx' \, dx_d
  &=
  \eps \int_\Omega \widetilde{B} : e^\eps(v) - \eps \int_\Omega B: e^\eps(v)
  \nonumber
  \\
  &=
  \frac{\eps}{2} \int_\Omega \overline{B} : e^\eps(v) - \eps \int_\Omega B: e^\eps(v),
  \label{eq:toto2}
\end{align}
where we recall that $\overline{B} = \widetilde{B} + \widetilde{B}^T$. We are going to successively bound the two terms of~\eqref{eq:toto2}.

The second term of~\eqref{eq:toto2} is easy to estimate. Using the fact that $\varphi$ does not depend on the variable $x_d$ and the definition of $B$, we compute
\begin{align}
  \left| \int_\Omega B : e^\eps(v) \right|
  &=
  \left| \int_\omega \int_{-1/2}^{1/2} J^i_{\alpha,j}\left( \frac{x'}{\eps},x_d \right) \partial_\alpha \varphi(x') \, e^\eps_{ij}(v)(x',x_d) \, dx' \, dx_d \right|
  \nonumber
  \\
  & \leq
  \left\| J\left( \frac{x'}{\eps},x_d \right) \right\|_{L^2(\Omega)} \, \| \nabla \varphi \|_{L^\infty(\omega)} \, \| e^\eps(v) \|_{L^2(\Omega)} 
  \nonumber
  \\
  & \leq
  C \, |\omega|^{1/2} \, \| J\|_{L^2(\mathcal{Y})} \, \| \nabla \varphi \|_{L^\infty(\omega)} \, \| e^\eps(v) \|_{L^2(\Omega)}.
  \label{eq:bound2_b}
\end{align}

The remainder of the proof is devoted to estimating the first term of~\eqref{eq:toto2}. We note that, by definition, $\overline{B}$ is in $(L^2(\Omega))^{d\times d}$ and its divergence belongs to $(L^2(\Omega))^d$, in view of~\eqref{eq:toto62_a}, \eqref{eq:toto62_ba} and~\eqref{eq:toto62_bb} and the fact that $\varphi \in W^{2,\infty}(\omega)$. We thus note that $\overline{B}$ has a well-defined normal trace (see Appendix~\ref{app:Hdiv} for details). Using the integration by parts relation~\eqref{eq:IPP_elas} (which is valid since $\overline{B}$ is a symmetric matrix) and~\eqref{eq:toto62_a}, we obtain that, for any $v \in V$,
\begin{align}
  \int_\Omega \overline{B} : e^\eps(v)
  =
  & - \int_\Omega v \cdot \div^\eps (\widetilde{B}^T)
  \nonumber
  \\
  & + \frac{1}{\eps} \left[\int_\omega \overline{B}_{d \beta}\left(\cdot,\frac{1}{2}\right) v_\beta\left(\cdot,\frac{1}{2}\right) - \int_\omega \overline{B}_{d\beta}\left(\cdot,-\frac{1}{2}\right) v_\beta\left(\cdot,-\frac{1}{2}\right) \right]
  \nonumber
  \\ 
  & + \frac{1}{\eps^2} \left[\int_\omega \overline{B}_{dd}\left(\cdot,\frac{1}{2}\right) v_d\left(\cdot,\frac{1}{2}\right) - \int_\omega \overline{B}_{dd}\left(\cdot,-\frac{1}{2}\right) v_d\left(\cdot,-\frac{1}{2}\right) \right].
  \label{eq:toto10_b}
\end{align}
We use~\eqref{eq:covid3}, \eqref{eq:toto62_ba} and~\eqref{eq:toto62_bb} to compute the first term of the right-hand side of~\eqref{eq:toto10_b}:
\begin{equation} \label{eq:covid4}
  \int_\Omega v \cdot \div^\eps (\widetilde{B}^T)
  =
  \frac{1}{\eps} \int_\Omega v_1 \, \overline{Z}_{11}^\eps \, \partial_1 \varphi
  +
  \frac{1}{\eps} \int_\Omega v_2 \, \partial_1 \left[ J^{1,\eps}_{1,2} \, \partial_1 \varphi \right],
\end{equation}
where $\overline{Z}_{11}^\eps$ is the function defined by $\dps \overline{Z}_{11}^\eps(x_1,x_2) = \overline{Z}_{11}\left( \frac{x_1}{\eps},x_2 \right)$ and likewise $\dps J^{1,\eps}_{1,2}(x_1,x_2) = J^1_{1,2} \left( \frac{x_1}{\eps},x_2 \right)$. We recast the second term of the right-hand side of~\eqref{eq:covid4} using an integration by parts and the fact that $J^1$ is skew-symmetric:
\begin{align}
  & \frac{1}{\eps} \int_\Omega v_2 \, \partial_1 \left[ J^{1,\eps}_{1,2} \, \partial_1 \varphi \right]
  \nonumber
  \\
  &=
  -
  \frac{1}{\eps} \int_\Omega J^{1,\eps}_{1,2} \, \partial_1 \varphi \, \partial_1 v_2
  \nonumber
  \\
  &=
  -
  \frac{1}{\eps} \int_\Omega J^{1,\eps}_{1,2} \, \partial_1 \varphi \, (\partial_1 v_2 + \partial_2 v_1)
  +
  \frac{1}{\eps} \int_\Omega J^{1,\eps}_{1,2} \, \partial_1 \varphi \, \partial_2 v_1
  \nonumber
  \\
  &=
  -
  2 \int_\Omega J^{1,\eps}_{1,2} \, \partial_1 \varphi \, e^\eps_{12}(v)
  -
  \frac{1}{\eps} \int_\Omega J^{1,\eps}_{2,1} \, \partial_1 \varphi \, \partial_2 v_1
  \nonumber
  \\
  &=
  -
  2 \int_\Omega J^{1,\eps}_{1,2} \, \partial_1 \varphi \, e^\eps_{12}(v)
  +
  \frac{1}{\eps} \int_\Omega \partial_2 J^{1,\eps}_{2,1} \, \partial_1 \varphi \, v_1
  \nonumber
  \\
  & \qquad \qquad
  -
  \frac{1}{\eps} \int_\omega J^1_{2,1} \left( \frac{\cdot}{\eps},\frac{1}{2} \right) \partial_1 \varphi \, v_1\left(\cdot,\frac{1}{2} \right)
  +
  \frac{1}{\eps} \int_\omega J^1_{2,1} \left( \frac{\cdot}{\eps},-\frac{1}{2} \right) \partial_1 \varphi \, v_1\left(\cdot,-\frac{1}{2} \right)
  \nonumber
  \\
  &=
  -
  2 \int_\Omega J^{1,\eps}_{1,2} \, \partial_1 \varphi \, e^\eps_{12}(v)
  +
  \frac{1}{\eps} \int_\Omega \overline{Z}_{11}^\eps \, \partial_1 \varphi \, v_1
  \nonumber
  \\
  & \qquad \qquad
  +
  \frac{1}{\eps} \int_\omega \overline{B}_{21} \left(\cdot,\frac{1}{2} \right) v_1\left(\cdot,\frac{1}{2} \right)
  -
  \frac{1}{\eps} \int_\omega \overline{B}_{21} \left(\cdot,-\frac{1}{2} \right) v_1\left(\cdot,-\frac{1}{2} \right),
  \label{eq:dur}
\end{align}
where we have used in the last line that, in view of~\eqref{eq:covid5} and the skew-symmetry of $J^2$ and $J^1$,
$$
\overline{B}_{21}\left(\cdot, \pm \frac{1}{2}\right)
=
J^1_{1,2} \left( \frac{\cdot}{\eps}, \pm \frac{1}{2}\right) \partial_1 \varphi
=
- J^1_{2,1} \left( \frac{\cdot}{\eps}, \pm \frac{1}{2}\right) \partial_1 \varphi.
$$
Collecting~\eqref{eq:toto10_b} and~\eqref{eq:covid4}, we deduce that
\begin{align}
  \int_\Omega \overline{B} : e^\eps(v)
  =
  & - \frac{1}{\eps} \int_\Omega v_1 \, \overline{Z}_{11}^\eps \, \partial_1 \varphi
  -
  \frac{1}{\eps} \int_\Omega v_2 \, \partial_1 \left[ J^{1,\eps}_{1,2} \, \partial_1 \varphi \right]
  \nonumber
  \\
  & + \frac{1}{\eps} \left[\int_\omega \overline{B}_{21}\left(\cdot,\frac{1}{2}\right) v_1\left(\cdot,\frac{1}{2}\right) - \int_\omega \overline{B}_{21}\left(\cdot,-\frac{1}{2}\right) v_1\left(\cdot,-\frac{1}{2}\right) \right]
  \nonumber
  \\ 
  & + \frac{1}{\eps^2} \left[\int_\omega \overline{B}_{22}\left(\cdot,\frac{1}{2}\right) v_2\left(\cdot,\frac{1}{2}\right) - \int_\omega \overline{B}_{22}\left(\cdot,-\frac{1}{2}\right) v_2\left(\cdot,-\frac{1}{2}\right) \right]
  \nonumber
  \\
  =
  & - \frac{2}{\eps} \int_\Omega v_1 \, \overline{Z}_{11}^\eps \, \partial_1 \varphi
  +
  2 \int_\Omega J^{1,\eps}_{1,2} \, \partial_1 \varphi \, e^\eps_{12}(v)
  \nonumber
  \\
  & + \frac{2}{\eps^2} \int_\omega J^2_{1,2}\left(\frac{\cdot}{\eps},\frac{1}{2}\right) \partial_1 \varphi \left[ v_2\left(\cdot,\frac{1}{2}\right) - v_2\left(\cdot,-\frac{1}{2}\right) \right]
  \nonumber
  \\
  =
  & - \frac{2}{\eps} \int_\Omega v_1 \, \overline{Z}_{11}^\eps \, \partial_1 \varphi
  +
  2 \int_\Omega J^{1,\eps}_{1,2} \, \partial_1 \varphi \, e^\eps_{12}(v)
  \nonumber
  \\
  & + 2 \int_\omega \int_{t=-1/2}^{1/2} J^2_{1,2}\left(\frac{x_1}{\eps},\frac{1}{2}\right) \partial_1 \varphi(x_1) \, e^\eps_{22}(v)(x_1,t) \, dt \, dx_1,
  \label{eq:covid6}
\end{align}
where we have used in the second equality the identity~\eqref{eq:dur}, the fact that $\dps \overline{B}_{22}\left(\cdot, \pm \frac{1}{2}\right) = 2 \, J^2_{1,2}\left(\frac{\cdot}{\eps},\frac{1}{2}\right) \partial_1 \varphi$ (see~\eqref{eq:covid5}) and in the third equality that $\eps^{-2} \, \partial_2 v_2 = e^\eps_{22}(v)$.

\medskip

We are now in position to bound the second and third terms of~\eqref{eq:covid6}. For the second one, we write
\begin{align} 
  \left| 2 \int_\Omega J^{1,\eps}_{1,2} \, \partial_1 \varphi \, e^\eps_{12}(v) \right|
  &\leq
  2 \, \left\| J^{1,\eps}_{1,2} \right\|_{L^2(\Omega)} \| \nabla \varphi \|_{L^\infty(\omega)} \, \| e^\eps(v) \|_{L^2(\Omega)}
  \nonumber
  \\
  & \leq
  C \, |\omega|^{1/2} \, \| J^1\|_{L^2(\mathcal{Y})} \, \| \nabla \varphi \|_{L^\infty(\omega)} \, \| e^\eps(v) \|_{L^2(\Omega)}.
  \label{eq:covid7}
\end{align}
For the third term of~\eqref{eq:covid6}, we write
\begin{align}
  & \left| 2 \int_\omega \int_{t=-1/2}^{1/2} J^2_{1,2}\left(\frac{x_1}{\eps},\frac{1}{2}\right) \partial_1 \varphi(x_1) \, e^\eps_{22}(v)(x_1,t) \, dt \, dx_1 \right|
  \nonumber
  \\
  &\leq
  2 \, \left\| J^2\left( \frac{\cdot}{\eps}, \frac{1}{2} \right) \right\|_{L^2(\omega)} \| \nabla \varphi \|_{L^\infty(\omega)} \, \| e^\eps(v) \|_{L^2(\Omega)} 
  \nonumber
  \\
  & \leq C \, |\omega|^{1/2} \, \left\| J^2\left(\cdot, \frac{1}{2}\right) \right\|_{L^2(Y)} \| \nabla \varphi \|_{L^\infty(\omega)} \, \| e^\eps(v) \|_{L^2(\Omega)} 
  \nonumber
  \\
  & \leq C \, |\omega|^{1/2} \, \| J^2 \|_{H^1(\mathcal{Y})} \, \| \nabla \varphi \|_{L^\infty(\omega)} \, \| e^\eps(v) \|_{L^2(\Omega)}. 
  \label{eq:covid8}
\end{align}
We claim that the first term of~\eqref{eq:covid6} satisfies the following bound:
\begin{equation} \label{eq:covid9}
  \left| \frac{2}{\eps} \int_\Omega v_1 \, \overline{Z}_{11}^\eps \, \partial_1 \varphi \right| \leq C \, |\omega|^{1/2} \, \| e^\eps(v) \|_{L^2(\Omega)} \, \| \nabla \varphi \|_{W^{1,\infty}(\omega)}.
\end{equation}
Collecting~\eqref{eq:toto2}, \eqref{eq:bound2_b}, \eqref{eq:covid6}, \eqref{eq:covid9}, \eqref{eq:covid7} and~\eqref{eq:covid8}, we obtain~\eqref{eq:covid2}.

\medskip

\noindent
{\bf Step~3.} We are now left with showing~\eqref{eq:covid9}. Recall here that $Y=(0,1)$, since $d=2$. To that aim, we split $\overline{Z}_{11}$ as
$$
\overline{Z}_{11}(x_1,x_2) = \overline{z}_{11}(x_1,x_2) + \mu(x_2) \qquad \text{with} \qquad \mu(x_2) = \int_0^1 \overline{Z}_{11}(s,x_2) \, ds,
$$
which ensures that $\dps \int_0^1 \overline{z}_{11}(s,x_2) \, ds = 0$. We next introduce $\dps \Lambda_{11}(x_1,x_2) = \int_0^{x_1} \overline{z}_{11}(s,x_2) \, ds$. Since $\overline{Z}_{11}$ is periodic with respect to its first variable, so is $\overline{z}_{11}$. Since the average (with respect to its first variable) of $\overline{z}_{11}$ vanishes, we see that $\Lambda_{11}$ is periodic of its first variable. We next write
\begin{equation} \label{eq:covid10}
  \frac{1}{\eps} \int_\Omega v_1 \, \overline{Z}_{11}^\eps \, \partial_1 \varphi
  =
  \frac{1}{\eps} \int_\Omega v_1(x) \, \partial_1 \Lambda_{11}\left( \frac{x_1}{\eps},x_2 \right) \partial_1 \varphi(x_1)
  +
  \frac{1}{\eps} \int_\Omega v_1(x) \, \mu(x_2) \, \partial_1 \varphi(x_1)
\end{equation}
and successively bound the two terms. For the first one, we have
$$
\frac{1}{\eps} \int_\Omega v_1(x) \, \partial_1 \Lambda_{11}\left( \frac{x_1}{\eps},x_2 \right) \partial_1 \varphi(x_1)
=
\int_\Omega v_1 \, \partial_1 \Lambda_{11}^\eps \, \partial_1 \varphi
=
- \int_\Omega \Lambda_{11}^\eps \, \partial_1 (v_1 \, \partial_1 \varphi),
$$
where we have set $\dps \Lambda_{11}^\eps(x_1,x_2) = \Lambda_{11}\left( \frac{x_1}{\eps},x_2 \right)$. Using that $\Lambda_{11}$ is periodic of its first variable, we deduce that the first term of~\eqref{eq:covid10} can be bounded as
\begin{align}
  \left| \frac{1}{\eps} \int_\Omega v_1(x) \, \partial_1 \Lambda_{11}\left( \frac{x_1}{\eps},x_2 \right) \partial_1 \varphi(x_1) \right|
  &\leq
  2 \, \left\| \Lambda_{11}^\eps \right\|_{L^2(\Omega)} \| v_1 \|_{H^1(\Omega)} \, \| \nabla \varphi \|_{W^{1,\infty}(\omega)}
  \nonumber
  \\
  &\leq
  C \, | \omega |^{1/2} \, \| \Lambda_{11} \|_{L^2(\mathcal{Y})} \, \| e(v) \|_{L^2(\Omega)} \, \| \nabla \varphi \|_{W^{1,\infty}(\omega)}
  \nonumber
  \\
  &\leq
  C \, | \omega |^{1/2} \, \| e^\eps(v) \|_{L^2(\Omega)} \, \| \nabla \varphi \|_{W^{1,\infty}(\omega)},
  \label{eq:dur2}
\end{align}
where we have used a Korn inequality (see Lemma~\ref{lem:Korn3}) in the second line.

To estimate the second term of~\eqref{eq:covid10}, we first observe, using Assumption~(ii), that
\begin{equation} \label{eq:covid11}
0 = \int_{\mathcal{Y}} \overline{Z}_{11} = \int_{\mathcal{Y}} \overline{z}_{11} + \int_{\mathcal{Y}} \mu = \int_{-1/2}^{1/2} \mu,
\end{equation}
where we have used that $\dps \int_0^1 \overline{z}_{11}(s,x_2) \, ds = 0$. Similarly, using Assumption~(iii), we write that
\begin{equation} \label{eq:covid12}
0 = \int_{\mathcal{Y}} x_2 \, \overline{Z}_{11} = \int_{\mathcal{Y}} x_2 \, \overline{z}_{11} + \int_{\mathcal{Y}} x_2 \, \mu = \int_{-1/2}^{1/2} x_2 \, \mu,
\end{equation}
where we have again used that the average of $\overline{z}_{11}$ with respect to its first variable vanishes. We thus note that the second term of~\eqref{eq:covid10} would vanish if $v_1$ was an affine function in term of $x_2$. This motivates the idea to use a Taylor-like expansion on $v_1$ to estimate the second term of~\eqref{eq:covid10}. More precisely, we write this term as follows:
\begin{align}
  & \frac{1}{\eps} \int_\Omega v_1(x_1,x_2) \, \mu(x_2) \, \partial_1 \varphi(x_1) \, dx
  \nonumber
  \\
  &=
  \frac{1}{\eps} \int_\Omega \left[ v_1\left(x_1,-\frac{1}{2}\right) + \int_{-1/2}^{x_2} \partial_2 v_1(x_1,s) \, ds \right] \mu(x_2) \, \partial_1 \varphi(x_1) \, dx
  \nonumber
  \\
  &=
  \frac{1}{\eps} \int_\Omega \left[ v_1\left(x_1,-\frac{1}{2}\right) + 2 \int_{-1/2}^{x_2} e_{12}(v)(x_1,s) \, ds - \int_{-1/2}^{x_2} \partial_1 v_2(x_1,s) \, ds \right] \mu(x_2) \, \partial_1 \varphi(x_1) \, dx
  \nonumber
  \\
  &=
  2 \int_\Omega \left[ \int_{-1/2}^{x_2} e^\eps_{12}(v)(x_1,s) \, ds \right] \mu(x_2) \, \partial_1 \varphi(x_1) \, dx - \frac{1}{\eps} \int_\Omega \left[ \int_{-1/2}^{x_2} \partial_1 v_2(x_1,s) \, ds \right] \mu(x_2) \, \partial_1 \varphi(x_1) \, dx,
  \label{eq:covid13}
\end{align}
where we have used~\eqref{eq:covid11} to obtain the last equality. Using a Fubini argument, we write
$$
\int_\Omega \left[ \int_{-1/2}^{x_2} e^\eps_{12}(v)(x_1,s) \, ds \right] \mu(x_2) \, \partial_1 \varphi(x_1) \, dx
=
\int_{(x_1,s) \in \Omega} e^\eps_{12}(v)(x_1,s) \, \partial_1 \varphi(x_1) \left[ \int_s^{1/2} \mu(x_2) \, dx_2 \right].
$$
Using that $\dps \left| \int_s^{1/2} \mu(x_2) \, dx_2 \right| \leq \| \mu \|_{L^2(-1/2,1/2)}$ for any $s \in (-1/2,1/2)$, we estimate the first term of~\eqref{eq:covid13} as follows:
\begin{align}
\left| \int_\Omega \left[ \int_{-1/2}^{x_2} e^\eps_{12}(v)(x_1,s) \, ds \right] \mu(x_2) \, \partial_1 \varphi(x_1) \, dx \right|
&\leq
\| \mu \|_{L^2(-1/2,1/2)} \, \int_{(x_1,s) \in \Omega} \left| e^\eps_{12}(v)(x_1,s) \, \partial_1 \varphi(x_1) \right|
\nonumber
\\
&\leq
\| \overline{Z} \|_{L^2(\mathcal{Y})} \, \| e^\eps_{12}(v) \|_{L^2(\Omega)} \, \| \nabla \varphi \|_{L^2(\Omega)}
\nonumber
\\
&\leq
C \, |\omega|^{1/2} \, \| e^\eps_{12}(v) \|_{L^2(\Omega)} \, \| \nabla \varphi \|_{L^\infty(\omega)}.
 \label{eq:covid14}
\end{align}
Using that $v_2(x_1,\cdot)$ vanishes for any $x_1 \in \partial \omega$, we integrate by parts in the second term of~\eqref{eq:covid13} and write
\begin{align*}
  & - \frac{1}{\eps} \int_\Omega \left[ \int_{-1/2}^{x_2} \partial_1 v_2(x_1,s) \, ds \right] \mu(x_2) \, \partial_1 \varphi(x_1) \, dx
  \\
  &=
   - \frac{1}{\eps} \int_\Omega \partial_1 \left[ \int_{-1/2}^{x_2} v_2(x_1,s) \, ds \right] \mu(x_2) \, \partial_1 \varphi(x_1) \, dx
  \\
  &=
  \frac{1}{\eps} \int_\Omega \left[ \int_{-1/2}^{x_2} v_2(x_1,s) \, ds \right] \mu(x_2) \, \partial_{11} \varphi(x_1) \, dx
  \\
  &=
  \frac{1}{\eps} \int_\Omega \left[ \int_{-1/2}^{x_2} \left\{ v_2\left(x_1,-\frac{1}{2}\right) + \int_{-1/2}^s \partial_2 v_2(x_1,s') \, ds' \right\} ds \right] \mu(x_2) \, \partial_{11} \varphi(x_1) \, dx
  \\
  &=
  \frac{1}{\eps} \int_\Omega \left[ \left(x_2+\frac{1}{2} \right) v_2\left(x_1,-\frac{1}{2}\right) + \int_{-1/2}^{x_2} \left\{ \int_{-1/2}^s \eps^2 \, e_{22}^\eps(v)(x_1,s') \, ds' \right\} ds \right] \mu(x_2) \, \partial_{11} \varphi(x_1) \, dx.
\end{align*}
The integral of the first term above vanishes in view of~\eqref{eq:covid11} and~\eqref{eq:covid12}. We thus deduce that the second term of~\eqref{eq:covid13} reads as
\begin{align*}
  & - \frac{1}{\eps} \int_\Omega \left[ \int_{-1/2}^{x_2} \partial_1 v_2(x_1,s) \, ds \right] \mu(x_2) \, \partial_1 \varphi(x_1) \, dx
  \\
  &=
  \eps \int_\Omega \left[ \int_{-1/2}^{x_2} \left\{ \int_{-1/2}^s e_{22}^\eps(v)(x_1,s') \, ds' \right\} ds \right] \mu(x_2) \, \partial_{11} \varphi(x_1) \, dx
  \\
  &=
  \eps \int_{(x_1,s') \in \Omega} e_{22}^\eps(v)(x_1,s') \, \partial_{11} \varphi(x_1) \left[ \int_{s=s'}^{1/2} \int_{x_2=s}^{1/2} \mu(x_2) \, ds \, dx_2 \right] dx_1 \, ds'.
\end{align*}
Using that $\dps \left| \int_{s=s'}^{1/2} \int_{x_2=s}^{1/2} \mu(x_2) \, ds \, dx_2 \right| \leq \| \mu \|_{L^2(-1/2,1/2)}$ for any $s' \in (-1/2,1/2)$, we estimate the second term of~\eqref{eq:covid13} as follows:
\begin{align}
  & \left| - \frac{1}{\eps} \int_\Omega \left[ \int_{-1/2}^{x_2} \partial_1 v_2(x_1,s) \, ds \right] \mu(x_2) \, \partial_1 \varphi(x_1) \, dx \right|
  \nonumber
  \\
  & \leq
  \eps \, \| \mu \|_{L^2(-1/2,1/2)} \, \int_{(x_1,s') \in \Omega} \left| e_{22}^\eps(v)(x_1,s') \, \partial_{11} \varphi(x_1) \right|
  \nonumber
  \\
  &\leq
  \eps \, \| \overline{Z} \|_{L^2(\mathcal{Y})} \, \| e^\eps_{22}(v) \|_{L^2(\Omega)} \, \| \nabla^2 \varphi \|_{L^2(\Omega)}
  \nonumber
  \\
  &\leq
  C \, \eps \, |\omega|^{1/2} \, \| e^\eps_{22}(v) \|_{L^2(\Omega)} \, \| \nabla^2 \varphi \|_{L^\infty(\omega)}.
 \label{eq:covid15}
\end{align}
Collecting~\eqref{eq:covid10}, \eqref{eq:dur2}, \eqref{eq:covid13}, \eqref{eq:covid14} and~\eqref{eq:covid15}, we deduce that
$$
\left| \frac{1}{\eps} \int_\Omega v_1 \, \overline{Z}_{11}\left( \frac{\cdot}{\eps},x_2 \right) \partial_1 \varphi \right|
\leq
C \, |\omega|^{1/2} \, \| e^\eps(v) \|_{L^2(\Omega)} \, \| \nabla \varphi \|_{W^{1,\infty}(\omega)},
$$
and thus~\eqref{eq:covid9}. This concludes the proof of Lemma~\ref{lemma:minZ2}.
\end{proof}

\subsection{Strong convergence result: the membrane case}\label{sec:membrane}


For any $u \in (H^1(\Omega))^d$, we define the norm $\| u \|_{H^1_\eps(\Omega)}$ as follows (compare with~\eqref{eq:defnorm}):
\begin{equation} \label{eq:defnorm_vec}
\|u\|_{H^1_\eps(\Omega)} := \sqrt{ \frac{\| u \|_{L^2_w(\Omega)}^2}{\left[ \max\left(|\omega|^{\frac{1}{d-1}}, \eps^2 \, |\omega|^{-\frac{1}{d-1}}\right) \right]^2} + \| e^\eps(u) \|_{L^2(\Omega)}^2 },
\end{equation}
where we recall that the weighted $L^2_w(\Omega)$ norm is defined by~\eqref{eq:def_L2_w}. The second term $\| e^\eps(u) \|_{L^2(\Omega)}$ is indeed the relevant energy norm for~\eqref{formvarelast}, and the scaling of the first term (as well as the use of a weighted $L^2$ norm) is motivated by Lemma~\ref{poincarebis2}. As a direct consequence of~\eqref{ineq:1_vec}, we have that
\begin{equation} \label{eq:poinK_corro_vec}
\forall u \in V, \quad \| e^\eps(u) \|_{L^2(\Omega)} \leq \| u \|_{H^1_\eps(\Omega)} \leq C \, \| e^\eps(u) \|_{L^2(\Omega)},
\end{equation}
for some constant $C>0$ independent of $\eps$ and $\omega$.


\medskip

In this section, we assume that we are in the membrane case~\eqref{eq:ass_membrane}. The aim of this section is to prove the following strong convergence result, which is our main result in the membrane case. Recall that, in this case, $u^\star$ belongs to $\VKL^{\mathcal{M}}$ (see Lemma~\ref{lem:ustarsym}).

\begin{theorem} \label{thconvforte2}
  Under Assumptions~\eqref{eq:inde_eps_vec} and~\eqref{eq:assump_g}, consider the solution $u^\eps$ to~\eqref{formvarelast} and its homogenized limit $u^\star$, solution to~\eqref{formvarelasthomog}. Assume that $A$ satisfies the symmetries~\eqref{hyp:symA}, that we are in the membrane case~\eqref{eq:ass_membrane} and that $\omega$ satisfies~\eqref{eq:shape_regul}. 
  We assume that
\begin{equation} \label{eq:hyp_regul_f}
\text{$f_\alpha \in L^2\left( \left(-\frac{1}{2}, \frac{1}{2}\right), H^1(\omega)\right)$ for any $1 \leq \alpha \leq d-1$ and $f_d \in L^2(\Omega)$}
\end{equation}
and that
\begin{equation} \label{eq:hyp_regul_h}
\text{$(h_\pm)_\alpha \in H^1(\omega)$ for any $1 \leq \alpha \leq d-1$ and $(h_\pm)_d \in L^2(\omega)$}.
\end{equation}
We also assume that $e(u^\star+g)$ belongs to $(W^{2,\infty}(\omega))^{d \times d}$ and that, for any $1 \leq \alpha,\beta \leq d-1$, the corrector $w^{\alpha \beta}$ defined by~\eqref{el-prcor1} satisfies $\dps w^{\alpha \beta} \in \left[ W^{1,\infty}\left(\mathbb{R}^{d-1} \times \left( -\demi, \demi \right) \right) \right]^d$. We additionally assume that $d=2$. We introduce the following two-scale expansion: for any $1 \leq \gamma \leq d-1$, let
$$
u_\gamma^{\eps,1}(x) := u_\gamma^\star(x') + \eps \, w_\gamma^{\alpha \beta} \left(\frac{x'}{\eps},x_d \right) e_{\alpha \beta}(u^\star+g)(x')
$$
for any $x=(x',x_d) \in \Omega$, and
$$
u_d^{\eps,1}(x) := \eps^2 \, w_d^{\alpha \beta} \left(\frac{x'}{\eps},x_d \right) e_{\alpha \beta}(u^\star+g)(x').
$$
Then, there exists a constant $C>0$ independent of $\eps$, $\omega$ (but depending on the constant $\eta$ of~\eqref{eq:shape_regul}), $u^\star$, $f$, $g$ and $h_\pm$ such that
\begin{multline} \label{eq:titi7_el}
  \| u^\eps - u^{\eps,1} \|_{H^1_\eps(\Omega)} \leq C \, \sqrt{\eps} \, \Big( |\omega|^{\frac{d-2}{2(d-1)}} \, \| e(u^\star+g) \|_{L^\infty(\omega)} \\ + \sqrt{\eps} \, |\omega|^{1/2} \, \| \nabla \big( e(u^\star+g) \big) \|_{W^{1,\infty}(\omega)} + \sqrt{\eps} \, {\cal N}^{\rm memb}_\eps(f,h_\pm) \Big),
\end{multline}
where the norm $\| \cdot \|_{H^1_\eps(\Omega)}$ is defined by~\eqref{eq:defnorm_vec} and where ${\cal N}^{\rm memb}_\eps(f,h_\pm)$ is defined by
\begin{multline} \label{eq:def_norme_N_pre}
  {\cal N}^{\rm memb}_\eps(f,h_\pm) = \sum_{\alpha=1}^{d-1} \| f_\alpha \|_{L^2(\Omega)} + \eps \, \| f_d \|_{L^2(\Omega)} + \eps \sum_{\alpha,\beta=1}^{d-1} \| \partial_\beta f_\alpha \|_{L^2(\Omega)} \\ + \sum_{\alpha=1}^{d-1} \| (h_\pm)_\alpha \|_{L^2(\omega)} + \eps \, \| (h_\pm)_d \|_{L^2(\omega)} + \eps \sum_{\alpha,\beta=1}^{d-1} \| \partial_\beta (h_\pm)_\alpha \|_{L^2(\omega)}.
\end{multline}
%
%
\end{theorem}

\begin{remark} \label{rem:regul_ustar_memb}
  As for the diffusion case (see Remark~\ref{rem:regul_ustar_diff}), we wish to point out that the assumption $e(u^\star + g) \in (W^{2,\infty}(\omega))^{d \times d}$ is a standard assumption when proving convergence rates of two-scale expansions (see, e.g.,~\cite[p.~28]{jikov}). Note that, in view of~\eqref{formvarelasthomog}, which can be recast as~\eqref{formvar:pb_membrane} in the membrane case, this assumption implies that $\m(f') + h'_+ + h'_-$ belongs to $(L^\infty(\omega))^{d-1}$.
\end{remark}

Up to lower order terms, we have 
$$
e^\eps(u^{\eps,1}+g)(x) \approx \left[ e_\alpha \otimes e_\beta + e(w^{\alpha \beta})\left(\frac{x'}{\eps},x_d\right) \right] e_{\alpha \beta}(u^\star+g)(x').
$$

\medskip

The proof of Theorem~\ref{thconvforte2} somewhat follows similar lines as the proof of Theorem~\ref{thconvforte_diffusion}. It relies on Lemma~\ref{lemma:minZ2}, which is however more involved than the corresponding result (i.e. Lemma~\ref{minZ}) of the scalar case, which is used to prove Theorem~\ref{thconvforte_diffusion}. In the statement of Theorem~\ref{thconvforte2}, the restriction to the case $d=2$ stems from the fact that Lemma~\ref{lemma:minZ2} has only been established in this case. Should Lemma~\ref{lemma:minZ2} hold in higher-dimensional settings, so would Theorem~\ref{thconvforte2}.

\medskip

Before turning to the proof of Theorem~\ref{thconvforte2}, we state a result similar to Lemma~\ref{potmoy2}, but only valid in the membrane case:

\begin{lemma} \label{potmoy3}
  Let $h \in H^1(\omega)$ and $v \in V$, where we recall that $V$ is defined by~\eqref{def:V2}. We assume that $v_d \in \O$. Let $1 \leq \alpha \leq d-1$. Then, for any $\dps z \in \left[ -\frac{1}{2},\frac{1}{2} \right]$, we have
  $$
  \left| \int_\omega \big( v_\alpha(\cdot,z)-\m(v_\alpha) \big) \, h \right| \leq \Big( 2 \, \eps \, \| h \|_{L^2(\omega)} + \eps^2 \, \| \partial_\alpha h \|_{L^2(\omega)} \Big) \, \| e^\eps(v) \|_{L^2(\Omega)}.
  $$
\end{lemma}

\begin{proof}[Proof of Lemma~\ref{potmoy3}]
Let $v \in V$ and let $1 \leq \alpha \leq d-1$. For any $z \in [-1/2,1/2]$, we have
\begin{multline} \label{eq:lundi1}
v_\alpha(\cdot,z)-\m(v_\alpha)
=
\int_{-1/2}^{1/2} \big( v_\alpha(\cdot,z)-v_\alpha(\cdot,t) \big) \, dt
\\ =
\int_{-1/2}^{1/2} \int_t^z \partial_d v_\alpha(\cdot,s) \, ds \, dt
=
T_1(\cdot,z) - T_2(\cdot,z),
\end{multline}
with
$$
T_1(\cdot,z) = \eps \int_{-1/2}^{1/2} \int_t^z \frac{2}{\eps} \, e_{\alpha d}(v)(\cdot,s) \, ds \, dt,
\qquad
T_2(\cdot,z) = \int_{-1/2}^{1/2} \int_t^z \partial_\alpha v_d(\cdot,s) \, ds \, dt.
$$
It follows that, for any $z \in [-1/2,1/2]$,
$$
\| T_1(\cdot,z) \|^2_{L^2(\omega)}
=
\int_\omega \left| T_1(\cdot,z) \right|^2
\leq
4 \, \eps^2 \int_\omega \left| \int_{-1/2}^{1/2} \left| e^\eps_{\alpha d}(v)(\cdot,s) \right| ds \right|^2
\leq
4 \, \eps^2 \, \| e^\eps (v) \|_{L^2(\Omega)}^2,
$$
and thus
\begin{equation} \label{eq:lundi3}
\left| \int_\omega T_1(\cdot,z) \, h \right| \leq \| h \|_{L^2(\omega)} \, \| T_1(\cdot,z) \|_{L^2(\omega)} \leq 2 \, \eps \, \| h \|_{L^2(\omega)} \, \| e^\eps (v) \|_{L^2(\Omega)}.
\end{equation}
We now turn to the second term of~\eqref{eq:lundi1} and write
\begin{align}
  \int_\omega T_2(x',z) \, h(x') \, dx' 
  &=
  \int_\omega \int_{-1/2}^{1/2} \int_t^z h(x') \, \partial_\alpha v_d(x',s) \, ds \, dt \, dx'
  \nonumber
  \\
  &=
  - \int_\omega \int_{-1/2}^{1/2} \int_t^z \partial_\alpha h(x') \, v_d(x',s) \, ds \, dt \, dx',
  \label{eq:lundi2}
\end{align}
where we have used an integration by part and the fact that $v_d(x',s) = 0$ for any $x' \in \partial \omega$.

We next write the following identities:
\begin{align*}
  v_d(x',s) &= v_d(x',-1/2) + \int_{-1/2}^s \partial_d v_d(x',s') \, ds',
  \\
  v_d(x',s) &= v_d(x',1/2) - \int_s^{1/2} \partial_d v_d(x',s') \, ds'.
\end{align*}
Adding them up, and using that $v_d \in \O$, we deduce that
$$
v_d(x',s) = \frac{1}{2} \int_{-1/2}^s \partial_d v_d(x',s') \, ds' - \frac{1}{2} \int_s^{1/2} \partial_d v_d(x',s') \, ds',
$$
and hence that, for any $s \in (-1/2,1/2)$,
$$
| v_d(x',s) | \leq \int_{-1/2}^{1/2} | \partial_d v_d(x',s') | \, ds'.
$$
Inserting this relation in~\eqref{eq:lundi2}, we thus obtain
$$
\left| \int_\omega T_2(\cdot,z) \, h \right|
\leq
\int_\omega \int_{-1/2}^{1/2} | \partial_\alpha h(x') \, \partial_d v_d(x',s') | \, ds' \, dx'
\leq
\eps^2 \, \| \partial_\alpha h \|_{L^2(\omega)} \, \| e^\eps(v) \|_{L^2(\Omega)}.
$$
Collecting~\eqref{eq:lundi1}, \eqref{eq:lundi3} and the above bound yields the claimed estimate and concludes the proof of Lemma~\ref{potmoy3}.
\end{proof}

\begin{proof}[Proof of Theorem~\ref{thconvforte2}]
The proof falls in four steps. In the first step, we correct for the boundary mismatch between $u^\eps$ and its approximation $u^{\eps,1}$. In Steps~2 and~3, we show that $u^{\eps,1}$ is close to $u^\eps$ in the bulk of the domain. The desired conclusion is reached in Step~4 by collecting all the estimates.
	
\medskip

\noindent
{\bf Step~1.} Let $\tau_\eps \in \mathcal{D}(\omega)$ such that $0 \leq \tau_\eps \leq 1$ in $\omega$ and such that $\tau_\eps(x') = 1$ for any $x' \in \omega$ such that $\text{dist}(x',\partial \omega) \geq \eps$. Since $\omega$ is smooth, we can choose $\tau_\eps$ such that $\eps \|\nabla \tau_\eps \|_{L^\infty(\omega)} \leq C$ for some $C>0$ independent of $\omega$ and $\eps$. We define $\omega_\eps := \{ x' \in \omega \text{ such that } \text{dist}(x',\partial \omega) \geq \eps \}$ and $\dps \Omega_\eps := \omega_\eps \times \left( -\frac{1}{2},\frac{1}{2} \right)$. Note that $|\Omega \setminus \Omega_\eps| \leq C \, \eps \, |\omega|^{\frac{d-2}{d-1}}$.

We introduce the function $v^{\eps,1}$ defined for $x = (x',x_d) \in \Omega$ by
\begin{align*}
  \forall 1 \leq \gamma \leq d-1, \qquad v^{\eps,1}_\gamma(x) & := u^\star_\gamma(x') + \eps \, \tau_\eps(x') \, w_\gamma^{\alpha \beta} \left( \frac{x'}{\eps}, x_d \right) e_{\alpha \beta}(u^\star+g)(x'),
  \\
  v^{\eps,1}_d(x) & := \eps^2 \, \tau_\eps(x') \, w_d^{\alpha \beta} \left( \frac{x'}{\eps}, x_d \right) e_{\alpha \beta}(u^\star+g)(x').
\end{align*}
By definition of $\tau_\eps$, we have $v^{\eps,1} \in V$ and $v^{\eps,1} = u^{\eps,1}$ in $\Omega_\eps$. In this first step of the proof, we bound $\| u^{\eps,1} - v^{\eps,1} \|_{H^1_\eps(\Omega)}$. We compute that 
\begin{gather*}
u^{\eps,1}_\gamma(x) - v^{\eps,1}_\gamma(x) = \eps \, (1-\tau_\eps(x')) \, w_\gamma^{\alpha \beta} \left( \frac{x'}{\eps},x_d \right) e_{\alpha \beta}(u^\star+g)(x'),
\\
u^{\eps,1}_d(x) - v^{\eps,1}_d(x) = \eps^2 \, (1-\tau_\eps(x')) \, w_d^{\alpha \beta} \left( \frac{x'}{\eps},x_d \right) e_{\alpha \beta}(u^\star+g)(x').
\end{gather*}
We thus get that
\begin{align}
  \sum_{\gamma=1}^{d-1} \| v^{\eps,1}_\gamma - u^{\eps,1}_\gamma \|^2_{L^2(\Omega)}
  &\leq
  C \, \eps^2 \, \| 1-\tau_\eps \|^2_{L^2(\Omega \setminus \Omega_\eps)} \sup_{1 \leq \alpha,\beta \leq d-1} \| w^{\alpha \beta} \|_{L^\infty}^2 \, \| e(u^\star+g) \|^2_{L^\infty(\omega)}
  \nonumber
  \\
  &\leq
  C \, \eps^2 \, |\Omega \setminus \Omega_\eps| \, \| e(u^\star+g) \|^2_{L^\infty(\omega)}
  \nonumber
  \\
  &\leq
  C \, \eps^3 \, |\omega|^{\frac{d-2}{d-1}} \, \| e(u^\star+g) \|^2_{L^\infty(\omega)}
  \nonumber
  \\
  &\leq
  C \, \eps \, |\omega|^{\frac{d-2}{d-1}} \, \left[ \max\left(|\omega|^{\frac{1}{d-1}}, \eps^2 \, |\omega|^{-\frac{1}{d-1}}\right) \right]^2 \, \| e(u^\star+g) \|^2_{L^\infty(\omega)},
  \label{diffapp3_L2_el_gamma}
\end{align}
where the last estimate stems from the fact that $\dps \eps \leq \max\left(|\omega|^{\frac{1}{d-1}}, \eps^2 \, |\omega|^{-\frac{1}{d-1}}\right)$. We also have that
\begin{align}
  |\omega|^{\frac{-2}{d-1}} \, \| v^{\eps,1}_d - u^{\eps,1}_d \|^2_{L^2(\Omega)}
  &\leq
  C \, \eps^5 \, |\omega|^{\frac{-2}{d-1}} \, |\omega|^{\frac{d-2}{d-1}} \, \| e(u^\star+g) \|^2_{L^\infty(\omega)}
  \nonumber
  \\
  &\leq
  C \, \eps \, |\omega|^{\frac{d-2}{d-1}} \, \left[ \max\left(|\omega|^{\frac{1}{d-1}}, \eps^2 \, |\omega|^{-\frac{1}{d-1}}\right) \right]^2 \, \| e(u^\star+g) \|^2_{L^\infty(\omega)},
  \label{diffapp3_L2_el_d}
\end{align}
where the last estimate stems from the fact that $\dps \eps^2 \, |\omega|^{\frac{-1}{d-1}} \leq \max\left(|\omega|^{\frac{1}{d-1}}, \eps^2 \, |\omega|^{-\frac{1}{d-1}}\right)$. Collecting~\eqref{diffapp3_L2_el_gamma} and~\eqref{diffapp3_L2_el_d}, we thus deduce that
\begin{equation} \label{diffapp3_L2_el}
  \| v^{\eps,1} - u^{\eps,1} \|^2_{L^2_w(\Omega)} \leq C \, \eps \, |\omega|^{\frac{d-2}{d-1}} \, \left[ \max\left(|\omega|^{\frac{1}{d-1}}, \eps^2 \, |\omega|^{-\frac{1}{d-1}}\right) \right]^2 \, \| e(u^\star+g) \|^2_{L^\infty(\omega)}.
\end{equation}
We next compute that $e^\eps(u^{\eps,1} - v^{\eps,1}) = E_0^\eps - E_1^\eps + \eps \, E_2^\eps$, where
\begin{align*}
  E_0^\eps(x) &= (1-\tau_\eps(x')) \, e(w^{\alpha \beta})\left( \frac{x'}{\eps},x_d \right) e_{\alpha \beta} (u^\star+g)(x'),
  \\
  E_1^\eps(x) &= \eps \, \nabla \tau_\eps(x') \otimes^s w^{\alpha \beta}\left( \frac{x'}{\eps},x_d \right) e_{\alpha \beta} (u^\star+g)(x'),
  \\
  E_2^\eps(x) &= (1-\tau_\eps(x')) \, w^{\alpha \beta} \left( \frac{x'}{\eps},x_d \right) \otimes^s \nabla(e_{\alpha \beta} (u^\star+g))(x'),
\end{align*}
where, for any vectors $q$ and $p$ in $\R^d$, $p \otimes^s q = (p \otimes q + q \otimes p)/2$. We bound the above three terms in $L^2(\Omega)$ norm, using that $w^{\alpha \beta} \in (W^{1,\infty})^d$, $e(u^\star+g)$ belongs to $(W^{1,\infty}(\omega))^{d \times d}$, $0 \leq \tau_\eps \leq 1$ and that $\eps \, \|\nabla \tau_\eps \|_{L^\infty(\omega)} \leq C$. We thus obtain that
\begin{align*}
  \|E_2^\eps\|^2_{L^2(\Omega)} & \leq C \, \sup_{1 \leq \alpha,\beta \leq d-1} \| w^{\alpha \beta} \|_{L^\infty}^2 \, \| \nabla \big( e(u^\star+g) \big) \|^2_{L^2(\Omega)}
  \\
  & \leq C \, |\omega| \, \| \nabla \big( e(u^\star+g) \big) \|^2_{L^\infty(\omega)},
  \\
  \| E_1^\eps\|^2_{L^2(\Omega)} & \leq C \, |\Omega \setminus \Omega_\eps| \, \sup_{1 \leq \alpha,\beta \leq d-1} \| w^{\alpha \beta} \|_{L^\infty}^2 \, \| e(u^\star+g) \|^2_{L^\infty(\omega)} 
  \\
  & \leq C \, \eps \, |\omega|^{\frac{d-2}{d-1}} \, \| e(u^\star+g) \|^2_{L^\infty(\omega)},
  \\
  \| E_0^\eps\|^2_{L^2(\Omega)} & \leq C \, |\Omega \setminus \Omega_\eps| \, \sup_{1 \leq \alpha, \beta \leq d-1} \| \nabla w^{\alpha \beta} \|_{L^\infty}^2 \, \| e(u^\star+g) \|^2_{L^\infty(\omega)}
  \\
  & \leq C \, \eps \, |\omega|^{\frac{d-2}{d-1}} \, \| e(u^\star+g) \|^2_{L^\infty(\omega)}.  
\end{align*}
This implies that
\begin{equation} \label{diffapp4}
  \| e^\eps(v^{\eps,1} - u^{\eps,1}) \|^2_{L^2(\Omega)} \leq C \, \eps \left( |\omega|^{\frac{d-2}{d-1}} \, \| e(u^\star+g) \|^2_{L^\infty(\omega)} + \eps \, |\omega| \, \| \nabla \big( e(u^\star+g) \big) \|^2_{L^\infty(\omega)} \right).
\end{equation}
Collecting~\eqref{diffapp3_L2_el} and~\eqref{diffapp4}, we deduce
\begin{equation} \label{diffapp3_H1_el}
  \| v^{\eps,1} - u^{\eps,1} \|^2_{H^1_\eps(\Omega)} \leq C \, \eps \left( |\omega|^{\frac{d-2}{d-1}} \, \| e(u^\star+g) \|^2_{L^\infty(\omega)} + \eps \, |\omega| \, \| \nabla \big( e(u^\star+g) \big) \|^2_{L^\infty(\omega)} \right),
\end{equation}
where we recall that the norm $\| \cdot \|_{H^1_\eps(\Omega)}$ is defined by~\eqref{eq:defnorm_vec}.

\medskip
  
\noindent 
{\bf Step~2.} We now bound $\overline{v}^\eps := u^\eps - v^{\eps,1}$. Using the coercivity of $A^\eps$, we have
\begin{align}
  c_- \| e^\eps(\overline{v}^\eps) \|_{L^2(\Omega)}^2
  &\leq
  \int_\Omega A^\eps e^\eps(\overline{v}^\eps) : e^\eps(\overline{v}^\eps)
  \nonumber
  \\
  &=
  \int_\Omega A^\eps e^\eps(u^\eps - u^{\eps,1}) : e^\eps(\overline{v}^\eps) + \int_\Omega A^\eps e^\eps(u^{\eps,1} - v^{\eps,1}) : e^\eps(\overline{v}^\eps).
  \label{estim2_el}
\end{align}
We bound the second term of~\eqref{estim2_el} using~\eqref{diffapp4}:
\begin{multline}\label{eq:titi_el}
  \left| \int_\Omega A^\eps e^\eps(u^{\eps,1} - v^{\eps,1}) : e^\eps(\overline{v}^\eps) \right| \\ \leq C \left( \sqrt{\eps} \, |\omega|^{\frac{d-2}{2(d-1)}} \, \| e(u^\star+g) \|_{L^\infty(\omega)} + \eps \, |\omega|^{1/2} \, \| \nabla \big( e(u^\star+g) \big) \|_{L^\infty(\omega)} \right) \| e^\eps(\overline{v}^\eps) \|_{L^2(\Omega)}.
\end{multline}
For the first term of~\eqref{estim2_el}, we write
$$
\int_\Omega A^\eps e^\eps(u^\eps - u^{\eps,1}) : e^\eps(\overline{v}^\eps) = \int_\Omega A^\eps e^\eps(u^\eps + g) : e^\eps(\overline{v}^\eps) - \int_\Omega A^\eps e^\eps( u^{\eps,1} + g) : e^\eps(\overline{v}^\eps). 
$$
We define the remainder terms $R^\eps_1$ and $R^\eps_2$ by
\begin{equation} \label{eq:r12}
  R^\eps_1 := \int_\Omega A^\eps e^\eps(u^\eps + g) : e^\eps(\overline{v}^\eps) - \int_\Omega \left[ \int_\Y A\left( e_\alpha \otimes e_\beta + e(w^{\alpha \beta}) \right) \right] e_{\alpha \beta}(u^\star+g) : e^\eps(\overline{v}^\eps)
\end{equation}
and
\begin{equation} \label{eq:r22}
  R^\eps_2 := \int_\Omega A^\eps e^\eps(u^{\eps,1} + g) : e^\eps(\overline{v}^\eps) - \int_\Omega A^\eps(x) \left( e_\alpha \otimes e_\beta + e(w^{\alpha \beta}) \left(\frac{x'}{\eps},x_d \right) \right) e_{\alpha \beta}(u^\star+g)(x') : e^\eps(\overline{v}^\eps)(x).
\end{equation}
We then compute that
\begin{multline} \label{eq:titi2_el}
  \int_\Omega A^\eps e^\eps(u^\eps - u^{\eps,1}) : e^\eps(\overline{v}^\eps)
  =
  R^\eps_1 - R^\eps_2 \\
  + \int_\Omega \left( \left[ \int_\Y A\left( e_\alpha \otimes e_\beta + e(w^{\alpha \beta}) \right) \right] - A^\eps \left( e_\alpha \otimes e_\beta + e(w^{\alpha \beta})_\eps \right) \right) e_{\alpha \beta}(u^\star + g) : e^\eps(\overline{v}^\eps),
\end{multline}
where we have used the short-hand notation $\dps e(w^{\alpha \beta})_\eps(x) = e(w^{\alpha \beta}) \left(\frac{x'}{\eps},x_d \right)$.

In the next step, we are going to show that
\begin{equation} \label{eq:titi5_el}
|R^1_\eps| + |R^2_\eps| \leq C \, \eps \, \Big( |\omega|^{1/2} \, \| \nabla \big( e(u^\star+g) \big) \|_{L^\infty(\omega)} + {\cal N}^{\rm memb}_\eps(f,h_\pm) \Big) \, \| e^\eps(\overline{v}^\eps) \|_{L^2(\Omega)},
\end{equation}
where ${\cal N}^{\rm memb}_\eps(f,h_\pm)$ is defined by~\eqref{eq:def_norme_N_pre}. We are thus left to bound the last term of~\eqref{eq:titi2_el}. To that aim, introduce the matrix-valued function
$$
Z_{\alpha \beta} := \left[ \int_{\mathcal{Y}} A \big( e_\alpha \otimes e_\beta + e(w^{\alpha \beta}) \big) \right] - A \left( e_\alpha \otimes e_\beta + e(w^{\alpha \beta}) \right)
$$
and recast~\eqref{eq:titi2_el} as
\begin{equation} \label{eq:titi3_el}
\int_\Omega A^\eps e^\eps(u^\eps - u^{\eps,1}) : e^\eps(\overline{v}^\eps) = R^\eps_1 - R^\eps_2 + \int_\Omega Z_{\alpha \beta} \left(\frac{x'}{\eps},x_d\right) e_{\alpha \beta}(u^\star + g)(x') : e^\eps(\overline{v}^\eps)(x) \, dx.
\end{equation}
In view of the properties~\eqref{el-prcor1_edp} satisfied by $w^{\alpha \beta}$, we observe that $Z_{\alpha \beta}$ satisfies the assumptions of Lemma~\ref{lemma:minZ2}. This is obvious for Assumptions~(i), (ii), (iv) and~(vi). Assumption~(v) stems from the second line of~\eqref{el-prcor1_edp} and from~\eqref{el-prcor1} where we choose $v(x) = x_d \, e_i$ as test function (for any $1 \leq i \leq d$), which yields
\begin{equation} \label{eq:titi6_el}
\forall 1 \leq \alpha,\beta \leq d-1, \quad \forall 1 \leq i \leq d, \quad \int_{\mathcal{Y}} A \big( e_\alpha \otimes e_\beta + e(w^{\alpha \beta}) \big) : (e_d \otimes e_i) = 0.
\end{equation}
Assumption~(iii) is satisfied because of a parity argument: the membrane corrector $w^{\alpha \beta}$ belongs to $\mathcal{E}^{d-1} \times \mathcal{O}$ for any $1 \leq \alpha,\beta \leq d-1$ (see Lemma~\ref{lem:symcorr}) and $A$ satisfies the symmetries~\eqref{hyp:symA}, which implies that $Z_{\alpha \beta}$ is a function which is even with respect to $x_d$ (and Assumption~(iii) hence holds). In addition, the function $e_{\alpha \beta}(u^\star+g)$ belongs to $W^{2,\infty}(\omega)$ and $\overline{v}^\eps$ belongs to $V$. We are thus in position to use Lemma~\ref{lemma:minZ2}. For any $1\leq \alpha,\beta \leq d-1$, we thus infer that
\begin{equation} \label{eq:titi4_el}
\left| \int_\Omega Z_{\alpha \beta} \left( \frac{x'}{\eps},x_d \right) e_{\alpha \beta}(u^\star+g)(x') : e^\eps(\overline{v}^\eps)(x) \, dx \right| \leq C \, \eps \, |\omega|^{1/2} \, \| \nabla \big( e(u^\star+g) \big) \|_{W^{1,\infty}(\omega)} \, \| e^\eps(\overline{v}^\eps) \|_{L^2(\Omega)}.
\end{equation}
Collecting~\eqref{estim2_el}, \eqref{eq:titi_el}, \eqref{eq:titi3_el}, \eqref{eq:titi5_el} and~\eqref{eq:titi4_el}, we have shown that
\begin{multline*}
\| e^\eps(\overline{v}^\eps) \|_{L^2(\Omega)} \leq C \Big( \eps \, |\omega|^{1/2} \, \| \nabla \big( e(u^\star+g) \big) \|_{W^{1,\infty}(\omega)} \\ + \sqrt{\eps} \, |\omega|^{\frac{d-2}{2(d-1)}} \, \| e(u^\star+g) \|_{L^\infty(\omega)} + \eps \, {\cal N}^{\rm memb}_\eps(f,h_\pm) \Big).
\end{multline*}
Since $\overline{v}^\eps$ belongs to $V$, we can use the estimate~\eqref{eq:poinK_corro_vec}, and we obtain from the above bound that
\begin{multline} \label{eq:tutu}
\| \overline{v}^\eps \|_{H^1_\eps(\Omega)} \leq C \, \sqrt{\eps} \Big( \sqrt{\eps} \, |\omega|^{1/2} \, \| \nabla \big( e(u^\star+g) \big) \|_{W^{1,\infty}(\omega)} \\ + |\omega|^{\frac{d-2}{2(d-1)}} \, \| e(u^\star+g) \|_{L^\infty(\omega)} + \sqrt{\eps} \, {\cal N}^{\rm memb}_\eps(f,h_\pm) \Big).
\end{multline}

\medskip

\noindent 
{\bf Step~3.} We show in this step the claimed estimate~\eqref{eq:titi5_el} on $R^\eps_1$ and $R^\eps_2$. The bound on $R^\eps_2$ comes from inserting in~\eqref{eq:r22} the definition of $u^{\eps,1}$. Using again the short-hand notation $\dps e(w^{\alpha \beta})_\eps(x) = e(w^{\alpha \beta}) \left(\frac{x'}{\eps},x_d \right)$ and $\dps w^{\alpha \beta}_\eps(x) = w^{\alpha \beta}\left(\frac{x'}{\eps},x_d \right)$, we write
\begin{align*}
  R^\eps_2
  &=
  \int_\Omega A^\eps \Big[ e^\eps(u^{\eps,1}+g) - \left( e_\alpha \otimes e_\beta + e(w^{\alpha \beta})_\eps \right) e_{\alpha \beta} (u^\star+g) \Big] : e^\eps(\overline{v}^\eps)
  \\
  &=
  \eps \int_\Omega A^\eps \Big[ w^{\alpha \beta}_\eps \otimes^s \nabla( e_{\alpha \beta} (u^\star +g)) \Big] : e^\eps(\overline{v}^\eps),
\end{align*}
where we recall that $p \otimes^s q = (p \otimes q + q \otimes p)/2$ for any vectors $q$ and $p$ in $\R^d$. We thus obtain
\begin{align}
  | R^\eps_2 |
  & \leq
  C \, \eps \, \| \nabla \big( e(u^\star+g) \big) \|_{L^2(\Omega)} \| e^\eps(\overline{v}^\eps) \|_{L^2(\Omega)}
  \nonumber
  \\
  & \leq
  C \, \eps \, |\omega|^{1/2} \, \| \nabla \big( e(u^\star+g) \big) \|_{L^\infty(\omega)} \| e^\eps(\overline{v}^\eps) \|_{L^2(\Omega)}.
  \label{eq:bound_R2eps_el}
\end{align}
We now bound $R^\eps_1$ defined by~\eqref{eq:r12}. Using the variational formulation~\eqref{formvarelast} with the test function $\overline{v}^\eps \in V$ defined in Step~2, the first term of $R^\eps_1$ reads
$$
\int_\Omega A^\eps e^\eps(u^\eps+g) : e^\eps(\overline{v}^\eps) = \int_\Omega f \cdot \overline{v}^\eps + \int_{\Gamma_\pm} h_\pm \cdot \overline{v}^\eps,
$$
where we recall that $\dps \Gamma_\pm = \omega \times \left\{ \pm \frac{1}{2} \right\}$. Using~\eqref{eq:titi6_el}, we see that $\dps \left[ \int_{\mathcal{Y}} A \big( e_\alpha \otimes e_\beta + e(w^{\alpha \beta}) \big) \right] e_d = 0$ for any $1 \leq \alpha,\beta \leq d-1$ and the second term of $R^\eps_1$ thus reads
\begin{align*}
  \int_\Omega \left[ \int_{\mathcal{Y}} A \big( e_\alpha \otimes e_\beta + e(w^{\alpha \beta}) \big) \right] e_{\alpha \beta}(u^\star+g) : e^\eps(\overline{v}^\eps) 
  &=
  \int_\Omega (K_{11}^\star)_{\alpha \beta \gamma \delta} \, e_{\alpha \beta}(u^\star+g) \, e_{\gamma \delta}(\overline{v}^\eps)
  \\
  &=
  \int_\Omega K_{11}^\star \, e'(u^\star+g) : e'(\overline{v}^\eps)
  \\
  &=
  \int_\omega K_{11}^\star \, e'(u^\star+g) : e'(\m(\overline{v}^\eps)),
\end{align*}
where we have used in the last line that $u^\star+g \in \GKL^\M$ and is thus independent from $x_d$. We proceed by using the variational formulation~\eqref{formvar:pb_membrane} of the homogenized problem in the membrane case and the fact that $\m(\overline{v}^\eps)$ belongs to $(H^1_0(\omega))^d$ (thus $\m((\overline{v}^\eps)') \in (H^1_0(\omega))^{d-1}$ is an admissible test function for~\eqref{formvar:pb_membrane}). We thus obtain that the second term of $R^\eps_1$ reads
\begin{align}
  \int_\Omega \left[ \int_{\mathcal{Y}} A \big( e_\alpha \otimes e_\beta + e(w^{\alpha \beta}) \big) \right] e_{\alpha \beta}(u^\star+g) : e^\eps(\overline{v}^\eps)
  &=
  \int_\omega K_{11}^\star \, e'(u^\star+g) : e'\big(\m((\overline{v}^\eps)')\big)
  \nonumber
  \\
  &=
  \int_\omega \big( \m(f') + h'_\pm \big) \cdot \m((\overline{v}^\eps)').
  \label{eq:pbmbending}
\end{align}
We thus deduce that
\begin{align}
  R^\eps_1
  &= \int_\Omega f \cdot \overline{v}^\eps - \int_\omega \m(f') \cdot \m((\overline{v}^\eps)') + \int_{\Gamma_+} h_+ \cdot \overline{v}^\eps - \int_\omega h'_+ \cdot \m((\overline{v}^\eps)') + \int_{\Gamma_-} h_- \cdot \overline{v}^\eps - \int_\omega h'_- \cdot \m((\overline{v}^\eps)')
  \nonumber
  \\
  &= \int_\Omega f \cdot \overline{v}^\eps - \int_\Omega f' \cdot \m((\overline{v}^\eps)') + \int_{\Gamma_+} h_+ \cdot \overline{v}^\eps - \int_\omega h_+ \cdot \m(\overline{v}^\eps) + \int_{\Gamma_-} h_- \cdot \overline{v}^\eps - \int_\omega h_- \cdot \m(\overline{v}^\eps),
  \label{eq:vac1}
\end{align}
where we have used in the last line, for the boundary terms, that $\overline{v}^\eps_d \in \O$ (recall indeed that, in the membrane case, $u^\eps_d \in \O$ in view of Lemma~\ref{lem:membrane} and $w_d^{\alpha \beta} \in \O$ in view of Lemma~\ref{lem:symcorr}), which implies that $\m((\overline{v}^\eps)_d) = 0$. 

Using again that $\m((\overline{v}^\eps)_d) = 0$, we next write, for the first terms of~\eqref{eq:vac1}, that
$$
\int_\Omega f \cdot \overline{v}^\eps - \int_\Omega f' \cdot \m((\overline{v}^\eps)') = \int_\Omega f_i \, \big(\overline{v}^\eps_i - \m(\overline{v}^\eps_i) \big),
$$
and thus, using Lemma~\ref{potmoy2} for the component $i=d$ and Lemma~\ref{potmoy3} for the other components (which is possible since $\overline{v}^\eps \in V$ and $\overline{v}^\eps_d \in \O$), we obtain
\begin{align}
  & \left| \int_\Omega f \cdot \overline{v}^\eps - \int_\Omega f' \cdot \m((\overline{v}^\eps)') \right|
  \nonumber
  \\
  &\leq
  \sum_{i=1}^d \int_{-1/2}^{1/2} \left| \int_\omega f_i(\cdot,z) \, \big( \overline{v}_i^\eps(\cdot, z) - \m(\overline{v}_i^\eps) \big) \right| \, dz
  \nonumber
  \\
  & \leq
  \|e^\eps( \overline{v}^\eps)\|_{L^2(\Omega)} \left( \int_{-1/2}^{1/2} \eps^2 \, \| f_d(\cdot,z) \|_{L^2(\omega)} + \sum_{\alpha=1}^{d-1} \int_{-1/2}^{1/2} 2 \, \eps \, \| f_\alpha(\cdot,z) \|_{L^2(\omega)} + \eps^2 \, \| \partial_\alpha f_\alpha(\cdot,z) \|_{L^2(\omega)} \right)
  \nonumber
  \\
  & \leq C \, \Big( \eps \, \| f' \|_{L^2(\Omega)} + \eps^2 \, \| f_d \|_{L^2(\Omega)} + \eps^2 \, \| \nabla' f' \|_{L^2(\Omega)} \Big) \, \|e^\eps( \overline{v}^\eps)\|_{L^2(\Omega)}
  \nonumber
  \\
  & \leq C \, \eps \, {\cal N}^{\rm memb}_\eps(f,h_\pm) \, \|e^\eps( \overline{v}^\eps)\|_{L^2(\Omega)},
  \label{eq:vac2}
\end{align}
where we recall that ${\cal N}^{\rm memb}_\eps(f,h_\pm)$ is defined by~\eqref{eq:def_norme_N_pre}. We next turn to the boundary terms of~\eqref{eq:vac1}. Using again Lemmas~\ref{potmoy2} and~\ref{potmoy3}, we write that
\begin{align}
  & \left| \int_{\Gamma_+} h_+ \cdot \overline{v}^\eps - \int_\omega h_+ \cdot \m(\overline{v}^\eps) \right|
  \nonumber
  \\
  &\leq
  \sum_{i=1}^d \left| \int_\omega (h_+)_i \, \big( \overline{v}^\eps_i - \m(\overline{v}^\eps_i) \big) \right|
  \nonumber
  \\
  &\leq
  C \, \Big( \eps \, \| (h_+)' \|_{L^2(\omega)} + \eps^2 \, \| (h_+)_d \|_{L^2(\omega)} + \eps^2 \, \| \nabla' (h_+)' \|_{L^2(\omega)} \Big) \, \| e^\eps(\overline{v}^\eps) \|_{L^2(\Omega)}
  \nonumber
  \\
  &\leq
  C \, \eps \, {\cal N}^{\rm memb}_\eps(f,h_\pm) \, \| e^\eps(\overline{v}^\eps) \|_{L^2(\Omega)},
  \label{eq:vac3}
\end{align}
and likewise for $h_-$.

Collecting~\eqref{eq:vac1}, \eqref{eq:vac2} and~\eqref{eq:vac3}, we deduce that
$$
|R^1_\eps| \leq C \, \eps \, {\cal N}^{\rm memb}_\eps(f,h_\pm) \, \|e^\eps( \overline{v}^\eps)\|_{L^2(\Omega)}.
$$
Collecting this estimate with~\eqref{eq:bound_R2eps_el}, we obtain the claimed bound~\eqref{eq:titi5_el}.
  
\medskip

\noindent
{\bf Step~4.} Collecting~\eqref{diffapp3_H1_el} and~\eqref{eq:tutu}, we deduce~\eqref{eq:titi7_el}, which concludes the proof of Theorem~\ref{thconvforte2}.
\end{proof}


\subsection{Strong convergence result: the bending case}\label{sec:bending}

Stating a similar strong convergence result in the bending case is a much more intricate task than in the membrane case. We stress here the fact that the arguments used for the proof of Theorem~\ref{thconvforte_diffusion} (on the scalar problem) and Theorem~\ref{thconvforte2} (on the membrane case) cannot be applied here. Indeed, notice that in the proof of Theorem~\ref{thconvforte2}, and more precisely in~\eqref{eq:pbmbending}, we have used $\m((\overline{v}^\eps)') \in (H^1_0(\omega))^{d-1}$ as a test function in the variational formulation~\eqref{formvar:pb_membrane} of the homogenized problem. However, in the bending case, it is not clear how to construct from $\overline{v}^\eps \in V$ an admissible test function for the homogenized problem~\eqref{formvar:pb_bending} (namely a test function in $H^2_0(\omega)$). Here, we manage to circumvent this difficulty by using a completely different strategy of proof, inspired by some arguments present in~\cite{destuynder1981comparaison} to handle homogeneous plates (for that case, we also refer to the introductory exposition provided in~\cite[Sec.~1.4.2]{adrien_phd}), but the adaptation of which to the heterogeneous case is far from immediate. Our analysis culminates in Theorem~\ref{thconvforte3_a}, our main result in the bending case.

\subsubsection{Technical results} \label{sec:tech_res_bending}

To state our main result for the bending case, we first need to state some intermediate results, which are collected in Lemmas~\ref{lem:lemma1}, \ref{lem:Sigmaab}, \ref{lem:lem00} and~\ref{lem:lem2} below.

\medskip

Let us define
\begin{equation} \label{eq:def_sigma_eps}
\sigma^\eps := A^\eps e^\eps(u^\eps + g).
\end{equation}
Using~\eqref{eq:inde_eps_vec} and Assumption~\eqref{eq:assump_g} (which implies that $e^\eps(g) = e(g)$), we infer from~\eqref{est:sig} that there exists $C>0$ independent of $\eps$ such that 
\begin{equation} \label{eq:bureau}
\|\sigma^\eps\|_{L^2(\Omega)} \leq C \left( \|f\|_{L^2(\Omega)} + \|e(g)\|_{L^2(\Omega)} + \|h_\pm\|_{L^2(\omega)} \right). 
\end{equation}
Theorem~\ref{limitel} provides information (in terms of $u^\star$) about $\dps \int_\Omega \sigma^\eps : e^\eps(v)$ when $v\in \VKL$. We indeed have
\begin{align*}
  \int_\Omega \sigma^\eps : e^\eps(v)
  &=
  \int_\Omega A^\eps e^\eps(u^\eps + g) : e^\eps(v)
  \\
  &=
  \int_\Omega f \cdot v + \int_{\Gamma_\pm} h_\pm \cdot v
  \\
  &=
  \int_\Omega f \cdot \widehat{v} + \int_{\Gamma_\pm} h_\pm \cdot \widehat{v} - \int_\Omega x_d \, f' \cdot \nabla' \widehat{v}_d - \int_{\Gamma_\pm} x_d \, h'_\pm \cdot \nabla' \widehat{v}_d
  \\
  &=
  \int_\omega \m(f) \cdot \widehat{v} + \int_\omega h_\pm \cdot \widehat{v} - \int_\omega \m(x_d \, f') \cdot \nabla' \widehat{v}_d - \demi \int_\omega (h'_+ - h'_-) \cdot \nabla' \widehat{v}_d
  \\
  &=
  \int_\omega K^\star \, \mathcal{P} (u^\star+g) : \mathcal{P} v,
\end{align*}
where we have used at the second line the variational formulation~\eqref{formvarelast} of the oscillatory problem and at the last line the variational formulation~\eqref{formvarelasthomog} of the homogenized problem, which is only valid for $v\in \VKL$.

To state a strong convergence theorem, we need to identify the limit of $\dps \int_\Omega \sigma^\eps : e^\eps(v)$ for any $v\in V$, and not only for $v\in \VKL$. To this aim, we introduce the following quantities, for any $1 \leq \alpha, \beta \leq d-1$:
\begin{equation} \label{eq:def_grand_Sigma}
  \Sigma^\eps_{\alpha\beta} := \sigma_{\alpha\beta}^\eps,
  \qquad
  \Sigma^\eps_{\alpha d} := \frac{1}{\eps} \, \sigma_{\alpha d}^\eps
  \qquad \text{and} \qquad
  \Sigma^\eps_{dd} := \frac{1}{\eps^2} \, \sigma_{dd}^\eps.
\end{equation}
By definition, we have $\dps \int_\Omega \sigma^\eps : e^\eps(v) = \int_\Omega \Sigma^\eps : e(v)$ for any $v \in V$. The main idea of the following lemmas is to show that there exists some $\Sigma^\star$ regular enough such that
$$
\forall v \in V, \qquad \int_\Omega \Sigma^\eps : e(v) = \int_\Omega \Sigma^\star : e(v),
$$
and to relate some components of $\Sigma^\star$ with $u^\star$.

\begin{lemma}\label{lem:lemma1}
There exists a symmetric matrix-valued field $\Sigma^\star := \left( \Sigma^\star_{ij} \right)_{1\leq i,j\leq d}$ such that, for any $1\leq \alpha, \beta \leq d-1$, we have
\begin{gather*}
  \Sigma^\star_{\alpha\beta} \in L^2(\Omega), \qquad \Sigma^\star_{\alpha d} \in L^2\left(\left(-\frac{1}{2}, \frac{1}{2}\right), H^{-1}(\omega)\right), \qquad \Sigma^\star_{dd} \in L^2\left(\left(-\frac{1}{2}, \frac{1}{2}\right), H^{-2}(\omega)\right),
  \\
  \partial_d \Sigma^\star_{\alpha d} \in L^2\left(\left(-\frac{1}{2}, \frac{1}{2}\right), H^{-1}(\omega)\right) \qquad \mbox{ and } \qquad \partial_d \Sigma^\star_{dd} \in L^2\left(\left(-\frac{1}{2}, \frac{1}{2}\right), H^{-2}(\omega)\right),
\end{gather*}
and such that, up to the extraction of a subsequence, 
\begin{gather}
  \Sigma_{\alpha \beta}^\eps \mathop{\rightharpoonup}_{\eps \to 0} \Sigma_{\alpha\beta}^\star \quad \mbox{ weakly in $L^2(\Omega)$},
  \label{eq:sigma_eps_alpha_beta}
  \\
  \Sigma^\eps_{\alpha d} \mathop{\rightharpoonup}_{\eps \to 0} \Sigma^\star_{\alpha d} \quad \mbox{ weakly in $\dps L^2\left( \left(-\frac{1}{2}, \frac{1}{2}\right), H^{-1}(\omega)\right)$},
  \label{eq:sigma_eps_alpha_d}
  \\
  \Sigma_{dd}^\eps \mathop{\rightharpoonup}_{\eps \to 0} \Sigma_{dd}^\star \quad \mbox{ weakly in $\dps L^2\left(\left(-\frac{1}{2}, \frac{1}{2}\right), H^{-2}(\omega)\right)$},
  \label{eq:sigma_eps_dd}
  \\
  \partial_d \Sigma^\eps_{\alpha d} \mathop{\rightharpoonup}_{\eps \to 0} \partial_d \Sigma^\star_{\alpha d} \quad \mbox{ weakly in $\dps L^2\left( \left(-\frac{1}{2}, \frac{1}{2}\right), H^{-1}(\omega)\right)$},
  \label{eq:d_sigma_eps_alpha_d}
  \\
  \partial_d \Sigma_{dd}^\eps \mathop{\rightharpoonup}_{\eps \to 0} \partial_d \Sigma_{dd}^\star \quad \mbox{ weakly in $\dps L^2\left(\left(-\frac{1}{2}, \frac{1}{2}\right), H^{-2}(\omega)\right)$}.
  \label{eq:d_sigma_eps_dd}
\end{gather} 
Furthermore, for any $\dps v \in \left( \mathcal{C}^\infty\left( \left[-\frac{1}{2}, \frac{1}{2}\right], \mathcal{D}(\omega)\right) \right)^d$ and for any $\eps$, we have
\begin{align}
  \int_\Omega \Sigma^\eps: e(v) & = \sum_{1\leq \alpha, \beta \leq d-1} \langle \Sigma_{\alpha\beta}^\star, e_{\alpha\beta}(v) \rangle_{L^2(\Omega)}
  \nonumber
  \\
  & + 2 \sum_{1\leq \alpha\leq d-1} \langle \Sigma_{\alpha d}^\star, e_{\alpha d}(v) \rangle_{L^2\left(\left(-\frac{1}{2}, \frac{1}{2}\right), H^{-1}(\omega)\right), L^2\left(\left(-\frac{1}{2}, \frac{1}{2}\right), H^1_0(\omega)\right)}
  \nonumber
  \\
  & + \langle \Sigma_{dd}^\star, e_{dd}(v) \rangle_{L^2\left(\left(-\frac{1}{2}, \frac{1}{2}\right), H^{-2}(\omega)\right), L^2\left(\left(-\frac{1}{2}, \frac{1}{2}\right), H^2_0(\omega)\right)}.
  \label{eq:bureau4}
\end{align}
In addition, there exists some $C>0$ independent of $\eps$ and $\omega$ such that
\begin{multline} \label{eq:stop6}
  \sum_{1\leq \alpha, \beta \leq d-1} \left\| \Sigma_{\alpha\beta}^\star \right\|_{L^2(\Omega)} + \sum_{1\leq \alpha \leq d-1} \left( \left\| \Sigma_{\alpha d}^\star \right\|_{L^2\left( \left( -\frac{1}{2}, \frac{1}{2}\right), H^{-1}(\omega)\right)} + \left\| \partial_d \Sigma_{\alpha d}^\star \right\|_{L^2\left( \left( -\frac{1}{2}, \frac{1}{2}\right), H^{-1}(\omega)\right)} \right) \\ + \left\| \Sigma^\star_{dd} \right\|_{L^2\left( \left( -\frac{1}{2}, \frac{1}{2}\right), H^{-2}(\omega)\right)} + \left\| \partial_d \Sigma^\star_{dd} \right\|_{L^2\left( \left( -\frac{1}{2}, \frac{1}{2}\right), H^{-2}(\omega)\right)} \leq C \left( \|f\|_{L^2(\Omega)} + \|e(g)\|_{L^2(\Omega)} + \|h_\pm\|_{L^2(\omega)} \right).
\end{multline}
\end{lemma}

\begin{proof}
We know from~\eqref{eq:bureau} that the sequence $\left( \sigma^\eps\right)_{\eps>0}$ is bounded in $(L^2(\Omega))^{d \times d}$. Thus, for any $1\leq \alpha, \beta \leq d-1$, there exists $\Sigma^\star_{\alpha\beta} \in L^2(\Omega)$ such that, up to the extraction of a subsequence, 
$$
\Sigma_{\alpha \beta}^\eps \mathop{\rightharpoonup}_{\eps \to 0} \Sigma_{\alpha\beta}^\star \quad \mbox{ weakly in $L^2(\Omega)$},
$$
and
\begin{equation} \label{eq:stop3}
  \sum_{1\leq \alpha, \beta \leq d-1} \left\| \Sigma_{\alpha\beta}^\star \right\|_{L^2(\Omega)} \leq C \left( \|f\|_{L^2(\Omega)} + \|e(g)\|_{L^2(\Omega)} + \|h_\pm\|_{L^2(\omega)} \right).
\end{equation}
The remainder of the proof falls in five steps.

\medskip

\noindent
{\bf Step~1.} We show that, for any $1\leq \alpha \leq d-1$, the sequence $\left( \Sigma_{\alpha d}^\eps \right)_{\eps>0}$ is bounded in $\dps L^2\left(\left(-\frac{1}{2}, \frac{1}{2}\right), H^{-1}(\omega)\right)$. Using the variational formulation~\eqref{formvarelast} of the oscillatory problem, we have that, for any $v = (v_i)_{1\leq i \leq d} \in V$ such that $v_d = 0$,
\begin{equation}\label{eq:eq1}
  \int_\Omega \Sigma_{\alpha d}^\eps \, \partial_d v_\alpha = \int_\Omega f_\alpha \, v_\alpha + \int_{\Gamma_\pm} (h_\pm)_\alpha \, v_\alpha - \int_\Omega \sigma_{\alpha \beta}^\eps \, e_{\alpha \beta}(v).
\end{equation}
For any $\dps w = (w_\alpha)_{1\leq \alpha \leq d-1} \in \left( L^2\left(\left(-\frac{1}{2}, \frac{1}{2}\right), H^1_0(\omega)\right) \right)^{d-1}$, we introduce, for any $1\leq \alpha \leq d-1$, the function $v^w_\alpha$ defined on $\Omega$ by
\begin{equation} \label{eq:def_vw_1}
  \forall (x',z) \in \omega \times \left(-\frac{1}{2}, \frac{1}{2}\right), \qquad v^w_\alpha(x',z) := \int_{-1/2}^z w_\alpha(x',t) \, dt.
\end{equation}
We also set $v^w_d = 0$. It then holds that $v^w = (v^w_i)_{1\leq i \leq d}$ belongs to $V$, and that it can be used as a test function in~\eqref{eq:eq1}. We then obtain that
\begin{equation}\label{eq:neweq1}
  \int_\Omega \Sigma_{\alpha d}^\eps \, w_\alpha = \int_\Omega f_\alpha \, v^w_\alpha + \int_{\Gamma_\pm} (h_\pm)_\alpha \, v^w_\alpha - \int_\Omega \sigma_{\alpha \beta}^\eps \, e_{\alpha \beta}(v^w).
\end{equation}
Using the fact that $\| v^w_\alpha \|_{L^2(\Omega)} \leq \|w_\alpha\|_{L^2(\Omega)}$ and that $\| \partial_\beta v^w_\alpha \|_{L^2(\Omega)} \leq \|\partial_\beta w_\alpha \|_{L^2(\Omega)}$ together with Lemma~\ref{lem:trace}, we obtain that there exists $C>0$ independent of $\eps$ and $\omega$ such that 
\begin{align*}
  \left| \int_\Omega \Sigma_{\alpha d}^\eps \, w_\alpha\right| & \leq \|f_\alpha\|_{L^2(\Omega)} \, \|v_\alpha^w\|_{L^2(\Omega)} + \| (h_+)_\alpha \|_{L^2(\omega)} \, \|v_\alpha^w\|_{L^2(\Gamma_+)} + \| (h_-)_\alpha \|_{L^2(\omega)} \, \|v_\alpha^w\|_{L^2(\Gamma_-)}
  \\
  & \quad + \frac{1}{2} \sum_{1\leq \alpha, \beta \leq d-1} \|\sigma_{\alpha \beta}^\eps\|_{L^2(\Omega)} \left( \| \partial_\alpha v^w_\beta \|_{L^2(\Omega)} + \|\partial_\beta v^w_\alpha\|_{L^2(\Omega)} \right)
  \\
  & \leq \|f_\alpha\|_{L^2(\Omega)} \, \|v_\alpha^w\|_{L^2(\Omega)} + \sqrt{2} \left( \| (h_+)_\alpha \|_{L^2(\omega)} + \| (h_-)_\alpha \|_{L^2(\omega)} \right) \left(\|v_\alpha^w\|_{L^2(\Omega)} + \|\partial_d v_\alpha^w\|_{L^2(\Omega)}\right)
  \\
  & \quad + \frac{1}{2} \sum_{1\leq \alpha, \beta \leq d-1} \|\sigma_{\alpha \beta}^\eps\|_{L^2(\Omega)} \left( \| \partial_\alpha v^w_\beta \|_{L^2(\Omega)} + \|\partial_\beta v^w_\alpha\|_{L^2(\Omega)} \right)
  \\
  & \leq \|f_\alpha\|_{L^2(\Omega)} \, \|w_\alpha\|_{L^2(\Omega)} + 2\sqrt{2} \left( \| (h_+)_\alpha \|_{L^2(\omega)} + \| (h_-)_\alpha \|_{L^2(\omega)} \right) \|w_\alpha\|_{L^2(\Omega)}
  \\
  & \quad + \frac{1}{2} \sum_{1\leq \alpha, \beta \leq d-1} \|\sigma_{\alpha \beta}^\eps\|_{L^2(\Omega)} \left( \| \partial_\alpha w_\beta \|_{L^2(\Omega)} + \| \partial_\beta w_\alpha \|_{L^2(\Omega)} \right).
\end{align*}
Using~\eqref{eq:bureau}, we deduce that
$$
\left| \int_\Omega \Sigma_{\alpha d}^\eps \, w_\alpha\right| \leq C \left( \|f\|_{L^2(\Omega)} + \|e(g)\|_{L^2(\Omega)} + \|h_\pm\|_{L^2(\omega)} \right) \|w\|_{L^2\left(\left( -\frac{1}{2}, \frac{1}{2}\right), H^1_0(\omega)\right)}.
$$
Since this bound holds for any $\dps w \in \left( L^2\left(\left(-\frac{1}{2}, \frac{1}{2}\right), H^1_0(\omega)\right) \right)^{d-1}$, we infer that there exists some $C>0$ independent of $\eps$ and $\omega$ such that
$$
\sum_{1\leq \alpha \leq d-1} \left\| \Sigma^\eps_{\alpha d} \right\|_{L^2\left( \left(-\frac{1}{2}, \frac{1}{2}\right), H^{-1}(\omega)\right)} \leq C \left( \|f\|_{L^2(\Omega)} + \|e(g)\|_{L^2(\Omega)} + \|h_\pm\|_{L^2(\omega)} \right).
$$
Hence, for any $1\leq \alpha \leq d-1$, there exists some $\dps \Sigma_{\alpha d}^\star \in L^2\left( \left(-\frac{1}{2}, \frac{1}{2}\right), H^{-1}(\omega)\right)$ such that, up to the extraction of a subsequence, 
$$
\Sigma^\eps_{\alpha d} \mathop{\rightharpoonup}_{\eps \to 0} \Sigma^\star_{\alpha d} \quad \mbox{ weakly in $\dps L^2\left( \left(-\frac{1}{2}, \frac{1}{2}\right), H^{-1}(\omega)\right)$}
$$
and
\begin{equation} \label{eq:stop4}
\sum_{1\leq \alpha \leq d-1} \left\| \Sigma^\star_{\alpha d} \right\|_{L^2\left( \left(-\frac{1}{2}, \frac{1}{2}\right), H^{-1}(\omega)\right)} \leq C \left( \|f\|_{L^2(\Omega)} + \|e(g)\|_{L^2(\Omega)} + \|h_\pm\|_{L^2(\omega)} \right).
\end{equation}

\medskip

\noindent
{\bf Step~2.} The convergence of $\left(\partial_d \Sigma_{\alpha d}^\eps\right)_{\eps >0}$ comes from the fact that, by definition of $\Sigma^\eps$, we have
\begin{equation} \label{eq:stop}
  - \div \Sigma^\eps = - \div^\eps \sigma^\eps = f \quad \text{in $\Omega$}.
\end{equation}
The $\alpha$-th component of this relation yields that
$$
- \partial_d \Sigma_{\alpha d}^\eps = f_\alpha + \partial_\beta \Sigma_{\alpha\beta}^\eps \quad \text{in $\Omega$},
$$
and thus, passing to the limit $\eps \to 0$ (along the adequate subsequence) in the sense of distributions, we obtain
$$
- \partial_d \Sigma_{\alpha d}^\star = f_\alpha + \partial_\beta \Sigma_{\alpha\beta}^\star \quad \text{in $\Omega$}.
$$
In addition, for any $1 \leq \alpha \leq d-1$, $\Sigma_{\alpha\beta}^\eps$ belongs to $L^2(\Omega)$ for any $1 \leq \beta \leq d-1$, thus $\partial_\beta \Sigma_{\alpha\beta}^\eps$ belongs to $\dps L^2\left( \left(-\frac{1}{2}, \frac{1}{2}\right), H^{-1}(\omega)\right)$, and therefore similarly for $\partial_d \Sigma_{\alpha d}^\eps$. For any $\dps \varphi \in L^2\left( \left(-\frac{1}{2}, \frac{1}{2}\right), H^1_0(\omega)\right)$, we next write that
\begin{equation} \label{eq:stop2}
\langle \partial_d \Sigma_{\alpha d}^\eps,\varphi \rangle_{L^2\left(\left(-\frac{1}{2}, \frac{1}{2}\right), H^{-1}(\omega)\right), L^2\left(\left(-\frac{1}{2}, \frac{1}{2}\right), H^1_0(\omega)\right)} = - \langle f_\alpha, \varphi \rangle_{L^2(\Omega)} + \langle \Sigma_{\alpha\beta}^\eps,\partial_\beta \varphi \rangle_{L^2(\Omega)},
\end{equation}
and thus
\begin{align*}
  \lim_{\eps \to 0} \ \langle \partial_d \Sigma_{\alpha d}^\eps,\varphi \rangle_{L^2\left(\left(-\frac{1}{2}, \frac{1}{2}\right), H^{-1}(\omega)\right), L^2\left(\left(-\frac{1}{2}, \frac{1}{2}\right), H^1_0(\omega)\right)}
  &=
  - \langle f_\alpha, \varphi \rangle_{L^2(\Omega)} + \langle \Sigma_{\alpha\beta}^\star,\partial_\beta \varphi \rangle_{L^2(\Omega)}
  \\
  &=
  \langle \partial_d \Sigma_{\alpha d}^\star,\varphi \rangle_{L^2\left(\left(-\frac{1}{2}, \frac{1}{2}\right), H^{-1}(\omega)\right), L^2\left(\left(-\frac{1}{2}, \frac{1}{2}\right), H^1_0(\omega)\right)}.
\end{align*}
We thus have shown that $\left(\partial_d \Sigma_{\alpha d}^\eps\right)_{\eps >0}$ converges weakly in $\dps L^2\left( \left(-\frac{1}{2}, \frac{1}{2}\right), H^{-1}(\omega)\right)$ to $\partial_d \Sigma_{\alpha d}^\star$. We also infer from~\eqref{eq:stop2} that
$$
\| \partial_d \Sigma_{\alpha d}^\eps \|_{L^2\left(\left(-\frac{1}{2}, \frac{1}{2}\right), H^{-1}(\omega)\right)} \leq \| f_\alpha \|_{L^2(\Omega)} + \sum_{\beta=1}^{d-1} \| \Sigma_{\alpha\beta}^\eps \|_{L^2(\Omega)} = \| f_\alpha \|_{L^2(\Omega)} + \sum_{\beta=1}^{d-1} \| \sigma_{\alpha\beta}^\eps \|_{L^2(\Omega)},
$$
and thus, using again~\eqref{eq:bureau}, we obtain that
\begin{equation} \label{eq:stop5}
  \sum_{1\leq \alpha \leq d-1} \left\| \partial_d \Sigma^\star_{\alpha d} \right\|_{L^2\left( \left(-\frac{1}{2}, \frac{1}{2}\right), H^{-1}(\omega)\right)} \leq C \left( \|f\|_{L^2(\Omega)} + \|e(g)\|_{L^2(\Omega)} + \|h_\pm\|_{L^2(\omega)} \right).
\end{equation}

\medskip

\noindent
{\bf Step~3.} We next turn to $\Sigma_{dd}^\eps$. For any $v = (v_i)_{1\leq i \leq d} \in V$ such that $v_\alpha = 0$ for all $1\leq \alpha \leq d-1$, we have, using~\eqref{formvarelast}, 
\begin{equation}\label{eq:eq2}
  \int_\Omega \Sigma_{dd}^\eps \, \partial_d v_d = \int_\Omega f_d \, v_d + \int_{\Gamma_\pm} (h_\pm)_d \, v_d - \int_\Omega \Sigma_{\alpha d}^\eps \, \partial_\alpha v_d.
\end{equation}
For any $\dps w_d \in L^2\left(\left(-\frac{1}{2}, \frac{1}{2}\right), H^2_0(\omega)\right)$, we introduce the function $v^w_d$ defined on $\Omega$ by
\begin{equation}\label{eq:vd}
  \forall (x',z) \in \omega \times \left(-\frac{1}{2}, \frac{1}{2}\right), \qquad v^w_d(x',z) := \int_{-1/2}^z w_d(x',t) \, dt,
\end{equation}
and we check that $\dps \partial_\alpha v^w_d \in L^2\left( \left(-\frac{1}{2},\frac{1}{2}\right), H^1_0(\omega)\right)$ for any $1 \leq \alpha \leq d-1$. In addition, we introduce $v^w_\alpha = 0$ for all $1 \leq \alpha \leq d-1$. We then have that $v^w = (v^w_i)_{1\leq i \leq d}$ belongs to $V$ and can thus be used as a function test in~\eqref{eq:eq2}. Following similar arguments as above, we obtain
\begin{align*}
  \left| \int_\Omega \Sigma_{dd}^\eps \, w_d \right|
  & \leq
  \|f_d\|_{L^2(\Omega)} \, \|v_d^w\|_{L^2(\Omega)} + \| (h_\pm)_d \|_{L^2(\omega)} \, \| v_d^w \|_{L^2(\Gamma_\pm)}
  \\
  & \quad + \sum_{1\leq \alpha\leq d-1} \| \Sigma_{\alpha d}^\eps \|_{L^2\left(\left(-\frac{1}{2}, \frac{1}{2}\right), H^{-1}(\omega)\right)} \, \| \partial_\alpha v^w_d \|_{L^2\left(\left(-\frac{1}{2}, \frac{1}{2}\right), H^1_0(\omega)\right)}
  \\
  & \leq C \left( \|f\|_{L^2(\Omega)} + \|e(g)\|_{L^2(\Omega)} + \|h_\pm\|_{L^2(\omega)} \right) \|w_d\|_{L^2\left(\left( -\frac{1}{2}, \frac{1}{2}\right), H^2_0(\omega)\right)},
\end{align*}
for some $C>0$ independent of $\omega$ and $\eps$. There thus exists some $\dps \Sigma_{dd}^\star \in L^2\left(\left( -\frac{1}{2}, \frac{1}{2}\right), H^{-2}(\omega)\right)$ such that, up to the extraction of a subsequence, 
$$
\Sigma_{dd}^\eps \mathop{\rightharpoonup}_{\eps \to 0} \Sigma_{dd}^\star \quad \mbox{ weakly in $\dps L^2\left(\left( -\frac{1}{2}, \frac{1}{2}\right), H^{-2}(\omega)\right)$}
$$
and 
\begin{equation} \label{eq:stop7}
\left\| \Sigma^\star_{dd} \right\|_{L^2\left(\left( -\frac{1}{2}, \frac{1}{2}\right), H^{-2}(\omega)\right)} \leq C \left( \|f\|_{L^2(\Omega)} + \|e(g)\|_{L^2(\Omega)} + \|h_\pm\|_{L^2(\omega)} \right).
\end{equation}

\medskip

\noindent
{\bf Step~4.} To study the convergence of $\left(\partial_d \Sigma_{dd}^\eps\right)_{\eps >0}$, we write the $d$-th component of~\eqref{eq:stop}, which reads
$$
- \partial_d \Sigma_{dd}^\eps = f_d + \partial_\beta \Sigma_{d\beta}^\eps.
$$
Using the weak convergence in $\dps L^2\left(\left(-\frac{1}{2}, \frac{1}{2}\right), H^{-1}(\omega)\right)$ of $\left(\Sigma_{d\beta}^\eps\right)_{\eps>0}$ to $\Sigma_{d\beta}^\star$ (established in Step~1) for any $1\leq \beta\leq d-1$, and using arguments similar as those of Step~2, we obtain that $\partial_d \Sigma_{dd}^\eps$ converges weakly in $\dps L^2\left( \left(-\frac{1}{2}, \frac{1}{2}\right), H^{-2}(\omega)\right)$ to $\partial_d \Sigma_{dd}^\star$. As in Step~2, we have
\begin{equation} \label{eq:stop8}
  \left\| \partial_d \Sigma^\star_{dd} \right\|_{L^2\left(\left( -\frac{1}{2}, \frac{1}{2}\right), H^{-2}(\omega)\right)} \leq C \left( \|f\|_{L^2(\Omega)} + \|e(g)\|_{L^2(\Omega)} + \|h_\pm\|_{L^2(\omega)} \right).
\end{equation}

\medskip

\noindent
{\bf Step~5.} Collecting~\eqref{eq:stop3}, \eqref{eq:stop4}, \eqref{eq:stop5}, \eqref{eq:stop7} and~\eqref{eq:stop8}, we obtain~\eqref{eq:stop6}. We are thus left with showing~\eqref{eq:bureau4}. Considering some smooth $\dps v \in \left( \mathcal{C}^\infty\left(\left[-\frac{1}{2}, \frac{1}{2}\right], \mathcal{D}(\omega)\right) \right)^d$ and using the weak convergences that we have just established, we have
\begin{align}
  \int_\Omega \Sigma^\eps: e(v) \mathop{\longrightarrow}_{\eps \to 0}
  & \sum_{1\leq \alpha, \beta \leq d-1} \langle \Sigma_{\alpha\beta}^\star, e_{\alpha\beta}(v) \rangle_{L^2(\Omega)}
  \nonumber
  \\
  & + 2 \sum_{1\leq \alpha\leq d-1} \langle \Sigma_{\alpha d}^\star, e_{\alpha d}(v) \rangle_{L^2\left(\left(-\frac{1}{2}, \frac{1}{2}\right), H^{-1}(\omega)\right), L^2\left(\left(-\frac{1}{2}, \frac{1}{2}\right), H^1_0(\omega)\right)}
  \nonumber
  \\
  & + \langle \Sigma_{dd}^\star, e_{dd}(v) \rangle_{L^2\left(\left(-\frac{1}{2}, \frac{1}{2}\right), H^{-2}(\omega)\right), L^2\left(\left(-\frac{1}{2}, \frac{1}{2}\right), H^2_0(\omega)\right)}.
  \label{eq:bureau2}
\end{align}
On the other hand, any such $v$ belongs to $V$ and we thus have, in view of~\eqref{formvarelast}, that
\begin{equation} \label{eq:bureau3}
  \int_\Omega \Sigma^\eps: e(v) = \int_\Omega \sigma^\eps: e^\eps(v) = \int_\Omega f \cdot v + \int_{\Gamma_+} h_+ \cdot v + \int_{\Gamma_-} h_- \cdot v. 
\end{equation}
In view of~\eqref{eq:bureau3}, we see that the left hand side of~\eqref{eq:bureau2} is actually independent of $\eps$. We thus deduce~\eqref{eq:bureau4}. This concludes the proof of Lemma~\ref{lem:lemma1}.
\end{proof}

\begin{lemma} \label{lem:lemma1_BC}
  The symmetric matrix-valued field $\Sigma^\star := \left( \Sigma^\star_{ij} \right)_{1\leq i,j\leq d}$ identified in Lemma~\ref{lem:lemma1} satisfies the following relations:
  \begin{equation} \label{eq:porto1}
    - \div \Sigma^\star = f \quad \text{in $\Omega$},
  \end{equation}
  and
  \begin{equation} \label{eq:porto2}
    \Sigma^\star \, n = h_+ \quad \text{on $\Gamma_+$}, \qquad \Sigma^\star \, n = h_- \quad \text{on $\Gamma_-$}.
  \end{equation}
\end{lemma}

\begin{proof}
We deduce from~\eqref{formvarelast} that, for any $v \in (\mathcal{D}(\Omega))^d$, we have
$$
\int_\Omega \Sigma^\eps: e(v) = \int_\Omega \sigma^\eps: e^\eps(v) = \int_\Omega f \cdot v,
$$
which implies that $- \div \Sigma^\eps = f$ in the sense of distributions. The convergences~\eqref{eq:sigma_eps_alpha_beta}, \eqref{eq:sigma_eps_alpha_d} and~\eqref{eq:sigma_eps_dd} ensure that $\Sigma^\eps$ converges to $\Sigma^\star$ in the sense of distributions. We hence get that~\eqref{eq:porto1} holds (in the sense of distributions).

Consider now some $\dps v \in \left( \mathcal{C}^\infty\left( \left[-\frac{1}{2}, \frac{1}{2}\right], \mathcal{D}(\omega)\right) \right)^d$. Observing that $\div \Sigma^\star \in (L^2(\Omega))^d$ (recall Assumption~\eqref{eq:hyp_regul_f}), we may define the trace of $\Sigma^\star \, n$ on $\Gamma_\pm$ by the duality relation
$$
\langle \Sigma^\star \, n, v \rangle_{\Gamma_\pm} = \langle \Sigma^\star, e(v) \rangle_\Omega + \int_\Omega v \cdot \div \Sigma^\star,
$$
where $\langle \Sigma^\star, e(v) \rangle_\Omega$ is defined as in the right-hand side of~\eqref{eq:bureau4}:
\begin{align*}
  \langle \Sigma^\star, e(v) \rangle_\Omega &= \sum_{1\leq \alpha, \beta \leq d-1} \langle \Sigma_{\alpha\beta}^\star, e_{\alpha\beta}(v) \rangle_{L^2(\Omega)}
  \\
  & + 2 \sum_{1\leq \alpha\leq d-1} \langle \Sigma_{\alpha d}^\star, e_{\alpha d}(v) \rangle_{L^2\left(\left(-\frac{1}{2}, \frac{1}{2}\right), H^{-1}(\omega)\right), L^2\left(\left(-\frac{1}{2}, \frac{1}{2}\right), H^1_0(\omega)\right)}
  \\
  & + \langle \Sigma_{dd}^\star, e_{dd}(v) \rangle_{L^2\left(\left(-\frac{1}{2}, \frac{1}{2}\right), H^{-2}(\omega)\right), L^2\left(\left(-\frac{1}{2}, \frac{1}{2}\right), H^2_0(\omega)\right)}.
\end{align*}
We then write
\begin{align*}
  \langle \Sigma^\star \, n,v \rangle_{\Gamma_\pm}
  &=
  \langle \Sigma^\star, e(v) \rangle_\Omega + \int_\Omega v \cdot \div \Sigma^\star
  \\
  &=
  \int_\Omega \Sigma^\eps : e(v) - \int_\Omega f \cdot v \qquad \text{[using~\eqref{eq:bureau4} and~\eqref{eq:porto1}]}
  \\
  &=
  \int_{\Gamma_+} h_+ \cdot v + \int_{\Gamma_-} h_- \cdot v,
\end{align*}
where the last line stems from the variational formulation of the original problem (see e.g.~\eqref{eq:bureau3}). We therefore deduce~\eqref{eq:porto2}.
\end{proof}
  

We have shown in Lemma~\ref{lem:lemma1} that, up to the extraction of a subsequence, the matrix $\left( \Sigma_{\alpha\beta}^\eps \right)_{1\leq \alpha,\beta\leq d-1}$ weakly converges to some $\left( \Sigma_{\alpha\beta}^\star \right)_{1\leq \alpha,\beta\leq d-1}$ in $(L^2(\Omega))^{(d-1)\times (d-1)}$. The aim of Lemma~\ref{lem:Sigmaab} is to characterize this weak limit, by actually providing an explicit expression for this matrix $\left( \Sigma_{\alpha\beta}^\star \right)_{1\leq \alpha,\beta\leq d-1}$. This will imply that the whole sequence $\left( \Sigma_{\alpha\beta}^\eps \right)_{1\leq \alpha,\beta\leq d-1}$ weakly converges, and not only a subsequence.

In Lemma~\ref{lem:Sigmaab} below, we are going to relate $\left( \Sigma_{\alpha\beta}^\star \right)_{1\leq \alpha,\beta\leq d-1}$, which is the limit of the stress tensor, with $u^\star$, which is the limit of the displacement (see~\eqref{eq:Sigmaformula}). For standard, scalar-valued, oscillatory diffusion problems of the form $- \div \left( \mathcal{A}_{\rm per}(\cdot/\eps) \nabla u^\eps \right) = f$ in some fixed domain $\Omega$, this amounts to writing that the limit of $\mathcal{A}_{\rm per}(\cdot/\eps) \nabla u^\eps$ is $\mathcal{A}^\star \nabla u^\star$, where $\mathcal{A}^\star$ is the homogenized matrix. Such a result can e.g. be obtained by studying the homogenized limit of the problem using the celebrated div-curl lemma (see~\cite{cherkaev}), and more precisely by studying the limit (in the sense of distributions) of $(p + \nabla w_p(\cdot/\eps))^T \mathcal{A}_{\rm per}(\cdot/\eps) \nabla u^\eps$ in two different manners ($w_p$ is of course the periodic corrector in the direction $p$). The idea of the approach is thus to pass to the limit (in two different manners) in the quantity $\dps \int_\Omega \phi \, (p + \nabla w_p(\cdot/\eps))^T \mathcal{A}_{\rm per}(\cdot/\eps) \nabla u^\eps$ for any $\phi \in \mathcal{D}(\Omega)$. In what follows, we are going to build upon this approach by considering the quantity $\dps \int_\Omega \phi \, e^\eps(v^\eps) : \sigma^\eps$ (see Steps~1 and~2 of the proof of Lemma~\ref{lem:Sigmaab}), where the stress tensor $\sigma^\eps$ of course plays the role of $\mathcal{A}_{\rm per}(\cdot/\eps) \nabla u^\eps$, where $v^\eps$ is an appropriately chosen function (essentially defined using the bending correctors $W^{\alpha\beta}$, see~\eqref{eq:def_v_eps} below) playing the role of $p \cdot x + \eps \, w_p(x/\eps)$, and where we have replaced the gradient operator by the scaled symmetric gradient operator (we have thus replaced $p + \nabla w_p(\cdot/\eps)$ by $e^\eps(v^\eps)$). For technical reasons, we are however unable to pass to the limit in $\dps \int_\Omega \phi \, e^\eps(v^\eps) : \sigma^\eps$ for any function $\phi \in \mathcal{D}(\Omega)$, but we have to restrict ourselves to functions $\phi$ independent of $x_d$. 
In order to obtain sufficient information to completely identify the limit $\left( \Sigma_{\alpha\beta}^\star \right)_{1\leq \alpha,\beta\leq d-1}$ of the stress tensor, we are going to consider an enriched family of bending correctors which are defined by~\eqref{eq:varWxi} below for a large class of functions $\xi$, and not only the bending corrector defined by~\eqref{el-prcor2}.

We now proceed in details and start by introducing the enriched bending correctors. For any $1\leq \alpha, \beta \leq d-1$ and any function $\xi \in L^2(-1/2,1/2)$, we denote by $W^{\alpha\beta,\xi} \in \mathcal{W}(\mathcal{Y})$ the unique solution to the problem
\begin{equation}\label{eq:varWxi}
  \forall v \in \mathcal{W}(\mathcal{Y}), \qquad \int_{\mathcal{Y}} A \Big( e\left(W^{\alpha\beta, \xi}\right) - \xi(x_d) \, e_\alpha \otimes e_\beta \Big) : e(v) = 0. 
\end{equation}
In contrast to the corrector problem~\eqref{el-prcor2}, we have replaced the factor $x_d$ there (which was reminiscent of the Kirchoff-Love space) by an arbitrary function $\xi$ only depending on $x_d$. For any $x_d \in (-1/2,1/2)$, we denote by $T^{\star,\xi}_{\alpha\beta}(x_d) \in \mathbb{R}_s^{d\times d}$ the symmetric matrix defined by
$$
T^{\star,\xi}_{\alpha\beta}(x_d) = \int_Y A(x',x_d) \Big( e\left( W^{\alpha\beta, \xi} \right) (x',x_d) - \xi(x_d) \, e_\alpha \otimes e_\beta \Big) \, dx'.
$$
Note that $T^{\star,\xi}_{\alpha\beta} = T^{\star,\xi}_{\beta\alpha}$, and that the mean in the above definition is only taken with respect to $x' \in Y$, and not with respect to $x=(x',x_d) \in \Y$.

In the specific case when $\xi = \Id$, we see that $W^{\alpha\beta,\Id} = W^{\alpha\beta}$ and $\dps \int_{-1/2}^{1/2} x_d \left( T^{\star,\Id}_{\alpha\beta}(x_d) \right)_{\gamma \delta} \, dx_d = (k^\star_{22})_{\alpha \beta \gamma \delta} = - (K^\star_{22})_{\alpha \beta \gamma \delta}$, where $K^\star_{22}$ is the homogenized matrix defined by~\eqref{eq:def_Kstar22} and where $k^\star_{22}$ is defined in~\eqref{eq:stop9} (the relation between $K^\star_{22}$ and $k^\star_{22}$ is shown in~\eqref{eq:stop9_bis} below). 

As in Lemma~\ref{lem:symcorr}, under the symmetry assumption~\eqref{hyp:symA} on $A$, we have that $W^{\alpha\beta,\xi} \in \O^{d-1}\times \E$ and $\dps \left( T^{\star,\xi}_{\alpha\beta} \right)_{\gamma \delta} \in \O$ for any odd function $\xi$ (and likewise $W^{\alpha\beta,\xi} \in \E^{d-1}\times \O$ and $\dps \left( T^{\star,\xi}_{\alpha\beta} \right)_{\gamma \delta} \in \E$ for any even function $\xi$).

\begin{lemma} \label{lem:Sigmaab}
Assume that $A$ satisfies the symmetries~\eqref{hyp:symA} and that we are in the bending case~\eqref{eq:ass_bending}. For any $1\leq \alpha, \beta, \gamma, \delta \leq d-1$, there exists a unique odd function $S^\star_{\alpha\beta\gamma\delta} \in L^2(-1/2,1/2)$ such that, for any odd function $\xi \in L^2(-1/2,1/2)$, 
\begin{equation}\label{eq:defSstar_imp}
\int_{-1/2}^{1/2} \xi(x_d) \, S^\star_{\alpha\beta \gamma \delta}(x_d) \, dx_d = \int_{-1/2}^{1/2} x_d \, \left( T^{\star,\xi}_{\alpha\beta}(x_d)\right)_{\gamma \delta} \, dx_d.
\end{equation}
In addition, $S^\star$ is given by
\begin{equation}\label{eq:defSstar_exp}
S^\star_{\alpha\beta\gamma\delta}(x_d) = (e_\alpha \otimes e_\beta) : \int_Y A(\cdot,x_d) \, \big( e(W^{\gamma \delta})(\cdot,x_d) - x_d \, e_\gamma \otimes e_\delta \big),
\end{equation}
it satisfies the bound
\begin{equation}\label{eq:lundi4}
\| S^\star_{\alpha\beta\gamma\delta} \|_{L^2(-1/2,1/2)} \leq \frac{c_+^2}{c_-} + c_+,
\end{equation}
for any $1\leq \alpha, \beta, \gamma, \delta \leq d-1$ (where we recall that $c_+$ and $c_-$ are the upper and lower bounds on the elasticity tensor $A$, respectively) and it also satisfies the following symmetries: $S^\star_{\alpha\beta \gamma\delta} = S^\star_{\beta\alpha \gamma\delta} = S^\star_{\alpha\beta \delta\gamma}$ (note however that $S^\star_{\alpha\beta \gamma\delta}$ and $S^\star_{\gamma\delta \alpha\beta}$ are in general not equal).

\medskip

Besides, in dimension $d=2$, the symmetric matrix $\left( \Sigma^\star_{\alpha\beta} \right)_{1\leq \alpha, \beta \leq d-1} \in (L^2(\Omega))^{(d-1)\times (d-1)}$ identified in Lemma~\ref{lem:lemma1} as the weak limit of $\left( \Sigma^\eps_{\alpha\beta} \right)_{1\leq \alpha, \beta \leq d-1}$ satisfies, for any $x'\in \omega$ and $x_d \in (-1/2,1/2)$,
\begin{equation}\label{eq:Sigmaformula}
\forall 1\leq \alpha, \beta \leq d-1, \qquad \Sigma^\star_{\alpha\beta}(x',x_d) = S^\star_{\alpha\beta\gamma\delta}(x_d) \, \partial_{\gamma\delta}(u^\star_d + g_d)(x'),
\end{equation}
where $u^\star_d$ is the solution to the homogenized problem~\eqref{formvar:pb_bending}. 
This weak limit is thus unique and independent of the subsequence considered in Lemma~\ref{lem:lemma1}.
\end{lemma}
The restriction to the two-dimensional case in the second part of the lemma only comes from the fact that its proof uses Lemma~\ref{lemma:minZ2}.

\begin{remark}
The only property of $u^\star$ that we use to prove~\eqref{eq:Sigmaformula} is actually the fact that, up to a subsequence extraction, $u^\eps_d$ weakly converges to $u^\star_d$. We do not use the fact that $u^\star_d$ is a solution to~\eqref{formvar:pb_bending}. For that reason, the relation~\eqref{eq:Sigmaformula} can be used to {\em identify} the homogenized equation, that is~\eqref{formvar:pb_bending}, as discussed below and in Appendix~\ref{app:diff_preuve2_bend}.
\end{remark}

The interest of~\eqref{eq:defSstar_imp} is of course that $S^\star$ is independent of $\xi$. Note that~\eqref{eq:defSstar_imp} trivially holds for any even function $\xi$, because both integrands are odd functions and thus both integrals vanish. The relation~\eqref{eq:Sigmaformula} is useful for a twofold reason. First, it allows to obtain the homogenized equation, as shown in Appendix~\ref{app:diff_preuve2_bend}. We have outlined above the fact that Lemma~\ref{lem:Sigmaab} may be understood as a div-curl lemma type result, and it is well-known, for standard, scalar-valued, oscillatory diffusion problems, that such a result directly yields the homogenized equation. Second, the relation~\eqref{eq:Sigmaformula} is going to be used in Lemma~\ref{lem:lem00} below to show additional regularity on $\Sigma^\star$, a pivotal ingredient to eventually obtain a strong convergence homogenization result.

\begin{proof}
For any $1\leq \alpha, \beta, \gamma, \delta \leq d-1$, we observe that $S^\star_{\alpha\beta \gamma \delta}$ defined by~\eqref{eq:defSstar_exp} indeed belongs to $L^2(-1/2,1/2)$. In addition, it is an odd function in view of the symmetries~\eqref{hyp:symA} of $A$ (which in particular imply that the bending correctors $W^{\gamma \delta}$ belong to $\O^{d-1}\times \E$, see Lemma~\ref{lem:symcorr}). 

We next compute that, for any $\xi \in L^2(-1/2,1/2)$,
\begin{align*}
  & \int_{-1/2}^{1/2} x_d \, \left( T^{\star,\xi}_{\alpha\beta}(x_d)\right)_{\gamma \delta} \, dx_d
  \\
  &=
  \int_{\Y} x_d \, e_\gamma \otimes e_\delta : A(x',x_d) \Big( e\left( W^{\alpha\beta, \xi} \right) (x',x_d) - \xi(x_d) \, e_\alpha \otimes e_\beta \Big) \, dx' \, dx_d
  \\
  &=
  \int_{\Y} \Big( x_d \, e_\gamma \otimes e_\delta - e\left( W^{\gamma\delta} \right) (x',x_d) \Big) : A(x',x_d) \Big( e\left( W^{\alpha\beta, \xi} \right) (x',x_d) - \xi(x_d) \, e_\alpha \otimes e_\beta \Big) \, dx' \, dx_d
  \\
  &=
  -\int_{\Y} \Big( e\left( W^{\alpha\beta, \xi} \right) (x',x_d) - \xi(x_d) \, e_\alpha \otimes e_\beta \Big) : A(x',x_d) \Big( e\left( W^{\gamma\delta} \right) (x',x_d) - x_d \, e_\gamma \otimes e_\delta \Big) \, dx' \, dx_d
  \\
  &=
  \int_{\Y} \xi(x_d) \, e_\alpha \otimes e_\beta : A(x',x_d) \Big( e\left( W^{\gamma\delta} \right) (x',x_d) - x_d \, e_\gamma \otimes e_\delta \Big) \, dx' \, dx_d
  \\
  &=
  \int_{-1/2}^{1/2} \xi(x_d) \, S^\star_{\alpha\beta \gamma \delta}(x_d) \, dx_d,
\end{align*}
where we have successively used the variational formulation~\eqref{eq:varWxi} with the test function $v = W^{\gamma\delta}$, the symmetry of $A$, the variational formulation~\eqref{el-prcor2} with the test function $v = W^{\alpha\beta, \xi}$, and the definition~\eqref{eq:defSstar_exp} of $S^\star$. We have thus obtained~\eqref{eq:defSstar_imp} for any $\xi \in L^2(-1/2,1/2)$. The symmetries of $S^\star$ stem from the fact that, for any $1\leq \alpha,\beta \leq d-1$, $T^{\star,\xi}_{\alpha\beta}$ is a symmetric matrix, and from the fact that $T^{\star,\xi}_{\alpha\beta} = T^{\star,\xi}_{\beta\alpha}$. 

The uniqueness of $S^\star_{\alpha\beta \gamma \delta}$ such that~\eqref{eq:defSstar_imp} holds comes from the following argument. Denoting by $L^2_{\rm odd}(-1/2,1/2)$ the set of odd functions of $L^2(-1/2,1/2)$, it can be easily checked that the application
$$
F_{\alpha\beta\gamma\delta}: \left\{ 
\begin{array}{ccc}
 L^2_{\rm odd}(-1/2,1/2) & \to & \mathbb{R} \\
 \xi & \mapsto & \dps \int_{-1/2}^{1/2} x_d \, \left( T^{\star,\xi}_{\alpha\beta}(x_d)\right)_{\gamma \delta} \, dx_d
\end{array} \right.
$$
is a continuous linear form. Hence, from the Riesz representation theorem in $L^2_{\rm odd}(-1/2,1/2)$, we obtain the existence of a unique function $S_{\alpha\beta\gamma\delta}^\star \in L^2_{\rm odd}(-1/2,1/2)$ satisfying~\eqref{eq:defSstar_imp}.

The bound~\eqref{eq:lundi4} is obtained as follows. Taking $v = W^{\gamma\delta}$ in the variational formulation~\eqref{el-prcor2} defining the corrector $W^{\gamma\delta}$, we obtain
\begin{equation} \label{eq:dimanche}
c_- \, \| e(W^{\gamma\delta}) \|^2_{L^2(\Y)} \leq c_+ \, \| e(W^{\gamma\delta}) \|_{L^2(\Y)} \, \| x_d \|_{L^2(\Y)},
\end{equation}
and thus $\| e(W^{\gamma\delta}) \|_{L^2(\Y)} \leq c_+/c_-$. We then infer~\eqref{eq:lundi4} from~\eqref{eq:defSstar_exp}, the Cauchy-Schwarz inequality and the above bound.
\medskip

We now turn to the second part of the lemma statement. Let us consider a subsequence of $\left(\Sigma^\eps\right)_{\eps >0}$ which weakly converges to some limit $\Sigma^\star$ in the sense of Lemma~\ref{lem:lemma1}. We are going to prove that, for any $1\leq \alpha, \beta \leq d-1$, $\Sigma_{\alpha\beta}^\star$ is given by~\eqref{eq:Sigmaformula}. Note that, although we only identify the $(d-1) \times (d-1)$ submatrix $\left( \Sigma_{\alpha\beta}^\star \right)_{1\leq \alpha,\beta\leq d-1}$, we use in the proof below the convergence of the whole $d \times d$ matrix $\Sigma^\eps$.

To obtain the desired result, we need to prove that, for any $1\leq \alpha, \beta \leq d-1$ and for any odd function $\xi \in L^2(-1/2,1/2)$, $\Sigma_{\alpha\beta}^\star$ satisfies
\begin{equation}\label{eq:fref}
\text{for a.a. $x' \in \omega$}, \quad \int_{-1/2}^{1/2} \xi(x_d) \, \Sigma_{\alpha\beta}^\star(x',x_d) \, dx_d = \left( \int_{-1/2}^{1/2} x_d \left( T^{\star,\xi}_{\alpha\beta}(x_d) \right)_{\gamma \delta} \, dx_d \right) \partial_{\gamma \delta}(u_d^\star + g_d)(x').
\end{equation}
This is indeed sufficient since both $S_{\alpha\beta\gamma\delta}^\star$ (by construction) and $\Sigma_{\alpha\beta}^\star(x',\cdot)$ are odd functions of $x_d$. The fact that $\Sigma_{\alpha\beta}^\star(x',\cdot)$ is odd stems from the fact that $u^\eps \in \O^{d-1}\times \E$ (see Lemma~\ref{lem:membrane}), that $g \in \O^{d-1}\times \E$ since it belongs to $\GKL^{\mathcal{B}}$ and thus $e^\eps_{\gamma \delta}(u^\eps+g) \in \O$, $e^\eps_{\gamma d}(u^\eps+g) \in \E$ and $e^\eps_{dd}(u^\eps+g) \in \O$. Using the symmetries~\eqref{hyp:symA} of $A$, we therefore obtain that $\Sigma_{\alpha\beta}^\eps(x',\cdot)$ is odd, and thus similarly for its limit $\Sigma_{\alpha\beta}^\star(x',\cdot)$. Similarly to~\eqref{eq:defSstar_imp}, the relation~\eqref{eq:fref} trivially holds for any even function $\xi \in L^2(-1/2,1/2)$, since both integrands are then odd functions of $x_d$, yielding vanishing integrals.

Using the density of the polynomial functions in $L^2(-1/2,1/2)$, and the fact that both sides of~\eqref{eq:fref} are linear and continuous applications of $\xi \in L^2(-1/2,1/2)$, it is sufficient to prove~\eqref{eq:fref} for any odd monomial function of the form $\dps \xi_q(x_d) = \frac{x_d^{2q+1}}{2q+1}$ for any $q \in \mathbb{N}$. We are thus left with showing that, for any $1 \leq \alpha, \beta \leq d-1$ and any $q \in \mathbb{N}$,
\begin{equation}\label{eq:qref}
\forall x' \in \omega, \quad \int_{-1/2}^{1/2} \xi_q(x_d) \, \Sigma_{\alpha\beta}^\star(x',x_d) \, dx_d = \left( \int_{-1/2}^{1/2} x_d \, \left( T^{\star,\xi_q}_{\alpha\beta}(x_d) \right)_{\gamma \delta} \, dx_d \right) \partial_{\gamma \delta}(u_d^\star + g_d).
\end{equation}
Throughout the remainder of this proof, we fix some $1\leq \alpha, \beta \leq d-1$ and some $q \in \mathbb{N}$. Let us define $\nu(x') = x_\alpha \, x_\beta / 2$ for any $x' \in \omega$ and let us introduce the vector-valued function $v^\eps = (v^\eps_i)_{1\leq i \leq d}$ defined by
\begin{equation} \label{eq:def_v_eps}
\begin{aligned}
  \forall 1 \leq \gamma \leq d-1, \qquad v_\gamma^\eps(x) & := \eps \, W_\gamma^{\alpha\beta, \xi_q}\left( \frac{x'}{\eps}, x_d\right) - \frac{x_d^{2q+1}}{2q+1} \, \partial_\gamma \nu(x'),
  \\
  v_d^\eps(x) &:= \eps^2 \, W_d^{\alpha\beta, \xi_q}\left( \frac{x'}{\eps}, x_d\right) + x_d^{2q} \, \nu(x').
\end{aligned}
\end{equation}
By construction, it holds that
\begin{equation} \label{eq:theatre5}
e^\eps(v^\eps)(x',x_d) = e\left( W^{\alpha\beta, \xi_q}\right)\left( \frac{x'}{\eps}, x_d\right) - \frac{1}{2} \, \frac{x_d^{2q+1}}{2q+1} \, \left(e_\alpha \otimes e_\beta + e_\beta \otimes e_\alpha \right) + \frac{2q}{\eps^2} \, x_d^{2q-1} \, \nu(x') \, e_d\otimes e_d.
\end{equation}
The proof of~\eqref{eq:qref} falls in three steps.

\medskip

\noindent {\bf Step~1.} Let $\phi \in \mathcal{D}(\omega)$. It holds that 
\begin{align}
  \int_\Omega \sigma^\eps : e^\eps(v^\eps) \, \phi
  & = \int_\Omega \Sigma^\eps : e(v^\eps) \, \phi
  \nonumber
  \\
  & = -\int_\Omega (\Sigma^\eps \nabla \phi) \cdot v^\eps - \int_\Omega \div \Sigma^\eps \cdot v^\eps \, \phi + \int_{\partial \Omega} \phi \, v^\eps \cdot (\Sigma^\eps n)
  \nonumber
  \\
  & = - \int_\Omega \sigma_{\gamma \delta}^\eps \, \partial_\gamma \phi \, v_\delta^\eps - \frac{1}{\eps} \int_\Omega \sigma_{\gamma d}^\eps \, \partial_\gamma \phi \, v_d^\eps - \int_\Omega \div \Sigma^\eps \cdot v^\eps \, \phi + \int_{\partial \Omega} \phi \, v^\eps \cdot (\Sigma^\eps n).
  \label{eq:theatre3}
\end{align}
We are going to pass to the limit $\eps \to 0$ in~\eqref{eq:theatre3}. First, using the fact that $-\div \Sigma^\eps = f$ in $\Omega$ and $\Sigma^\eps n = h_\pm$ on $\Gamma_\pm$, we obtain
$$
-\int_\Omega \div \Sigma^\eps \cdot v^\eps \, \phi + \int_{\partial \Omega} \phi \, v^\eps \cdot (\Sigma^\eps n) = \int_\Omega f \cdot v^\eps \, \phi + \int_{\Gamma_\pm} \phi \, v^\eps \cdot h_\pm,
$$
and thus, using the specific expression of $v^\eps$, we get
\begin{align*}
  & \lim_{\eps \to 0} - \int_\Omega \div \Sigma^\eps \cdot v^\eps \, \phi + \int_{\partial \Omega} \phi \, v^\eps \cdot (\Sigma^\eps n)
  \\
  & = \int_\Omega f_d(x) \, x_d^{2q} \, \nu(x') \, \phi(x') \, dx - \int_\Omega f_\gamma(x) \, \frac{x_d^{2q+1}}{2q+1} \, \partial_\gamma\nu(x') \, \phi(x') \, dx
  \\
  & + (1/2)^{2q} \int_\omega \phi(x') \, \nu(x') \, (h_++h_-)_d(x') - \frac{(1/2)^{2q+1}}{2q+1} \int_\omega \phi(x') \, \partial_\gamma \nu(x') \, (h_+-h_-)_\gamma(x').
\end{align*}
Second, using~\eqref{eq:sigma_eps_alpha_beta}, we have
\begin{align*}
  & \lim_{\eps \to 0} \int_\Omega \sigma_{\gamma \delta}^\eps \, \partial_\gamma \phi \, v_\delta^\eps
  \\
  &= \lim_{\eps \to 0} \int_\Omega \Sigma_{\gamma \delta}^\eps \, \partial_\gamma \phi \, v_\delta^\eps
  \\
  &= - \int_\Omega \Sigma_{\gamma \delta}^\star(x) \, \partial_\gamma \phi(x') \, \partial_\delta \nu(x') \, \frac{x_d^{2q+1}}{2q+1} \, dx
  \\
  & = \left\langle \int_{-1/2}^{1/2} \partial_\gamma \Sigma_{\gamma\delta}^\star(\cdot, x_d) \, \frac{x_d^{2q+1}}{2q+1} \, dx_d, \phi \, \partial_\delta \nu \right\rangle_{\mathcal{D}'(\omega), \mathcal{D}(\omega)} + \int_\Omega \phi(x') \, \Sigma_{\gamma\delta}^\star(x) \, \partial_{\gamma \delta} \nu(x') \, \frac{x_d^{2q+1}}{2q+1} \, dx
  \\
  & = \left\langle \partial_\delta \nu \left( \int_{-1/2}^{1/2} \partial_\gamma \Sigma_{\gamma\delta}^\star(\cdot,x_d) \, \frac{x_d^{2q+1}}{2q+1} \, dx_d \right), \phi \right\rangle_{\mathcal{D}'(\omega), \mathcal{D}(\omega)} + \int_\omega \phi \left( \int_{-1/2}^{1/2} \frac{x_d^{2q+1}}{2q+1} \, \Sigma_{\alpha\beta}^\star(\cdot,x_d) \, dx_d \right).
\end{align*}
Third, we compute that
\begin{align*}
  \frac{1}{\eps} \int_\Omega \sigma_{\gamma d}^\eps \, \partial_\gamma \phi \, v_d^\eps
  &=
  \frac{1}{\eps} \int_\Omega \sigma_{\gamma d}^\eps(x) \, \partial_\gamma \phi(x') \, \nu(x') \, x_d^{2q} \, dx + \eps \int_\Omega \sigma_{\gamma d}^\eps(x',x_d) \, W_d^{\alpha\beta, \xi_q}\left( \frac{x'}{\eps}, x_d\right) \partial_\gamma \phi(x') \, dx
  \\
  & = \int_\Omega \Sigma_{\gamma d}^\eps(x) \, \partial_\gamma \phi(x') \, \nu(x') \, x_d^{2q} \, dx + \eps \int_\Omega \sigma_{\gamma d}^\eps(x) \, W_d^{\alpha\beta, \xi_q}\left( \frac{x'}{\eps}, x_d\right) \partial_\gamma \phi(x') \, dx.
\end{align*}
Using that $\sigma^\eps$ is bounded in $L^2(\Omega)$ (see~\eqref{eq:bureau}), that $\Omega \ni x=(x',x_d) \to \partial_\gamma \phi(x') \, \nu(x') \, x_d^{2q}$ belongs to $L^2\big( (-1/2,1/2), H^1_0(\omega) \big)$ and the weak convergence~\eqref{eq:sigma_eps_alpha_d}, we deduce that
\begin{align*}
  & \lim_{\eps \to 0} \frac{1}{\eps} \int_\Omega \sigma_{\gamma d}^\eps \, \partial_\gamma \phi \, v_d^\eps
  \\
  & = \left\langle \int_{-1/2}^{1/2} x_d^{2q} \, \Sigma_{\gamma d}^\star(\cdot, x_d) \, dx_d, \partial_\gamma \phi \, \nu \right\rangle_{\mathcal{D}'(\omega), \mathcal{D}(\omega)}
  \\
  & = - \left\langle \int_{-1/2}^{1/2} x_d^{2q} \, \partial_\gamma \Sigma_{\gamma d}^\star(\cdot, x_d) \, dx_d, \phi \, \nu \right\rangle_{\mathcal{D}'(\omega), \mathcal{D}(\omega)} - \left\langle \int_{-1/2}^{1/2} x_d^{2q} \, \Sigma_{\gamma d}^\star(\cdot, x_d) \, dx_d, \phi \, \partial_\gamma \nu \right\rangle_{\mathcal{D}'(\omega), \mathcal{D}(\omega)}
  \\
  & = - \left\langle \nu \int_{-1/2}^{1/2} x_d^{2q} \, \partial_\gamma \Sigma_{\gamma d}^\star(\cdot, x_d) \, dx_d, \phi \right\rangle_{\mathcal{D}'(\omega), \mathcal{D}(\omega)} - \left\langle \partial_\gamma \nu \int_{-1/2}^{1/2} x_d^{2q} \, \Sigma_{\gamma d}^\star(\cdot, x_d) \, dx_d, \phi \right\rangle_{\mathcal{D}'(\omega), \mathcal{D}(\omega)}.
\end{align*}
Collecting these results and using the symmetry of $\Sigma^\star$, we are in position to pass to the limit $\eps \to 0$ in~\eqref{eq:theatre3}. We thus obtain that, in the sense of distributions in $\mathcal{D}'(\omega)$,
\begin{align}
  & \lim_{\eps \to 0} \int_{-1/2}^{1/2} \sigma^\eps(\cdot,x_d) : e^\eps(v^\eps)(\cdot,x_d) \, dx_d
  \nonumber
  \\
  &= (1/2)^{2q} \, \nu \, (h_++h_-)_d + \nu \int_{-1/2}^{1/2} x_d^{2q} \, \big( \partial_\gamma \Sigma_{\gamma d}^\star(\cdot, x_d) + f_d(\cdot, x_d) \big) \, dx_d
  \nonumber
  \\
  & - \frac{(1/2)^{2q+1}}{2q+1} \, \partial_\gamma \nu \, (h_+-h_-)_\gamma + \partial_\gamma \nu \int_{-1/2}^{1/2} \left( x_d^{2q} \, \Sigma_{\gamma d}^\star(\cdot,x_d) - \frac{x_d^{2q+1}}{2q+1} \left(\partial_\delta \Sigma_{\gamma\delta}^\star(\cdot,x_d) + f_\gamma(\cdot,x_d) \right) \right) dx_d
  \nonumber
  \\
  & - \int_{-1/2}^{1/2} \frac{x_d^{2q+1}}{2q+1} \, \Sigma_{\alpha\beta}^\star(\cdot,x_d) \, dx_d.
  \label{eq:theatre4}
\end{align}
The above expression can be simplified. We first note, using that $\partial_\delta \Sigma_{\gamma \delta}^\star + \partial_d \Sigma_{\gamma d}^\star = -f_\gamma$ (see~\eqref{eq:porto1}), that the second term of the third line of~\eqref{eq:theatre4} reads
\begin{align*}
  & \partial_\gamma \nu \int_{-1/2}^{1/2} \left( x_d^{2q} \, \Sigma_{\gamma d}^\star(\cdot,x_d) - \frac{x_d^{2q+1}}{2q+1} \left(\partial_\delta \Sigma_{\gamma\delta}^\star(\cdot,x_d) + f_\gamma(\cdot,x_d) \right) \right) dx_d
  \\
  & = \partial_\gamma \nu \int_{-1/2}^{1/2} \left( x_d^{2q} \, \Sigma_{\gamma d}^\star(\cdot,x_d) + \frac{x_d^{2q+1}}{2q+1} \, \partial_d \Sigma_{\gamma d}^\star(\cdot,x_d) \right) dx_d
  \\
  & = \partial_\gamma \nu \int_{-1/2}^{1/2} \partial_d \left( \frac{x_d^{2q+1}}{2q+1} \, \Sigma_{\gamma d}^\star(\cdot,x_d) \right)
  \\
  & = \partial_\gamma \nu \, \frac{(1/2)^{2q+1}}{2q+1} \left( \Sigma_{\gamma d}^\star(\cdot,1/2) + \Sigma_{\gamma d}^\star(\cdot,-1/2) \right)
  \\
  & = \partial_\gamma \nu \, \frac{(1/2)^{2q+1}}{2q+1} \big( (h_+)_\gamma - (h_-)_\gamma \big),
\end{align*}
the last equality being obtained using the boundary conditions of $\Sigma^\star \, n$ on $\Gamma_\pm$ (see~\eqref{eq:porto2}). The third line of~\eqref{eq:theatre4} thus vanishes.
%
%
%

Turning to the second line of~\eqref{eq:theatre4}, using the symmetry of $\Sigma^\star$ and noting that $\partial_\gamma \Sigma_{\gamma d}^\star + f_d = - \partial_d \Sigma_{dd}^\star$, we obtain
\begin{align*}
  & \nu \int_{-1/2}^{1/2} x_d^{2q} \, \big( \partial_\gamma \Sigma_{\gamma d}^\star(\cdot,x_d) + f_d(\cdot,x_d) \big) \, dx_d
  \\
  &=
  - \nu \int_{-1/2}^{1/2} x_d^{2q} \, \partial_d \Sigma_{dd}^\star(\cdot,x_d) \, dx_d
  \\
  &=
  2q \, \nu \int_{-1/2}^{1/2} x_d^{2q-1} \, \Sigma_{dd}^\star(\cdot,x_d) \, dx_d - \nu \, (1/2)^{2q} \, \big( \Sigma_{dd}^\star(\cdot,1/2) - \Sigma_{dd}^\star(\cdot,-1/2) \big)
  \\
  &=
  2q \, \nu \int_{-1/2}^{1/2} x_d^{2q-1} \, \Sigma_{dd}^\star(\cdot,x_d) \, dx_d - \nu \, (1/2)^{2q} \, \big( (h_+)_d + (h_-)_d \big),
\end{align*}
the last equality being again obtained using the boundary conditions of $\Sigma^\star \, n$ on $\Gamma_\pm$ (see again~\eqref{eq:porto2}). 
We thus deduce from~\eqref{eq:theatre4} that, in $\mathcal{D}'(\omega)$, 
$$
\lim_{\eps \to 0} \int_{-1/2}^{1/2} \sigma^\eps(\cdot,x_d) : e^\eps(v^\eps)(\cdot,x_d) \, dx_d = 2 q \, \nu \int_{-1/2}^{1/2} x_d^{2q-1} \, \Sigma_{dd}^\star(\cdot,x_d) \, dx_d - \int_{-1/2}^{1/2} \frac{x_d^{2q+1}}{2q+1} \, \Sigma_{\alpha\beta}^\star(\cdot,x_d) \, dx_d.
$$

\medskip

\noindent
{\bf Step~2.} We are now going to compute again the limit of $\sigma^\eps : e^\eps(v^\eps)$, but by considering this product in a different manner. Using the symmetry of $A^\eps$, it holds that
$$
\sigma^\eps : e^\eps(v^\eps) = A^\eps e^\eps(u^\eps+g) : e^\eps(v^\eps) = A^\eps e^\eps(v^\eps) : e^\eps(u^\eps+g).
$$
Using~\eqref{eq:theatre5}, we have
\begin{align*}
  A^\eps(x) e^\eps(v^\eps)(x)
  &=
  A^\eps(x)\left[ e\left( W^{\alpha\beta, \xi_q}\right)\left( \frac{x'}{\eps},x_d\right) - \frac{x_d^{2q+1}}{2q+1} \, e_\alpha \otimes e_\beta \right] + \frac{2q}{\eps^2} \, x_d^{2q-1} \, \nu(x') \, A^\eps(x) \, e_d\otimes e_d
  \\
  &= T^\eps(x) + U^\eps(x)
\end{align*}
with 
$$
T^\eps(x):= A^\eps(x) \left[ e\left( W^{\alpha\beta, \xi_q}\right)\left( \frac{x'}{\eps},x_d\right) - \frac{x_d^{2q+1}}{2q+1} \, e_\alpha \otimes e_\beta \right]
$$
and 
$$
U^\eps(x) := \frac{2q}{\eps^2} \, x_d^{2q-1} \, \nu(x') \, A^\eps(x) \, e_d\otimes e_d.
$$ 
For any $\phi \in \mathcal{D}(\omega)$, we compute that
\begin{align}
  \int_\Omega \sigma^\eps : e^\eps(v^\eps) \, \phi
  & = \int_\Omega A^\eps e^\eps(v^\eps) : e^\eps(u^\eps +g) \, \phi
  \nonumber
  \\
  & = \int_\Omega T^\eps : e^\eps(u^\eps + g) \, \phi + \int_\Omega U^\eps : e^\eps(u^\eps + g) \, \phi
  \nonumber
  \\
  & = \int_\Omega \left(T^\eps - T^{\star,\xi_q}_{\alpha\beta}\right) : e^\eps(u^\eps) \, \phi + \int_\Omega \left(T^\eps - T^{\star,\xi_q}_{\alpha\beta}\right) : e^\eps(g) \, \phi
  \nonumber
  \\
  & \qquad \qquad \qquad + \int_\Omega T^{\star,\xi_q}_{\alpha\beta} : e^\eps(u^\eps + g) \, \phi + \int_\Omega U^\eps : e^\eps(u^\eps + g) \, \phi.
  \label{eq:theatre6}
\end{align}
We successively study the limit of the four terms of the right-hand side of~\eqref{eq:theatre6} when $\eps \to 0$.

We first consider the third term of~\eqref{eq:theatre6}, namely $\dps \int_\Omega T^{\star,\xi_q}_{\alpha\beta} : e^\eps(u^\eps + g) \, \phi$, and claim that
\begin{equation} \label{eq:theatre7}
  \forall 1 \leq i \leq d, \quad \forall x_d \in (-1/2,1/2), \qquad \left( T^{\star,\xi_q}_{\alpha\beta}(x_d) \right)_{id} = 0.
\end{equation}
Indeed, for any fixed $1 \leq i \leq d$ and for any $\lambda \in L^2(-1/2,1/2)$, introduce $\dps \Lambda(x_d) := C + \int_{-1/2}^{x_d} \lambda(t) \, dt$ where the constant $C$ is chosen such that the mean of $\Lambda$ over $(-1/2,1/2)$ vanishes and define next $v := (v_j)_{1 \leq j \leq d} \in \mathcal{W}(\mathcal{Y})$ by $v_i(x) = \Lambda(x_d)$ and $v_j = 0$ for $j \neq i$. Using $v$ as test function in~\eqref{eq:varWxi} (with of course $\xi = \xi_q$) yields that $\dps \int_{-1/2}^{1/2} \lambda(x_d) \, T^{\star,\xi_q}_{\alpha\beta}(x_d) : e_i \otimes e_d \, dx_d = 0$. Since this equality holds for any $\lambda \in L^2(-1/2,1/2)$, we obtain~\eqref{eq:theatre7}. We therefore can write
$$
\int_\Omega T^{\star,\xi_q}_{\alpha\beta} : e^\eps(u^\eps + g) \, \phi = \int_\Omega \left(T^{\star,\xi_q}_{\alpha\beta}\right)_{\gamma \delta} \, e^\eps_{\gamma\delta}(u^\eps + g) \, \phi = \int_\Omega \left(T^{\star,\xi_q}_{\alpha\beta}\right)_{\gamma \delta} \, e_{\gamma\delta}(u^\eps + g) \, \phi.
$$
We know that the sequence $(u^\eps)_{\eps>0}$ weakly converges in $(H^1(\Omega))^d$ to $u^\star$. Consequently, for all $1\leq \gamma, \delta \leq d-1$, the sequence $\left(e_{\gamma \delta}(u^\eps+g) \right)_{\eps >0}$ weakly converges in $L^2(\Omega)$ to $e_{\gamma \delta}(u^\star + g)$. Since $u^\star$ and $g$ belong to $\GKL^{\mathcal{B}}$, we have $e_{\gamma \delta}(u^\star + g)(x) = -x_d \, \partial_{\gamma \delta}(u^\star_d + g_d)(x')$. We thus deduce that
\begin{equation} \label{eq:theatre8}
\lim_{\eps \to 0} \int_\Omega T^{\star,\xi_q}_{\alpha\beta} : e^\eps(u^\eps + g) \, \phi = - \int_\Omega x_d \, \left(T^{\star,\xi_q}_{\alpha\beta}(x_d)\right)_{\gamma \delta} \, \partial_{\gamma \delta}(u^\star_d + g_d)(x') \, \phi(x') \, dx. 
\end{equation}

We next consider the first term of~\eqref{eq:theatre6}, recast it as
$$
\int_\Omega \left(T^\eps - T^{\star,\xi_q}_{\alpha\beta}\right) : e^\eps(u^\eps) \, \phi
=
\int_\Omega Z\left(\frac{x'}{\eps},x_d\right) : e^\eps(u^\eps)(x',x_d) \, \phi(x') \, dx' \, dx_d
$$
with
$$
Z(x',x_d) = A(x',x_d) \left[ e\left( W^{\alpha\beta, \xi_q}\right)(x',x_d) - \frac{x_d^{2q+1}}{2q+1} \, e_\alpha \otimes e_\beta \right] - T^{\star,\xi_q}_{\alpha\beta}(x_d),
$$
and bound it using Lemma~\ref{lemma:minZ2}. We check that $u^\eps \in V$ and recall that $\phi \in \mathcal{D}(\omega) \subset W^{2,\infty}(\omega)$. The matrix-valued function $Z$ defined above satisfies the assumptions of Lemma~\ref{lemma:minZ2}: this is obvious for assumptions~(i) and~(vi). Assumption~(v) holds in view of the boundary condition implied by~\eqref{eq:varWxi} and of~\eqref{eq:theatre7}, and Assumption~(iv) stems from the fact that
$$
\div Z = - \div T^{\star,\xi_q}_{\alpha\beta} = - e_i \left[ \partial_j \left( T^{\star,\xi_q}_{\alpha\beta} \right)_{ij} \right] = - e_i \left[ \partial_d \left( T^{\star,\xi_q}_{\alpha\beta} \right)_{id} \right] = 0,
$$
where we have used that $T^{\star,\xi_q}_{\alpha\beta}$ only depends on $x_d$ and eventually~\eqref{eq:theatre7}. Assumptions~(ii) and~(iii) are also satisfied, in view of the fact that $\dps \int_Y Z(x',x_d) \, dx' = 0$ for any $x_d \in (-1/2,1/2)$, which is a direct consequence of the definition of $T^{\star,\xi_q}_{\alpha\beta}$. We are thus in position to use Lemma~\ref{lemma:minZ2}, and deduce that
$$
\left| \int_\Omega \left(T^\eps - T^{\star,\xi_q}_{\alpha\beta} \right) : e^\eps(u^\eps) \, \phi \right| \leq C \, \eps \, |\omega|^{1/2} \, \|\nabla \phi\|_{W^{1,\infty}(\omega)} \, \| e^\eps(u^\eps) \|_{L^2(\Omega)},
$$
and hence, using the a priori bound on $e^\eps(u^\eps)$ (see~\eqref{est:sig}, \eqref{eq:inde_eps_vec} and Assumption~\eqref{eq:assump_g}, which implies that $e^\eps(g) = e(g)$), we get
\begin{equation} \label{eq:theatre9}
\lim_{\eps \to 0} \int_\Omega \left(T^\eps - T^{\star,\xi_q}_{\alpha\beta} \right) : e^\eps(u^\eps) \, \phi = 0.
\end{equation}

We next consider the second term of~\eqref{eq:theatre6}. Recalling Assumption~\eqref{eq:assump_g}, we compute that
\begin{equation} \label{eq:theatre14}
\int_\Omega \left(T^\eps - T^{\star,\xi_q}_{\alpha\beta}\right) : e^\eps(g) \, \phi
=
\int_\Omega \left(T^\eps - T^{\star,\xi_q}_{\alpha\beta}\right) : e(g) \, \phi
\mathop{\longrightarrow}_{\eps \to 0}
0,
\end{equation}
where the last convergence comes from the fact that the oscillatory function $T^\eps$ weakly converges to its $Y$-average (see Lemma~\ref{limmoyenne}), which is exactly $T^{\star,\xi_q}_{\alpha\beta}$.

We now turn to the fourth term of~\eqref{eq:theatre6} and write
\begin{align*}
  \int_\Omega U^\eps : e^\eps(u^\eps + g) \, \phi
  & = \frac{2q}{\eps^2} \int_\Omega x_d^{2q-1} \, \nu(x') \, A^\eps(x) \, e_d\otimes e_d : e^\eps(u^\eps +g)(x) \, \phi(x') \, dx
  \\
  & = \frac{2q}{\eps^2} \int_\Omega x_d^{2q-1} \, \nu(x') \, \sigma_{dd}^\eps(x) \, \phi(x') \, dx
  \\
  & = 2q \int_\Omega x_d^{2q-1} \, \nu(x') \, \Sigma_{dd}^\eps(x) \, \phi(x') \,dx.
\end{align*}
Using the convergence~\eqref{eq:sigma_eps_dd}, we deduce that
\begin{equation} \label{eq:theatre10}
\lim_{\eps \to 0} \int_\Omega U^\eps : e^\eps(u^\eps + g) \, \phi = \left\langle 2q \, \nu \int_{-1/2}^{1/2} x_d^{2q-1} \, \Sigma_{dd}^\star(\cdot,x_d) \, dx_d, \phi \right\rangle_{\mathcal{D}'(\omega), \mathcal{D}(\omega)}.
\end{equation}
Collecting~\eqref{eq:theatre6}, \eqref{eq:theatre8}, \eqref{eq:theatre9}, \eqref{eq:theatre14} and~\eqref{eq:theatre10}, we thus have proved that, in the sense of distributions in $\mathcal{D}'(\omega)$, 
\begin{multline*}
\lim_{\eps \to 0} \int_{-1/2}^{1/2} \sigma^\eps(\cdot,x_d) : e^\eps(v^\eps)(\cdot,x_d) \, dx_d \\ = 2q \, \nu \int_{-1/2}^{1/2} x_d^{2q-1} \, \Sigma_{dd}^\star(\cdot,x_d) \, dx_d - \left( \int_{-1/2}^{1/2} x_d \left( T^{\star,\xi_q}_{\alpha\beta}(x_d) \right)_{\gamma \delta} \, dx_d \right) \partial_{\gamma \delta}(u^\star_d + g_d).
\end{multline*}

\medskip

\noindent
{\bf Step~3.} Collecting the results obtained at Step~1 and Step~2, we obtain
$$
\int_{-1/2}^{1/2} \frac{x_d^{2q+1}}{2q+1} \, \Sigma_{\alpha\beta}^\star(\cdot,x_d) \, dx_d = \left( \int_{-1/2}^{1/2} x_d \left( T^{\star,\xi_q}_{\alpha\beta}(x_d) \right)_{\gamma \delta} \, dx_d \right) \partial_{\gamma \delta}(u^\star_d + g_d),
$$
which is the desired result~\eqref{eq:qref}. This concludes the proof of Lemma~\ref{lem:Sigmaab}.
\end{proof}




\medskip

Under some stronger regularity assumptions on $A$, $f$, $g$ and $h_\pm$, the following lemma states that the matrix $\Sigma^\star$ built in Lemma~\ref{lem:lemma1} is actually more regular than shown there. Before stating this lemma, we first recall that the homogenized matrix $K_{22}^\star$ defined in Theorem~\ref{limitel} satisfies
\begin{equation} \label{eq:lundi5}
  \forall \tau \in \R_s^{(d-1)\times(d-1)}, \qquad \frac{c_-}{12} \ \tau : \tau \leq K_{22}^\star \, \tau : \tau \leq c_+ \left( \frac{c_+}{c_-} + 1 \right) \tau : \tau.
\end{equation}
The lower bound is given by~\eqref{eq:titi11} below and the upper bound is obtained by observing that
$$
K_{22}^\star \, \tau : \tau = \int_\Y A \big( e(W^\tau) - x_d \, \tau \big) : \big( e(W^\tau) - x_d \, \tau \big) = - \int_\Y x_d \, A \big( e(W^\tau) - x_d \, \tau \big) : \tau,
$$
where $W^\tau = \tau_{\alpha \beta} \, W^{\alpha\beta}$, and hence (using a bound similar to~\eqref{eq:dimanche} on $e(W^\tau)$)
$$
K_{22}^\star \, \tau : \tau \leq c_+ \, |\tau| \, \big( \| e(W^\tau) \|_{L^2(\Y)} + |\tau| \big) \leq c_+ \left( \frac{c_+}{c_-} + 1 \right) |\tau|^2.
$$

\begin{lemma}\label{lem:lem00}
Assume that $d=2$, that $A$ satisfies the symmetries~\eqref{hyp:symA} and that we are in the bending case~\eqref{eq:ass_bending}. Assume furthermore that $f$ and $h_\pm$ satisfy the regularity assumptions~\eqref{eq:hyp_regul_f} and~\eqref{eq:hyp_regul_h}, and that
\begin{equation} \label{eq:hyp_regul_g}
  g_d \in H^4(\omega).
\end{equation}
Then, it holds that, for all $1\leq \alpha, \beta \leq d-1$,
$$
\Sigma_{\alpha \beta}^\star \in L^2\left( \left(-\frac{1}{2}, \frac{1}{2}\right), H^2(\omega)\right), \quad \Sigma_{\alpha d}^\star \in L^2\left( \left(-\frac{1}{2}, \frac{1}{2}\right), H^1(\omega)\right) \quad \text{and} \quad \Sigma_{dd}^\star \in L^2(\Omega).
$$
Under the shape regularity assumption~\eqref{eq:shape_regul} on $\omega$, we have the bounds
\begin{align}
  \forall 1 \leq \alpha \leq d-1, \ \ \| \Sigma_{\alpha d}^\star \|_{L^2(\Omega)} & \leq C \, \overline{{\cal N}}^{\rm bend}_\Omega(f,g,h_\pm),
  \label{eq:borne_L2_Sigma_star_alphad}
  \\
  \| \Sigma_{dd}^\star \|_{L^2(\Omega)} & \leq C \, {\cal N}^{\rm bend}(f,g,h_\pm),
  \label{eq:borne_L2_Sigma_star_dd}
\end{align}
for some constant $C$ only depending on $c_-$ and $c_+$ and where ${\cal N}^{\rm bend}(f,g,h_\pm)$ and $\overline{{\cal N}}^{\rm bend}_\Omega(f,g,h_\pm)$ are defined by
\begin{multline} \label{eq:def_norme_N_bend}
  {\cal N}^{\rm bend}(f,g,h_\pm) = \| f_d \|_{L^2(\Omega)} + \sum_{\alpha,\beta=1}^{d-1} \| \partial_\beta f_\alpha \|_{L^2(\Omega)} + \| \nabla^4_{d-1} g_d \|_{L^2(\omega)} \\ + \| (h_\pm)_d \|_{L^2(\omega)} + \sum_{\alpha,\beta=1}^{d-1} \| \partial_\beta (h_\pm)_\alpha \|_{L^2(\omega)},
\end{multline}
%
%
and
\begin{equation} \label{eq:def_norme_N_bend_full}
  \overline{{\cal N}}^{\rm bend}_\Omega(f,g,h_\pm) = \| f_\alpha \|_{L^2(\Omega)} + \| (h_\pm)_\alpha \|_{L^2(\omega)} + |\omega|^{\frac{1}{d-1}} \, {\cal N}^{\rm bend}(f,g,h_\pm).
\end{equation}
%
%
\end{lemma}

It is easy to see in the proof below that we also have the following additional bounds:
\begin{gather*}
\sum_{\alpha,\beta,\rho,\tau=1}^{d-1} \| \partial_{\rho\tau} \Sigma_{\alpha \beta}^\star \|_{L^2(\Omega)} + \sum_{\alpha,\rho=1}^{d-1} \| \partial_\rho \Sigma_{\alpha d}^\star \|_{L^2(\Omega)} \leq C \, {\cal N}^{\rm bend}(f,g,h_\pm),
\\
\sum_{\alpha,\beta=1}^{d-1} \| \Sigma_{\alpha \beta}^\star \|_{L^2(\Omega)} \leq C \, |\omega|^{\frac{2}{d-1}} \, {\cal N}^{\rm bend}(f,g,h_\pm), \quad \sum_{\alpha,\beta,\rho=1}^{d-1} \| \partial_\rho \Sigma_{\alpha \beta}^\star \|_{L^2(\Omega)} \leq C \, |\omega|^{\frac{1}{d-1}} \, {\cal N}^{\rm bend}(f,g,h_\pm).
\end{gather*}
The restriction to the two-dimensional case only comes from the fact that the proof of the above result uses Lemma~\ref{lem:Sigmaab}.

\begin{proof}
We start by proving that $u^\star_d\in H^4(\omega)$. The homogenized equation (see~\eqref{formvar:pb_bending}) reads
\begin{multline*}
\nabla^2_{d-1} : K_{22}^\star \nabla^2_{d-1} u_d^\star = \m(f_d) + (h_+)_d + (h_-)_d + \div' \left( \m(x_d \, f') \right) \\ + \frac{1}{2} \div' \left( h'_+-h'_- \right) - \nabla^2_{d-1} : K_{22}^\star \nabla^2_{d-1} g_d \quad \mbox{in $\mathcal{D}'(\omega)$}.
\end{multline*}
The assumptions~\eqref{eq:hyp_regul_f}, \eqref{eq:hyp_regul_g} and~\eqref{eq:hyp_regul_h} on $f$, $g$ and $h_\pm$ and the bound~\eqref{eq:lundi5} on $K_{22}^\star$ imply that
$$
\nabla^2_{d-1} : K_{22}^\star \nabla^2_{d-1} u_d^\star \in L^2(\omega) 
$$
with $\| \nabla^2_{d-1} : K_{22}^\star \nabla^2_{d-1} u_d^\star \|_{L^2(\omega)} \leq C \, {\cal N}^{\rm bend}(f,g,h_\pm)$ for some $C$ only depending on $c_-$ and $c_+$. We furthermore have $u^\star_d \in H^2_0(\omega)$. Thus, by standard elliptic regularity, we obtain that $u_d^\star \in H^4(\omega)$. A scaling argument (using the shape regularity assumption~\eqref{eq:shape_regul} on $\omega$) shows that $\| \nabla^4_{d-1} u_d^\star \|_{L^2(\omega)} \leq C \, {\cal N}^{\rm bend}(f,g,h_\pm)$, $\| \nabla^3_{d-1} u_d^\star \|_{L^2(\omega)} \leq C \, |\omega|^{\frac{1}{d-1}} \, {\cal N}^{\rm bend}(f,g,h_\pm)$ and $\| \nabla^2_{d-1} u_d^\star \|_{L^2(\omega)} \leq C \, |\omega|^{\frac{2}{d-1}} \, {\cal N}^{\rm bend}(f,g,h_\pm)$ for some $C$ only depending on $c_-$ and $c_+$. Using now Lemma~\ref{lem:Sigmaab} (and more precisely~\eqref{eq:Sigmaformula}) and~\eqref{eq:lundi4}, we deduce that, for all $1\leq \alpha, \beta \leq d-1$, we have $\dps \Sigma_{\alpha \beta}^\star \in L^2\left( \left(-\frac{1}{2}, \frac{1}{2}\right), H^2(\omega)\right)$ with, for all $1 \leq \rho,\tau \leq d-1$,
\begin{gather*}
\left\| \partial_{\rho\tau} \Sigma^\star_{\alpha\beta} \right\|_{L^2(\Omega)} \leq C \, {\cal N}^{\rm bend}(f,g,h_\pm),
\\
\left\| \partial_\rho \Sigma^\star_{\alpha\beta} \right\|_{L^2(\Omega)} \leq C \, |\omega|^{\frac{1}{d-1}} \, {\cal N}^{\rm bend}(f,g,h_\pm),
\\
\left\| \Sigma^\star_{\alpha\beta} \right\|_{L^2(\Omega)} \leq C \, |\omega|^{\frac{2}{d-1}} \, {\cal N}^{\rm bend}(f,g,h_\pm),
\end{gather*}
for some $C$ only depending on $c_-$ and $c_+$. 

\medskip
 
We next prove that $\dps \Sigma_{\alpha d}^\star \in L^2\left( \left( -\frac{1}{2}, \frac{1}{2} \right), H^1(\omega) \right)$ for all $1\leq \alpha \leq d-1$. Passing to the limit $\eps \to 0$ in~\eqref{eq:neweq1} and using~\eqref{eq:sigma_eps_alpha_d} and~\eqref{eq:sigma_eps_alpha_beta}, it holds that, for any $\dps w \in \left( L^2\left( \left( -\frac{1}{2}, \frac{1}{2} \right), H^1_0(\omega) \right) \right)^{d-1}$,
\begin{align*}
  & \left\langle \Sigma_{\alpha d}^\star, w_\alpha \right\rangle_{L^2\left( \left( -\frac{1}{2}, \frac{1}{2} \right), H^{-1}(\omega) \right),L^2\left( \left( -\frac{1}{2}, \frac{1}{2} \right), H^1_0(\omega) \right)}
  \\
  & =
  \int_\Omega f_\alpha \, v^w_\alpha + \int_{\Gamma_\pm} (h_\pm)_\alpha \, v^w_\alpha - \int_\Omega \Sigma_{\alpha \beta}^\star \, e_{\alpha \beta}(v^w)
  \\
  & = \int_\Omega f_\alpha \, v^w_\alpha + \int_{\Gamma_\pm} (h_\pm)_\alpha \, v^w_\alpha + \int_\Omega \partial_\beta \Sigma_{\alpha \beta}^\star \, v^w_\alpha \quad \text{[since $v^w$ vanishes on $\partial \omega \times (-1/2,1/2)$]}
  \\
  & = \int_\Omega w_\alpha(x',z) \left( \int_z^{1/2} \big( f_\alpha(x',\cdot) + \partial_\beta \Sigma_{\alpha\beta}^\star(x',\cdot) \big) \right) dx' \, dz + \int_\Omega w_\alpha(x',z) \, (h_+)_\alpha(x') \, dx' \, dz,
\end{align*}
where $v^w$ is defined in terms of $w$ by~\eqref{eq:def_vw_1} (and thus vanishes on $\Gamma_-$) and where we have used an integration by parts with respect to the $d$-th variable in the last line. Since $w$ is arbitrary, this implies that, for any $1 \leq \alpha \leq d-1$,
\begin{equation}\label{eq:Sigad}
  \Sigma_{\alpha d}^\star(x',z) = (h_+)_\alpha(x') + \int_z^{1/2} \big( f_\alpha(x',\cdot) + \partial_\beta \Sigma_{\alpha\beta}^\star(x',\cdot) \big) \quad \mbox{ in $\mathcal{D}'(\Omega)$}.
\end{equation}
Since $\dps \Sigma_{\alpha\beta}^\star \in L^2\left( \left(-\frac{1}{2}, \frac{1}{2}\right), H^2(\omega) \right)$ for all $1\leq \alpha, \beta \leq d-1$ and in view of the regularity~\eqref{eq:hyp_regul_f} and~\eqref{eq:hyp_regul_h} on $f$ and $h_+$, we infer from~\eqref{eq:Sigad} that $\dps \Sigma_{\alpha d}^\star \in L^2\left( \left(-\frac{1}{2}, \frac{1}{2}\right), H^1(\omega)\right)$ for all $1 \leq \alpha \leq d-1$. In addition, we have the following bounds: for all $1 \leq \rho \leq d-1$,
\begin{gather*}
  \left\| \partial_\rho \Sigma^\star_{\alpha d} \right\|_{L^2(\Omega)} \leq C \, {\cal N}^{\rm bend}(f,g,h_\pm),
  \\
\left\| \Sigma^\star_{\alpha d} \right\|_{L^2(\Omega)} \leq C \left( \| f_\alpha \|_{L^2(\Omega)} + \| (h_\pm)_\alpha \|_{L^2(\omega)} + |\omega|^{\frac{1}{d-1}} \, {\cal N}^{\rm bend}(f,g,h_\pm) \right),
\end{gather*}
for some $C$ only depending on $c_-$ and $c_+$.
 
\medskip

We next prove that $\Sigma_{dd}^\star \in L^2(\Omega)$. Let $\dps w_d \in L^2\left( \left( -\frac{1}{2}, \frac{1}{2} \right), H^2_0(\omega)\right)$ and consider $v^w$ defined by~\eqref{eq:vd}. Passing to the limit $\eps \to 0$ in~\eqref{eq:eq2} and using~\eqref{eq:sigma_eps_dd} and~\eqref{eq:sigma_eps_alpha_d} and the fact that $v^w$ vanishes on $\Gamma_-$, we obtain that
\begin{align*}
  & \left\langle \Sigma_{dd}^\star, w_d \right\rangle_{L^2\left( \left( -\frac{1}{2}, \frac{1}{2} \right), H^{-2}(\omega) \right),L^2\left( \left( -\frac{1}{2}, \frac{1}{2} \right), H^2_0(\omega) \right)}
  \\
  &=
  \int_\Omega f_d \, v^w_d + \int_{\Gamma_\pm} (h_\pm)_d \, v^w_d - \int_\Omega \Sigma_{\alpha d}^\star \, \partial_\alpha v^w_d
  \\
  & =
  \int_\Omega w_d(x',z) \left( \int_z^{1/2} \big( f_d(x',\cdot) + \partial_\alpha \Sigma_{\alpha d}^\star(x',\cdot) \big) \right) dx' \, dz + \int_\Omega w_d(x',z) \, (h_+)_d(x') \, dx' \, dz,
\end{align*}
which implies, since $w_d$ is arbitrary, that
$$
\Sigma_{dd}^\star(x',z) = (h_+)_d(x') + \int_z^{1/2} \big( f_d(x',\cdot) + \partial_\alpha \Sigma_{\alpha d}^\star(x',\cdot) \big) \quad \mbox{ in $\mathcal{D}'(\Omega)$}.
$$
In view of~\eqref{eq:hyp_regul_f} and~\eqref{eq:hyp_regul_h} and since $\dps \Sigma_{\alpha d}^\star \in L^2\left( \left(-\frac{1}{2}, \frac{1}{2}\right), H^1(\omega)\right)$ for all $1\leq \alpha \leq d-1$, we deduce that $\Sigma_{dd}^\star \in L^2(\Omega)$ with
$$
\left\| \Sigma^\star_{dd} \right\|_{L^2(\Omega)} \leq C \, {\cal N}^{\rm bend}(f,g,h_\pm),
$$
for some $C$ only depending on $c_-$ and $c_+$. This concludes the proof of Lemma~\ref{lem:lem00}.
\end{proof}

A consequence of Lemmas~\ref{lem:lemma1} and~\ref{lem:lem00} is the following result. 

\begin{lemma}\label{lem:lem2}
Under the same assumptions as in Lemma~\ref{lem:lem00}, it holds that
$$
\forall v \in V, \qquad \int_\Omega \left( \Sigma^\eps - \Sigma^\star \right) : e(v) = 0.
$$
\end{lemma}
We emphasize that this result holds without taking the limit $\eps \to 0$. This is critical, since we will use that equality later on for functions $v$ that depend on $\eps$. 

\begin{proof}
Let $\eps >0$ and $v\in V$. Since the space $\dps \left( \mathcal{C}^\infty\left( \left[-\frac{1}{2}, \frac{1}{2}\right], \mathcal{D}(\omega) \right) \right)^d$ is dense in $V$, there exists a regularization of $v$, namely $\dps \phi_n \in \left( \mathcal{C}^\infty \left( \left[-\frac{1}{2}, \frac{1}{2}\right], \mathcal{D}(\omega)\right) \right)^d$ such that $\| e^\eps(v-\phi_n) \|_{L^2(\Omega)} \leq 1/n$ for any $n \in \N^\star$.

Then, using Lemma~\ref{lem:lemma1} (and more precisely~\eqref{eq:bureau4}), we compute that
\begin{align*}
  \int_\Omega \Sigma^\eps : e(v) =
  & \int_\Omega \Sigma^\eps : e(\phi_n) + \int_\Omega \Sigma^\eps : e(v-\phi_n)
  \\
  =
  & \sum_{1\leq \alpha, \beta \leq d-1} \langle \Sigma_{\alpha\beta}^\star, e_{\alpha\beta}(\phi_n) \rangle_{L^2(\Omega)}
  \\
  & + 2 \sum_{1\leq \alpha\leq d-1} \langle \Sigma_{\alpha d}^\star, e_{\alpha d}(\phi_n) \rangle_{L^2\left(\left(-\frac{1}{2}, \frac{1}{2}\right), H^{-1}(\omega)\right), L^2\left(\left(-\frac{1}{2}, \frac{1}{2}\right), H^1_0(\omega)\right)}
  \\
  & + \langle \Sigma_{dd}^\star, e_{dd}(\phi_n) \rangle_{L^2\left(\left(-\frac{1}{2}, \frac{1}{2}\right), H^{-2}(\omega)\right), L^2\left(\left(-\frac{1}{2}, \frac{1}{2}\right), H^2_0(\omega)\right)} + \int_\Omega \Sigma^\eps : e(v-\phi_n).
\end{align*}
Using now Lemma~\ref{lem:lem00} (and hence the fact that all components of $\Sigma^\star$ belong to $L^2(\Omega)$), we get
\begin{align}
  \int_\Omega \Sigma^\eps : e(v)
  &=
  \int_\Omega \Sigma^\star : e(\phi_n) + \int_\Omega \Sigma^\eps : e(v-\phi_n)
  \nonumber
  \\
  &=
  \int_\Omega \Sigma^\star : e(v) + \int_\Omega \Sigma^\star : e(\phi_n - v) + \int_\Omega \Sigma^\eps : e(v-\phi_n).
  \label{eq:bureau5}
\end{align}
Moreover, it holds that 
$$
\left| \int_\Omega \Sigma^\eps : e(v-\phi_n) \right| = \left| \int_\Omega \sigma^\eps : e^\eps(v-\phi_n) \right| \leq \frac{1}{n} \, \|\sigma^\eps\|_{L^2(\Omega)}
$$
and likewise
$$
\left| \int_\Omega \Sigma^\star : e(\phi_n-v) \right| \leq \| \Sigma^\star \|_{L^2(\Omega)} \, \|e(v-\phi_n) \|_{L^2(\Omega)} \leq \| \Sigma^\star \|_{L^2(\Omega)} \|e^\eps(v - \phi_n) \|_{L^2(\Omega)} \leq \frac{1}{n} \, \| \Sigma^\star \|_{L^2(\Omega)}.
$$
Thus, letting $n$ go to $+\infty$ in~\eqref{eq:bureau5}, we obtain that $\dps \int_\Omega \Sigma^\eps : e(v) = \int_\Omega \Sigma^\star : e(v)$. This concludes the proof of Lemma~\ref{lem:lem2}.
\end{proof}

\subsubsection{Main result (bending case)}

We are now in position to state and prove our main result in the bending case. As in the membrane case (see~\eqref{eq:defnorm_vec}), for any $u \in (H^1(\Omega))^d$, we define the norm
$$
\|u\|_{H^1_\eps(\Omega)} := \sqrt{ \frac{\| u \|_{L^2_w(\Omega)}^2}{\left[ \max\left(|\omega|^{\frac{1}{d-1}}, \eps^2 \, |\omega|^{-\frac{1}{d-1}}\right) \right]^2} + \| e^\eps(u) \|_{L^2(\Omega)}^2 },
$$
where the weighted $L^2_w(\Omega)$ norm is defined by~\eqref{eq:def_L2_w}. This definition in particular implies that, on the space $V$, the norms $\| e^\eps(u) \|_{L^2(\Omega)}$ and $\| u \|_{H^1_\eps(\Omega)}$ are equivalent, with equivalence constants which are independent of $\eps$ and $\omega$ (see~\eqref{eq:poinK_corro_vec}).

\begin{theorem} \label{thconvforte3_a}
Under Assumptions~\eqref{eq:inde_eps_vec} and~\eqref{eq:assump_g}, consider the solution $u^\eps$ to~\eqref{formvarelast} and its homogenized limit $u^\star$, solution to~\eqref{formvarelasthomog}. Assume that $d=2$, that $A$ satisfies the symmetries~\eqref{hyp:symA}, that we are in the bending case~\eqref{eq:ass_bending} and that $\omega$ satisfies~\eqref{eq:shape_regul}. Assume furthermore that the regularity assumptions~\eqref{eq:hyp_regul_f}, \eqref{eq:hyp_regul_g} and~\eqref{eq:hyp_regul_h} on $f$, $g$ and $h_\pm$ are satisfied. For any $x = (x',x_d) \in \Omega$ and any $1 \leq \gamma \leq d-1$, let 
\begin{align*}
u_\gamma^{\eps,1}(x)
& :=
u_\gamma^\star(x) + \eps \, W_\gamma^{\alpha \beta} \left(\frac{x'}{\eps},x_d \right) \partial_{\alpha \beta}(u^\star_d + g_d)(x')
\\
&=
-x_d \, \partial_\gamma u_d^\star(x') + \eps \, W_\gamma^{\alpha \beta} \left(\frac{x'}{\eps},x_d \right) \partial_{\alpha \beta}(u^\star_d + g_d)(x')
\end{align*}
and 
$$
u_d^{\eps,1}(x) := u_d^\star(x') + \eps^2 \, W_d^{\alpha \beta} \left(\frac{x'}{\eps},x_d \right) \partial_{\alpha \beta}(u^\star_d + g_d)(x'),
$$
where, for any $1 \leq \alpha,\beta \leq d-1$, $W^{\alpha \beta}$ is the corrector defined by~\eqref{el-prcor2}.

We also assume that $\nabla^2 (u_d^\star+g_d)$ belongs to $\left( W^{2,\infty}(\omega) \right)^{d \times d}$ and that, for any $1 \leq \alpha,\beta \leq d-1$, we have $\dps W^{\alpha \beta} \in \left[ W^{1,\infty}\left(\mathbb{R}^{d-1} \times \left( -\demi, \demi \right) \right) \right]^d$. Then, there exists a constant $C>0$ independent of $\eps$, $\omega$ (but depending on the constant $\eta$ of~\eqref{eq:shape_regul}), $u^\star$, $f$, $g$ and $h_\pm$ such that
\begin{multline} \label{eq:theatre13}
  \| u^\eps - u^{\eps,1} \|_{H^1_\eps(\Omega)} \leq C \, \sqrt{\eps} \, \Big( |\omega|^{\frac{d-2}{2(d-1)}} \, \| \nabla^2 (u_d^\star+g_d) \|_{L^\infty(\omega)} + \sqrt{\eps} \, |\omega|^{1/2} \, \| \nabla^3 (u_d^\star+g_d) \|_{W^{1,\infty}(\omega)} \\ + \sqrt{\eps} \ \overline{{\cal N}}^{\rm bend}_\Omega(f,g,h_\pm) + \eps^{3/2} \, {\cal N}^{\rm bend}(f,g,h_\pm) \Big),
\end{multline}
where we recall that $\| \cdot \|_{H^1_\eps(\Omega)}$ is defined by~\eqref{eq:defnorm_vec}, and where $\overline{{\cal N}}^{\rm bend}_\Omega(f,g,h_\pm)$ and ${\cal N}^{\rm bend}(f,g,h_\pm)$ are defined by~\eqref{eq:def_norme_N_bend_full} and~\eqref{eq:def_norme_N_bend}.
\end{theorem}
Up to lower order terms, we have 
$$
e^\eps(u^{\eps,1}+g)(x) \approx \left[ e(W^{\alpha \beta})\left(\frac{x'}{\eps},x_d\right) - x_d \, e_\alpha \otimes e_\beta \right] \partial_{\alpha \beta}(u^\star_d+g_d)(x').
$$

Similarly to the above results, the restriction to the two-dimensional case stems from the fact that we use Lemma~\ref{lemma:minZ2} (and its consequences, in particular Lemma~\ref{lem:lem2}) in the proof. Should Lemma~\ref{lemma:minZ2} hold in higher-dimensional settings, so would Theorem~\ref{thconvforte3_a}.

\begin{remark} \label{rem:regul_ustar_bend}
  As for the diffusion case (see Remark~\ref{rem:regul_ustar_diff}) and the elasticity membrane case (see Remark~\ref{rem:regul_ustar_memb}), we wish to point out that the assumption $\nabla^2 (u^\star_d + g_d) \in (W^{2,\infty}(\omega))^{d \times d}$ is a standard assumption when proving convergence rates of two-scale expansions (see, e.g.,~\cite[p.~28]{jikov}). Note that, in view of~\eqref{formvarelasthomog}, which can be recast as~\eqref{formvar:pb_bending} in the bending case, this assumption implies that $\dps \m(f_d) + (h_+)_d + (h_-)_d + \div' \left[ \m(x_d \, f') + \frac{1}{2} (h'_+-h'_-) \right]$ belongs to $L^\infty(\omega)$.
\end{remark}

\begin{proof} The proof falls in six steps. In the first step, we correct for the boundary mismatch between $u^\eps$ and its approximation $u^{\eps,1}$. In Step~2, we decompose the error in the bulk of the domain between $u^{\eps,1}$ and $u^\eps$ in three contributions, which are each estimated in Steps~3--5. We collect all the estimates in Step~6 to conclude the proof.

\medskip

\noindent
{\bf Step~1}: Let $\tau_\eps \in \mathcal{D}(\omega)$ such that $0 \leq \tau_\eps \leq 1$ in $\omega$ and such that $\tau_\eps(x') = 1$ for any $x' \in \omega$ such that $\text{dist}(x',\partial \omega) \geq \eps$. Since $\omega$ is smooth, we can choose $\tau_\eps$ such that $\eps \|\nabla \tau_\eps \|_{L^\infty(\omega)} \leq C$ for some $C>0$ independent of $\omega$ and $\eps$. We define $\omega_\eps := \{ x' \in \omega \text{ such that } \text{dist}(x',\partial \omega) \geq \eps \}$ and $\dps \Omega_\eps := \omega_\eps \times \left( -\frac{1}{2},\frac{1}{2} \right)$. Note that $|\Omega \setminus \Omega_\eps| \leq C \, \eps \, |\omega|^{\frac{d-2}{d-1}}$. 

We introduce the function $v^{\eps,1}$ defined for $x = (x',x_d) \in \Omega$ by
\begin{align*}
  \forall 1 \leq \gamma \leq d-1, \qquad v_\gamma^{\eps,1}(x) := u_\gamma^\star(x) + \eps \, \tau_\eps \, W^{\alpha \beta}_\gamma \left( \frac{x'}{\eps},x_d \right) \partial_{\alpha \beta}(u^\star_d + g_d)(x'),
  \\
  v_d^{\eps,1}(x) := u_d^\star(x') + \eps^2 \, \tau_\eps \, W^{\alpha \beta}_d \left( \frac{x'}{\eps}, x_d \right) \partial_{\alpha \beta}(u^\star_d + g_d)(x').
\end{align*}
By definition of $\tau_\eps$, we have $v^{\eps,1} \in V$ and $v^{\eps,1} = u^{\eps,1}$ in $\Omega_\eps$. In this first step of the proof, we bound $\| u^{\eps,1} - v^{\eps,1} \|_{H^1_\eps(\Omega)}$. We compute that 
\begin{gather*}
u^{\eps,1}_\gamma(x) - v^{\eps,1}_\gamma(x) = \eps \, (1-\tau_\eps(x')) \, W_\gamma^{\alpha \beta} \left( \frac{x'}{\eps},x_d \right) \partial_{\alpha \beta}(u^\star_d+g_d)(x'),
\\
u^{\eps,1}_d(x) - v^{\eps,1}_d(x) = \eps^2 \, (1-\tau_\eps(x')) \, W_d^{\alpha \beta} \left( \frac{x'}{\eps},x_d \right) \partial_{\alpha \beta}(u^\star_d+g_d)(x').
\end{gather*}
We thus get that
\begin{align}
  \sum_{\gamma=1}^{d-1} \| v^{\eps,1}_\gamma - u^{\eps,1}_\gamma \|^2_{L^2(\Omega)}
  &\leq
  C \, \eps^2 \, \| 1-\tau_\eps \|^2_{L^2(\Omega \setminus \Omega_\eps)} \sup_{1 \leq \alpha,\beta \leq d-1} \| W^{\alpha \beta} \|_{L^\infty}^2 \, \| \nabla^2 (u^\star_d+g_d) \|^2_{L^\infty(\omega)}
  \nonumber
  \\
  &\leq
  C \, \eps^2 \, |\Omega \setminus \Omega_\eps| \, \| \nabla^2 (u^\star_d+g_d) \|^2_{L^\infty(\omega)}
  \nonumber
  \\
  &\leq
  C \, \eps^3 \, |\omega|^{\frac{d-2}{d-1}} \, \| \nabla^2 (u^\star_d+g_d) \|^2_{L^\infty(\omega)}
  \nonumber
  \\
  &\leq
  C \, \eps \, |\omega|^{\frac{d-2}{d-1}} \, \left[ \max\left(|\omega|^{\frac{1}{d-1}}, \eps^2 \, |\omega|^{-\frac{1}{d-1}}\right) \right]^2 \, \| \nabla^2 (u^\star_d+g_d) \|^2_{L^\infty(\omega)},
  \label{diffapp3_L2_el_bend_gamma}
\end{align}
where the last estimate stems from the fact that $\dps \eps \leq \max\left(|\omega|^{\frac{1}{d-1}}, \eps^2 \, |\omega|^{-\frac{1}{d-1}}\right)$. We also have that
\begin{align}
  |\omega|^{\frac{-2}{d-1}} \, \| v^{\eps,1}_d - u^{\eps,1}_d \|^2_{L^2(\Omega)}
  &\leq
  C \, \eps^5 \, |\omega|^{\frac{-2}{d-1}} \, |\omega|^{\frac{d-2}{d-1}} \, \| \nabla^2 (u^\star_d+g_d) \|^2_{L^\infty(\omega)}
  \nonumber
  \\
  &\leq
  C \, \eps \, |\omega|^{\frac{d-2}{d-1}} \, \left[ \max\left(|\omega|^{\frac{1}{d-1}}, \eps^2 \, |\omega|^{-\frac{1}{d-1}}\right) \right]^2 \, \| \nabla^2 (u^\star_d+g_d) \|^2_{L^\infty(\omega)},
  \label{diffapp3_L2_el_bend_d}
\end{align}
where the last estimate stems from the fact that $\dps \eps^2 \, |\omega|^{\frac{-1}{d-1}} \leq \max\left(|\omega|^{\frac{1}{d-1}}, \eps^2 \, |\omega|^{-\frac{1}{d-1}}\right)$. Collecting~\eqref{diffapp3_L2_el_bend_gamma} and~\eqref{diffapp3_L2_el_bend_d}, we thus deduce that
\begin{equation} \label{diffapp3_L2_el_bend}
  \| v^{\eps,1} - u^{\eps,1} \|^2_{L^2_w(\Omega)} \leq C \, \eps \, |\omega|^{\frac{d-2}{d-1}} \, \left[ \max\left(|\omega|^{\frac{1}{d-1}}, \eps^2 \, |\omega|^{-\frac{1}{d-1}}\right) \right]^2 \, \| \nabla^2 (u^\star_d+g_d) \|^2_{L^\infty(\omega)}.
\end{equation}
We next compute that $e^\eps(u^{\eps,1} - v^{\eps,1}) = E_0^\eps - E_1^\eps + \eps \, E_2^\eps$, where
\begin{align*}
  E_0^\eps(x) &= (1-\tau_\eps(x')) \, e(W^{\alpha\beta}) \left(\frac{x'}{\eps}, x_d \right) \partial_{\alpha \beta}(u^\star_d+g_d)(x'),
  \\
  E_1^\eps(x) &= \eps \, \nabla \tau_\eps(x') \otimes^s W^{\alpha\beta}\left(\frac{x'}{\eps}, x_d \right) \partial_{\alpha \beta}(u^\star_d+g_d)(x'),
  \\
  E_2^\eps(x) &= (1-\tau_\eps(x')) \, W^{\alpha\beta}\left(\frac{x'}{\eps}, x_d \right) \otimes^s \nabla \left(\partial_{\alpha \beta}(u^\star_d+g_d) \right)(x'),
\end{align*}
where we recall that $\otimes^s$ is the symmetrized tensor product: $p \otimes^s q = (p \otimes q + q \otimes p)/2$. Using the fact that $W^{\alpha\beta} \in (W^{1,\infty})^d$, that $u_d^\star+g_d$ belongs to $W^{3,\infty}(\omega)$, that $0 \leq \tau_\eps \leq 1$ and $\eps \, \|\nabla \tau_\eps\|_{L^\infty} \leq C$, we obtain that
\begin{align*}
  \| E_0^\eps \|^2_{L^2(\Omega)} & \leq C \, |\Omega \setminus \Omega_\eps| \, \sup_{1\leq \alpha, \beta \leq d-1} \| \nabla W^{\alpha \beta} \|_{L^\infty}^2 \, \| \nabla^2 (u^\star_d+g_d) \|^2_{L^\infty(\omega)}
  \\
  & \leq C \, \eps \, |\omega|^{\frac{d-2}{d-1}} \, \| \nabla^2 (u^\star_d+g_d) \|^2_{L^\infty(\omega)},
  \\
  \| E_1^\eps \|^2_{L^2(\Omega)} & \leq C \, |\Omega \setminus \Omega_\eps| \, \sup_{1\leq \alpha, \beta\leq d-1} \| W^{\alpha\beta} \|_{L^\infty}^2 \, \| \partial_{\alpha \beta}(u^\star_d+g_d) \|_{L^\infty(\omega)}^2
  \\
  & \leq C \, \eps \, |\omega|^{\frac{d-1}{d-2}} \, \| \nabla^2 (u^\star_d+g_d) \|_{L^\infty(\omega)}^2,
  \\
  \| E_2^\eps \|_{L^2(\Omega)}^2 & \leq C \, \sup_{1\leq \alpha, \beta \leq d-1} \| W^{\alpha\beta} \|_{L^\infty}^2 \, \| \nabla^3 (u^\star_d+g_d) \|_{L^2(\Omega)}^2
  \\
  & \leq C \, |\omega| \, \| \nabla^3 (u^\star_d+g_d) \|_{L^\infty(\omega)}^2.
\end{align*}
This implies that
\begin{equation}\label{eq:rel00_a}
 \| e^\eps(u^{\eps,1} - v^{\eps,1}) \|_{L^2(\Omega)}^2 \leq C \, \eps \left( \eps \, |\omega| \, \| \nabla^3 (u^\star_d+g_d) \|_{L^\infty(\omega)}^2 + |\omega|^{\frac{d-2}{d-1}} \, \| \nabla^2 (u^\star_d+g_d) \|^2_{L^\infty(\omega)} \right).
\end{equation}
Collecting~\eqref{diffapp3_L2_el_bend} and~\eqref{eq:rel00_a}, we deduce
\begin{equation} \label{diffapp3_H1_el_bend}
  \| v^{\eps,1} - u^{\eps,1} \|^2_{H^1_\eps(\Omega)} \leq C \, \eps \left( \eps \, |\omega| \, \| \nabla^3 (u^\star_d+g_d) \|_{L^\infty(\omega)}^2 + |\omega|^{\frac{d-2}{d-1}} \, \| \nabla^2 (u^\star_d+g_d) \|^2_{L^\infty(\omega)} \right),
\end{equation}
where we recall that the norm $\| \cdot \|_{H^1_\eps(\Omega)}$ is defined by~\eqref{eq:defnorm_vec}.

\medskip

\noindent
{\bf Step~2}: We now bound $\overline{v}^\eps := u^\eps - v^{\eps,1}$. Using the coercivity of $A^\eps$, we have
\begin{equation} \label{eq:tgv1_a}
c_- \| e^\eps(\overline{v}^\eps) \|_{L^2(\Omega)}^2 \leq \int_\Omega A^\eps e^\eps(\overline{v}^\eps) : e^\eps(\overline{v}^\eps) = T_1 ^\eps + T_2^\eps,
\end{equation}
where 
$$
T_1^\eps := \int_\Omega A^\eps e^\eps(u^\eps - u^{\eps,1}) : e^\eps(\overline{v}^\eps) \qquad \mbox{ and } \qquad T_2^\eps := \int_\Omega A^\eps e^\eps(u^{\eps,1} - v^{\eps,1}) : e^\eps(\overline{v}^\eps).
$$
The term $T_2^\eps$ can easily be bounded using~\eqref{eq:rel00_a}:
\begin{equation} \label{eq:tgv3_a}
| T_2^\eps | \leq C \sqrt{\eps} \left( \sqrt{\eps} \, |\omega|^{1/2} \, \| \nabla^3 (u^\star_d+g_d) \|_{L^\infty(\omega)} + |\omega|^{\frac{d-2}{2(d-1)}} \, \| \nabla^2 (u^\star_d+g_d) \|_{L^\infty(\omega)} \right) \| e^\eps(\overline{v}^\eps) \|_{L^2(\Omega)}.
\end{equation}
We now turn to bounding $T_1^\eps$. We introduce the matrix-valued field $\widetilde{\Sigma}^\star := \left( \widetilde{\Sigma}^\star_{ij} \right)_{1\leq i,j \leq d}$ defined by
$$
\widetilde{\Sigma}^\star_{\alpha\beta} = \Sigma^\star_{\alpha\beta} \ \text{ for any $1 \leq \alpha, \beta \leq d-1$} \qquad \mbox{ and } \qquad \widetilde{\Sigma}^\star_{id} = \widetilde{\Sigma}^\star_{di} = 0 \ \text{ for any $1 \leq i \leq d$},
$$
where we recall that $\Sigma^\star_{\alpha\beta}$ is the weak limit of $\sigma^\eps_{\alpha\beta}$, and is given by~\eqref{eq:Sigmaformula} in terms of $u^\star$ and $S^\star$. We also recall that the entries $\dps \{ S^\star_{\alpha\beta\gamma\delta} \}_{1 \leq \alpha,\beta,\gamma,\delta \leq d-1}$ of the tensor $S^\star$ are given by~\eqref{eq:defSstar_exp}. To simplify the notations in the proof below, we set $S^\star_{id\gamma\delta} = S^\star_{dj\gamma\delta} = 0$ for any $1 \leq i,j \leq d$ and any $1 \leq \gamma,\delta \leq d-1$, and we thus have, by definition of $\widetilde{\Sigma}^\star$ and in view of~\eqref{eq:Sigmaformula}, that
\begin{equation}\label{eq:Sigmaformula_general}
\forall 1\leq i,j \leq d, \qquad \widetilde{\Sigma}^\star_{ij}(x',x_d) = S^\star_{ij\gamma\delta}(x_d) \, \partial_{\gamma\delta}(u^\star_d + g_d)(x').
\end{equation}
We write
\begin{equation} \label{eq:tgv4_a}
T_1^\eps = \int_\Omega A^\eps e^\eps(u^\eps + g) : e^\eps(\overline{v}^\eps) - \int_\Omega A^\eps e^\eps(u^{\eps,1} + g) : e^\eps(\overline{v}^\eps) = R_1^\eps - R_2^\eps + R_3^\eps
\end{equation}
with
\begin{align*}
  R_1^\eps & := \int_\Omega A^\eps e^\eps(u^\eps + g) : e^\eps(\overline{v}^\eps) - \int_\Omega \widetilde{\Sigma}^\star : e^\eps(\overline{v}^\eps),
  \\ 
  R_2^\eps & := \int_\Omega A^\eps e^\eps(u^{\eps,1} + g) : e^\eps(\overline{v}^\eps) - \int_\Omega A^\eps \left( e(W^{\alpha\beta})_\eps - x_d \, e_\alpha \otimes e_\beta \right) \partial_{\alpha\beta}(u^\star_d+g_d) : e^\eps(\overline{v}^\eps),
  \\ 
  R_3^\eps & := \int_\Omega \Big[ \widetilde{\Sigma}^\star - A^\eps \left( e(W^{\alpha\beta})_\eps - x_d \, e_\alpha \otimes e_\beta \right) \partial_{\alpha \beta}(u_d^\star+g_d) \Big] : e^\eps(\overline{v}^\eps),
\end{align*}
where we have used the short-hand notation $\dps e(W^{\alpha \beta})_\eps(x) = e(W^{\alpha \beta}) \left(\frac{x'}{\eps},x_d \right)$. We successively estimate the three terms in the right-hand side of~\eqref{eq:tgv4_a}.

\medskip

\noindent
{\bf Step~3.} Let us first estimate $R_3^\eps$. We write
\begin{align}
  R_3^\eps
  &=
  \int_\Omega \Big[ \widetilde{\Sigma}^\star_{ij} - (e_i \otimes e_j) : A^\eps \left( e(W^{\alpha\beta})_\eps - x_d \, e_\alpha \otimes e_\beta \right) \partial_{\alpha \beta}(u_d^\star+g_d) \Big] \, e^\eps_{ij}(\overline{v}^\eps)
  \nonumber
  \\
  &=
  \int_\Omega \partial_{\alpha \beta}(u_d^\star+g_d) \Big[ S^\star_{ij\alpha\beta} - (e_i \otimes e_j) : A^\eps \left( e(W^{\alpha\beta})_\eps - x_d \, e_\alpha \otimes e_\beta \right) \Big] \, e^\eps_{ij}(\overline{v}^\eps)
  \nonumber
  \\
  &=
  \int_\Omega \partial_{\alpha \beta}(u_d^\star+g_d)(x') \, Z_{\alpha\beta}\left(\frac{x'}{\eps},x_d\right) : e^\eps(\overline{v}^\eps)(x) \, dx,
  \label{eq:theatre12}
\end{align}
where we have used~\eqref{eq:Sigmaformula_general} in the second line and where the matrix-valued function $Z_{\alpha\beta}$ of the third line is defined by
$$
\left[ Z_{\alpha\beta}(x',x_d) \right]_{ij} = S^\star_{ij\alpha\beta}(x_d) - (e_i \otimes e_j) : A(x',x_d) \left( e(W^{\alpha\beta})(x',x_d) - x_d \, e_\alpha \otimes e_\beta \right).
$$
We are aiming at using Lemma~\ref{lemma:minZ2}. We observe that $\overline{v}^\eps \in V$ and that $\partial_{\alpha \beta}(u_d^\star+g_d)$ only depends on $x'$ and belongs to $W^{2,\infty}(\omega)$. The function $Z_{\alpha\beta}$ obviously satisfies Assumption~(i) and~(vi). For any $x' \in \R^{d-1}$ and $x_d = \pm 1/2$, and for any $1 \leq i \leq d$, we compute that
\begin{multline*}
e_i^T \, Z_{\alpha\beta} \, e_d = \left[ Z_{\alpha\beta} \right]_{id} = - (e_i \otimes e_d) : A \left( e(W^{\alpha\beta}) - x_d \, e_\alpha \otimes e_\beta \right) \\ = - e_i \cdot \left[ A \left( e(W^{\alpha\beta}) - x_d \, e_\alpha \otimes e_\beta \right) e_d \right]
\end{multline*}
which vanishes in view of the boundary condition in~\eqref{el-prcor2_edp}. Assumption~(v) thus holds. To verify Assumption~(iv), we first note that the second term in $Z_{\alpha\beta}$ is divergence free, and then compute that, for any $1 \leq i \leq d$,
$$
\left[ \div Z_{\alpha\beta} \right]_i = \partial_j \left[ Z_{\alpha\beta} \right]_{ij} = \partial_j S^\star_{ij\alpha\beta} = \partial_d S^\star_{id\alpha\beta} = 0,
$$
where we have used that $S^\star$ only depends on $x_d$. We have thus shown Assumption~(iv). We are now left with verifying Assumptions~(ii) and~(iii). To that aim, we compute
$$
\int_Y \left[ Z_{\alpha\beta}(\cdot,x_d) \right]_{ij} = S^\star_{ij\alpha\beta}(x_d) - (e_i \otimes e_j) : \int_Y A(\cdot,x_d) \left( e(W^{\alpha\beta})(\cdot,x_d) - x_d \, e_\alpha \otimes e_\beta \right).
$$
If $1 \leq i,j \leq d-1$, then this quantity vanishes as a direct consequence of~\eqref{eq:defSstar_exp}, and we thus have that $\dps \int_{\mathcal{Y}} \left[ Z_{\alpha\beta} \right]_{\gamma\delta} = \int_{\mathcal{Y}} x_d \, \left[ Z_{\alpha\beta} \right]_{\gamma\delta} = 0$ for any $1 \leq \gamma,\delta \leq d-1$. If $j=d$ (and likewise if $i=d$ since $Z_{\alpha\beta}$ is a symmetric matrix), then the first term in the definition of $\left[ Z_{\alpha\beta} \right]_{ij}$ vanishes (because $S^\star_{id\alpha\beta} = 0$ by definition) and we thus have, for $r=0,1$, that
$$
\int_{\mathcal{Y}} x_d^r \, \left[ Z_{\alpha\beta} \right]_{id} = -\int_{\mathcal{Y}} x_d^r \, A \big( e(W^{\alpha \beta}) - x_d \, e_\alpha \otimes e_\beta \big) : (e_d \otimes e_i) = 0,
$$
where the last equality comes from choosing $\dps v(x) = \left( \frac{x_d^{r+1}}{r+1} - C_r \right) e_i$ as test function (for any $1 \leq i \leq d$) in~\eqref{el-prcor2}, where the constant $C_r$ is adjusted such that the mean of $v$ over $(-1/2,1/2)$ vanishes (we obviously have $C_0 = 0$). This implies that Assumptions~(ii) and~(iii) indeed hold. We are thus in position to use Lemma~\ref{lemma:minZ2}, and deduce from~\eqref{eq:theatre12} that
$$
| R_3^\eps | \leq C \, \eps \, |\omega|^{1/2} \, \| \nabla^3 (u^\star_d+g_d) \|_{W^{1,\infty}(\omega)} \, \| e^\eps(\overline{v}^\eps) \|_{L^2(\Omega)}.
$$

\medskip

\noindent
{\bf Step~4.} Inserting in the definition of $R_2^\eps$ the expression of $u^{\eps,1}$, and using the short-hand notation $\dps W^{\alpha \beta}_\eps(x) = W^{\alpha \beta}\left(\frac{x'}{\eps},x_d \right)$, we find
\begin{align*}
  R_2^\eps
  &=
  \int_\Omega A^\eps \left[ e^\eps(u^{\eps,1}+g) - \left( e(W^{\alpha\beta} )_\eps - x_d \, e_\alpha \otimes e_\beta \right) \partial_{\alpha\beta}(u^\star_d+g_d) \right] : e^\eps(\overline{v}^\eps)
  \\
  &=
  \eps \int_\Omega A^\eps \left[ W^{\alpha\beta}_\eps \otimes^s \nabla\left( \partial_{\alpha \beta}(u_d^\star + g_d)\right) \right] : e^\eps(\overline{v}^\eps),
\end{align*}
and thus
$$
| R_2^\eps | \leq C \, \eps \, \| \nabla^3 (u_d^\star+g_d) \|_{L^2(\Omega)} \, \| e^\eps(\overline{v}^\eps) \|_{L^2(\Omega)} \leq C \, \eps \, |\omega|^{1/2} \, \| \nabla^3 (u_d^\star+g_d) \|_{L^\infty(\omega)} \, \| e^\eps(\overline{v}^\eps) \|_{L^2(\Omega)}.
$$

\medskip

\noindent
{\bf Step~5.} To bound $R_1^\eps$, we are going to use the same notations as in Lemmas~\ref{lem:lemma1}, \ref{lem:lem00} and~\ref{lem:lem2}. Since the last row and the last column of $\widetilde{\Sigma}^\star$ vanish and since the other entries of $\widetilde{\Sigma}^\star$ are equal to those of $\Sigma^\star$, we can write
\begin{align*}
  R_1^\eps
  &=
  \int_\Omega A^\eps e^\eps(u^\eps + g) : e^\eps(\overline{v}^\eps) - \int_\Omega \widetilde{\Sigma}^\star : e^\eps(\overline{v}^\eps)
  \\
  &=
  \int_\Omega \sigma^\eps : e^\eps(\overline{v}^\eps) - \int_\Omega \widetilde{\Sigma}^\star_{\alpha\beta} \, e_{\alpha\beta}(\overline{v}^\eps)
  \\
  &=
  \int_\Omega \Sigma^\eps : e(\overline{v}^\eps) - \int_\Omega \Sigma^\star_{\alpha\beta} \, e_{\alpha\beta}(\overline{v}^\eps)
  \\
  &=
  \int_\Omega \Sigma^\star : e(\overline{v}^\eps) - \int_\Omega \Sigma^\star_{\alpha\beta} \, e_{\alpha\beta}(\overline{v}^\eps),
\end{align*}
where, in the last line, we have used Lemma~\ref{lem:lem2} (recall that $\overline{v}^\eps$ belongs to $V$). We deduce that
$$
R_1^\eps = \int_\Omega \Sigma^\star_{\alpha d} \, e_{\alpha d}(\overline{v}^\eps) + \int_\Omega \Sigma^\star_{dd} \, e_{dd}(\overline{v}^\eps) = \eps \int_\Omega \Sigma^\star_{\alpha d} \, e^\eps_{\alpha d}(\overline{v}^\eps) + \eps^2 \int_\Omega \Sigma^\star_{dd} \, e^\eps_{dd}(\overline{v}^\eps),
$$
and thus, using~\eqref{eq:borne_L2_Sigma_star_alphad} and~\eqref{eq:borne_L2_Sigma_star_dd},
\begin{align*}
| R_1^\eps |
& \leq
C \, \eps \left( \eps \, \| \Sigma^\star_{dd} \|_{L^2(\Omega)} + \sum_{\alpha=1}^{d-1} \| \Sigma^\star_{\alpha d} \|_{L^2(\Omega)} \right) \| e^\eps(\overline{v}^\eps) \|_{L^2(\Omega)}
\\
& \leq
C \, \eps \left( \eps \, {\cal N}^{\rm bend}(f,g,h_\pm) + \overline{{\cal N}}^{\rm bend}_\Omega(f,g,h_\pm) \right) \| e^\eps(\overline{v}^\eps) \|_{L^2(\Omega)}.
\end{align*}

\medskip

\noindent
{\bf Step~6.} Collecting~\eqref{eq:tgv1_a}, \eqref{eq:tgv3_a}, \eqref{eq:tgv4_a} and the bounds shown in Steps~3--5, we obtain that
\begin{multline*}
\| e^\eps(\overline{v}^\eps) \|_{L^2(\Omega)} \leq C \, \sqrt{\eps} \, \Big( |\omega|^{\frac{d-2}{2(d-1)}} \, \| \nabla^2 (u_d^\star+g_d) \|_{L^\infty(\omega)} + \sqrt{\eps} \, |\omega|^{1/2} \, \| \nabla^3 (u_d^\star+g_d) \|_{W^{1,\infty}(\omega)} \\ + \sqrt{\eps} \ \overline{{\cal N}}^{\rm bend}_\Omega(f,g,h_\pm) + \eps^{3/2} \, {\cal N}^{\rm bend}(f,g,h_\pm) \Big).
\end{multline*}
Since $\overline{v}^\eps$ belongs to $V$, we can use the estimate~\eqref{eq:poinK_corro_vec}, and we obtain from the above bound that
\begin{multline*}
\| \overline{v}^\eps \|_{H^1_\eps(\Omega)} \leq C \, \sqrt{\eps} \, \Big( |\omega|^{\frac{d-2}{2(d-1)}} \, \| \nabla^2 (u_d^\star+g_d) \|_{L^\infty(\omega)} + \sqrt{\eps} \, |\omega|^{1/2} \, \| \nabla^3 (u_d^\star+g_d) \|_{W^{1,\infty}(\omega)} \\ + \sqrt{\eps} \ \overline{{\cal N}}^{\rm bend}_\Omega(f,g,h_\pm) + \eps^{3/2} \, {\cal N}^{\rm bend}(f,g,h_\pm) \Big).
\end{multline*}
Collecting this bound with~\eqref{diffapp3_H1_el_bend}, we deduce~\eqref{eq:theatre13}, which concludes the proof of Theorem~\ref{thconvforte3_a}.
\end{proof}





\appendix

\section{Rescaling of the problems~\eqref{pb:diff1} and~\eqref{pbelas_eps}} \label{app:scaling}

We briefly show here the equivalence between~\eqref{pb:diff1}, \eqref{pb:diff2} and~\eqref{var:diff} on the one hand, and~\eqref{pbelas_eps}, \eqref{pbelas} and~\eqref{formvarelast} on the other hand.

\subsection{The diffusion case} \label{app:scaling_scalaire}

We begin with the diffusion case, and the problem~\eqref{pb:diff1} posed on the thin plate $\Omega^\eps$. Its variational formulation is to find $\widetilde{u}^\eps \in V^\eps$ such that, for any $\widetilde{v} \in V^\eps$,
$$
\int_{\Omega^\eps} \mathcal{A}^\eps \nabla \widetilde{u}^\eps \cdot \nabla \widetilde{v} = \int_{\Omega^\eps} \widetilde{f}^\eps \, \widetilde{v} - \int_{\Omega^\eps} \mathcal{A}^\eps \nabla g^\eps \cdot \nabla \widetilde{v} + \eps \int_\omega h^\eps_+ \, \widetilde{v}\left(\cdot,\frac{\eps}{2}\right) + \eps \int_\omega h^\eps_- \, \widetilde{v}\left(\cdot,-\frac{\eps}{2}\right).
$$
After a change of variables, this equation reads as
\begin{multline*}
\eps \int_\Omega \mathcal{A}^\eps(x',\eps \, x_d) \, \nabla \widetilde{u}^\eps(x',\eps \, x_d) \cdot \nabla \widetilde{v}(x',\eps \, x_d) = \eps \int_\Omega \widetilde{f}^\eps(x',\eps \, x_d) \, \widetilde{v}(x',\eps \, x_d) \\ - \eps \int_\Omega \mathcal{A}^\eps(x',\eps \, x_d) \, \nabla g^\eps(x') \cdot \nabla \widetilde{v}(x',\eps \, x_d) + \eps \int_\omega h^\eps_+ \, \widetilde{v}\left(\cdot,\frac{\eps}{2}\right) + \eps \int_\omega h^\eps_- \, \widetilde{v}\left(\cdot,-\frac{\eps}{2}\right),
\end{multline*}
where we have explicitely used in the notation that $g^\eps$ is independent of the last variable. Using the rescaling~\eqref{eq:scaling_u_scal} for the load and the solution, and using the same rescaling on the test function as for $\widetilde{u}^\eps$, we can recast the above variational formulation as finding $u^\eps \in V$ such that, for any $v \in V$,
$$
\int_\Omega A^\eps \nabla^\eps u^\eps \cdot \nabla^\eps v = \int_\Omega f^\eps \, v - \int_\Omega A^\eps \nabla^\eps g^\eps \cdot \nabla^\eps v + \int_\omega h^\eps_+ \, v\left(\cdot,\frac{1}{2}\right) + \int_\omega h^\eps_- \, v\left(\cdot,-\frac{1}{2}\right),
$$
which is exactly~\eqref{var:diff}. We next use the following integration by part result, to deduce~\eqref{pb:diff2} from~\eqref{var:diff}: for any vector-valued function $z$ and any scalar-valued function $v$ (with $z$ and $v$ sufficiently regular),
\begin{equation} \label{eq:IPP_scalaire}
\int_\Omega z \cdot \nabla^\eps v = - \int_\Omega v \, \div^\eps z + \int_{\partial \Omega} z_\alpha \, v \, n_\alpha + \frac{1}{\eps} \int_{\partial \Omega} z_d \, v \, n_d. 
\end{equation}

\subsection{The elasticity case} \label{app:scaling_vectoriel}

We next turn to the elasticity case. The variational formulation of the problem~\eqref{pbelas_eps} posed on the thin plate $\Omega^\eps$ is to find $\widetilde{u}^\eps \in V^\eps$ such that, for any $\widetilde{v} \in V^\eps$,
$$
\int_{\Omega^\eps} \mathcal{A}^\eps e(\widetilde{u}^\eps) : e(\widetilde{v}) = \int_{\Omega^\eps} \widetilde{f}^\eps \cdot \widetilde{v} - \int_{\Omega^\eps} \mathcal{A}^\eps e(\widetilde{g}^\eps) : e(\widetilde{v}) + \eps \int_\omega \widetilde{h}^\eps_+ \cdot \widetilde{v}\left(\cdot,\frac{\eps}{2}\right) + \eps \int_\omega \widetilde{h}^\eps_- \cdot \widetilde{v}\left(\cdot,-\frac{\eps}{2}\right).
$$
After a change of variables, this equation reads as
\begin{multline*}
\eps \int_\Omega \mathcal{A}^\eps(x',\eps \, x_d) \, e(\widetilde{u}^\eps)(x',\eps \, x_d) : e(\widetilde{v})(x',\eps \, x_d) = \eps \int_\Omega \widetilde{f}^\eps(x',\eps \, x_d) \cdot \widetilde{v}(x',\eps \, x_d) \\ - \eps \int_\Omega \mathcal{A}^\eps(x',\eps \, x_d) \, e(\widetilde{g}^\eps)(x',\eps \, x_d) : e(\widetilde{v})(x',\eps \, x_d) + \eps \int_\omega \widetilde{h}^\eps_+ \cdot \widetilde{v}\left(\cdot,\frac{\eps}{2}\right) + \eps \int_\omega \widetilde{h}^\eps_- \cdot \widetilde{v}\left(\cdot,-\frac{\eps}{2}\right).
\end{multline*}
Using the rescaling~\eqref{eq:scaling_u_el} for the load and the solution, and using the same rescaling on the test function as for $\widetilde{u}^\eps$ (and observing that $f^\eps(x) \cdot v(x) = \widetilde{f}^\eps(x',\eps \, x_d) \cdot \widetilde{v}(x',\eps \, x_d)$ and $h^\eps(x) \cdot v(x) = \widetilde{h}^\eps(x',\eps \, x_d) \cdot \widetilde{v}(x',\eps \, x_d)$), we can recast the above variational formulation as finding $u^\eps \in V$ such that, for any $v \in V$,
$$
\int_\Omega A^\eps e^\eps(u^\eps) : e^\eps(v) = \int_\Omega f^\eps \cdot v - \int_\Omega A^\eps e^\eps(g^\eps) : e^\eps(v) + \int_\omega h^\eps_+ \cdot v\left(\cdot,\frac{1}{2}\right) + \int_\omega h^\eps_- \cdot v\left(\cdot,-\frac{1}{2}\right),
$$
which is exactly~\eqref{formvarelast}. We next use the following integration by part result, to deduce~\eqref{pbelas} from~\eqref{formvarelast}: for any function $z$ such that $z(x)$ is a symmetric matrix and for any vector-valued function $v$ (with $z$ and $v$ sufficiently regular),
\begin{equation} \label{eq:IPP_elas}
\int_\Omega e^\eps(v) : z = - \int_\Omega v \cdot \div^\eps z + \int_{\partial \Omega} v_\gamma \, n_\delta \, z_{\gamma \delta} + \frac{1}{\eps} \int_{\partial \Omega} (v_\gamma \, n_d + v_d \, n_\gamma) \, z_{d\gamma} + \frac{1}{\eps^2} \int_{\partial \Omega} v_d \, n_d \, z_{dd}.
\end{equation}

\section{$H_{\rm div}$ space} \label{app:Hdiv}

We recall here some results about the $H_{\rm div}$ space. Let $\Omega \subset \R^d$ be a bounded regular domain of $\R^d$. We define the space
$$
H_{\rm div}(\Omega):= \left\{ v \in (L^2(\Omega))^d, \ \ \mbox{\rm div} \, v \in L^2(\Omega) \right\},
$$
which is a Hilbert space for the scalar product
$$
\forall v,w \in H_{\rm div}(\Omega), \quad \langle v,w \rangle = \int_\Omega v \cdot w + \int_\Omega (\mbox{\rm div} \, v) \ (\mbox{\rm div} \, w).
$$
The space $\left(\mathcal{C}^\infty(\overline{\Omega})\right)^d$ is dense in $H_{\rm div}(\Omega)$. Let us consider the normal trace application
$$
\gamma_n : \left\{ 
\begin{array}{ccc}
\left(\mathcal{C}^\infty(\overline{\Omega})\right)^d & \to & \mathcal{C}^0(\partial \Omega)\\
v & \mapsto & (v \cdot n)|_{\partial \Omega}
\end{array}
\right. ,
$$
where $n$ denotes the unit exterior normal vector to $\partial \Omega$. The application $\gamma_n$ can be uniquely extended as a continuous application from $H_{\rm div}(\Omega)$ to $H^{-1/2}(\partial \Omega)$, and the following Stokes formula holds: 
$$
\forall v \in H_{\rm div}(\Omega), \quad \forall w\in H^1(\Omega), \quad \int_\Omega v \cdot \nabla w + \int_\Omega w \; \mbox{\rm div} \, v = \langle \gamma_n(v), \gamma_0(w) \rangle_{H^{-1/2}(\partial \Omega), H^{1/2}(\partial \Omega)}, 
$$
where $\gamma_0$ denotes the trace application from $H^1(\Omega)$ to $H^{1/2}(\partial \Omega)$.



\section{Korn inequalities} \label{app:korn}

For the sake of completeness, we provide here a proof of the Korn inequality we have stated in~\eqref{eq:korn}. We start by recalling a well-known result:

\begin{lemma}[Korn's inequality in $H^1$, see~\cite{ciarlet1988mathematical}] \label{simple2}
Let $\Omega \subset \R^d$ be a bounded regular domain. Then, there exists a constant $C(\Omega)>0$ such that, for any $u \in (H^1(\Omega))^d$, we have
$$
\| u \|_{H^1(\Omega)}^2 \leq C(\Omega) \left( \|u\|^2_{L^2(\Omega)} + \|e(u)\|^2_{L^2(\Omega)} \right).
$$
\end{lemma}

In the specific case of functions which vanish on the boundary of the domain, we have the following result.

\begin{lemma}[Korn's inequality in $H_0^1$, see~\cite{ciarlet1988mathematical}] \label{Korn2}
Let $\Omega \subset \R^d$ be a regular domain. For any $u \in (H_0^1(\Omega))^d$, we have
$$
\|\nabla u\|_{L^2(\Omega)} \leq \sqrt{2} \, \| e(u) \|_{L^2(\Omega)}.
$$
\end{lemma} 

In this work, we need a slightly modified version of Lemma~\ref{Korn2}, since we work in $V$ and not in $(H_0^1(\Omega))^d$.

\begin{lemma}\label{lem:Korn3}
Let $V$ be given by~\eqref{def:V2}. Then, there exists a constant $C$ such that
\begin{equation}\label{Korn}
\forall u \in V, \quad \| u \|_{H^1(\Omega)} \leq C \, \| e(u) \|_{L^2(\Omega)}.
\end{equation}
\end{lemma}

\begin{proof}
We argue by contradiction and assume that~\eqref{Korn} does not hold. Then, for any $n \in \N^\star$, there exists $u_n \in V$ such that 
\begin{equation}\label{u<e(u)}
\|u_n\|_{H^1(\Omega)} = 1, \qquad \| e(u_n) \|_{L^2(\Omega)} \leq \frac{1}{n}.
\end{equation}
Since $u_n$ is bounded in $(H^1(\Omega))^d$, there exists $u \in (H^1(\Omega))^d$ such that, up to the extraction of a subsequence, $\dps u_n \mathop{\rightharpoonup}_{n\to +\infty} u$ weakly in $(H^1(\Omega))^d$. Thus, when $n\to+\infty$, we have that $u_n$ strongly converges to $u$ in $(L^2(\Omega))^d$ and $e(u_n)$ weakly converges to $e(u)$ in $(L^2(\Omega))^{d\times d}$. The estimate~\eqref{u<e(u)} yields that $\dps \| e(u_n) \|_{L^2(\Omega)} \mathop{\longrightarrow}_{n\to +\infty} 0$, which implies that $e(u)=0$ and thus that $u$ is a rigid displacement. In addition, $u\in V$, and therefore $u=0$. Using Lemma~\ref{simple2}, we also deduce from~\eqref{u<e(u)} that
$$
\| u_n \|^2_{L^2(\Omega)} \geq \frac{1}{C(\Omega)} - \frac{1}{n^2}. 
$$
Passing to the limit $n \to \infty$ and using the fact that $u_n$ strongly converges to $u$ in $(L^2(\Omega))^d$, we obtain that $\dps \| u \|^2_{L^2(\Omega)} \geq 1/C(\Omega)$, which provides a contradiction with the fact that $u=0$.
\end{proof}


\section{Proofs of the homogenization limits}\label{app:diff}

In this section, for the sake of completeness, we provide a proof of Theorem~\ref{limitdiff} (resp. Theorem~\ref{limitel}), which was originally shown in~\cite[Theorem~8.1]{caillerieDiffusion} (resp.~\cite[Theorem~6.2]{caillerieElasticite}). 
We use below the following well-known result, where we recall (see~\eqref{eq:def_Y_Y}) that $Y = (0,1)^{d-1}$.

\begin{lemma} \label{limmoyenne}
Let $B$ be a function in $\dps L^2_{\rm loc} \left( \mathbb{R}^{d-1} \times \left( -\frac{1}{2}, \frac{1}{2} \right) \right)$ such that, for any $\dps z \in \left( -\frac{1}{2}, \frac{1}{2} \right)$, the function $B(\cdot,z)$ is $Y$-periodic. Then, for any $\dps z \in \left( -\frac{1}{2}, \frac{1}{2} \right)$ and any bounded $\omega \subset \mathbb{R}^{d-1}$,
$$
B\left(\frac{\cdot}{\eps},z \right) \underset{\eps \rightarrow 0}{\rightharpoonup} \frac{1}{|Y|} \int_Y B(y',z) \, dy' \quad \text{ weakly in $L^2(\omega)$}.
$$
\end{lemma}

\subsection{Proof of Theorem~\ref{limitdiff}} \label{app:diff_preuve1}

To identify the homogenized problem, we use the oscillating test function method. Let $\phi \in \mathcal{D}(\omega)$ and $\dps v(x) := \phi(x') + \eps \, w^\alpha\left(\frac{x'}{\eps},x_d \right) \partial_\alpha \phi(x')$ for any $x = (x',x_d) \in \Omega$. By definition, $v \in V$ and it is thus an admissible test function in~\eqref{var:diff}. We note that $\dps \nabla^\eps v(x) = \nabla \phi(x') + \nabla w^\alpha\left(\frac{x'}{\eps},x_d \right) \partial_\alpha \phi(x') + \eps \, w^\alpha\left(\frac{x'}{\eps},x_d \right) \nabla \partial_\alpha \phi(x')$. Using this function $v$ as test function in~\eqref{var:diff}, we get
\begin{equation} \label{formvarfctest}
  c^\eps(u^\eps,\phi) + r^\eps(u^\eps,\phi) = d^\eps(\phi) + s^\eps(\phi),
\end{equation}
where
\begin{align*}
  c^\eps(u^\eps,\phi) &:= \int_\Omega A^\eps(x) \nabla^\eps u^\eps(x) \cdot \left( \nabla w^\alpha\left(\frac{x'}{\eps},x_d \right) + e_\alpha \right) \partial_\alpha \phi(x'),
  \\
  r^\eps(u^\eps,\phi) &:= \eps \int_\Omega w^\alpha\left(\frac{x'}{\eps},x_d \right) A^\eps(x) \nabla^\eps u^\eps(x) \cdot \nabla \partial_\alpha \phi(x'),
  \\
  d^\eps(\phi) &:= b^\eps(\phi) - \int_\Omega A^\eps(x) \nabla^\eps g(x') \cdot \nabla w^\alpha\left(\frac{x'}{\eps},x_d \right) \partial_\alpha \phi(x'),
\end{align*}
and 
\begin{multline*}
s^\eps(\phi) := \eps \int_\Omega f(x) \, w^\alpha\left(\frac{x'}{\eps},x_d \right) \partial_\alpha \phi(x') - \eps \int_\Omega w^\alpha\left(\frac{x'}{\eps},x_d \right) A^\eps(x) \nabla^\eps g(x') \cdot \nabla \partial_\alpha \phi(x') \\ + \eps \int_{\Gamma_\pm} h_\pm(x') \, w^\alpha\left(\frac{x'}{\eps},\pm \frac{1}{2} \right) \partial_\alpha \phi(x'),
\end{multline*}
where $\dps \Gamma_\pm = \omega \times \left\{ \pm \frac{1}{2} \right\}$. The following limits are immediate in view of~\eqref{bornesigma} (which, together with~\eqref{eq:inde_eps}, implies that $\| \nabla^\eps u^\eps \|_{L^2(\Omega)}$ is bounded) and of the fact that $\nabla^\eps g = \nabla g$ (because $g$ does not depend on $x_d$):
$$
r^\eps(u^\eps,\phi) \underset{\eps \rightarrow 0}{\rightarrow} 0 \qquad \text{and} \qquad s^\eps(\phi) \underset{\eps \rightarrow 0}{\rightarrow} 0.
$$
We now identify the limit of $c^\eps(u^\eps,\phi)$. Since $A$ is symmetric, we have
$$
c^\eps(u^\eps,\phi) = \int_\Omega \nabla^\eps u^\eps(x) \cdot A^\eps(x) \left( \nabla w^\alpha\left(\frac{x'}{\eps},x_d \right) + e_\alpha \right) \partial_\alpha \phi(x').
$$
Introducing the vector-valued function $\dps Z_\alpha = A \left( \nabla w^\alpha + e_\alpha \right)$, we write, using an integration by parts (see~\eqref{eq:IPP_scalaire}) and the fact that $\phi \in \mathcal{D}(\omega)$,
\begin{align*}
  c^\eps(u^\eps,\phi) &= \int_\Omega \nabla^\eps u^\eps(x) \cdot Z_\alpha\left(\frac{x'}{\eps},x_d \right) \partial_\alpha \phi(x')
  \\
  &= -\int_\Omega u^\eps(x) \, \div^\eps \left[ Z_\alpha\left(\frac{x'}{\eps},x_d \right) \partial_\alpha \phi(x') \right] \pm \frac{1}{\eps} \int_{\Gamma_\pm} u^\eps\left(x',\pm \demi\right) \partial_\alpha \phi(x') \, Z_\alpha\left(\frac{x'}{\eps},\pm \demi \right) \cdot e_d. 
\end{align*}
Using the second line of the corrector equation~\eqref{prcor2}, we see that the above integral over $\Gamma_\pm$ vanishes. Since $Z_\alpha$ is divergence-free and $\partial_\alpha \phi$ is independent of $x_d$, we compute that
$$
\div^\eps \left[ Z_\alpha\left(\frac{x'}{\eps},x_d \right) \partial_\alpha \phi(x') \right]
=
\left[ Z_\alpha\left(\frac{x'}{\eps},x_d \right) \cdot e_\beta \right] \partial_{\alpha \beta} \phi(x')
=
Z_\alpha\left(\frac{x'}{\eps},x_d \right) \cdot \partial_\alpha \nabla \phi(x').
$$
We thus deduce that
$$
c^\eps(u^\eps,\phi) = - \int_\Omega u^\eps(x) \left[ A\left(\frac{x'}{\eps},x_d \right) \left( \nabla w^\alpha\left(\frac{x'}{\eps},x_d \right) + e_\alpha \right) \right] \cdot \partial_\alpha \nabla \phi(x').
$$
Using the Rellich theorem, it holds that, up to the extraction of a subsequence, $(u^\eps)_{\eps >0}$ converges strongly in $L^2(\Omega)$ to $u^\star$. Besides, since $\| \nabla^\eps u^\eps \|_{L^2(\Omega)}$ is bounded, we have that $\| \eps^{-1} \partial_d u^\eps \|_{L^2(\Omega)}$ is bounded, and therefore $\partial_d u^\star=0$. Since $u^\star\in V$, we thus get that $u^\star \in H^1_0(\omega)$.

Using Lemma~\ref{limmoyenne} and the strong convergence of $(u^\eps)_{\eps>0}$ to $u^\star$ in $L^2(\Omega)$, we obtain
$$
c^\eps(u^\eps,\phi) \underset{\eps \rightarrow 0}{\rightarrow} c^\star(u^\star,\phi) := -\int_\Omega u^\star \left[ \int_Y A(\cdot,x_d) \left( \nabla w^\alpha\left(\cdot, x_d \right) + e_\alpha \right) \right] \cdot \partial_\alpha \nabla \phi(x'). 
$$
Using that $u^\star$ and $\phi$ do not depend on $x_d$, we infer that
\begin{align*}
  c^\star(u^\star,\phi) &= -\int_\omega u^\star \left[ \int_{\mathcal{Y}} A \left( \nabla w^\alpha + e_\alpha \right) \right] \cdot \partial_\alpha \nabla \phi \qquad \text{[$u^\star$ and $\phi$ ind. of $x_d$]}
  \\
  &= \int_\omega \left[ \int_{\mathcal{Y}} A \left( \nabla w^\alpha + e_\alpha \right) \right] \partial_\alpha u^\star \cdot \nabla \phi \qquad \text{[$\phi$ belongs to $\mathcal{D}(\omega)$]}
  \\
  &= \int_\omega A^\star \nabla' u^\star \cdot \nabla' \phi.
\end{align*}
The last equality stems from the fact that
$$
A^\star_{\alpha \beta} = \int_{\mathcal{Y}} A \left( \nabla w^\alpha + e_\alpha \right) \cdot \left( \nabla w^\beta + e_\beta \right) = \int_{\mathcal{Y}} A \left( \nabla w^\alpha + e_\alpha \right) \cdot e_\beta,
$$
where the last equality comes from using the test function $v=w^\beta$ in~\eqref{prcor1}. 

We next turn to the term $d^\eps(\phi)$. Using again Lemma~\ref{limmoyenne} and similar arguments as above, it holds that
\begin{align*}
  \lim_{\eps \to 0} d^\eps(\phi)
  &= \int_\omega (\m(f)+h_\pm) \, \phi - \int_\omega \left[ \int_{\mathcal{Y}} A \left( \nabla w^\alpha + e_\alpha \right) \right] \partial_\alpha g \cdot \nabla \phi
  \\
  &= \int_\omega (\m(f)+h_\pm) \, \phi - \int_\omega A^\star \nabla' g \cdot \nabla' \phi.
\end{align*}
Passing to the limit $\eps \to 0$ in~\eqref{formvarfctest}, we thus obtain $\dps \int_\omega A^\star \nabla' u^\star \cdot \nabla' \phi = \int_\omega (\m(f)+h_\pm) \, \phi - \int_\omega A^\star \nabla' g \cdot \nabla' \phi$, which means that $u^\star \in H^1_0(\omega)$ is a solution to~\eqref{eq:fv_star}.

By definition, $A^\star$ is obviously symmetric. The coercivity of $A^\star$ can be obtained by standard arguments: for any $\xi = (\xi_\alpha)_{1 \leq \alpha \leq d-1} \in \R^{d-1}$, we compute
$$
\xi^T A^\star \xi
=
\int_{\mathcal{Y}} A (\nabla w^\xi + \xi) \cdot (\nabla w^\xi + \xi)
\geq
c_- \int_{\mathcal{Y}} \left| \nabla w^\xi + \xi \right|^2
\geq
c_- \int_{\mathcal{Y}} \left| \nabla' w^\xi + \xi \right|^2
\geq
c_- \left| \int_{\mathcal{Y}} (\nabla' w^\xi + \xi) \right|^2,
$$
where $w^\xi = \xi_\alpha \, w^\alpha$. Using next the periodicity of $w^\xi$ with respect to its first $d-1$ variables, we obtain that $\xi^T A^\star \xi \geq c_- \, \xi^T \xi$ and thus the coercivity of $A^\star$. The homogenized problem~\eqref{eq:fv_star} is thus well-posed, and hence $u^\star$ is uniquely defined. The whole sequence $u^\eps$ (and not only a subsequence) therefore converges to $u^\star$. This concludes the proof of Theorem~\ref{limitdiff}.

\subsection{Proof of Theorem~\ref{limitel} in the membrane case} \label{app:diff_preuve2_memb}

Using the method of the oscillating test function, we are going to prove Theorem~\ref{limitel} in the membrane case. The proof falls in seven steps. The decoupling assumption~\eqref{hyp:symA} on $A$, and the fact that we are in the membrane case~\eqref{eq:ass_membrane}, is only used in Step 3b. All the other steps are also valid in the bending case, and actually do not use the decoupling assumption~\eqref{hyp:symA}. They are thus written without assuming that~\eqref{hyp:symA} and~\eqref{eq:ass_membrane} hold. The bending case is considered in Appendix~\ref{app:diff_preuve2_bend}, using a different strategy of proof.

\medskip

\noindent
{\bf Step~1: structure of $u^\star$.} We begin by showing that $u^\star$ belongs to $\VKL$. We recall (see~\eqref{est:sig}, \eqref{eq:inde_eps_vec} and~\eqref{eq:assump_g}) that $\| e^\eps(u^\eps) \|_{L^2(\Omega)}$ is bounded. This implies that $\| \partial_d u^\eps_d \|_{L^2(\Omega)} \leq C \eps^2$. The function $u^\eps$ converges to $u^\star$ weakly in $\left( H^1(\Omega) \right)^d$, therefore $\partial_d u^\star_d = 0$.

From the bound on $\| e^\eps(u^\eps) \|_{L^2(\Omega)}$, we also get that $\| \partial_d u^\eps_\alpha + \partial_\alpha u^\eps_d \|_{L^2(\Omega)} \leq C \eps$. Using again that $u^\eps$ converges to $u^\star$ weakly in $\left(H^1(\Omega)\right)^d$, we deduce that $\partial_d u^\star_\alpha + \partial_\alpha u^\star_d = 0$. Since $\partial_d u^\star_d = 0$, there exists some $\widetilde{u}^\star$ in $\left(H^1(\omega)\right)^d$ (which is independent of $x_d$) such that $u^\star_\alpha = \widetilde{u}_\alpha^\star - x_d \, \partial_\alpha \widetilde{u}^\star_d$ and $u^\star_d = \widetilde{u}^\star_d$.
 
The function $u^\eps$ belongs to $V$ and converges to $u^\star$ weakly in $\left(H^1(\Omega)\right)^d$. We thus obtain that $\m(u^\eps) \in (H^1_0(\omega))^d$ and $\m(u^\eps) \underset{\eps \rightarrow 0}{\rightharpoonup} \m(u^\star) = \widetilde{u}^\star$ weakly in $(H^1(\omega))^d$. We hence have $\widetilde{u}^\star \in (H^1_0(\omega))^d$. Since $x_d \, \nabla \widetilde{u}^\star_d = \widetilde{u}^\star - u^\star$, we have $\nabla \widetilde{u}^\star_d = 0$ on $\partial \omega$ and $\nabla \widetilde{u}^\star_d \in (H^1_0(\omega))^d$. We have shown previously that $\widetilde{u}^\star_d \in H^1_0(\omega)$, and thus $u^\star_d \in H^2_0(\omega)$. We hence have that $u^\star \in \VKL$. As for any element of $\VKL$, we can associate to $u^\star$ a function $\widehat{u}^\star \in \left(H^1_0(\omega)\right)^{d-1} \times H^2_0(\omega)$, and this element turns out to be $\widehat{u}^\star = \widetilde{u}^\star$.

\medskip

\noindent
{\bf Step~2: oscillating test function.} To identify the homogenized limit of~\eqref{formvarelast}, we make use of the oscillating test function method. Let $\widehat{\phi} \in (\mathcal{D}(\omega))^d$. By defining $\phi$ as $\phi_\alpha = \widehat{\phi}_\alpha - x_d \, \partial_\alpha \widehat{\phi}_d$ and $\phi_d = \widehat{\phi}_d$, we get that $\phi \in \VKL$. Let us define, for any $1\leq \gamma \leq d-1$ and any $x=(x',x_d) \in \Omega$,
\begin{align*}
  v_\gamma(x) &:= \phi_\gamma(x) + \eps \left[ w^{\alpha \beta}_\gamma \left(\frac{x'}{\eps},x_d \right) e_{\alpha \beta} (\widehat{\phi})(x') + W^{\alpha \beta}_\gamma \left(\frac{x'}{\eps},x_d \right) \partial_{\alpha \beta} \widehat{\phi}_d(x') \right],
  \\
v_d(x) &:= \phi_d(x) + \eps^2 \left[ w^{\alpha \beta}_d \left(\frac{x'}{\eps},x_d \right) e_{\alpha \beta} (\widehat{\phi})(x') + W^{\alpha \beta}_d \left(\frac{x'}{\eps},x_d \right) \partial_{\alpha \beta} \widehat{\phi}_d(x') \right].
\end{align*}
By definition, $v$ belongs to $V$ and is thus an admissible test function in~\eqref{formvarelast}. We note that
\begin{multline*}
e^\eps(v)(x) = \left[ e_\alpha \otimes e_\beta + e(w^{\alpha \beta})\left(\frac{x'}{\eps},x_d \right) \right] e_{\alpha \beta}(\widehat{\phi})(x') \\ + \left[ -x_d \, e_\alpha \otimes e_\beta + e(W^{\alpha \beta})\left(\frac{x'}{\eps},x_d \right) \right] \partial_{\alpha \beta}\widehat{\phi}_d(x') + \eps \, R^\eps(x),
\end{multline*}
where $\| R^\eps \|_{L^2(\Omega)} \leq C$. Using $v$ as test function in~\eqref{formvarelast}, we get
\begin{equation}
\label{el-formvarfctest}
c^\eps(u^\eps,\phi) + r^\eps(u^\eps,\phi) = d^\eps(\phi) + s^\eps(\phi),
\end{equation}
where
\begin{multline*}
  c^\eps(u^\eps,\phi)
  :=
  \int_\Omega A^\eps(x) \, e^\eps(u^\eps)(x) : \left( \left[ e_\alpha \otimes e_\beta + e(w^{\alpha \beta})\left(\frac{x'}{\eps},x_d \right) \right] e_{\alpha \beta}(\widehat{\phi})(x') \right. \\ \left. + \left[ -x_d \, e_\alpha \otimes e_\beta + e(W^{\alpha \beta})\left(\frac{x'}{\eps},x_d \right) \right] \partial_{\alpha \beta}\widehat{\phi}_d(x') \right),
\end{multline*}
\begin{align*}
  r^\eps(u^\eps,\phi)
  &:=
  \eps \int_\Omega A^\eps e^\eps(u^\eps) : R^\eps,
  \\
  d^\eps(\phi)
  &:=
  \int_\Omega f \cdot \phi + \int_\omega h_\pm \cdot \phi\left(\cdot,\pm \frac{1}{2}\right) - c^\eps(g,\phi),
\end{align*}
and
\begin{align*}
  s^\eps(\phi) & := \eps \int_\Omega f_\gamma(x) \left[ w^{\alpha \beta}_\gamma \left(\frac{x'}{\eps},x_d \right) e_{\alpha \beta}(\widehat{\phi})(x') + W^{\alpha \beta}_\gamma \left(\frac{x'}{\eps},x_d \right) \partial_{\alpha \beta} \widehat{\phi}_d(x') \right]
  \\
  & + \eps^2 \int_\Omega f_d(x) \left[ w^{\alpha \beta}_d \left(\frac{x'}{\eps},x_d \right) e_{\alpha \beta}(\widehat{\phi})(x') + W^{\alpha \beta}_d \left(\frac{x'}{\eps},x_d \right) \partial_{\alpha \beta} \widehat{\phi}_d(x') \right]
  \\
  & + \eps \int_{\Gamma_\pm} (h_\pm(x'))_\gamma \left[ w^{\alpha \beta}_\gamma \left(\frac{x'}{\eps},\pm \frac{1}{2} \right) e_{\alpha \beta} (\widehat{\phi})(x') + W^{\alpha \beta}_\gamma \left(\frac{x'}{\eps},\pm \frac{1}{2} \right) \partial_{\alpha \beta} \widehat{\phi}_d(x') \right]
  \\
  & + \eps^2 \int_{\Gamma_\pm} (h_\pm(x'))_d \left[ w^{\alpha \beta}_d \left(\frac{x'}{\eps},\pm \frac{1}{2} \right) e_{\alpha \beta} (\widehat{\phi})(x') + W^{\alpha \beta}_d \left(\frac{x'}{\eps},\pm \frac{1}{2} \right) \partial_{\alpha \beta} \widehat{\phi}_d(x') \right]
  \\
  & - \eps \int_\Omega A^\eps \, e^\eps(g) : R^\eps.
\end{align*}
The following limits are immediate in view of the $L^2$ bound on $e^\eps(u^\eps)$ and $R^\eps$ and the fact that $g \in \GKL$:
$$
r^\eps(u^\eps,\phi) \underset{\eps \rightarrow 0}{\rightarrow} 0 \qquad \text{and} \qquad s^\eps(\phi) \underset{\eps \rightarrow 0}{\rightarrow} 0. 
$$

\medskip

\noindent
{\bf Step~3: limit of $c^\eps(u^\eps,\phi)$.} We proceed in three steps to identify the limit of $c^\eps(u^\eps,\phi)$. Only the second step is restricted to the membrane case.

\medskip

\noindent
{\bf Step~3a: integration by parts.} Since $A$ is symmetric in the sense that $A_{ijkl} = A_{klij}$, we write
\begin{equation} \label{eq:titi9}
c^\eps(u^\eps,\phi)
= \int_\Omega e^\eps(u^\eps)(x) : \left[ Z_{\alpha \beta}^w\left(\frac{x'}{\eps},x_d \right) e_{\alpha \beta}(\widehat{\phi})(x') + Z_{\alpha \beta}^W\left(\frac{x'}{\eps},x_d \right) \partial_{\alpha \beta}\widehat{\phi}_d(x') \right],
\end{equation}
where $x=(x',x_d) \in \Omega$ and where the matrix-valued functions $Z_{\alpha \beta}^w$ and $Z_{\alpha \beta}^W$ are defined by
$$
Z_{\alpha \beta}^w = A \left( e_\alpha \otimes e_\beta + e(w^{\alpha \beta}) \right) \qquad \text{and} \qquad Z_{\alpha \beta}^W = A \left( -x_d \, e_\alpha \otimes e_\beta + e(W^{\alpha \beta}) \right).
$$
We recall the integration by parts relation~\eqref{eq:IPP_elas}: for any function $z$ such that $z(x)$ is a symmetric matrix and for any vector-valued function $v$ (with $z$ and $v$ sufficiently regular), we have
$$
\int_\Omega e^\eps(v) : z = - \int_\Omega v \cdot \div^\eps z + \int_{\partial \Omega} v_\gamma \, n_\delta \, z_{\gamma \delta} + \frac{1}{\eps} \int_{\partial \Omega} (v_\gamma \, n_d + v_d \, n_\gamma) \, z_{d\gamma} + \frac{1}{\eps^2} \int_{\partial \Omega} v_d \, n_d \, z_{dd}.
$$
Using an integration by parts in~\eqref{eq:titi9} (we note that, in view of~\eqref{eq:symm_A}, both $Z_{\alpha \beta}^w$ and $Z_{\alpha \beta}^W$ are symmetric matrices, for any $1 \leq \alpha,\beta \leq d-1$) and the fact that $\widehat{\phi} \in (\mathcal{D}(\omega))^d$, we obtain
\begin{align}
c^\eps(u^\eps,\phi)
&= -\int_\Omega u^\eps(x) \cdot \div^\eps \left[ Z_{\alpha \beta}^w\left(\frac{x'}{\eps},x_d \right) e_{\alpha \beta}(\widehat{\phi})(x') + Z_{\alpha \beta}^W\left(\frac{x'}{\eps},x_d \right) \partial_{\alpha \beta}\widehat{\phi}_d(x') \right]
\nonumber
\\
& \quad \pm \frac{1}{\eps} \int_{\Gamma_\pm} u^\eps(x) \cdot e_\gamma \ \left[ Z_{\alpha \beta}^w\left(\frac{x'}{\eps},x_d \right) e_{\alpha \beta}(\widehat{\phi})(x') + Z_{\alpha \beta}^W\left(\frac{x'}{\eps},x_d \right) \partial_{\alpha \beta}\widehat{\phi}_d(x') \right]_{d\gamma}
\nonumber
\\
& \quad \pm \frac{1}{\eps^2} \int_{\Gamma_\pm} u^\eps(x) \cdot e_d \ \left[ Z_{\alpha \beta}^w\left(\frac{x'}{\eps},x_d \right) e_{\alpha \beta}(\widehat{\phi})(x') + Z_{\alpha \beta}^W\left(\frac{x'}{\eps},x_d \right) \partial_{\alpha \beta}\widehat{\phi}_d(x') \right]_{dd},
\label{eq:titi8}
\end{align}
where, on the top (resp. bottom) face $\Gamma_+$ (resp. $\Gamma_-$) of $\Omega$, we have $x_d = 1/2$ (resp. $x_d = -1/2$) in the last two lines of~\eqref{eq:titi8}. In view of the boundary conditions satisfied by the correctors $w^{\alpha \beta}$ and $W^{\alpha \beta}$ (see the second line of~\eqref{el-prcor1_edp} and of~\eqref{el-prcor2_edp}), we have that $Z_{\alpha \beta}^w \, e_d = Z_{\alpha \beta}^W \, e_d = 0$ on $\mathcal{Y}_\pm$. The integrand of the last two lines of~\eqref{eq:titi8} therefore vanishes on $\Gamma_\pm$.

We next compute, for any function $z_\eps$ such that $z_\eps(x)$ is a symmetric matrix and for any scalar-valued function $\psi$, that
\begin{align*}
  \left[ \div^\eps (z_\eps \, \psi) \right]_\gamma
  &=
  \psi \left[ \div^\eps z_\eps \right]_\gamma + (z_\eps)_{\gamma \delta} \, \partial_\delta \psi + \frac{1}{\eps} \, (z_\eps)_{\gamma d} \, \partial_d \psi,
  \\
  \left[ \div^\eps (z_\eps \, \psi) \right]_d
  &=
  \psi \left[ \div^\eps z_\eps \right]_d + \frac{1}{\eps} \, (z_\eps)_{d \delta} \, \partial_\delta \psi + \frac{1}{\eps^2} \, (z_\eps)_{dd} \, \partial_d \psi.
\end{align*}
Using the corrector equation satisfied by $w^{\alpha \beta}$ and $W^{\alpha \beta}$ (see the first line of~\eqref{el-prcor1_edp} and of~\eqref{el-prcor2_edp}), which implies that $Z_{\alpha \beta}^w$ and $Z_{\alpha \beta}^W$ are divergence free, we compute that
\begin{align*}
  & e_\gamma \cdot \div^\eps \left[ Z_{\alpha \beta}^w\left(\frac{x'}{\eps},x_d \right) e_{\alpha \beta}(\widehat{\phi})(x') + Z_{\alpha \beta}^W\left(\frac{x'}{\eps},x_d \right) \partial_{\alpha \beta}\widehat{\phi}_d(x') \right]
  \\
  &=
  \left[ Z_{\alpha \beta}^w\left(\frac{x'}{\eps},x_d \right) \right]_{\gamma \delta} \partial_\delta e_{\alpha \beta}(\widehat{\phi})(x') + \left[ Z_{\alpha \beta}^W\left(\frac{x'}{\eps},x_d \right) \right]_{\gamma \delta} \partial_{\delta \alpha \beta}\widehat{\phi}_d(x')
  \\
  &=
  e_\gamma \cdot \left[ Z_{\alpha \beta}^w\left(\frac{x'}{\eps},x_d \right) \nabla e_{\alpha \beta}(\widehat{\phi})(x') + Z_{\alpha \beta}^W\left(\frac{x'}{\eps},x_d \right) \nabla \partial_{\alpha \beta}\widehat{\phi}_d(x') \right]
\end{align*}
and likewise
\begin{align*}
  & e_d \cdot \div^\eps \left[ Z_{\alpha \beta}^w\left(\frac{x'}{\eps},x_d \right) e_{\alpha \beta}(\widehat{\phi})(x') + Z_{\alpha \beta}^W\left(\frac{x'}{\eps},x_d \right) \partial_{\alpha \beta}\widehat{\phi}_d(x') \right]
  \\
  &=
  \frac{1}{\eps} \, e_d \cdot \left[ Z_{\alpha \beta}^w\left(\frac{x'}{\eps},x_d \right) \nabla e_{\alpha \beta}(\widehat{\phi})(x') + Z_{\alpha \beta}^W\left(\frac{x'}{\eps},x_d \right) \nabla \partial_{\alpha \beta}\widehat{\phi}_d(x') \right].
\end{align*}
We thus deduce from~\eqref{eq:titi8} that
$$
c^\eps(u^\eps,\phi) = -\int_\Omega u^\eps_\gamma(x) \ e_\gamma \cdot \left[ Z_{\alpha \beta}^w\left(\frac{x'}{\eps},x_d \right) \nabla e_{\alpha \beta}(\widehat{\phi})(x') + Z_{\alpha \beta}^W\left(\frac{x'}{\eps},x_d \right) \nabla \partial_{\alpha \beta}\widehat{\phi}_d(x') \right] - c^\eps_d(u^\eps,\phi),
$$
with
\begin{equation} \label{eq:def_ceps_d}
c^\eps_d(u^\eps,\phi) = \frac{1}{\eps} \int_\Omega u^\eps_d(x) \ e_d \cdot \left[ Z_{\alpha \beta}^w\left(\frac{x'}{\eps},x_d \right) \nabla e_{\alpha \beta}(\widehat{\phi})(x') + Z_{\alpha \beta}^W\left(\frac{x'}{\eps},x_d \right) \nabla \partial_{\alpha \beta}\widehat{\phi}_d(x') \right],
\end{equation}
and we recast the above expression as
\begin{multline} \label{eq:dimanche1}
c^\eps(u^\eps,\phi) = (\eps-1) \, c^\eps_d(u^\eps,\phi) - \int_\Omega u^\eps(x) \cdot A\left(\frac{x'}{\eps},x_d \right) \left( \left[ e_\alpha \otimes e_\beta + e(w^{\alpha \beta})\left(\frac{x'}{\eps},x_d \right) \right] \nabla e_{\alpha \beta}(\widehat{\phi})(x') \right. \\ \left. + \left[ -x_d \, e_\alpha \otimes e_\beta + e(W^{\alpha \beta})\left(\frac{x'}{\eps},x_d \right) \right] \nabla \partial_{\alpha \beta}\widehat{\phi}_d(x') \right).
\end{multline}

\medskip

\noindent
{\bf Step~3b: bound on $c^\eps_d(u^\eps,\phi)$ using the membrane case assumptions.} We are now going to use the decoupling assumption~\eqref{hyp:symA} and the fact that we are in the membrane case~\eqref{eq:ass_membrane} to estimate $c^\eps_d(u^\eps,\phi)$. For any $\dps (x',z) \in \omega \times (-1/2,1/2)$, we have
$$
u^\eps_d(x',z) = u^\eps_d\left(x',-\frac{1}{2} \right) + \int_{-1/2}^z e_{dd}(u^\eps)
$$
and
$$
u^\eps_d(x',z) = u^\eps_d\left(x',\frac{1}{2} \right) - \int_z^{1/2} e_{dd}(u^\eps),
$$
which implies that
$$
u^\eps_d(x',z) = \frac{1}{2} \left[ u^\eps_d\left(x',-\frac{1}{2} \right) + u^\eps_d\left(x',\frac{1}{2} \right) \right] + \frac{1}{2} \int_{-1/2}^z e_{dd}(u^\eps) - \frac{1}{2} \int_z^{1/2} e_{dd}(u^\eps).
$$
In the membrane case, we have $u^\eps_d \in \O$ (see Lemma~\ref{lem:membrane}), and therefore
$$
u^\eps_d(x',z) = \frac{\eps^2}{2} \int_{-1/2}^z e^\eps_{dd}(u^\eps) - \frac{\eps^2}{2} \int_z^{1/2} e^\eps_{dd}(u^\eps).
$$
In view of the bound~\eqref{est:sig} (together with~\eqref{eq:inde_eps_vec} and~\eqref{eq:assump_g}), we obtain that $\| u^\eps_d \|_{L^2(\Omega)} \leq C \, \eps^2$. Using the periodicity of $Z_{\alpha \beta}^w$ and $Z_{\alpha \beta}^W$ with respect to their first $d-1$ variables (for any $1 \leq \alpha,\beta \leq d-1$), we deduce from this bound and~\eqref{eq:def_ceps_d} that
\begin{equation} \label{eq:special_memb}
\left| c^\eps_d(u^\eps,\phi) \right| \leq C \, \eps.
\end{equation}
We are now in position to pass to the limit $\eps \to 0$ in~\eqref{eq:dimanche1}. The first term converges to zero in view of~\eqref{eq:special_memb}, and we use the Rellich theorem (from which we infer that $(u^\eps)_{\eps>0}$ converges strongly to $u^\star$ in $\left(L^2(\Omega)\right)^d$) to handle the second term. Using Lemma~\ref{limmoyenne}, we deduce that
$$
c^\eps(u^\eps,\phi) \underset{\eps \rightarrow 0}{\rightarrow} c^\star(u^\star,\phi)
$$
with
\begin{multline} \label{eq:dimanche2}
- c^\star(u^\star,\phi) 
=
\int_\Omega u^\star \cdot \left[ \left\{ \int_Y A(\cdot,x_d) \left( e_\alpha \otimes e_\beta + e(w^{\alpha \beta})\left(\cdot,x_d\right) \right) \right\} \nabla e_{\alpha \beta}(\widehat{\phi}) \right. \\ \left. + \left\{ \int_Y A(\cdot,x_d) \left( -x_d \, e_\alpha \otimes e_\beta + e(W^{\alpha \beta})\left(\cdot,x_d\right) \right) \right\} \nabla \partial_{\alpha \beta}\widehat{\phi}_d \right].
\end{multline}

\medskip

\noindent
{\bf Step~3c: conclusion of Step~3.} The sequel of the proof does not need the decoupling assumption~\eqref{hyp:symA}. We have shown at the beginning of the proof that $u^\star = \widehat{u}^\star - x_d \, \nabla \widehat{u}^\star_d$ with $\widehat{u}^\star$ independent of $x_d$. Using that $\widehat{\phi}$ is also independent of $x_d$, we recast~\eqref{eq:dimanche2} as
\begin{align}
  - c^\star(u^\star,\phi)
  &=
  \int_\Omega (\widehat{u}^\star - x_d \, \nabla \widehat{u}^\star_d) \cdot \left[ \left\{ \int_Y A(\cdot,x_d) \left( e_\alpha \otimes e_\beta + e(w^{\alpha \beta})\left(\cdot,x_d\right) \right) \right\} \nabla e_{\alpha \beta}(\widehat{\phi}) \right.
  \nonumber
  \\
  & \left. \qquad + \left\{ \int_Y A(\cdot,x_d) \left( -x_d \, e_\alpha \otimes e_\beta + e(W^{\alpha \beta})\left(\cdot,x_d\right) \right) \right\} \nabla \partial_{\alpha \beta}\widehat{\phi}_d \right]
  \nonumber
  \\
  &= \int_\omega (\widehat{u}^\star)_\gamma \left[ (k^\star_{11})_{\alpha \beta \gamma \delta} \ \partial_\delta e_{\alpha \beta}(\widehat{\phi}) + (k^\star_{12})_{\alpha \beta \gamma \delta} \ \partial_\delta \partial_{\alpha \beta}\widehat{\phi}_d \right]
  \nonumber
  \\
  & \qquad - \partial_\gamma \widehat{u}^\star_d \left[ (k^\star_{21})_{\alpha \beta \gamma \delta} \ \partial_\delta e_{\alpha \beta}(\widehat{\phi}) + (k^\star_{22})_{\alpha \beta \gamma \delta} \ \partial_\delta \partial_{\alpha \beta}\widehat{\phi}_d \right]
  \nonumber
  \\
  & \qquad + \int_\omega (\widehat{u}^\star)_d \left[ (k^\star_{11})_{\alpha \beta d \delta} \ \partial_\delta e_{\alpha \beta}(\widehat{\phi}) + (k^\star_{12})_{\alpha \beta d \delta} \ \partial_\delta \partial_{\alpha \beta}\widehat{\phi}_d \right],
  \label{eq:vendredi}
\end{align}
with
\begin{align}
  (k^\star_{11})_{\alpha \beta \gamma \delta} &:= e^T_\gamma \left[ \int_\Y A \big( e(w^{\alpha \beta}) + e_\alpha \otimes e_\beta \big) \right] e_\delta,
  \label{eq:def_petit_k_11}
  \\
  (k^\star_{12})_{\alpha \beta \gamma \delta} &:= e^T_\gamma \left[ \int_\Y A \big( e(W^{\alpha \beta}) - x_d \, e_\alpha \otimes e_\beta \big) \right] e_\delta
  \nonumber
\end{align}
and likewise when $\gamma$ is replaced by $d$, and where $k^\star_{22}$ and $k^\star_{21}$ are defined by
\begin{align}
  (k^\star_{22})_{\alpha \beta \gamma \delta} &:= e^T_\gamma \left[ \int_\Y A \big( e(W^{\alpha \beta}) - x_d \, e_\alpha \otimes e_\beta \big) \, x_d \right] e_\delta,
  \label{eq:stop9}
  \\
  (k^\star_{21})_{\alpha \beta \gamma \delta} &:= e^T_\gamma \left[ \int_\Y A \big( e(w^{\alpha \beta}) + e_\alpha \otimes e_\beta \big) \, x_d \right] e_\delta.
  \label{eq:def_petit_k_21}
\end{align}

\medskip

\noindent
{\bf Step~4: properties of the homogenized tensors.} Using~\eqref{el-prcor1} with the test function $w^{\gamma \delta} \in \mathcal{W}(\mathcal{Y})$, respectively $W^{\gamma \delta} \in \mathcal{W}(\mathcal{Y})$, we see that
$$
(K^\star_{11})_{\alpha \beta \gamma \delta} = \int_\Y A \big( e(w^{\alpha \beta}) + e_\alpha \otimes e_\beta \big) : e_\gamma \otimes e_\delta = (k^\star_{11})_{\alpha \beta \gamma \delta},
$$
respectively 
$$
(K^\star_{12})_{\alpha \beta \gamma \delta} = - \int_\Y x_d \, A \big( e(w^{\alpha \beta}) + e_\alpha \otimes e_\beta \big) : e_\gamma \otimes e_\delta = - (k^\star_{21})_{\alpha \beta \gamma \delta},
$$
where $k^\star_{11}$ and $k^\star_{21}$ are defined by~\eqref{eq:def_petit_k_11} and~\eqref{eq:def_petit_k_21}, respectively. Similarly, using the test function $W^{\gamma \delta} \in \mathcal{W}(\mathcal{Y})$ in~\eqref{el-prcor2}, we observe that
\begin{equation} \label{eq:stop9_bis}
(K^\star_{22})_{\alpha \beta \gamma \delta} = - \int_\Y x_d \, A \big( e(W^{\alpha \beta}) - x_d \, e_\alpha \otimes e_\beta \big) : e_\gamma \otimes e_\delta = - (k^\star_{22})_{\alpha \beta \gamma \delta}.
\end{equation}
Using that $A$ is symmetric (i.e. $A_{ijkl} = A_{klij}$), we also write that
\begin{align*}
  (K^\star_{12})_{\alpha \beta \gamma \delta}
  &=
  \int_\Y A \big( e(W^{\gamma \delta}) - x_d \, e_\gamma \otimes e_\delta \big) : \big( e(w^{\alpha \beta}) + e_\alpha \otimes e_\beta \big)
  \\
  &=
  \int_\Y A \big( e(W^{\gamma \delta}) - x_d \, e_\gamma \otimes e_\delta \big) : e_\alpha \otimes e_\beta 
  \\
  &=
  (k^\star_{12})_{\gamma \delta \alpha \beta},
\end{align*}
where we have used~\eqref{el-prcor2} with the test function $w^{\alpha \beta} \in \mathcal{W}(\mathcal{Y})$ in the second line.

\medskip

We also claim that, for any $1 \leq \alpha, \beta, \delta \leq d-1$, we have
\begin{equation} \label{eq:vendredi3}
  (k^\star_{11})_{\alpha \beta d \delta} = (k^\star_{12})_{\alpha \beta d \delta} = 0.
\end{equation}
Indeed, since $X : \R^d \ni x \mapsto x_d \, e_\delta$ is an admissible test function in~\eqref{el-prcor1}, we can write 
\begin{align*}
  0
  &= 
  \int_{\mathcal{Y}} A \big( e(w^{\alpha \beta}) + e_\alpha \otimes e_\beta \big) : e(X)
  \\
  &= 
  \frac{1}{2} \int_{\mathcal{Y}} A \big( e(w^{\alpha \beta}) + e_\alpha \otimes e_\beta \big) : \big( e_\delta \otimes e_d + e_d \otimes e_\delta \big)
  \\
  &= 
  \int_{\mathcal{Y}} A \big( e(w^{\alpha \beta}) + e_\alpha \otimes e_\beta \big) : e_d \otimes e_\delta \qquad \text{[$A$ is symmetric: $A_{ijkl} = A_{jikl}$]}
  \\
  &=
  e^T_d \left[ \int_\Y A \big( e(w^{\alpha \beta}) + e_\alpha \otimes e_\beta \big) \right] e_\delta
  \\
  &= 
  (k^\star_{11})_{\alpha \beta d \delta}.
\end{align*}
Likewise, since $X$ is also an admissible test function in~\eqref{el-prcor2}, we can write 
\begin{align*}
  0
  &= 
  \int_{\mathcal{Y}} A \big( e(W^{\alpha \beta}) - x_d \, e_\alpha \otimes e_\beta \big) : e(X)
  \\
  &= 
  \int_{\mathcal{Y}} A \big( e(W^{\alpha \beta}) - x_d \, e_\alpha \otimes e_\beta \big) : e_d \otimes e_\delta \qquad \text{[$A$ is symmetric: $A_{ijkl} = A_{jikl}$]}
  \\
  &=
  e^T_d \left[ \int_\Y A \big( e(W^{\alpha \beta}) - x_d \, e_\alpha \otimes e_\beta \big) \right] e_\delta
  \\
  &= 
  (k^\star_{12})_{\alpha \beta d \delta}.
\end{align*}
This hence proves~\eqref{eq:vendredi3}.

\medskip

\noindent
{\bf Step~5: identification of $c^\star(u^\star,\phi)$.} Using the properties of $k^\star_{11}$, $k^\star_{22}$, $k^\star_{12}$ and $k^\star_{21}$ that we have shown in Step~4, we infer from~\eqref{eq:vendredi} that
\begin{multline*}
  - c^\star(u^\star,\phi)
  =
  \int_\omega (\widehat{u}^\star)_\gamma \left[ (K^\star_{11})_{\alpha \beta \gamma \delta} \ \partial_\delta e_{\alpha \beta}(\widehat{\phi}) + (K^\star_{12})_{\gamma \delta \alpha \beta} \ \partial_\delta \partial_{\alpha \beta}\widehat{\phi}_d \right]
  \\
  + \partial_\gamma \widehat{u}^\star_d \left[ (K^\star_{12})_{\alpha \beta \gamma \delta} \ \partial_\delta e_{\alpha \beta}(\widehat{\phi}) + (K^\star_{22})_{\alpha \beta \gamma \delta} \ \partial_\delta \partial_{\alpha \beta}\widehat{\phi}_d \right].
\end{multline*}
Using the symmetry properties of $K^\star_{11}$ and $K^\star_{22}$ (namely $(K^\star_{11})_{\alpha \beta \gamma \delta} = (K^\star_{11})_{\gamma \delta \alpha \beta}$ and likewise for $K^\star_{22}$), we recast the above as 
\begin{multline}
  - c^\star(u^\star,\phi)
  =
  \int_\omega (\widehat{u}^\star)_\gamma \left[ (K^\star_{11})_{\gamma \delta \alpha \beta} \ \partial_\delta e_{\alpha \beta}(\widehat{\phi}) + (K^\star_{12})_{\gamma \delta \alpha \beta} \ \partial_\delta \partial_{\alpha \beta}\widehat{\phi}_d \right]
  \\
  + \partial_\gamma \widehat{u}^\star_d \left[ ((K^\star_{12})^T)_{\gamma \delta \alpha \beta} \ \partial_\delta e_{\alpha \beta}(\widehat{\phi}) + (K^\star_{22})_{\gamma \delta \alpha \beta} \ \partial_\delta \partial_{\alpha \beta}\widehat{\phi}_d \right].
  \label{eq:vendredi2}
\end{multline}
We next note that, by simple tensor algebra,
\begin{align*}
(\widehat{u}^\star)' \cdot (K^\star_{11} : e_\alpha \otimes e_\beta) \nabla' e_{\alpha \beta}(\widehat{\phi})
&=
(\widehat{u}^\star)_\gamma \ (K^\star_{11} : e_\alpha \otimes e_\beta)_{\gamma \delta} \ \partial_\delta e_{\alpha \beta}(\widehat{\phi})
\\
&=
(\widehat{u}^\star)_\gamma \ (K^\star_{11})_{\gamma \delta \alpha \beta} \ \partial_\delta e_{\alpha \beta}(\widehat{\phi}).
\end{align*}
We thus deduce from~\eqref{eq:vendredi2} that
\begin{multline*}
  - c^\star(u^\star,\phi)
  =
  \int_\omega (\widehat{u}^\star)' \cdot \left[ (K^\star_{11} : e_\alpha \otimes e_\beta) \nabla' e_{\alpha \beta}(\widehat{\phi}) + (K^\star_{12} : e_\alpha \otimes e_\beta) \nabla' \partial_{\alpha \beta}\widehat{\phi}_d \right]
  \\
  + \nabla' \widehat{u}^\star_d \cdot \left[ ((K^\star_{12})^T : e_\alpha \otimes e_\beta) \nabla' e_{\alpha \beta}(\widehat{\phi}) + (K^\star_{22} : e_\alpha \otimes e_\beta) \nabla' \partial_{\alpha \beta}\widehat{\phi}_d \right].
\end{multline*}
Using an integration by parts and the symmetry of the matrices $K^\star_{11} : e_\alpha \otimes e_\beta$, $K^\star_{12} : e_\alpha \otimes e_\beta$, $(K^\star_{12})^T : e_\alpha \otimes e_\beta$ and $K^\star_{22} : e_\alpha \otimes e_\beta$, we compute
\begin{multline*}
  c^\star(u^\star,\phi)
  =
  \int_\omega e_{\alpha \beta}(\widehat{\phi}) \, \Big[ (K^\star_{11} : e_\alpha \otimes e_\beta) : \nabla' (\widehat{u}^\star)' \Big] + \partial_{\alpha \beta}\widehat{\phi}_d \, \Big[ (K^\star_{12} : e_\alpha \otimes e_\beta) : \nabla' (\widehat{u}^\star)' \Big]
  \\
  + e_{\alpha \beta}(\widehat{\phi}) \, \Big[ ((K^\star_{12})^T : e_\alpha \otimes e_\beta) : \nabla^2_{d-1} \widehat{u}^\star_d \Big] + \partial_{\alpha \beta}\widehat{\phi}_d \, \Big[ (K^\star_{22} : e_\alpha \otimes e_\beta) : \nabla^2_{d-1} \widehat{u}^\star_d \Big].
\end{multline*}
Note that the boundary term in the integration by parts vanishes because $\widehat{\phi} \in (\mathcal{D}(\omega))^d$. We thus deduce that
\begin{align*}
  c^\star(u^\star,\phi)
  &=
  \int_\omega \Big( K^\star_{11} : e'(\widehat{\phi}') \Big) : \nabla' (\widehat{u}^\star)' + \Big( K^\star_{12} : \nabla^2_{d-1} \widehat{\phi}_d \Big) : \nabla' (\widehat{u}^\star)'
  \\
  & \qquad + \Big( (K^\star_{12})^T : e'(\widehat{\phi}') \Big) : \nabla^2_{d-1} \widehat{u}^\star_d + \Big( K^\star_{22} : \nabla^2_{d-1} \widehat{\phi}_d \Big) : \nabla^2_{d-1} \widehat{u}^\star_d
  \\
  &=
  \int_\omega \Big( K^\star_{11} : e'(\widehat{\phi}') \Big) : e'((\widehat{u}^\star)') + \Big( K^\star_{12} : \nabla^2_{d-1} \widehat{\phi}_d \Big) : e'((\widehat{u}^\star)')
  \\
  & \qquad + \Big( (K^\star_{12})^T : e'(\widehat{\phi}') \Big) : \nabla^2_{d-1} \widehat{u}^\star_d + \Big( K^\star_{22} : \nabla^2_{d-1} \widehat{\phi}_d \Big) : \nabla^2_{d-1} \widehat{u}^\star_d
  \\
  &=
  \int_\omega K^\star \, \mathcal{P} u^\star : \mathcal{P} \phi,
\end{align*}
where we have used in the second equation the symmetry of the matrices $K^\star_{11} : e'(\widehat{\phi}')$ and $K^\star_{12} : \nabla^2_{d-1} \widehat{\phi}_d$, and the symmetry of $K^\star$ in the last line.

\medskip

\noindent
{\bf Step~6: homogenized equation.} We are now in position to pass to the limit $\eps \to 0$ in~\eqref{el-formvarfctest}. Indeed, we have now identified the limit of the two terms of the left-hand side and that of $s^\eps(\phi)$ on the right-hand side. We are thus left with identifying the limit of $d^\eps(\phi)$. Since $\phi = \widehat{\phi} - x_d \, \nabla \widehat{\phi}_d$ where $\widehat{\phi}$ is independent of $x_d$, we compute that
\begin{align*}
  d^\eps(\phi)
  &=
  \int_\Omega f \cdot \phi + \int_\omega h_\pm \cdot \phi\left(\cdot,\pm \frac{1}{2}\right) - c^\eps(g,\phi)
  \\
  &=
  \int_\omega (\m(f)+h_++h_-) \cdot \widehat{\phi} - \int_\omega \m(x_d \, f_\alpha) \, \partial_\alpha \widehat{\phi}_d - \frac{1}{2} \int_\omega \big( (h_+)_\alpha - (h_-)_\alpha \big) \, \partial_\alpha \widehat{\phi}_d - c^\eps(g,\phi).
\end{align*}
Using similar computations as above to identify the limit of the last term (recall that, to identify the limit of $c^\eps(u^\eps,\phi)$, the only property we have used on $u^\eps$ is that it converges strongly in $(L^2(\Omega))^d$ to some $u^\star \in \GKL$, and that we have assumed that $g \in \GKL$), we obtain that
$$
\lim_{\eps \to 0} d^\eps(\phi)
=
\int_\omega (\m(f)+h_++h_-) \cdot \widehat{\phi} - \int_\omega \m(x_d \, f_\alpha) \, \partial_\alpha \widehat{\phi}_d - \frac{1}{2} \int_\omega \big( (h_+)_\alpha - (h_-)_\alpha \big) \, \partial_\alpha \widehat{\phi}_d - \int_\omega K^\star \, \mathcal{P} g : \mathcal{P} \phi.
$$
We have thus shown that, for any $\widehat{\phi} \in \left(\mathcal{D}(\omega)\right)^d$, we have
\begin{multline*}
\int_\omega K^\star \, \mathcal{P} u^\star : \mathcal{P} \phi = \int_\omega (\m(f)+h_++h_-) \cdot \widehat{\phi} - \int_\omega \m(x_d \, f_\alpha) \, \partial_\alpha \widehat{\phi}_d \\ - \frac{1}{2} \int_\omega \big( (h_+)_\alpha - (h_-)_\alpha \big) \, \partial_\alpha \widehat{\phi}_d - \int_\omega K^\star \, \mathcal{P} g : \mathcal{P} \phi,
\end{multline*}
where $\phi$ is defined from $\widehat{\phi}$ as explained at the beginning of Step~2.

We next claim that 
\begin{multline*}
\int_\omega K^\star \, \mathcal{P} u^\star : \mathcal{P} v = \int_\omega (\m(f)+h_++h_-) \cdot \widehat{v} - \int_\omega \m(x_d \, f_\alpha) \, \partial_\alpha \widehat{v}_d \\ - \frac{1}{2} \int_\omega \big( (h_+)_\alpha - (h_-)_\alpha \big) \, \partial_\alpha \widehat{v}_d - \int_\omega K^\star \, \mathcal{P} g : \mathcal{P} v
\end{multline*}
for any $v \in \VKL$, as stated in~\eqref{formvarelasthomog}. Indeed, for any $v$ in $\VKL$ (and denoting $\widehat{v}$ the associated function in $\left(H^1_0(\omega)\right)^{d-1} \times H^2_0(\omega)$) and for any $\widehat{\phi} \in \left(\mathcal{D}(\omega)\right)^d$ (with $\phi$ defined as above), we compute
\begin{align*}
  \| \P \phi - \P v \|_{L^2(\omega)}^2
  &=
  \| e'(\widehat{\phi}') - e'(\widehat{v}') \|_{L^2(\omega)}^2 + \| \nabla^2_{d-1} \widehat{\phi}_d - \nabla^2_{d-1} \widehat{v}_d \|_{L^2(\omega)}^2
  \\
  & \leq \| \widehat{\phi}' - \widehat{v}' \|^2_{H^1(\omega)} + \| \widehat{\phi}_d - \widehat{v}_d \|_{H^2(\omega)}^2.
\end{align*}
The density of $\left(\mathcal{D}(\omega)\right)^{d-1}$ in $\left(H^1_0(\omega)\right)^{d-1}$ for the $H^1$ norm and of $\mathcal{D}(\omega)$ in $H^2_0(\omega)$ for the $H^2$ norm allows us to conclude.

\medskip

\noindent
{\bf Step~7: coercivity of the homogenized tensor.} We eventually show that $K^\star$ is coercive. Let $\sigma$ and $\tau$ be in $\R_s^{(d-1)\times(d-1)}$, and let $w^\sigma := \sigma_{\alpha \beta} \, w^{\alpha \beta}$ and $W^\tau = \tau_{\alpha \beta} \, W^{\alpha\beta}$ where $w^{\alpha \beta}$ and $W^{\alpha\beta}$ are the correctors. We then have
\begin{align}
  \frac{1}{c_-} \, K^\star \begin{pmatrix} \sigma \\ \tau \end{pmatrix} : \begin{pmatrix} \sigma \\ \tau \end{pmatrix}
  &=
  \frac{1}{c_-} \, \int_\Y A \big( e(w^\sigma) + \sigma + e(W^\tau) - x_d \, \tau \big) : \big( e(w^\sigma) + \sigma + e(W^\tau) - x_d \, \tau \big)
  \nonumber
  \\
  &\geq \int_\Y \big| e(w^\sigma) + \sigma + e(W^\tau) - x_d \, \tau \big|^2
  \nonumber
  \\
  &\geq \int_\Y \big| e'(w^\sigma + W^\tau) + \sigma - x_d \, \tau \big|^2
  \nonumber
  \\
  & = \int_\Y \big| e'(w^\sigma + W^\tau) + \sigma \big|^2 + x_d^2 \, \tau : \tau - 2 \, x_d \, \tau : \big( e'(w^\sigma + W^\tau) + \sigma \big)
  \nonumber 
  \\
  & \geq \sum_{\alpha,\beta=1}^{d-1} \left| \int_\Y e_{\alpha \beta}(w^\sigma + W^\tau) + \sigma_{\alpha \beta} \right|^2 + \frac{1}{12} \, \tau : \tau - 2 \int_\Y x_d \, \tau : e'(w^\sigma + W^\tau),
  \label{eq:titi10}
\end{align}
where we have used Jensen inequality and the fact that $\dps \int_\Y x_d \, \tau : \sigma = 0$. Since $w^\sigma$ is periodic with respect to its first $d-1$ variables, we have $\dps \int_\Y e_{\alpha \beta}(w^\sigma) = 0$, and likewise for $W^\tau$. In addition, for any $1 \leq \alpha,\beta \leq d-1$, we have, using the periodicity of $w^\sigma + W^\tau$ with respect to its first $d-1$ variables, that
$$
\int_\Y x_d \, \frac{\partial (w^\sigma_\alpha + W^\tau_\alpha)}{\partial x_\beta} = 0.
$$
We thus deduce from~\eqref{eq:titi10} that
\begin{equation} \label{eq:titi11}
K^\star \begin{pmatrix} \sigma \\ \tau \end{pmatrix} : \begin{pmatrix} \sigma \\ \tau \end{pmatrix} \geq c_- \left( \sigma : \sigma + \frac{1}{12} \, \tau : \tau \right),
\end{equation}
which implies that $K^\star$ is coercive, in the sense that there exists $\widetilde{c}_- > 0$ such that $\dps K^\star \begin{pmatrix} \sigma \\ \tau \end{pmatrix} : \begin{pmatrix} \sigma \\ \tau \end{pmatrix} \geq \widetilde{c}_- \begin{pmatrix} \sigma \\ \tau \end{pmatrix} : \begin{pmatrix} \sigma \\ \tau \end{pmatrix}$.

As a consequence, the bilinear form in the left-hand side of the homogenized problem~\eqref{formvarelasthomog} is coercive in $\VKL$. We indeed compute that, for any $v \in \VKL$,
\begin{align*}
  \int_\omega K^\star \, \mathcal{P} v : \mathcal{P} v
  &\geq
  \widetilde{c}_- \left( \| e'(\widehat{v}') \|_{L^2(\omega)}^2 + \| \nabla^2_{d-1} \widehat{v}_d \|_{L^2(\omega)}^2 \right)
  \\
  &\geq
  \widetilde{c}_- \left( \demi \| \nabla' \widehat{v}' \|_{L^2(\omega)}^2 + C \| \nabla' \widehat{v}_d \|_{H^1(\omega)}^2 \right),
\end{align*}
where we have used the Korn inequality in $H^1_0$ (see Lemma~\ref{Korn2}) for $\widehat{v}' \in (H^1_0(\omega))^{d-1}$ and the Poincar\'e inequality for each component of $\nabla' \widehat{v}_d \in (H^1_0(\omega))^{d-1}$. Using the Poincar\'e inequality for each component of $\widehat{v} \in (H^1_0(\omega))^d$, we infer that
$$
\int_\omega K^\star \, \mathcal{P} v : \mathcal{P} v \geq c \left( \| \widehat{v}' \|_{H^1(\omega)}^2 + \| \widehat{v}_d \|_{H^2(\omega)}^2 \right),
$$
thus the coercivity claim. The homogenized problem~\eqref{formvarelasthomog} is therefore well-posed. This hence shows that the whole sequence $u^\eps$ (and not only a subsequence) converges to $u^\star$, and concludes the proof of Theorem~\ref{limitel}.

\subsection{Proof of Theorem~\ref{limitel} in the bending case} \label{app:diff_preuve2_bend}


Under the decoupling assumption~\eqref{hyp:symA}, we consider here the bending case~\eqref{eq:ass_bending}. We also assume that $d=2$ (in order to be able to use~\eqref{eq:Sigmaformula}). We are going to establish that the homogenized limit $u^\star$ of $u^\eps$ is a solution to~\eqref{formvar:pb_bending}, which is exactly~\eqref{formvarelasthomog} in the bending case (see Lemma~\ref{lem:ustarsym}). Using the results established in Appendix~\ref{app:diff_preuve2_memb} (see in particular Step~7 of the proof there, which implies that $K^\star_{22}$ and the associated bilinear form are coercive), we already know that~\eqref{formvar:pb_bending} is well-posed.

Our proof relies on the tools we have introduced in Section~\ref{sec:tech_res_bending}. Let $\sigma^\eps$ be defined by~\eqref{eq:def_sigma_eps}, and let $\Sigma^\eps$ be defined by~\eqref{eq:def_grand_Sigma}. The variational formulation~\eqref{formvarelast} satisfied by $u^\eps$ reads
$$
\forall v \in V, \qquad \int_\Omega \Sigma^\eps : e(v) = \int_\Omega \sigma^\eps : e^\eps(v) = \int_\Omega A^\eps e^\eps(u^\eps + g) : e^\eps(v) = \int_\Omega f \cdot v + \int_{\Gamma_\pm} h_\pm \cdot v.
$$
For any $\psi \in H^2_0(\omega)$, we consider the above equation with $\dps v = \left( \begin{array}{c} - x_d \, \nabla' \psi \\ \psi \end{array} \right) \in V$ (note that $v \in \VKL^{\mathcal{B}}$). We compute that $[ e(v) ]_{\alpha\beta} = - x_d \, \partial_{\alpha\beta} \psi$ and $[ e(v) ]_{\alpha d} = [ e(v) ]_{dd} = 0$. We hence obtain
\begin{multline*}
- \sum_{\alpha,\beta=1}^{d-1} \int_\Omega \Sigma^\eps_{\alpha\beta} \, x_d \, \partial_{\alpha\beta} \psi = \int_\omega \big[ \m(f_d) + (h_+)_d + (h_-)_d \big] \, \psi \\ - \int_\omega \m(x_d \, f') \cdot \nabla' \psi - \frac{1}{2} \int_\omega \big[ h'_+ - h'_- \big] \cdot \nabla' \psi.
\end{multline*}
In Lemma~\ref{lem:lemma1}, we have established some weak convergence results of $\Sigma^\eps$ to some $\Sigma^\star$. In particular, using~\eqref{eq:sigma_eps_alpha_beta}, we infer from the above equation that
\begin{multline} \label{eq:kevinn1}
- \sum_{\alpha,\beta=1}^{d-1} \int_\Omega \Sigma^\star_{\alpha\beta} \, x_d \, \partial_{\alpha\beta} \psi = \int_\omega \big[ \m(f_d) + (h_+)_d + (h_-)_d \big] \, \psi \\ - \int_\omega \m(x_d \, f') \cdot \nabla' \psi - \frac{1}{2} \int_\omega \big[ h'_+ - h'_- \big] \cdot \nabla' \psi.
\end{multline}
Successively using~\eqref{eq:Sigmaformula}, \eqref{eq:defSstar_exp}, \eqref{eq:stop9_bis} and the symmetry of $K^\star_{22}$, we recast the first term of the above equation as
\begin{align*}
  & \sum_{\alpha,\beta=1}^{d-1} \int_{-1/2}^{1/2} x_d \, \Sigma^\star_{\alpha\beta}(\cdot,x_d)
  \\
  &= \sum_{\alpha,\beta,\gamma,\delta=1}^{d-1} \partial_{\gamma_\delta} (u^\star_d + g_d) \int_{-1/2}^{1/2} x_d \, S^\star_{\alpha\beta\gamma\delta}(x_d)
  \\
  &= \sum_{\alpha,\beta,\gamma,\delta=1}^{d-1} \partial_{\gamma_\delta} (u^\star_d + g_d) \int_{-1/2}^{1/2} x_d \, (e_\alpha \otimes e_\beta) : \int_Y A(\cdot,x_d) \, \big( e(W^{\gamma \delta})(\cdot,x_d) - x_d \, e_\gamma \otimes e_\delta \big)
  \\
  &= - \sum_{\alpha,\beta,\gamma,\delta=1}^{d-1} (K^\star_{22})_{\gamma \delta \alpha \beta} \, \partial_{\gamma_\delta}(u^\star_d + g_d)
  \\
  &= - \sum_{\alpha,\beta,\gamma,\delta=1}^{d-1} (K^\star_{22})_{\alpha \beta \gamma \delta} \, \partial_{\gamma_\delta}(u^\star_d + g_d).
\end{align*}
We thus deduce from~\eqref{eq:kevinn1} that, for any $\psi \in H^2_0(\omega)$,
\begin{multline*}
  \int_\omega K^\star_{22} \, \nabla^2_{d-1} (u^\star_d + g_d) : \nabla^2_{d-1} \psi = \int_\omega \big[ \m(f_d) + (h_+)_d + (h_-)_d \big] \, \psi \\ - \int_\omega \m(x_d \, f') \cdot \nabla' \psi - \frac{1}{2} \int_\omega \big[ h'_+ - h'_- \big] \cdot \nabla' \psi,
\end{multline*}
which is exactly~\eqref{formvar:pb_bending}.

\section*{Acknowledgments}


The work of FL is partly supported by ONR and EOARD. FL acknowledges the continuous support from these two agencies, in particular through the current grants ONR N00014-25-1-2299 and EOARD FA8655-24-1-7057. AL acknowledges some financial support from Inria. The authors warmly thank Annie Raoult for carefully reading a preliminary version of this manuscript and Gr\'egoire Allaire for enlightening discussions on the topic. 








\bibliographystyle{plain} 
\bibliography{biblio_theorique_long}	

\begin{thebibliography}{10}

\bibitem{allaire2012shape}
G.~Allaire.
\newblock {\em Shape optimization by the homogenization method}, volume 146 of
  {\em Applied Mathematical Sciences}.
\newblock Springer, New York, 2002.

\bibitem{bensoussan2011asymptotic}
A.~Bensoussan, J.-L. Lions, and G.~Papanicolaou.
\newblock {\em Asymptotic analysis for periodic structures}, volume 374.
\newblock American Mathematical Soc., 2011.

\bibitem{livre_blanc_lebris}
X.~Blanc and C.~Le~Bris.
\newblock {\em Homogenization theory for multiscale problems: An introduction},
  volume~21 of {\em Modeling, Simulation and Applications}.
\newblock Springer, 2023.

\bibitem{caillerieDiffusion}
D.~Caillerie.
\newblock Homog\'en\'eisation des \'equations de la diffusion stationnaire dans
  les domaines cylindriques aplatis.
\newblock {\em RAIRO Analyse num{\'e}rique}, 15(4):295--319, 1981.

\bibitem{caillerieElasticite}
D.~Caillerie.
\newblock Thin elastic and periodic plates.
\newblock {\em Mathematical Methods in the Applied Sciences}, 6(1):159--191,
  1984.

\bibitem{ciarlet1988mathematical}
P.G. Ciarlet.
\newblock {\em Mathematical Elasticity: Volume I: three-dimensional
  elasticity}.
\newblock North-Holland, 1988.

\bibitem{ciarlet1988mathematical_vol2}
P.G. Ciarlet.
\newblock {\em Mathematical Elasticity: Volume II: theory of plates}.
\newblock North-Holland, 1997.

\bibitem{ciarlet1979justification}
P.G. Ciarlet and P.~Destuynder.
\newblock Justification of the two-dimensional linear plate model.
\newblock {\em J. Mec.}, 18(2):315--344, 1979.

\bibitem{CD}
D.~Cioranescu and P.~Donato.
\newblock {\em An introduction to homogenization}.
\newblock Oxford University Press, New York, 1999.

\bibitem{dauge}
M.~Dauge and I.~Gruais.
\newblock Asymptotics of arbitrary order for a thin elastic clamped plate, {I}.
  {O}ptimal error estimates.
\newblock {\em Asymptotic Analysis}, 13(2):167--197, 1996.

\bibitem{destuynder1981comparaison}
P.~Destuynder.
\newblock Comparaison entre les mod\`eles tridimensionnels et bidimensionnels
  de plaques en {\'e}lasticit{\'e}.
\newblock {\em RAIRO Analyse num{\'e}rique}, 15(4):331--369, 1981.

\bibitem{efendiev2009multiscale}
Y.~Efendiev and T.~Hou.
\newblock {\em {M}ultiscale {F}inite {E}lement {M}ethods: {T}heory and
  {A}pplications}, volume~4 of {\em Surveys and Tutorials in the Applied
  Mathematical Sciences}.
\newblock Springer New York, 2009.

\bibitem{MsFEM-ELLL}
V.~Ehrlacher, A.~Leb\'ee, F.~Legoll, and A.~Lesage.
\newblock Multiscale {F}inite {E}lement {M}ethods for elastic heterogeneous
  plates.
\newblock {\em in preparation}.

\bibitem{engquist2008asymptotic}
B.~Engquist and P.E. Souganidis.
\newblock Asymptotic and numerical homogenization.
\newblock {\em Acta Numerica}, 17:147--190, 2008.

\bibitem{griso2018homogenization}
G.~Griso and B.~Miara.
\newblock Homogenization of periodically heterogeneous thin beams.
\newblock {\em Chinese Annals of Mathematics, Series B}, 39(3):397--426, 2018.

\bibitem{gustafsson2003non}
B.~Gustafsson and J.~Mossino.
\newblock Non-periodic explicit homogenization and reduction of dimension: the
  linear case.
\newblock {\em IMA Journal of Applied Mathematics}, 68(3):269--298, 2003.

\bibitem{gustafsson2006compensated}
B.~Gustafsson and J.~Mossino.
\newblock Compensated compactness for homogenization and reduction of
  dimension: the case of elastic laminates.
\newblock {\em Asymptotic Analysis}, 47(1-2):139--169, 2006.

\bibitem{hornung2014derivation}
P.~Hornung, S.~Neukamm, and I.~Vel{\v{c}}i{\'c}.
\newblock Derivation of a homogenized nonlinear plate theory from {3D}
  elasticity.
\newblock {\em Calculus of Variations and Partial Differential Equations},
  51(3-4):677--699, 2014.

\bibitem{hornung2018stochastic}
P.~Hornung, M.~Pawelczyk, and I.~Vel{\v{c}}i{\'c}.
\newblock Stochastic homogenization of the bending plate model.
\newblock {\em Journal of Mathematical Analysis and Applications},
  458(2):1236--1273, 2018.

\bibitem{jikov}
V.V. Jikov, S.M. Kozlov, and O.A. Ole{\u\i}nik.
\newblock {\em Homogenization of differential operators and integral
  functionals}.
\newblock Springer-Verlag, Berlin, 1994.

\bibitem{kohn_vogelius86}
R.V. Kohn and M.~Vogelius.
\newblock A new model for thin plates with rapidly varying thickness. {III}:
  Comparison of different scalings.
\newblock {\em Quarterly of Applied Mathematics}, 44(1):35--48, 1986.

\bibitem{le2005systemes}
C.~Le~Bris.
\newblock {\em Syst{\`e}mes multi-{\'e}chelles: mod{\'e}lisation et
  simulation}, volume~47 of {\em Math\'ematiques et applications}.
\newblock Springer, 2005.

\bibitem{lebris_legoll_jcp}
C.~Le~Bris and F.~Legoll.
\newblock Examples of computational approaches for elliptic, possibly
  multiscale {PDE}s with random inputs.
\newblock {\em J. Comput. Physics}, 328:455--473, 2017.

\bibitem{adrien_phd}
A.~Lesage.
\newblock {\em Multi-scale approaches for the computation and the optimization
  of heterogeneous plates}.
\newblock PhD thesis, Universit\'e Paris-Est, 2020.
\newblock (available at {\tt https://theses.hal.science/tel-03587182}).

\bibitem{tomasz2000plates}
T.~Lewinski and J.J. Telega.
\newblock {\em Plates, laminates and shells: asymptotic analysis and
  homogenization}, volume~52 of {\em Advances in Mathematics for Applied
  Sciences}.
\newblock World Scientific, 2000.

\bibitem{marohnic2015homogenization}
M.~Marohni{\'c} and I.~Vel{\v{c}}i{\'c}.
\newblock Homogenization of bending theory for plates: the case of oscillations
  in the direction of thickness.
\newblock {\em Communications on Pure and Applied Analysis}, 14(6):2151--2168,
  2015.

\bibitem{marohnic2016non}
M.~Marohni{\'c} and I.~Vel{\v{c}}i{\'c}.
\newblock Non-periodic homogenization of bending-torsion theory for
  inextensible rods from {3D} elasticity.
\newblock {\em Annali di Matematica Pura ed Applicata (1923-)},
  195(4):1055--1079, 2016.

\bibitem{cherkaev}
F.~Murat and L.~Tartar.
\newblock H-convergence.
\newblock In A.~Cherkaev and R.~V. Kohn, editors, {\em Topics in the
  mathematical modelling of composite materials}, volume~31 of {\em Progress in
  Nonlinear Differential Equations and their Applications}, pages 21--44.
  Birkhauser, 1997.

\bibitem{velvcic2015derivation}
I.~Vel{\v{c}}i{\'c}.
\newblock On the derivation of homogenized bending plate model.
\newblock {\em Calculus of Variations and Partial Differential Equations},
  53(3-4):561--586, 2015.

\end{thebibliography}

\end{document}